\documentclass[12pt,oneside]{report}
\usepackage{amsmath,amssymb,alltt}
\usepackage[amsmath,thmmarks]{ntheorem}
\usepackage{setspace}
\usepackage{color}
\usepackage{enumerate}

\usepackage[a4paper, left=3cm, right=2cm, top=3cm, bottom=3cm]{geometry}

\newcommand{\aut}{\operatorname{Aut}}

\newcommand{\out}{\operatorname{Out}}
\newcommand{\Irr}{\operatorname{Irr}}
\newcommand{\sym}{\operatorname{Sym}}
\newcommand{\alt}{\operatorname{Alt}}
\newcommand{\dih}{\operatorname{Dih}}
\newcommand{\sdih}{\operatorname{SDih}}
\newcommand{\syl}{\operatorname{Syl}}

\newcommand{\SL}{\mathrm{SL} }
\newcommand{\PSL}{\mathrm{PSL}}
\newcommand{\PGL}{\mathrm{PGL}}
\newcommand{\PSp}{\mathrm{PSp}}
\newcommand{\Sp}{\mathrm{Sp}}

\newcommand{\B}{\mathrm{B}}

\newcommand{\GL}{\mathrm{GL}}
\newcommand{\GF}{\mathrm{GF}}

\newcommand{\HN}{\mathrm{HN}}
\newcommand{\HS}{\mathrm{HS}}

\newcommand{\C}{\mathrm{C}}

\newcommand{\GO}{\mathrm{GO}}
\newcommand{\Q}{\mathrm{Q}}
\newcommand{\wt}[1]{\widetilde{#1}}

\newcommand{\CC}{{\mathcal{C}}}

\def \<{\langle }
\def \>{\rangle }
\def \bs {\backslash }
\renewcommand{\bar}{\overline}
\def \inv {^{-1}}
\def \GF {\mathrm{GF} }
\def \PSp {\mathrm{PSp} }
\def \GL {\mathrm{GL} }
\def \sym {\mathrm{Sym} }
\def \alt {\mathrm{Alt} }
\def \dih {\mathrm{Dih} }
\def \syl {\mathrm{Syl} }
\def \bs {\backslash }
\renewcommand{\bar}{\overline}
\renewcommand{\leq}{\leqslant}
\renewcommand{\geq}{\geqslant}

\def \End {\mathrm{End} }

\newcommand{\e}{{\epsilon }}

\def \b{\beta }
\def \a{\alpha }
\def \g{\gamma }
\def \d{\delta }
\def \l{\lambda }

\def\m{\mu }

\input cyracc.def

\theoremseparator{.}

\newtheorem*{thmA}{Theorem A}
\newtheorem*{thmB}{Theorem B}
\newtheorem*{thmC}{Theorem C}

\newtheorem*{Ithm}{Theorem}
\newtheorem*{Ihyp}{Hypothesis}

\newtheorem{lemma}{Lemma}[chapter]
\newtheorem{thm}[lemma]{Theorem}
\newtheorem{cor}[lemma]{Corollary}
\newtheorem{hyp}[lemma]{Hypothesis}
\newtheorem{notation1}[lemma]{Notation}

\theoremstyle{plain} \theorembodyfont{\upshape}

\newtheorem*{hypA}{Hypothesis A}
\newtheorem*{hypB}{Hypothesis B}
\newtheorem*{hypC}{Hypothesis C}

\newtheorem{defi}[lemma]{Definition}

\theoremstyle{nonumberplain} \theoremheaderfont{\normalfont\itshape}
\theorembodyfont{\normalfont} \theoremseparator{.}
\theoremsymbol{\ensuremath{\square}}
\newtheorem{proof}{Proof}

\title{Odd Characterizations of Almost Simple Groups}

\author{Sarah Astill}

\begin{document}

\begin{titlepage}
  \vspace*{3cm}
    \begin{center}
          {\Huge{\sc{Odd Characterizations of Almost Simple Groups}}}\\
    \vspace{0.2in}
          {\Large{\sc{3-Local and Character Theoretic Methods}}}\\
    \vspace{0.5in}
    by\\
    \vspace{0.5in}
   {\Large{\sc{Sarah Astill}}}\\
    \vspace{5cm}
    \begin{singlespace}
    A thesis submitted to\\
    The University of Birmingham\\
    for the degree of\\
    {\sc{Doctor of Philosophy}}\\
    \end{singlespace}
\end{center}
\vfill
    \begin{singlespace}
  \hfill
  \parbox{2.8in}{School of Mathematics\\
          The University of Birmingham\\
          \date}
\end{singlespace}
\end{titlepage}
\normalsize
\chapter*{Abstract}\thispagestyle{empty}
In this PhD thesis we discuss methods of recognizing finite groups by the structure of  normalizers
of certain $3$-subgroups. We explain a method for characterizing groups using character theoretic
and block theoretic methods and we use these methods to characterize $\alt(9)$. Furthermore, we
describe a particular hypothesis related to the $3$-local structure of finite groups of local
characteristic $3$ and characterize two almost simple proper extensions of
$\Omega^+_8(2)$ as examples of the local approach to group recognition. We also give a $3$-local characterization of the
sporadic simple group $\HN$ using local methods whilst applying a character theoretic result.

\pagestyle{empty} \tableofcontents \thispagestyle{empty}
\addtocontents{toc}{\protect\thispagestyle{empty}} \clearpage \pagestyle{plain}
\pagenumbering{arabic} \setcounter{page}{1}

\chapter*{Introduction}\addcontentsline{toc}{chapter}{\protect\numberline{}Introduction}

In \cite{Higman-plocalconditions}, Higman gives the following weak analogue of the  Brauer--Fowler
theorem for the prime three.

\begin{Ithm}
There are a finite number of finite simple groups $G$ with more than one conjugacy class of
elements of order three such that for some integer $n$, $|C_G(x)|\leq n$ for every element of order
three, $x$ in $G$.
\end{Ithm}
The Brauer--Fowler Theorem itself says that there are a finite number of finite simple  groups with
a given centralizer of an involution. This, together with the Feit--Thompson Theorem suggested that
finite simple groups could be classified by the structure of involution centralizers. Furthermore,
much as the proof of the Brauer--Fowler Theorem relies on the fact that two involutions generate a
dihedral group, Higman's analogue relies on a well known observation that a group generated by
elements  $a$ and $b$ of order three such that $ab$ also has order three has an abelian normal
subgroup of index three. In \cite{Hartley} Hartley and Kuzucuo{\v{g}}lu proved using the
classification of finite simple groups (see \cite{GLS1})  that for any two natural numbers, $n$ and
$k$, there are a finite number of finite simple groups $G$ containing an element $x$ of order $n$
such that $|C_G(x)|\leq k$. However the elementary nature of the proof of Higman's statement
reminds us that elements of order three have a special role in finite group theory and also
provides hope that some simple groups can be recognized from the structure of their
$3$-centralizers independently of the classification of finite simple groups (see Section XVI in
\cite{GorensteinClassification} for further discussion of this). In fact, during the 1960's and 1970's, Higman and some of his
students worked towards odd characterizations of some simple groups using character theoretic
methods (see for example  \cite{Higman} and \cite{PrincePSp43}). The methods work particularly well
in characteristic three. Note that in a finite group $G$, the subgroups $N_G(P)$,  where $P$ is a
non-trivial $p$-subgroup of $G$ ($p$ a prime), are called the $p$-local subgroups of $G$. There are
many recent examples of so-called $3$-local characterizations of simple groups. See for example
\cite{AstillM12}, \cite{ParkerKorchaginaRowleyCo3}, \cite{SalarianCo1}. In particular, in the
qualifying thesis which preceded this thesis \cite{AstillMPhilThesis}, the following theorems were
proven.

\begin{Ithm}
Let $G$ be a finite group with subgroups $A \cong B$  such that
$C:=A \cap B$ contains a Sylow $3$-subgroup of both $A$ and $B$ and
such that no non-trivial normal subgroup of $C$ is normal in both
$A$ and $B$. Suppose further that
\begin{enumerate}[$(i)$]
    \item $G=\<A,B\>$;
    \item $A/O_3(A)\cong B/O_3(B)\cong \GL_2(3)$;
    \item $O_3(A)$ and $O_3(B)$ are natural modules with respect to
    the actions of $A/O_3(A)$ and $B/O_3(B)$ respectively; and
    \item for $S\in \mathrm{Syl}_3(C)$,
    $N_G(\mathcal{Z}(S))=N_A(\mathcal{Z}(S))=N_B(\mathcal{Z}(S))$.
\end{enumerate}
Then $G \cong \mathrm{M}_{12}$ or
$\mathrm{PSL}_3(3)$.
\end{Ithm}

\begin{Ithm} Let
$\mathcal{G}$ be a finite group with non-conjugate subgroups $A_1$
and $A_2$ such that $A_{12}:=A_1 \cap A_2$ contains a Sylow
$3$-subgroup of both $A_1$ and $A_2$, and such that no non-trivial
normal subgroup of $A_{12}$ is normal in both $A_1$ and $A_2$.
Suppose further that, for $i=1,2$,
\begin{enumerate}[$(i)$]
    \item $|O_3(A_i)|=3^5$;
    \item $A_i/O_3(A_i) \cong
    \mathrm{GL}_2(3)$;
    \item $O_3(O_3(A_i))/\mathcal{Z}(A_i)$ and
    $\mathcal{Z}(O_3(A_i))/O_3(A_i)'$ are natural modules with respect to the action of
$A_i/O_3(A_i)$; and
    \item $N_\mathcal{G}(O_3(A_i)')=A_i$.
\end{enumerate}
Then
$\mathcal{G}\cong G_2(3)$.
\end{Ithm}
Both theorems recognize a group which is \textit{rank 2} in the sense  that there are two subgroups
properly containing a given Sylow $3$-subgroup. The characterizations rely on two character
theoretic results by Smith and Tyrer and by Feit and Thompson (see Theorems \ref{Feit-Thompson} and
\ref{Smith-Tyrer} in Section \ref{section-recognition results} of this thesis). In $3$-local characterizations, we often need to determine the structure of centralizers of elements of order three. Once we have such information we must use it to determine
the structure of an involution centralizer. Character theoretic results allow us to restrict
the size and structure of such subgroups by using information related to the normalizer of a Sylow
$3$-subgroup. The Smith--Tyrer Theorem can be useful in determining the structure of a group if the
target group is $p$-soluble of length one and such structures appear surprisingly often in
\cite{AstillMPhilThesis}. The theorem is not used to such a large extent in this thesis (for
example we must deal with non-soluble centralizers). In fact, it has been observed that,
to $3$-locally recognize certain groups it is often necessary to develop character theoretic
arguments related to the specific $3$-local subgroups one encounters. To be more specific, in the
final chapter of this thesis, we recognize the Harada--Norton sporadic simple group, $\HN$  (see Chapter
\ref{Chapter-HN}). The group $\HN$ has a subgroup isomorphic $3 \times \alt(9)$. However the
information we acquire through $3$-local analysis only allows us to see a small part of this
subgroup. Thus in Chapter \ref{chapter-Alt9} of this thesis we present a proof of the following
theorem.
\begin{thmA}
Let $G$ be a finite group with $J\leq G$ such that $J$ is elementary abelian of order $27$. Suppose
$H=N_G(J)$ is isomorphic to a $3$-local subgroup of $\alt(9)$ of  shape $3^3.\sym(4)$. If
$O_{3'}(C_G(x))=1$ for every element of order three $x$ in $H$, then $G=H$ or $G \cong
\mathrm{Alt}(9)$.
\end{thmA}
The proof of this result is highly character theoretic and deals with a fixed isomorphism type  of
local subgroup and as such is tailored towards the situation arising in the $\HN$ recognition
result. However the method is most likely applicable to many situations involving $3$-local recognition of small groups. The character theoretic proof uses Suzuki's theory of special classes as described in
Chapter \ref{chapter-Alt9}. It also develops some character theoretic methods which were possibly
used by Higman and students in the 1970's. These methods involve blocks of characters and detailed
calculations. In fact we use a computer algebra package for some of these calculations and the code
is available on request. The proof also uses some local
methods to finally recognize the simple group $\alt(9)$ together with a theorem of Aschbacher.

In Chapter \ref{chaper general hypothesis} we consider groups satisfying a particular hypothesis.
This hypothesis is related to a major programme of research led by Meierfrankenfeld, Stellmacher
and Stroth. The programme aims to understand  \textit{groups of local characteristic $p$} (see
\cite{MSS-overview}). Given a group $X$ and a prime $p$, $X$ is said to be of
\textit{characteristic $p$} if $C_X(O_p(X))\leq O_p(X)$. Given a group $G$ and a prime $p$ dividing $|G|$, $G$ is of \textit{local characteristic $p$} if every
$p$-local subgroup is of characteristic $p$ and $G$  is of \textit{parabolic characteristic $p$} if every $p$-local subgroup which contains a Sylow $p$-subgroup is of characteristic $p$. A group $G$ is \textit{almost simple} if a subgroup $H \trianglelefteq G$ is non-abelian and simple and $G$ is isomorphic to a subgroup of $\aut(H)$. Almost simple groups of Lie type defined over fields of characteristic $p$ have local characteristic $p$ and several of the sporadic simple groups have a prime divisor
$p$ of the group order for which they are of either local or parabolic characteristic $p$ and therefore in some sense
mimic the local behavior of groups of Lie type in characteristic $p$. In this thesis we do not explicitly consider groups of local characteristic $p$ however we consider the following hypothesis which has application towards the understanding of such groups.
\begin{Ihyp}
Let $G$ be a finite group and let $Z$ be the centre of a Sylow $3$-subgroup of $G$ with
$Q:=O_3(C_G(Z))$.  Suppose that
\begin{enumerate}[$(i)$]
\item $Q\cong 3_+^{1+4}$;
\item $C_G(Q)\leq Q$; and
\item for some $x \in G\bs N_G(Z)$, $[Z,Z^x]=1$.
\end{enumerate}
\end{Ihyp}
The third condition is to say that $Z$ is not weakly closed in $C_G(Z)$ with respect to $G$.  Five
sporadic simple groups satisfy this hypothesis as well as several simple and almost simple groups
of Lie type in defining characteristics 2 and 3. Thus the configuration is exceptional as it admits
sporadic groups and simple groups of local characteristic 2. Full analysis of this hypothesis will form
part of a future project however we begin the analysis in this thesis. In particular, we replace
condition $(iii)$ with the following stronger condition.
\begin{list}{$(iii)$}{}
\item  $Z \neq Z^x \leq Q$ for some $x \in G$.
\end{list}

In Chapter \ref{chaper general hypothesis} we examine groups satisfying our hypothesis and produce a
list of local properties which such groups have. These properties are then used in Chapter
\ref{Chapter-O8Plus2} where we consider groups with an additional hypothesis as we prove the
following theorem.

\begin{thmB}
Let $G$ be a finite group and let $Z$ be the centre of a Sylow $3$-subgroup of $G$ with $Q:=O_3(C_G(Z))$.
Suppose that
\begin{enumerate}[$(i)$]
\item $Q\cong 3_+^{1+4}$;
\item $C_G(Q)\leq Q$; and
\item $Z \neq Z^x \leq Q$ for some $x \in G$.
\end{enumerate}
Furthermore assume that $C_G(Z)/Q \cong \SL_2(3)$ or $C_G(Z)/Q \cong
\SL_2(3)\times 2$ and the action of $O^2(C_G(Z)/Q)\cong \SL_2(3)$ on $Q/Z$ has one non-central
chief factor. Then $G\cong \Omega_8^+(2).3$ or $G \cong
\Omega_8^+(2).\sym(3)$.
\end{thmB}
This result has a direct application in a further recognition result in preparation  by Parker and
Stroth which aims to recognize the exceptional group of Lie type $^2E_6(2)$ (and its almost simple
extensions) as groups of parabolic characteristic $3$. The almost simple groups each have a section to
which Theorem B applies. The proof of Theorem B requires us to recognize firstly that a group
satisfying the hypothesis has a proper normal subgroup and secondly that a normal subgroup is
isomorphic to the simple orthogonal group $\Omega_8^+(2)$. After gathering both $3$-local and $2$-local information about groups satisfying the hypothesis of Theorem B we are able to use transfer results to recognize abelian quotients. We finally make use of a theorem due to Smith \cite{SmithOrthogonal} to recognize the simple subgroup.

The Harada--Norton sporadic simple group also satisfies the hypothesis we describe in Chapter
\ref{chaper general hypothesis} and in the final chapter of this thesis we give a  proof of the following result.

\begin{thmC}
Let $G$ be a finite group and let $Z$ be the centre of a Sylow $3$-subgroup of $G$ with $Q:=O_3(C_G(Z))$.
Suppose that
\begin{enumerate}[$(i)$]
\item $Q\cong 3_+^{1+4}$;
\item $C_G(Q)\leq Q$;
\item $Z \neq Z^x \leq Q$ for some $x \in G$; and
\item $C_G(Z)/Q \cong 2^{.}\alt(5)$.
\end{enumerate}
Then $G$ is isomorphic to the sporadic simple group $\HN$.
\end{thmC}
We apply the general theory from Chapter \ref{chaper general hypothesis} to understand  the
$3$-local structure of groups satisfying the hypothesis of Theorem C. We also apply Theorem A to
recognize a $3$-centralizer of shape $3\times \alt(9)$. However the majority of the proof involves
$2$-local analysis. This is because in order to eventually recognize the simple group we apply a theorem of Segev. Segev's recognition result requires us to determine the
structure of two conjugacy classes of involution centralizer. Both involution
centralizers are non-soluble which, as described previously, can make identification more difficult. Moreover, both involution centralizers have small Sylow $3$-subgroups which further complicates our determination of the group structure.

We conclude this introduction with some discussion of transfer and the scope for further work. In
the proof of Theorem B we are forced to work ``at the top of the group" when we prove that our
group has proper derived subgroup. This requires results which use the
transfer homomorphism which is an essential tool when working with groups which are almost simple proper extensions. The recognition of $\HN$ could also be extended in this way to recognize the
almost simple group $\aut(\HN)\sim \HN.2$. However, as we see with the characterization of
$\Omega_8^+(2).\sym(3)$, proving the existence of an index two subgroup is difficult. The
transfer results can only be used once a Sylow $2$-subgroup has been found and after we have
gathered a great deal of information about fusion of elements of order two. Such things are not
observed until the later stages of the proof. However, future work will include extending Theorem C to the almost simple case and perhaps such
work will lead to a faster way to recognize proper $2$-quotients.
We mention also that perhaps the first case
to consider in relation to our general hypothesis is the simple group $\PSL_4(3)$. This has also
been characterized (see \cite{AstillPSL43}) with the following theorem.
\begin{Ithm}
Let $G$ be a finite group and let $Z$ be the centre of a Sylow $3$-subgroup of $G$. Suppose that
\begin{enumerate}[$(i)$]
\item $Q:=O_3(C_G(Z))\cong 3_+^{1+4}$;
\item $C_G(Q)\leq Q$;
\item $Z \neq Z^x \leq Q$ for some $x \in G$; and
\item $C_G(Z)/Q \cong \SL_2(3)$  and the action of
$C_G(Z)/Q$ on $Q/Z$ has two non-central chief factors.
\end{enumerate}
Then either $G \cong \PSL_4(3)$ or $G$ is isomorphic to a maximal parabolic subgroup of $\PSL_4(3)$
of shape $3^3.\SL_3(3)$.
\end{Ithm}
This result will also be extended to recognize the four almost simple extensions of $\PSL_4(3)$.
Furthermore, future work will
include characterizing all groups which satisfy our hypothesis with $Z \neq Z^x \leq Q$. This involves recognizing
$F_4(2)$ and $\aut(F_4(2))$ and of course proving that no further examples exist. Finally, we
remark that, as described previously,  the hypothesis we consider in this thesis can be weakened to consider groups in which $Z$ is not weakly closed in $C_G(Z)$ with respect to $G$. This, of course, is a much wider project and involves recognizing many more almost simple groups. However the work in this thesis makes a contribution to such an investigation and describes and develops methods which will certainly be applicable. In particular, it is likely that the character theoretic results in Chapter \ref{chapter-Alt9} can be extended and could prove to be a vital tool in such a project. A first extension of the character theoretic methods could, for example, be to characterize groups $G$ with a Sylow $3$-subgroup $S\cong 3 \wr 3$ such that $C_G(\mathcal{Z}(S))$ has characteristic $3$.

Finally, all groups in this thesis are finite.  We note that $\sym(n)$ and $\alt(n)$ denote the symmetric and alternating groups of degree $n$ and $\dih(n)$ denotes the dihedral group of order $n$ and $Q_n$ the quaternion group of order $n$. Notation for classical groups follows \cite{Aschbacher}. All other groups and notation for group extensions follows the {\sc Atlas} \cite{atlas} conventions. In particular, if $S$ is a group, $p$ is a prime and $n \in \mathbb{N}$ then $S\cong p^n$ means that $S$ is an
elementary abelian subgroup of that order. If a group $G$ has a normal subgroup $N$ of isomorphism
type  $A$ with $G/N$ of isomorphism type $B$ then we say that $G$ has shape $A.B$ or $G \sim A.B$.
If furthermore the extension is split, this is to say $G$ has a subgroup $M$ isomorphic to $B$ such
that $G=NM$, then we use the notation $G \sim A:B$ unless $[M,N]=1$ in which case $G\cong A \times
B$. If the extension is non-split then we denote this by $G \sim A^.B$. If $N,M \vartriangleleft G$ with $G=MN$, $[M,N]=1$ and $1 \neq \mathcal{Z}(M)\leq \mathcal{Z}(N)$ then we write $G=M \ast N$ and say $G$ is a central product of $M$ and $N$. Furthermore, if $G$ is a
group and $x \in G$ then $x^G$ represents the conjugacy class of $G$ containing $x$ (so $x^G=\{x^g
|g \in G\}$). If $1 \neq x \in G$ is in the centre of a Sylow $p$-subgroup then we say that $x$ is $p$-central in $G$. If a group $A$ acts on a group $G$ and $a \in A$ and $g \in G$ then $[g,a]=g\inv g^a$.
Further group theory notation and terminology is standard as in \cite{Aschbacher} and
\cite{stellmacher} except that $\mathcal{Z}(G)$ denotes the centre of a group $G$. The character theoretic notation used in Chapter \ref{chapter-Alt9} follows
\cite{Isaacs}.


\chapter{Preliminary Results}\label{chapter prelims}
We begin this thesis with some preliminary results whose proofs can mostly be found in \cite{Aschbacher},  \cite{Gorenstein} and \cite{stellmacher}.

\section{General Group Theoretic Results}

\begin{lemma}[Frattini Argument]\cite[3.2.7, p66]{stellmacher}\label{frattini}
Suppose that $p$ is a prime and that $H\trianglelefteq G$ with $P\in \syl_{p}(H)$. Then $G=N_{G}(P)H$.
\end{lemma}

\begin{lemma}[Dedekind's Modular Law]\cite[1.14]{Aschbacher}\label{dedekind}
Suppose that $A$, $B$ and $C$ are subgroups of a group $G$ such that $B\leq C$. Then \[AB\cap
C=(A\cap C)B.\]
\end{lemma}
Note that a commutator $[a,b]$ is defined to equal $a\inv a^b$. Also commutators are left defined so $[a,b,c]$ means $[[a,b],c]$.
\begin{lemma}[Three Subgroup Lemma]\cite[1.5.6, p26]{stellmacher}
Let $G$ be a group and let $A,B,C \leq G$. If $[A,B,C]=[B,C,A]=1$ then $[C,A,B]=1$.
\end{lemma}

\begin{defi}
Let $p$ be a prime and let $E$ be a non-abelian $p$-group. If $E'=\mathcal{Z}(E)=\Phi(E)$ is cyclic of order $p$ then $E$
is said to be extraspecial.
\end{defi}
\begin{lemma}\cite[Thm 20.5]{doerk-hawkes}\label{prelim-exraspecial}
Let $E$ be an extraspecial $p$-group. Exactly one of the following holds.
\begin{enumerate}[$(i)$]
  \item $p=2$ and $E$ is a central product of $n$ copies of
  $\mathrm{Dih}(8)$.
  \item $p=2$ and $E$ is a central product of $n-1$ copies of
  $\mathrm{Dih}(8)$ and one copy of $Q_8$.
  \item $p\neq 2$ and $E$ has exponent $p$.
  \item $p\neq 2$ and $E$ has exponent $p^2$.
\end{enumerate}
We denote such groups as $E \cong 2_+^{1+2n}$, $2_-^{1+2n}$, $p_+^{1+2n}$ and $p_-^{1+2n}$
respectively.
\end{lemma}
It is well known that $\mathrm{Dih}(8)*\mathrm{Dih}(8)\cong Q_8*Q_8$ so the description of
extraspecial $2$-groups given here is not unique.
\begin{thm}\label{extraspecial outer automorphisms}
Suppose that $p$ is a prime and that $E$ is an extraspecial $p$-group of order $p^{1+2n}$.
\begin{enumerate}[$(i)$]
  \item If $p=2$ and $E \cong  2_+^{1+2n}$ then
  $\mathrm{Out}(E)\cong \mathrm{O}^+_{2n}(2)$.
  \item If $p=2$ and $E \cong  2_-^{1+2n}$ then
  $\mathrm{Out}(E)\cong \mathrm{O}^-_{2n}(2)$.
  \item If $p$ is odd and $E \cong  p_+^{1+2n}$ then
  $\mathrm{Out}(E)\cong \Sp_{2n}(p).C_{p-1}$.
  \item If $p$ is odd and $E \cong  p_-^{1+2n}$ then
  $\mathrm{Out}(E)\cong p_+^{1+2(n-1)}.\Sp_{2n-2}(p).C_{p-1}$.
\end{enumerate}
\end{thm}
\begin{proof}
See  \cite[20.8, 20.9]{doerk-hawkes}.
\end{proof}
The following result is well known. A proof can be found for example in \cite[1.18]{AstillMPhilThesis}.
\begin{lemma}\label{exactly 2 q-8's in extraspecial group}
The extraspecial $2$-group $2_+^{1+4}$ contains exactly two subgroups isomorphic to $Q_8$ and they
commute and $2_+^{1+4}$ contains exactly 12 elements of order four.
\end{lemma}

\begin{defi}
Let $A$ be a group acting on a group $G$. The action of $A$ on $G$
is coprime if $|A|$ and $|G|$ are coprime.
\end{defi}

\begin{thm}[Coprime Action]\label{coprime action}
Suppose $A$ is a group acting on the group $G$ and suppose the
action of $A$ on $G$ is coprime. The following hold.
\begin{enumerate}[$(i)$]
  \item  $G=[G,A]C_G(A)$ and if $G$ is abelian, then $G=[G,A]\times
  C_G(A)$.
  \item $[G,A]=[G,A,A]$.
  \item $C_{G/N}(A)=C_G(A) N/N$ for any $A$-invariant
  $N\trianglelefteq G$.
  \item If $A$ is an elementary abelian $p$-group ($p$ is a prime) of order at least $p^2$ then $G=\<C_G(a) \mid a \in
  A^\#\>=\<C_G(A_1)\mid [A:A_1]=p\>$.
\end{enumerate}
\end{thm}
\begin{proof}
For $(i)$ and $(ii)$ see \cite[8.2.7, p187]{stellmacher}. For $(iii)$
see \cite[8.2.2, p184]{stellmacher}. For $(iv)$ see \cite[8.3.4,
p193]{stellmacher}.
\end{proof}

\begin{lemma}\label{prelims-extraspecial and a coprime aut}
Let $p$ be a prime and let $Q$ be an extraspecial $p$-group. Suppose that $\a$ is a non-trivial automorphism of $Q$ of order coprime to $p$ with $[\mathcal{Z}(Q),\a]=1$. Then either
\begin{enumerate}[$(i)$]
\item $Q=[Q,\a]$ and $C_Q(\a)=\mathcal{Z}(Q)$; or
\item $C_Q(\a)$ and $[Q,\a]$ are both extraspecial with $Q=C_{Q}(\a) [Q,\a]$ and $C_{Q}(\a)\cap  [Q,\a]=\mathcal{Z}(Q)$.
\end{enumerate}
\end{lemma}
\begin{proof} By coprime action on an abelian group, we have $Q/\mathcal{Z}(Q)=[Q/\mathcal{Z}(Q),\a] \times C_{Q/\mathcal{Z}(Q)}(\a)$. Hence if $Q=[Q,\a]$ then $C_Q(\a)=\mathcal{Z}(Q)$. So suppose that $Q\neq [Q,\a]$.
Since $\a$ is non-trivial,  $Q\neq C_Q(\a)$ and so we have that $1< [Q,\a]<Q$. Notice that $[C_Q(\a),Q,\a]\leq [\mathcal{Z}(Q),\a]=1$ and $[C_Q(\a),\a,Q]=[1,Q]=1$ so by the three subgroup lemma, $[Q,\a,C_Q(\a)]=1$. Consider $\mathcal{Z}(C_Q(\a))$. This commutes with $\<C_Q(\a),[Q,\a]\>=Q$ and so $1 \neq \mathcal{Z}(C_Q(\a))\leq \mathcal{Z}(Q)$ and since $\mathcal{Z}(Q)$ is cyclic of order $p$,  $\mathcal{Z}(C_Q(\a))=\mathcal{Z}(Q)$. Similarly, $\mathcal{Z}([Q,\a])=\mathcal{Z}(Q)$.  If $C_Q(\a)=\mathcal{Z}(Q)$ then, because $Q=[Q,\a]C_Q(\a)$, it follows that $Q=[Q,\a]$ which is not the case. Similarly if $[Q,\a]=\mathcal{Z}(Q)$ then, $Q=C_Q(\a)$, which contradicts that $\a\neq 1$. Hence $C_Q(\a)$ and $[Q,\a]$ are both extraspecial and it follows immediately from coprime action that $Q=C_{Q}(\a) [Q,\a]$ and $C_{Q}(\a)\cap  [Q,\a]=\mathcal{Z}(Q)$.
\end{proof}

\begin{thm}[Thompson]\cite[2.1, p337]{Gorenstein}\label{thompson-nilpotent}
Let $Q$ be a group and $\a$ an automorphism of $Q$ of prime order such that $C_Q(\a)=1$. Then $Q$ is nilpotent.
\end{thm}

\begin{thm}[Burnside]\cite[5.1.4, p174]{Gorenstein}\label{Burnside-p'-automorphism}
Let $p$ be a prime and let $P$ be a $p$-group and $\a$ an automorphism of $P$ of order prime to $p$. If $\a$ centralizes
$P/\Phi(P)$ then $\a=1$.
\end{thm}

\begin{thm}[Gasch\"{u}tz]\cite[p63]{GLS2}\label{Gaschutz}
Let $p$ be a prime and let $A$ be an abelian normal $p$-subgroup of a group $G$. Suppose that $S \in \syl_p(G)$. Then there is a
complement to $A$ in $G$ if and only if there is a complement to $A$ in $S$.
\end{thm}

\begin{defi}
Let $G$ be  group and let $p$ be a prime. Set $n$ to be the order of a largest abelian $p$-subgroup of $G$ and set $\mathcal{A}:=\{A \leq G||A|=n\}$. Then the Thompson subgroup of $G$ is $J(G):=\<A\mid A \in \mathcal{A}\>$.
\end{defi}
See \cite{stellmacher}, for example, for properties of the Thompson subgroup. We use the following property many times in this thesis.
\begin{lemma}\label{conjugation in thompson subgroup}
Let $G$ be a group, $p$ be a prime and $S\in \syl_p(G)$. Suppose $J(S)$ is abelian and suppose $a,b
\in J(S)$ are conjugate in $G$. Then $a$ and $b$ are conjugate in $N_G(J(S))$.
\end{lemma}
\begin{proof}
Suppose $a^g=b$ for some $g \in G$. Notice first that it follows immediately from the definition of
the Thompson subgroup that $J(S)^g=J(S^g)$. Now $J(S),J(S^g)\leq C_G(b)$. Let $P,Q \in
\syl_p(C_G(b))$ such that $J(S)\leq P$ and $J(S^g)\leq Q$. Again, by the definition of the Thompson
subgroup, it is clear that $J(S)\leq P$ implies $J(S)=J(P)$ and similarly $J(S^g)= J(Q)$. By
Sylow's Theorem, there exists $x \in C_G(b)$ such that $Q^x=P$ and so $J(S)=J(P)=
J(Q)^x=J(S)^{gx}$. Thus $gx \in N_G(J(S))$ and $a^{gx}=b^x=b$ as required.
\end{proof}

\begin{lemma}\cite[7.1.5, p167]{stellmacher}\label{Burnside conjugation lemma}
Let $G$ be a group, $p$ be a prime and $S$ be a Sylow $p$-subgroup of $G$.  If $A_1$ and $A_2$ are
normal subsets of $S$ which are conjugate in $G$ then they are conjugate in $N_G(S)$.
\end{lemma}

In many of the calculations in the proof of Theorem B and Theorem C we  often switch between a group with nilpotence class two
and its abelian quotient modulo the centre. The following lemma allows us to adjust between the
two groups.
\begin{lemma}\label{Prelims-p centralizers on class 2 groups}                                                                   
Let $P$ be a class two group and $\a$ an automorphism of $P$ that centralizes $\mathcal{Z}(P)$. Then
$[C_{P/\mathcal{Z}(P)}(\a):C_P(\a)/\mathcal{Z}(P)]$ divides  $|\mathcal{Z}(P)|$.
\end{lemma}
\begin{proof}
Define a map \[\begin{array}{rlrcl}
    \phi & : & C_{P/\mathcal{Z}(P)}(\a) & \longrightarrow &\mathcal{Z}(P) \\
     \; &\;  & \mathcal{Z}(P)x & \longmapsto & [x,\a].\\
\end{array}\] Then $\phi$ is a well defined map since $\a$ commutes with $\mathcal{Z}(P)$. Moreover $\phi$ is
a homomorphism and the kernel is $C_P(\a)/\mathcal{Z}(P)$.
\end{proof}

\begin{lemma}\label{Prelims 2^8 3^2 Dih(8)}
Let $X$ be a group with an elementary abelian subgroup $E\vartriangleleft X$ of order $2^{2n}$ such
that $C_X(E)=E$. Let $S \in \syl_2(X)$ and suppose that whenever $E<R\trianglelefteq S$ with $R/E$
elementary abelian and $|R/E|=2^a$ we have  $|C_E(R)|\leq 2^{2n-a-1}$. Then $E$ is characteristic in
$S$.
\end{lemma}
\begin{proof}
First observe that since $C_X(E)=E$, $X/E$ is a group of outer automorphisms of $E$. Let $\a$ be an
automorphism of $S$ such that $E^\a\neq E$. Then $R:=EE^\a\trianglelefteq S$. Since $E^\a$ is elementary
abelian, we have that $E^\a/(E \cap E^\a)\cong EE^\a/E=R/E$ is elementary abelian and $E \cap E^\a$ is central in
$R$. If $|R/E|=2^a$ then $|E \cap E^\a|=2^{2n-a}$ so $|C_E(R)|\geq |E \cap
E^\a|=2^{2n-a}$ which is a contradiction.
\end{proof}

\section{Results Using Transfer}

The transfer homomorphism is a useful tool in group theory and is used to identify proper normal subgroups. See Chapter 7 in \cite{stellmacher} for  a
definition of the transfer homomorphism and related results. We state four transfer related results in this section and apply one of these in Lemma \ref{Prelims-4*4*4 transfer} to prove a result which is required in Chapter \ref{Chapter-HN}.

A group $G$ is said to have a normal $p$-complement ($p$ a prime) if $O_{p'}(G)=O^p(G)$.  This is to
say that $G$ has a normal subgroup $N$ of order prime to $p$ such that $G=NP$ for $P \in
\syl_p(G)$.
\begin{thm}[Burnside's Normal $p$-complement Theorem]\cite[7.2.1, p169]{stellmacher}\label{Burnside-normal p complement}
Let $G$ be a group and let $p$ be a prime. Suppose that $P\in \syl_p(G)$ such that $C_G(P)=N_G(P)$. Then $G$ has a normal $p$-complement.
\end{thm}

\begin{lemma}[Thompson's Transfer Lemma]\cite[12.1.1, p338]{stellmacher}\label{ThompsonTransfer}
Let $G$ be a group and $S\in \syl_2(G)$. Suppose that there exists a maximal subgroup $U< S$ and an
involution $t\in S$ such that $t^G \cap U = \emptyset$. Then $t$ is not contained in $O^2(G)$.
\end{lemma}

\begin{thm}[Gr\"{u}n]\cite[7.4.2]{Gorenstein}\label{GrunsThm}
Let $G$ be a group, $p$ a prime and $S\in \syl_p(G)$. Then $S \cap G'=\<S \cap N_G(S)',S \cap P'|P
\in \syl_p(G)\>$.
\end{thm}

\begin{thm}[Extremal Transfer]\cite[15.15, p92]{GLS2}\label{extremal transfer}
Let $G$ be a group and let $p$ be a prime with $P \in \syl_p(G)$. Suppose $Q\lhd P$ and $[P:Q]=p$ and $x \in P\bs Q$. If
$x^G \cap P\subset Q \cup Qx$ then either $x \notin O^p(G)$ or there exists $g \in G$ such that
$x^g \in Q$ and $C_P(x^g)\in \syl_p(C_G(x^g))$.
\end{thm} Note that $x^G \cap P\subset Q \cup Qx$ holds automatically if $p=2$.

The following lemma is an application of Lemma
\ref{extremal transfer} that will be needed in Chapter \ref{Chapter-HN}. Note that given a $p$-group $S$, we set $\Omega(S)=\<x \mid x^p=1\>$.
\begin{lemma}\label{Prelims-4*4*4 transfer}
Let $G$ be a group and $4\times 4 \times 4 \cong A \leq G$ with $C_G(A)=A$. Set $X:=N_G(A)$ and assume that $X\sim
4^3: (2 \times \GL_3(2))$ contains a Sylow $2$-subgroup of $G$. Furthermore suppose that there
exists an involution $u \in X \bs O^2(X)$ such that $C_G(u)\cong 2 \times \sym(8)$. Then $u \notin
O^2(G)$.  In particular, $O^2(G)\neq G$.
\end{lemma}
\begin{proof}
Let $Y:=O^2(X)$ then $Y/A\cong \GL_3(2)$ and $u \notin Y$. We assume for a contradiction that for
some $g \in G$, $r:=u^g \in Y$ and so we apply Lemma \ref{extremal transfer} to see that
$C_X(r)$ contains a Sylow $2$-subgroup of $C_G(r)\cong 2 \times \sym(8)$. Observe first that $r
\notin A$ because no element of order four in $G$ squares to $r$.

Set $V:=\Omega(A)\cong 2^3$ and let $S\in \syl_2(C_X(r))$. Then $|S|=2^8$ and therefore $|S \cap A|\geq 2^4$. It
follows that $S \cap A\cong 4 \times 4$ since $Ar\in Y/A\cong \GL_3(2)$ acts faithfully on $V$ and
therefore $|C_V(r)|\leq 2^2$. In particular, $|C_A(r)|=2^4$ and so $SA\in \syl_2(X)$.

Since $X/A\cong 2\times \GL_3(2)$, $2 \times \dih(8)\cong SA/A \cong S /(A \cap S)=S/C_A(r)$. Set
$S_0:=S \cap Y$ then $r \in S_0$ and we have that $\dih(8)\cong S_0A/A \cong S_0 /(A \cap
S_0)=S_0/C_A(r)$. Since $r \in \mathcal{Z}(S)$, $C_A(r)r\in \mathcal{Z}(S_0/C_A(r))$. Therefore $S_0/\<C_A(r),r\>\cong
2 \times 2$. Let $C_A(r)<R<S$ such that $|R/C_A(r)|=2$ and $S=S_0R$ and $[R,S_0]\leq C_A(r)$. This
is possible as $S/C_A(r)\cong 2 \times \dih(8)$. We have therefore that $[R,S_0]$, $S_0 \cap R \leq
C_A(r) \leq \<C_A(r),r\>$ and so $S/\<C_A(r),r\> \cong 2 \times 2 \times 2$. Now
$\<C_A(r),r\>/\<r\>\cong C_A(r)\cong 4 \times 4$. Hence, $S/\<r\>\sim (4 \times 4). (2 \times 2
\times 2)$  which is a subgroup of $C_G(r)/\<r\>\cong \sym(8)$. However, a $2$-subgroup of
$\sym(8)$ has non-abelian derived subgroup which supplies us with a contradiction. Thus $u \notin
O^2(G)$.
\end{proof}

\section{Strongly Closed Subgroups}

We now define strongly $p$-embedded ($p$ a prime) subgroups as well as strongly closed  and weakly closed subgroups of a group. In Chapter \ref{chapter-Alt9} we prove a result concerning groups with a certain strongly $3$-embedded subgroup. Groups with a strongly $2$-embedded subgroup are well understood thanks to  a theorem due to Bender (see \cite{bender-stronglyembedded}).
\begin{defi}
Let $G$ be a group and $H \leq G$ with a prime $p$ dividing $|H|$. We say that $H$ is strongly
$p$-embedded if for all $g \in G\bs H$, $p \nmid |H \cap H^g|$. If $p=2$ we say that $H$ is
strongly embedded in $G$.
\end{defi}

\begin{defi}
Let $G$ be a group with subgroups $P \leq R\leq G$.
\begin{enumerate}[$(i)$]
\item We say that $P$ is strongly closed in $R$ with respect to $G$ if for all $g \in G$ and for all $x \in P$, $x^g \in R$ implies $x^g \in P$. Alternatively, for all $g \in G$, $P^g \cap R \leq P$.

\item We say that $P$ is weakly closed in $R$ with respect to $G$ if for all $g \in G$, $P^g \leq R$ implies that  $P^g = P$.
\end{enumerate}
\end{defi}

In the following lemma we use the notation $m(G)$ to be the order of the largest elementary abelian
$2$-subgroup of a group $G$. The result is due to Goldschmidt (see \cite{Goldschmidt-2-fusion}) but
is stated in the presented form and proven also in \cite{Yoshida}.
\begin{lemma}\label{prelim-strongly closed}
Let $E$ be an  elementary abelian $2$-subgroup of a group $G$ and let $E\leq S \in \syl_2(N_G(E))$.
Assume that for each $x \in S \bs E$, $m(E)> m(S/E) +m(C_E(x))$. Then $E$ is strongly closed in $S$
with respect to $G$. In particular, $S \in\syl_2(G)$.
\end{lemma}

In Chapters \ref{Chapter-O8Plus2} and \ref{Chapter-HN} we show that certain abelian $2$-subgroups are strongly closed in a  Sylow
$2$-subgroup of a group with a view to applying the following theorem due to Goldschmidt. The result is an essential part of the $2$-local analysis required to determine a centralizer of an involution.

Recall that given a $p$-group $S$, we set $\Omega(S)=\<x \mid x^p=1\>$.
\begin{thm}[Goldschmidt]\cite[p370]{stellmacher} \label{goldschmidt}
Let $S$ be a Sylow 2-subgroup of a group $G$ and let $A$ be an abelian subgroup of $S$ such that $A$ is
strongly closed in $S$ with respect to $G$. Suppose that $G = \<A^G\>$ and $O_{2'}(G) = 1$.
Then $G = F^*(G)$ and $A = O_2(G)\Omega(S)$.
\end{thm}

\section{Representation Theoretic Results}

In order to understand a group we often identify this group acting on a vector space and use
representation  theoretic results.  When one group $G$ acts on a $p$-group $V$ say ($p$ a prime) we
may consider sections of $V$ on which $G$ acts irreducibly and call these chief factors. Note that
in this thesis we often consider groups acting on elementary abelian $p$-groups.  We consider such
groups as vector spaces and call them modules. However we continue to write such groups
multiplicatively.

\begin{defi}
Let $G$ be a group which acts on a group $p$-group $V$. Consider a sequence  $1=V_0\vartriangleleft V_1
\vartriangleleft \hdots \vartriangleleft V_n=V$ where each $V_i$ is a $G$-invariant subgroup of $V$
and each $V_i\vartriangleleft V_{i+1}$ is maximal with respect to being $G$-invariant. We say that
the series is a $G$-chief series and that each factor $V_{i+1}/V_i$ is a $G$-chief factor. Moreover
if $[V_{i+1}/V_i,G]=1$ then we say that $V_{i+1}/V_i$ is a central $G$-chief factor and non-central
otherwise.
\end{defi}

\begin{defi}
Let $G$ be a group acting on an elementary abelian $p$-group $V$. We say that $G$ acts
quadratically on $V$ if $[V,G,G]=1$.
\end{defi}
See Chapter 9 in \cite{stellmacher} for results concerning quadratic action. We require the following such result.

\begin{lemma}\cite[9.1.1, p226]{stellmacher}\label{quadractic action lemma}
Let $V$ be an elementary abelian $p$-group and $G$ a group acting quadratically on $V$. Then
$G/C_G(V)$ is an elementary abelian $p$-group.
\end{lemma}

\begin{lemma}\label{lem-cenhalfspace}
Suppose that $p$ is a prime  and that $V$ is an elementary abelian $p$-group and let $x$ be an
automorphism of $V$.
\begin{enumerate}[$(i)$]
\item Then $V/C_{V}(x)\cong [V,x]$.
\item If $p=2$ and $x$ has order $2$,
then $C_{V}(x)\geq [V,x]$ and $|C_{V}(x)|^{2}\geq |V|$.
\end{enumerate}
\end{lemma}
\begin{proof}
$(i)$ This is Lemma 8.4.1 in \cite{stellmacher}.

$(ii)$ This follows because the action of $x$ on $V$ is necessarily quadratic because
$[v,x]^x=(v\inv v^x)^x=(vv^x)^x=v^xv=[v,x]$ and so
$[v,x,x]=[v,x]\inv[v,x]^x=[v,x][v,x]^x=[v,x][v,x]=1$ and so  $C_{V}(x)\geq [V,x]$ and by part $(i)$,  $|C_{V}(x)|^{2}\geq |V|$.
\end{proof}

\begin{lemma}\label{lem-conjinvos}
Let $G$ be a finite group and $V\trianglelefteq G$ be an elementary abelian $2$-group. Suppose that
$r\in G$ is an involution such that $C_{V}(r)=[V,r]$. Then
\begin{enumerate}[$(i)$]
\item every involution in $Vr$ is conjugate to $r$; and
\item
$|C_{G}(r)|=|C_{V}(r)||C_{G/V}(Vr)|$.
\end{enumerate}
\end{lemma}
\begin{proof}
$(i)$ Let $t\in Vr$ be an involution. Then $t=qr$, for some $q\in V$. Since $t^2=1$, we have that
$1=qrqr=[q,r]$ as $r$ and $q$ have order at most two. So $q\in C_{V}(r)=[V,r]$. So
$q=q_{1}rq_{1}r$, for some $q_{1}\in V$, and therefore $t=q_{1}rq_{1}rr=r^{q_{1}}$ and so $t$ is
conjugate to $r$ by an element of $V$.

$(ii)$ Define a homomorphism, $\phi:C_{G}(r)\rightarrow C_{G/V}(Vr)$  by $\phi(x)=Vx$. Then $\ker
\phi=C_{V}(r)$. Moreover, if $Vy\in  C_{G/V}(Vr)$ then $Vr^y=Vr$. Hence, using $(i)$ we see that
there exists $q \in V$ such that $r^y=r^q$. Therefore $q\inv y \in C_G(r)$ and of course $Vq\inv y
=Vy$ and so $\phi(q\inv y)=Vy$. Therefore $\phi$ is surjective. Thus, by an isomorphism theorem,
$C_{G}(r)/C_{V}(r)\cong C_{G/V}(Vr)$ and $|C_{G}(r)|=|C_{V}(r)||C_{G/V}(Vr)|$, as required.
\end{proof}

During the proof of Theorems B and C we observe $3$-elements acting on elementary abelian $2$-groups. The
following lemma allows us to convert information about the fixed space of a $3$-element into
information about the fixed space of certain $2$-elements.

\begin{lemma}\label{Prelims-centralizers of invs on a vspace which invert a 3}
Let $G$ be a group with a normal $2$-subgroup $V$ which is elementary abelian of order
$2^n$. Suppose $t$ and $w$ are in $G$ such that $Vt$ has order two and $Vw$ has order three and
$Vt$ inverts $Vw$. If $|C_V(w)|=2^a$ then $|C_V(t)|\leq 2^{(n+a)/2}$.
\end{lemma}
\begin{proof}
Since $Vt$ inverts $Vw$, we have that $Vw=Vtw^2t$ and so $Vw^2=Vtw^2tw=VtVt^w$. Therefore $C_V(t)
\cap C_V(t^w) \leq C_V(w^2)=C_V(w)$. We have that $|V| \geq
|C_V(t)C_V(t^w)|=|C_V(t)||C_V(t^w)|/|C_V(t) \cap C_V(t^w)|$ and so $2^n \geq |C_V(t)|^2/2^a$
which implies $|C_V(t)|\leq 2^{(n+a)/2}$.
\end{proof}

\begin{lemma}\label{lemma making new modules}
Let $p$ be a prime and let $X\cong \SL_2(p)$ act on a vector space $U$ over $\GF(p)$.
Suppose that $V$ and $W$ are distinct irreducible and isomorphic $\GF(p)X$-submodules of $U$
such that $U=V \bigoplus W$. Then $U$ contains exactly $p+1$ $X$-submodules and each is isomorphic
to $V$.
\end{lemma}
\begin{proof}
Since $V$ and $W$ are isomorphic, there is an isomorphism $\phi: V \rightarrow W$. Consider the
sets $V_i=\{(v^i,\phi(v))|v \in V\}$ where $i=1,2,\hdots,p-1$ and where multiplication is
coordinate-wise. Then each $V_i$ is a $\GF(p)X$-submodule which is isomorphic to $V$.  Thus
$\{V,W,V_1,\hdots,V_{p-1}\}$ is a set of $p+1$ $\GF(p)X$-invariant submodules.

By \cite[2.8.8]{GLS3}, the splitting field for $\SL_2(p)$ is $\GF(p)$ and by
\cite[25.8]{Aschbacher}, this means that $\mathrm{End}_{\GF(p)X}(V)\cong \GF(p)$.  We now apply Theorem 3.5.6 in \cite[p79]{Gorenstein} which says that the number
of distinct irreducible $\GF(p)X$-submodules of $U$ is $(p^2-1)/(p-1)=p+1$ since $p=|\mathrm{End}_{\GF(p)X}(V)|$. Hence $U$ contains exactly $p+1$ $X$-submodules each isomorphic
to $V$.
\end{proof}

In this thesis, we will often consider natural $\SL_n(p)$-modules for $p$ a prime. We  will often observe a group $G \cong \SL_n(p)$ acting naturally on an elementary abelian $p$-group $V$ which we view as a vector space and call the natural $G$-module.

\begin{lemma}\label{prelims-natural sl23 mods}
Let $G\cong \SL_2(3)$ and suppose that $G$ acts on an elementary abelian $3$-group $V$ of order
nine. Then either $V$ is a natural $G$-module or $V$ has a trivial $G$-submodule.
\end{lemma}
\begin{proof}
Since $G$ acts on $V$ there is a homomorphism from $G$ to $\GL(V)\cong \GL_2(3)$. Moreover the
kernel of the homomorphism, $C_G(V)$, is a normal subgroup of $G$ therefore $|C_G(V)|=1,2,8$ or
$24$. So assume that $G$ acts non-trivially on $V$. If $|C_G(V)|=1$ then there is an injective
homomorphism from $G$ into $\GL(V)$ and it follows that $G\cong \SL(V)$ and so $V$ is a natural
$G$-module. If $|C_G(V)|=2$ then there is an injective homomorphism from $G/\mathcal{Z}(G)\cong \alt(4)$ into
$\GL_2(3)$ which is not possible. If $|C_G(V)|=8$ then let $S \in \syl_3(G)$ then $1 \neq C_V(S) \neq V$
and so $C_V(S)$ is a trivial $G$-submodule.
\end{proof}

\begin{lemma}\cite[3.20 $(iii)$]{ParkerRowley-book}\label{Parker-Rowley-SL2(q)-splitting}
Let $X\cong \SL_2(3)$ and $S \in \syl_3(X)$. Suppose that $X$ acts on an elementary  abelian
$3$-group $V$  such that $V=\<C_V(S)^X\>$, $C_V(X)=1$ and $[V,S,S]=1$. Then $V$ is a direct product
of natural modules for $X$.
\end{lemma}

\begin{lemma}\label{prelims-extraspecial 2^5 in GL(4,3)}
Let $E\leq \GL_4(3)$ such that $|E|=2^5$, and $|\Phi(E)|\leq 2$. Furthermore let $S \leq \GL_4(3)$ be elementary abelian of order nine
such that $S$ acts faithfully on $E$. If $Q_8 \cong A \cong B$ with $A\neq B$ both $S$-invariant
subgroups of $E$, then $E\cong 2_+^{1+4}$ and $E$ is uniquely determined up to conjugation in
$\GL_4(3)$.
\end{lemma}
\begin{proof}
Note that $E$ is non-abelian since $A,B \leq E$. Therefore $|E/\Phi(E)|=2^4$  is acted on
faithfully by ${S}$. Hence, $S$ is isomorphic to a subgroup of $\GL_4(2)$. Now observe that $\GL_4(2)$ has Sylow $3$-subgroups of order nine which
contain an element of order three which acts fixed-point-freely on the natural module. Thus any
${S}$-invariant subgroup of $E$ properly containing $\Phi(E)$ has order $2^3$ or $2^5$. Since $A$
and $B$ are distinct and normalized by $S$, we have $E=AB$. Suppose $|\mathcal{Z}(E)|>2$. Then
$|\mathcal{Z}(E)|=8$ is ${S}$-invariant. By coprime action,
$\mathcal{Z}(E)=\<C_{\mathcal{Z}(E)}(s)|s \in S^\#\>$. Thus there exists $s \in S^\#$ such that
$C_{\mathcal{Z}(E)}(s)>\Phi(E)$. Since $E=AB$, we find $a\in A$ and $b\in B$ such that $ab \in
C_{\mathcal{Z}(E)}(s)\bs \Phi(E)$. Then, as $s$ normalizes $A$ and $B$,  $s$ must centralize $a$
and $b$. Now $C_E(s)$ is $S$-invariant with $|C_E(s) \cap A|\geq 4$ and $|C_E(s) \cap B|\geq 4$. It
follows that $[E,s]=1$ which is a contradiction. Thus $\mathcal{Z}(E)=\Phi(E)$ and so $E$ is extraspecial and
by Lemma \ref{prelim-exraspecial}, $E\cong 2_+^{1+4}$.

Since $E$ is extraspecial, $[E:E']=2^4$. Therefore there are sixteen 1-dimensional representations of $E$ over $\GF(3)$.
Moreover there is a $4$-dimensional representation of $E$ since $E \leq \GL_4(3)$. Since
$16+4^2=32=|E|$, this accounts for all the irreducible representations of $E$ over $\GF(3)$. Hence there is a unique
$4$-dimensional representation of $E$ and so there is one conjugacy class of such subgroups in
$\GL_4(3)$.
\end{proof}

In \cite{AstillMPhilThesis} a complete proof of the following well known result due to Higman is given. We state a definition of $\GF(2)\SL_2(2^n)$-module.
\begin{defi}
Let $q=2^n$ and suppose that $G \cong \SL_2(q)$. Let $V$ be an irreducible finite-dimensional $\GF(2)G$-module such that $\End_{\GF(2)G}(V)\cong
\GF(q)$ and $V$ is a 2-dimensional $\GF(q)G$-module.
\end{defi}
\begin{thm}[Higman]\label{Higman's SL2 Thm}
Let $X$ be a group and $Q:=O_2(X)$ where $X/Q\cong \SL_2(2^n)$ for $n\geq 2$. If an element of order
three in $X/Q$ acts fixed-point-freely on $Q$ then $Q$ is elementary abelian and is a direct sum of
natural $\GF(2) X/Q$-modules.
\end{thm}

Modules for $\SL_2(2^n)$ where described in Section 2.6 and Chapter 8 in \cite{AstillMPhilThesis}.
In particular the next lemma follows from Lemma 8.5  in \cite{AstillMPhilThesis}.
\begin{lemma}\label{prelim-alt5 action}
Let $X\cong \alt(5)\cong \SL_2(4)$ act irreducibly on an elementary abelian $2$-group $V$ such that
an element of order three acts fixed-point-freely. Then  $|V|=2^4$ and $V$ is a natural module for
$X$ over $\GF(2)$.
\end{lemma}

%

\section{Recognition Results}\label{section-recognition results}

The following theorem by Higman is proved in \cite{Higman} using the Suzuki method (see Chapter
\ref{chapter-Alt9}).

\begin{thm}[Higman]\cite{Higman}\label{Higman Alt6}
Let $G$ be a simple group with a Sylow $3$-subgroup, $S$, which is elementary abelian of order
nine. Suppose $G$ has more than one conjugacy class of elements of order three and $C_G(s)=S$ for
each $1 \neq s \in S$. Then $G \cong \alt(6)$.
\end{thm}

In \cite{PrincePSp43}, Prince completes an earlier characterization result by Hayden \cite{HaydenPSp43} to recognize the groups $\PSp_4(3)\cong \Omega^-_6(2)$, $\aut(\PSp_4(3))\cong \mathrm{SO}^-_6(2)$ and $\Sp_6(2)\cong
\mathrm{SO}_7(2)$.
\begin{thm}[Prince]\label{prince}
Let $G$ be a group and suppose $a \in G$ has order 3 such that the following hold.
 \begin{enumerate}[$(i)$]
 \item $C_G(a)$ has shape $3^{1+2}_+.\SL_2(3)$;
 \item there exists $J\leq C_G(a)$ which is elementary abelian of order 27 and normalizes no
 non-trivial $3'$-subgroup of $G$.
 \end{enumerate}If $a$ is not conjugate to its inverse in $G$ then either $G$ has a normal subgroup of index 3 or
$G\cong \PSp_4(3)\cong \Omega^-_6(2)$. If $a$ is conjugate to its inverse in $G$ then
either $G=N_G(\<a\>)$ or $G\cong \aut(\PSp_4(3))\cong \mathrm{SO}^-_6(2)$ or $G\cong \Sp_6(2)\cong
\mathrm{SO}_7(2)$.
\end{thm}
\begin{proof}
Theorem 1 and 2 from \cite{PrincePSp43} give the result under the assumption that $C_G(a)$ is
isomorphic to the centralizer in $\PSp_4(3)$ of a $3$-central element of order three. Lemma 6 in
\cite{ParkerFi22} says that $C_G(a)$ has this isomorphism type provided it has shape
$3_+^{1+2}.\SL_2(3)$ and $O_2(C_G(a))=1$. However the condition $O_2(C_G(a))=1$ is guaranteed by
part $(ii)$ of the hypothesis.
\end{proof}
In Chapter \ref{Chapter-O8Plus2} we need to distinguish between $\mathrm{SO}^-_6(2)$ and
$\mathrm{SO}_7(2)$ and we require some theory about the subgroup structure of both groups.

\begin{lemma}\label{Prelims-distinguishPSp62}
If $G\cong \PSp_6(2)\cong \mathrm{SO}_7(2)$ and $J\leq G$ is elementary abelian of order $27$ then
there exist three distinct subgroups of $J$ of order three, $A_1,A_2,A_3$ such that $C_G(A_i)\cong
3 \times \sym(6)$ for each $i \in \{1,2,3\}$.
\end{lemma}
\begin{proof}
Let $\{e_1,f_1,e_2,f_2,e_3,f_3\}$ be a symplectic basis where $\{e_i,f_i\}$ is a hyperbolic pair.
Then $N_G(\<e_i,f_i\>)\cong \sym(3)\times \PSp_4(2)\cong \sym(3) \times \sym(6)$. In particular
there exists an element of order three $x$ in $G$ such that $C_G(x) \cong 3 \times \sym(6)$. We may
assume $x\in J$.  Since this element of order three is non $3$-central, there are at least three
conjugates of $x$ in $J$.
\end{proof}

\begin{lemma}\label{Prelims-PSp4(3) involutions}
Let $G \cong \PSp_4(3)\cong \Omega^-_6(2)$. Suppose that $t \in G$ is an involution and $R \in \syl_3(C_G(t))$ such that $|R|=9$.  The following hold.
\begin{enumerate}[$(i)$]
\item  We have that $t$ is $2$-central in $G$ and $C_G(t)\sim 2_+^{1+4}.(3 \times 3).2$.
\item $O_2(C_{G}(t))\cong 2_+^{1+4}$.
\item If $Q_1\cong Q_2\cong Q_8$ are distinct subgroups of $C_G(t)$ such that $[Q_i,R]=Q_i$ for $i \in \{1,2\}$, then $[Q_1,Q_2]=1$ and $Q_1Q_2=O_2(C_{G}(t)) \cong 2_+^{1+4}$.
\end{enumerate}
\end{lemma}
\begin{proof}
$(i)$ We see in \cite{atlas} that $G\cong \PSp_4(3)\cong \Omega^-_6(2)$ has two classes of involutions and only one class commutes with a
subgroup of order nine. Thus $t$ is $2$-central in $G$. We
observe that $C_{G}(t)\sim (\Sp_2(3) \ast \Sp_2(3)).2\sim 2^. (\alt(4) \wr 2)$. Moreover, we may describe this group as $C_{G}(t)\sim 2_+^{1+4}.(3 \times 3).2$.

$(ii)$ Clearly $O_2(C_{G}(t))$ contains a subgroup, $Q$ say, isomorphic to $2_+^{1+4}$. Since $C_{G}(t)\sim 2^. \alt(4) \wr 2$ which clearly has no larger normal Sylow $2$-subgroup, $O_2(C_{G}(t))\cong 2_+^{1+4}$.

$(iii)$ Since
$O_2(C_{G}(t))\cong 2_+^{1+4}$ and $Q_1$ and $Q_2$ are both normalized by $R\in
\syl_3(C_H(t))$, it follows that $\<Q_1,Q_2\>\leq O_2(C_G(t))$. Moreover, $Q_1\neq Q_2$
and by Lemma \ref{exactly 2 q-8's in extraspecial group}, $2^{1+4}_+$ has exactly two subgroups isomorphic to $Q_8$, therefore we have $[Q_1,Q_2]=1$ and $Q_1Q_2\cong 2_+^{1+4}$.
\end{proof}

\begin{lemma}\label{prelims-PSp43 normalizer J}
\begin{enumerate}[$(i)$]
\item Let $G \cong \PSp_4(3)\cong \Omega^-_6(2)$ or $\aut(\PSp_4(3))\cong \mathrm{SO}^-_6(2)$.
Suppose also that $J \leq G$ is elementary abelian of order $27$. Then ${N_G(J)}\sim
3^3.\sym(4)$ or ${N_G(J)}\sim 3^3.(\sym(4)\times 2)$.

\item If $G \cong \aut(\PSp_4(3))\cong \mathrm{SO}^-_6(2)$ and $w \notin G'$ is an involution with $9 \mid |C_G(w)|$, then $C_G(w)\cong 2 \times \sym(6)$. Also, $G$ has no element of order three which commutes with a subgroup isomorphic to $ \sym(6)$.
\end{enumerate}
\end{lemma}
\begin{proof}
$(i)$ It is clear from the statement of Theorem \ref{prince} that $G$ has an elementary abelian subgroup $J$ of order $27$. We can now observe from, for example, \cite{atlas} that ${N_G(J)}\sim 3^3.\sym(4)$ or ${N_G(J)}\sim 3^3.(\sym(4)\times 2)$.

Part $(ii)$ is easily checked in \cite{atlas}.
\end{proof}

\begin{lemma}\label{Prelims-sym9}
Let $G$ be a group of order $3^42$ with $S \in \syl_3(G)$ and $T \in \syl_2(G)$ and
$J\vartriangleleft G$ elementary abelian of order 27. Suppose that $Z:=\mathcal{Z}(S)$ has order three and
$Z \leq C_S(T)\neq S$. Then $G\cong C_{\sym(9)}(\hspace{0.5mm} (1,2,3)(4,5,6)(7,8,9) \hspace{0.5mm} )$.
\end{lemma}
\begin{proof}
We have that $T$ normalizes $Z$ and $J$ and so by Maschke's Theorem, there exists a subgroup $K \leq J$ such that $K$ is a $T$-invariant complement to $Z$ in $J$. Set $L:=KT$ then $K\trianglelefteq L$ and $[G:L]=9$. Suppose that $N \leq L$ and that $N$ is normal in $G$. If $3\mid |N|$ then $N\cap \mathcal{Z}(S) \neq 1$ which is a contradiction since $Z \nleq K$. So $N$ is a $2$-group which implies $N=1$ otherwise $G$ has a central involution. Hence there is an injective homomorphism from $G$ into $\sym(9)$. Moreover there is a map from $G$ into the centralizer in $\sym(9)$ of the centre of a Sylow $3$-subgroup. Since  $| C_{\sym(9)}(\hspace{0.5mm} (1,2,3)(4,5,6)(7,8,9) \hspace{0.5mm}|=|G|$, we have an isomorphism.
\end{proof}

\begin{thm}[Prince]\cite{PrinceSym9}\label{princeSym9}
Let $G$ be a group and suppose $x \in G$ has order 3 such that $C_G(x)\cong
C_{\sym(9)}(\hspace{0.5mm} (1,2,3)(4,5,6)(7,8,9) \hspace{0.5mm} )$ and there exists $J\leq C_G(x)$
which is elementary abelian of order 27 and normalizes no non-trivial $3'$-subgroup of $G$. Then
either $J\vartriangleleft G$ or $G \cong \sym(9)$.
\end{thm}

\begin{thm}[Aschbacher]\cite{AschbacherM12}\label{Ascbacher M12}
Let $G$ be a finite group with an involution $s$. Set $L:=C_G(s)$, $Q:=O_2(L)$ and choose $X\in
\syl_3(L)$. Assume that $Q$ is extraspecial of order 32, $L/Q\cong \sym(3)$, $C_Q(X)=\langle s\rangle$
and $s$ is not weakly closed in $Q$ with respect to $G$. Then either $G$ has shape
$2^3.\PSL_3(2)$ or $G\cong \alt(8)$, $\alt(9)$ or $\mathrm{M}_{12}$.
\end{thm}

The following theorem will be vital in the proof of Theorem B in Chapter \ref{Chapter-O8Plus2}. Note that our notation for orthogonal groups follows \cite{Aschbacher}. In particular, for a natural number $n$  and a prime $p$ and $\e \in \{1,-1\}$,  $\Omega_{2n}^{\e}(p)$ is the derived subgroup of  $\mathrm{SO}_{2n}^{\e}(p)$.

\begin{thm}[Smith]\cite{SmithOrthogonal}\label{Smith-Orthog}
Let $G$ be a finite group and let $t$ be an involution in $G=O^2(G)$. Suppose $F^\ast(C_G(t))$ is
extraspecial of order $2^9$ and $C_G(t)/O_2(C_G(t))\cong \sym(3) \times \sym(3) \times \sym(3)$ and $C_G(O_2(C_G(t)))\leq O_2(C_G(t))$.
Then either $O_{2'}(G)t \in \mathcal{Z}(G/O_{2'}(G))$ or $G \cong \Omega_8^+(2)$.
\end{thm}

The following three theorems are all required to recognize
certain sections of a group satisfying Hypothesis C.
\begin{thm}[Parker--Rowley]\cite{ParkerRowleyA8}\label{ParkerRowleyA8}
Let $G$ be a finite group with $R :=\<a, b\>$ an elementary abelian Sylow 3-subgroup of $G$ of order
nine. Assume the following hold. \begin{enumerate}[$(i)$]
 \item  $C_G(R)=R$ and $N_G(R)/C_G(R)\cong Dih(8)$.
 \item $C_G(a)\cong 3\times \alt(5)$ and $N_G(\<a\>)$ is isomorphic to the diagonal subgroup
 of index two in  $\sym(3)\times \sym(5)$.
 \item $C_G(b)\leq N_G(R)$, $C_G(b)/R\cong 2$ and $N_G(\<b\>)/R\cong 2\times 2$.
 \end{enumerate}Then $G$ is isomorphic to $\alt(8)$.
\end{thm}

\begin{cor}\label{Cor-ParkerRowleyA8}
Let $G$ be a group and $\alt(8)\cong H\leq G$ such that for $R \in \syl_3(H)$ and each $r \in
R^\#$, $C_G(r) \leq H$. Then $G=H$.
\end{cor}
\begin{proof}
Suppose $R$ is not a Sylow $3$-subgroup of $G$. Then there exists $R<S\in \syl_3(G)$. Therefore
$R<N_S(R)$ and $1 \neq r \in \mathcal{Z}(N_S(R))\cap R$. Therefore $N_S(R)\leq C_G(r)\leq H$ which is a
contradiction. Thus $R \in \syl_3(G)$. Pick $a,b\in R$ such that $C_H(a)\cong 3 \times \alt(5)$ and
$C_H(b)\leq N_H(R)$. Now we check the hypotheses of Theorem \ref{ParkerRowleyA8}. We have that for
any $r \in R^\#$, $C_G(R)\leq C_G(r) \leq H$ and so $C_G(R)=C_H(R)=R$. So consider $N_G(R)/C_G(R)$
which is isomorphic to a subgroup of $\GL_2(3)$. Since $R \in \syl_3(G)$, $N_G(R)/R$ is a
$2$-group. Also $N_H(R)/R\cong \dih(8)$. Suppose $N_G(R)/R \cong \sdih(16)$. Then $N_G(R)$ is
transitive on $R^\#$ which is a contradiction. Therefore  $N_G(R)=N_H(R)$ and $N_G(R)/C_G(R)\cong
\dih(8)$ so $(i)$ is satisfied. Now $C_G(a)=C_H(a)$ and there exists some $x \in H$ that inverts $a$. Therefore $N_H(\<a\>)=C_H(a)\<x\>\leq H$. Similarly $C_G(b)=C_H(b)$ and there exists some $y \in H$ that  inverts $b$. Therefore $N_H(\<b\>)=C_H(b)\<y\>\leq H$.
Thus $(ii)$ and $(iii)$ are satisfied so $G=H\cong \alt(8)$.
\end{proof}

\begin{thm}[Aschbacher]\cite{AschbacherHS}\label{Aschbacher-HS}
Let $G$ be a group with an involution $t$ and set $H:=C_G(t)$. Let $V \leq G$ such that  $V\cong 2 \times 2
\times 2$ and set $M:=N_G(V)$. Suppose that
 \begin{enumerate}[$(i)$]
 \item $O_2(H)\cong 4 *2_+^{1+4}$ and $H/O_2(H) \cong \sym(5)$; and
 \item $V \leq O_2(H)$, $O_2(M)\cong 4 \times 4 \times 4$ and $M/O_2(M)\cong \GL_3(2)$.
\end{enumerate} Then $G$ is isomorphic to the sporadic simple group $\HS$.
\end{thm}

\begin{thm}[Segev]\cite{SegevHN}\label{Segev-HN}
Let G be a finite group containing two involutions $u$ and $t$ such that $C_G(u)\cong (2 ^. \HS) : 2$
and $C_G(t)\sim 2_+^{1+8}.(\alt(5)\wr 2)$ with $C_G(O_2(C_G(t)))\leq O_2(C_G(t))$. Then $G \cong \HN$.
\end{thm}

Finally we present the following two theorems by Feit and Thompson and by Smith and Tyrer which
have both proved to be very useful in odd characterizations (see \cite{AstillM12} and
\cite{AstillG23} for example). Both theorems are required in the proof of Theorem A in Chapter \ref{chapter-Alt9}.

\begin{thm}[Feit--Thompson] \cite{FeitThompson}\label{Feit-Thompson}
Let $G$ be a finite group containing a subgroup, $X$, of order three such that $C_G(X)=X$. Then one
of the following holds:
\begin{enumerate}[$(i)$]
    \item $G$ contains a nilpotent normal subgroup, $N$, such that $G/N\cong \sym(3)$ or
    $C_3$;
    \item $G$ contains an elementary abelian normal 2-subgroup, $N$, such that
    $G/N\cong \mathrm{Alt}(5)$; or
    \item $G\cong\mathrm{PSL}_2(7)$.
\end{enumerate}
\end{thm}
The result can be found in  \cite{FeitThompson} however the additional information in conclusion
$(ii)$ that $N$ is elementary abelian uses a theorem of Higman (see \ref{Higman's SL2 Thm}).

\begin{defi}
A group $G$ is $p$-soluble if every composition factor of $G$ is
either a $p$-group or a $p'$-group.
\end{defi}Consider the following series
\[1\trianglelefteq O_{p'}(G) \trianglelefteq O_{p',p}(G)
\trianglelefteq O_{p',p,p'}(G) \trianglelefteq \hdots\] where
$O_{p',p}(G)$ is the preimage in $G$ of $O_p(G/O_{p'}(G))$ and
$O_{p',p,p'}(G)$ is the preimage in $G$ of $O_{p'}(G/O_{p',p}(G))$
and so on. This series defines a minimal factorization of $G$ into
$p$ and $p'$ factors (minimal in the sense that the number of
factors is as few as possible). We call this the lower $p$-series
for $G$.
\begin{defi}
A $p$-soluble group $G$ has length $n$ if there are $n$ factors in
the lower $p$-series for $G$ which are $p$-groups. In particular,
$G$ is $p$-soluble of length one if $G=O_{p',p,p'}(G)$.
Alternatively, $G$ is $p$-soluble of length one if for any Sylow
$p$-subgroup, $S$, of $G$, $O_{p'}(G)S\trianglelefteq G$.
\end{defi}

\begin{thm}[Smith--Tyrer] \cite{Smith-Tyrer}\label{Smith-Tyrer}
Let $G$ be a finite group and let $P$ be a Sylow $p$-subgroup of $G$ for an odd prime $p$. Suppose
$P$ is abelian and $[N_G(P):C_G(P)]=2$. If $P$ is non-cyclic, then $O^p(G) <G$ or $G$ is
$p$-soluble of length 1.
\end{thm}


\chapter{Character Theoretic Results and a 3-Local Recognition of $\alt(9)$}\label{chapter-Alt9}

Let $G$ be a finite group and let $J \leq G$ be  a $3$-subgroup of $G$.
Suppose that $J$ is elementary abelian of order 27 and $N_G(J)$ is isomorphic to the normalizer in
$\alt(9)$ of an elementary abelian subgroup of order $27$. Then we say that $G$ has a $3$-local
subgroup, $N_G(J)$, of $\alt(9)$-type. The theorem we  prove in this chapter is the following.

\begin{thmA}
Let $G$ be a finite group and suppose that $H$ is a $3$-local subgroup of
$G$ of $\alt(9)$-type. If $O_{3'}(C_G(x))=1$ for every element $x$ of
order three in $ H$ then $G=H$ or $G \cong \mathrm{Alt}(9)$.
\end{thmA}
There are two isomorphism types of groups $X$ with $X/O_3(X)\cong \sym(4)$ acting faithfully on an
elementary abelian subgroup $O_3(X)$ of order 27. Both are isomorphic to subgroups of $\sym(9)$
however only one embeds into $\alt(9)$ which is the type we consider in Theorem A. In
\cite{PrincePSp43}, Prince characterizes $\PSp_4(3)$ which has a $3$-local subgroup with the same
shape, $3^3:\sym(4)$, as $H$ but with a different isomorphism type. Furthermore some of the methods
used in the proof of Theorem A date back to Prince and to Higman's odd characterizations.
Higman characterized the nine finite simple groups with smallest order which have more than one
conjugacy class of elements of order three and he lists these in \cite{Higman-plocalconditions}. He
characterized the groups assuming the order of each $3$-centralizer. Most of these calculations
were never published but his methods would certainly have involved detailed character calculations,
in particular, the Suzuki method and possibly arguments involving blocks of characters. It is
likely he would have calculated part of the character table and then used calculations involving
structure constants to obtain an upper bound for the group order. These methods are particularly relevant when the $3$-structure of a
group is small. In such situations local group theoretic arguments become very difficult because it
is often not possible to control the size and the structure of an involution centralizer and
therefore one has no control of the group order. Character theory may however allow calculation of
an upper bound for the group order and therefore come to the rescue when local methods fail.

The proof of Theorem A uses a combination of local and character theoretic methods. We briefly
describe Suzuki's theory of special classes and define some necessary $p$-block theory. We then
begin to work under the hypothesis of Theorem A. This leads us to three possibilities for the
$3$-centralizer structure of $G$ and we consider each case separately. The first case describes a
situation when, in some sense, $H=N_G(J)$ has full control of the $3$-centralizer structure of $G$
(Hypothesis \ref{case 1 hyp}). An argument involving Suzuki's theory of special classes proves that
$G=H$. We only calculate the part of the table which is necessary and so the calculations are
somewhat delicate. In the second case we consider the possibility that the $3$-centralizers of $G$
are small but $H$ does not control the $3$-structure (Hypothesis \ref{case 2 hyp}). We require some
detailed calculations involving blocks of characters to reach a contradiction. The character
calculations involve using the character table of $H$ to calculate part of the character table of
$G$. From here structure constants are calculated and an upper bound for $|G|$ is found. The
complexity of the calculation is far greater than in the first case. Finally, in the third case we
recognize $\alt(9)$. A calculation of Higman's to recognize $\alt(6)$ using the Suzuki method
allows us to see that $G$ has a $3$-centralizer isomorphic to $3 \times \alt(6)$ (Hypothesis
\ref{case 3 hyp}). It is then possible to use local group theory to determine the structure of an
involution centralizer. We may then recognize $\alt(9)$ using a recent theorem of Aschbacher
\cite{AschbacherM12}.

It is likely that Higman used character calculations to recognize
$\alt(9)$ from the order of each $3$-centralizer. We note that the
block character theoretic methods used in Case 2 were extended in Case 3 to repeat Higman's calculation of part of the
character table of $\alt(9)$. From here an upper bound for the
group order was  obtained. However, this leaves the difficulty of
recognizing $\alt(9)$ from its group order and knowledge of the
$3$-structure. Thus we present here the local arguments only.

We may deduce a corollary concerning strongly $3$-embedded groups immediately from Theorem A. Recall
that a subgroup $H$ of a group $G$ is strongly $p$-embedded ($p$ a prime divisor of $|H|$) if $p\mid |H|$, $H
\neq G$ and $p \nmid |H \cap H^g|$ for all $g \in G \bs H$.

\begin{cor}
If $G$ and $H$ satisfy the hypothesis of Theorem A then $H$ cannot be a strongly $3$-embedded
subgroup of $G$.
\end{cor}

Character notation follows \cite{Collins-CT} and \cite{Isaacs} in particular for a character $\chi$
of a  group $H \leq G$ the induced character is labeled $\chi^G$. If $\chi$, $\psi$ are two
characters of a group $G$ then $(\chi,\psi)_G$ is the inner product in $G$ and equals
$\frac{1}{|G|}\sum_{g \in G}\chi(g)\bar{\psi(g)}$.

\section{Preliminary Results}\label{CharacterTheory-prelims}

We will apply Theorem A in Chapter \ref{Chapter-HN}. To do so we need the following lemma which gives conditions which guarantee that a group is isomorphic to a $3$-local subgroup of $\alt(9)$-type.
\begin{lemma}\label{Alt9 prelims 3cubedsym4}
Let $H$ be a group of order $3^42^3$ with $S \in \syl_3(H)$. Suppose the following hold.
\begin{enumerate}[$(i)$]
 \item  $J\trianglelefteq H$ is elementary abelian of order 27.
 \item $Z:=\mathcal{Z}(S)$ has order three with $C_H(Z)=S$.
 \item There is an involution $t\in N_H(S)$ such that $C_J(t)\neq 1$.
\end{enumerate}
Then $H$ is isomorphic to a $3$-local subgroup of $\alt(9)$-type.
\end{lemma}
\begin{proof}
Let $T \in \syl_2(H)$. First observe that  since $J\vartriangleleft S$, $Z=\mathcal{Z}(S)\leq J$.
Therefore, $C_H(J)\leq C_H(Z)=S$ and since $J$ is not central in $S$, $C_H(J)=J$ and so $H/J$ acts
faithfully on $J$. Hence $H/J$ is isomorphic to a subgroup of $\GL_3(3)$. By Sylow's Theorem, $H$
has four Sylow $3$-subgroups or $S\trianglelefteq H$. However, if $S\trianglelefteq H$ then
$\mathcal{Z}(S)\trianglelefteq H$ and then $C_H(Z)$ would not be a three group. Thus $H$ has four
Sylow $3$-subgroups.

If $V$ is a $\GF(3)$-vector space of dimension three then we observe that a Sylow $2$-subgroup of $\GL(V)$ has order $2^5$. The normalizer in $\GL(V)$ of a subspace of dimension two is isomorphic to $2 \times \GL_2(3)$. Notice that this normalizer contains a Sylow $2$-subgroup of $\GL(V)$. We conclude from this that $T\in \syl_2(H)$ preserves a subgroup of $J$ of order nine. Let $W$ be such a  $T$-invariant subgroup of $J$ of order nine. Suppose that $W\trianglelefteq H$. Then $W \cap S\trianglelefteq S$ so $Z \leq W$. We see that $C_H(W)=J$ since no involution centralizes $Z \leq W$ and $|\mathcal{Z}(S)|=3$. Thus, $H/J$ is isomorphic to a subgroup of $\GL_2(3)$ and then it follows from $|H|$ that $H/J\cong \SL_2(3)$.
Now, by hypothesis, an involution $t$ normalizes $S$ with $C_J(t)\neq 1$. Hence $Jt\in \mathcal{Z}(H/J)$ and $Jt$ inverts $W$. Moreover, $|C_J(t)|=3$. We have that $C_H(t)\sim 3.\SL_2(3)$ and furthermore we have that $H=JC_H(t)$. Therefore $C_J(t)\trianglelefteq \<J,C_H(t)\>=H$. This is a contradiction since $C_J(t)\trianglelefteq H$ implies that $Z \leq C_J(t)$. Thus $W\ntrianglelefteq H$.

Now suppose that $Z \leq W$. We have seen that $H$ has four Sylow $3$-subgroups and these are permuted by $T$. Therefore, $T$ permutes the centres of the Sylow $3$-subgroups. Since $T$ normalizes $W$, we have that $W=\bigcup Z^H$. However this implies that $W\trianglelefteq H$ which is not the case. So $Z \nleq W$. This implies that no proper non-trivial subgroup of $W$ is normal in $H$ for if it were then this normal subgroup would contain $Z$.

Now we set $X:=WT$. Then $[H:X]=9$.
Consider the core of $X$ in $H$, $K:=\bigcap_{g \in H}X^g\vartriangleleft H$. Then $K \cap J\leq
W$ and $K \cap J\trianglelefteq H$. Therefore $K \cap J=1$ and so $K$ is a $2$-group. However now we have $[K,J]\leq K \cap J=1$
and so $K \leq C_H(J)=J$. Thus $K=1$. Hence $H$ is isomorphic to a subgroup of $\sym(9)$. Furthermore, $H$
embeds into the normalizer in $\sym(9)$ of $\<(1,2,3),(4,5,6),(7,8,9)\>$. There are three possible isomorphism types for $H$ and only one of these is contained in $\alt(9)$. The three subgroups are

\[A=\langle (123),(456),(789),(147)(258)(369),(23)(89), (14)(25)(36)(78)\rangle;\]
\[B=\langle
(123),(456),(789),(147)(258)(369),(23)(89), (14)(25)(36)\rangle;\] and
\[C=\langle
(123),(456),(789),(147)(258)(369),(23)(56)(89), (23)(89)\rangle.\]

Notice that $H\ncong B$ since we calculate that the centralizer of a $3$-central element of order three in $B$ is not a $3$-group (because $(14)(25)(36)$ centralizes $(123)(456)(789)$).
Notice also that by hypothesis, an involution $t$ in $N_H(S)$, has non-trivial centralizer on $J$. However an involution in $C$ which normalizes a Sylow $3$-subgroup of $C$ acts fixed-point-freely on $C$ (the involution $(23)(56)(89)$). Thus we conclude that $H \cong A$ is isomorphic to a subgroup of $\alt(9)$.
\end{proof}

\begin{lemma}\label{self-normalizing extraspecial 3^3}
Let $G$ be a group and $S \in \syl_3(G)$ such that $S \cong
3_+^{1+2}$ and $N_G(\mathcal{Z}(S))=S$. Then $G=O_{3'}(G)S$. In particular, if
$O_{3'}(G)=1$ then $G=S$.
\end{lemma}
\begin{proof}
Let $Z:=\mathcal{Z}(S)$. Note that $N_G(S)\leq N_G(Z)$ implies that $S=N_G(S)$. Suppose $Z \neq
Y\leq S$ with $Y^g=Z$ for some $g \in G$. Then $\<S,S^g\> \leq N_G(\<Z,Y\>)$ and $\<Z,Y\>$ is
self-centralizing so $\<S,S^g\>/\<Z,Y\>$ is isomorphic to $\SL_2(3)$. Therefore $N_G(S)$ contains
an involution which is a contradiction. Thus no subgroup of $S$ distinct from $Z$ is conjugate to
$Z$. Now by Gr\"{u}n's Theorem (Theorem \ref{GrunsThm}), $S \cap G'=\<N_G(S)',S \cap P'|P \in
\syl_3(G)\>$. Since $N_G(S)'=S'=Z$ and $S \cap P'\leq Z$ for any Sylow $3$-subgroup $P$ of $G$, $S
\cap G'=Z$. So $Z \in \syl_3(G')$ and $Z$ is necessarily self normalizing in $G'$. Thus $G'$ has a
normal $3$-complement (by Burnside's normal $p$-complement Theorem) $N=O_{3'}(G')$ such that
$G'=NZ$. Now, by a Lemma \ref{frattini} (Frattini), $G=G'N_G(Z)=NZN_G(Z)=NZ S=NS$ and clearly
$N=O_{3'}(G)$.
\end{proof}

The following observation allows
the calculation of structure constants in a group in which all
$3$-centralizers are known. It is stated as a theorem and proved in
\cite{FeitThompson} however it was probably known well before this.

\begin{lemma}\label{FT 333 group}
Let $X$ be a group with elements $a$, $b$ and $c$ each of order $3$ such that $c=ab$ and
$X:=\<a,b\>$. Then $X$ has an abelian normal subgroup of index $3$.
\end{lemma}

\begin{lemma}\label{triangle group}
Let $G$ be a group with elements $a$, $b$ and $c$ each of order
three such that $ab=c$ and such that $\<a\>$, $\<b\>$ and $\<c\>$
are not all $G$-conjugate. Then there exists an element of order
three $z \in X=\<a,b\>$ such that $X \leq C_G(z)$.
\end{lemma}
\begin{proof}
By Lemma \ref{FT 333 group}, $X$ has an abelian normal subgroup of index three, $N$ say. Let $S \in
\syl_3(X)$. Since $\<a\>$, $\<b\>$ and $\<c\>$ are not all conjugate, $|S|\geq 9$. Therefore $1\neq
S \cap N \trianglelefteq S$ and so $\mathcal{Z}(S) \cap N \neq 1$. Choose $z \in \mathcal{Z}(S) \cap N$ of order three
then $z $ commutes with $\<S,N\>=X$ and therefore $X \leq C_G(z)$.
\end{proof}

Given a group $F$ and elements $x$, $y$, $z$ in $F$, $a^F_{xyz}$
denotes the number of pairs $(a,b) \in x^F\times y^F$ such that
$ab=z$. This integer is called a structure constant and may be
calculated from the character table of $F$ using the formula:
\[a^F_{xyz}= \frac{|F|}{|C_F(x)||C_F(y)|} \sum_{\chi \in
\mathrm{Irr}(F)} \frac{\chi(x) \chi(y)\bar{\chi(z)}}{\chi(1)}.\] We
introduce the  notation:
\[\a^F_{xyz}:= \sum_{\chi \in
\mathrm{Irr}(F)} \frac{\chi(x) \chi(y)\bar{\chi(z)}}{\chi(1)}.\]
Therefore $\a^F_{xyz}=a^F_{xyz}\frac{|C_F(x)||C_F(y)|}{|F|}$.

\subsection{Suzuki's Theory of Special Classes}\label{suzuki special classes}

We now present Suzuki's definition of special classes and the
Suzuki method. See \cite{Collins-CT} and \cite{Curtis-Reiner} for
further details.

\begin{defi}\label{defi-special classes}
Let $G$ be a group and $H$ a subgroup of $G$. Suppose that
$\mathcal{C}=\bigcup_{i=1}^n\mathcal{C}_i$ is a union of $H$-conjugacy classes of $H$. Then
$\mathcal{C}$  is called a set of special classes in $H$ provided the following hold:
\begin{enumerate}[$(i)$]
\item $C_G(h)\leq H$ for all $h \in \mathcal{C}$;
\item $\mathcal{C}_i^G \cap \mathcal{C}=\mathcal{C}_i$ for all
$1\leq i \leq n$; and
\item if $h \in \mathcal{C}$ and $\<h\>=\<f\>$ then $f \in \mathcal{C}$.
\end{enumerate}
\end{defi}
Suppose that $G$ is a group, $H \leq G$ and $\CC$ is a set of
special classes in $H$. Set $\mathrm{CF}(H)$ to be the
$\mathbb{C}$-space of all class functions of $H$ and set \[W:=\{\phi
\in\mathrm{CF}(H)|\phi(h)=0 ~\mathrm{for~ all}~ h \in H \bs
\mathcal{C}\}.\] Let $\{\psi_1,\hdots,\psi_r\}=\mathrm{Irr}(H)$ and
$\{\theta_1,\hdots,\theta_s\}=\mathrm{Irr}(G)$.

\begin{lemma}\label{special classes 1}
Suppose $\lambda,\mu\in W$. Then $\lambda^G(x)=\lambda(x)$ for all
$x \in \CC$ and for all $x \in G \bs \CC^G$.  Furthermore
$(\lambda^G,\mu^G)_G=(\lambda,\mu)_H$.
\end{lemma}
\begin{proof}
See \cite[p111]{Collins-CT}.
\end{proof}

\begin{lemma}\label{special classes 2}
Suppose $\mathcal{C}=\bigcup_{i=1}^n\mathcal{C}_i$ is a set of
special classes in $H \leq G$. Then the set of all class functions
of $H$ which vanish on $H \bs \mathcal{C}$ is a
$\mathbb{C}$-subspace of the set of all class functions of $H$ and
has dimension $n$.
\end{lemma}
\begin{proof}
See \cite[14.6, p348]{Curtis-Reiner}.
\end{proof}

Let $\mathcal{B}_W:=\{\lambda_1,\hdots,\lambda_n\}$ be a basis for
$W$  and define a matrix $A=(a_{ij})$ such that
\[\lambda_i=\sum_{j=1}^r a_{ij}\psi_j.\] Now consider the class
functions of $G$, $\mu_i=\lambda_i^G$. Each can be written as a
linear combination of irreducible characters of $G$ (these are
$\theta_1,\hdots,\theta_s$) and so let $B=(b_{ij})$ be the matrix
such that \[\mu_i=\sum_{j=1}^s b_{ij}\theta_j.\]

\begin{thm}[Suzuki]\label{suzuki}
There exist uniquely defined $C=(c_{ij})$ such that the $(i,j)$'th entry of $CA$ is
$\psi_j(x_i)$ and the $(i,j)$'th entry of $CB$ is $\theta_j(x_i)$ where $x_i$ is a
representative from the $H$-conjugacy class $\mathcal{C}_i$.
\end{thm}
\begin{proof}
See \cite[14.11, p351]{Curtis-Reiner}
\end{proof}
\begin{cor}\label{cor-gamma ad cij's}
$C=(c_{ij})$ is such that
$\gamma_i:=\sum_{j=1}^r\psi_j(x_i)\psi_j=\sum_{k=1}^n
c_{ik}\lambda_k$.
\end{cor}
\begin{proof}
We have $\gamma_i:=\sum_{j=1}^r\psi_j(x_i)\psi_j$ and by Suzuki's
Theorem, $\psi_j(x_i)=\sum_{k=1}^nc_{ik}a_{kj}$ so
\[\begin{array}{rcl}
    \sum_{j=1}^r\psi_j(x_i)\psi_j&=&\sum_{j=1}^r\sum_{k=1}^n
c_{ik}a_{kj}\psi_j\\
\; &=& \sum_{k=1}^n c_{ik}\sum_{j=1}^r a_{kj}\psi_j\\
\; &=& \sum_{k=1}^n c_{ik}\lambda_k
\end{array}\] because $\lambda_k=\sum_{j=1}^r a_{kj}\psi_j$ by
definition of $A$. Hence $\gamma_i=\sum_{k=1}^n c_{ik}\lambda_k$.
\end{proof}

Therefore given the character table of $H$ it is possible to determine the constants $c_{ij}$ by
calculating each $\gamma_i$ and writing as a linear combination of basis  elements using the basis
$\mathcal{B}_W$ above. The Suzuki method involves making a careful choice of basis $\mathcal{B}_W$.
A good choice of basis will make it easier to determine candidates for $B$ and therefore candidates
for the character table of $G$.


\subsection{Some $p$-Block Theory}\label{block theory}
We present the relevant block theory for use in the proof of Theorem A. See \cite{Isaacs} and
\cite{Feit-RepTheory} for further details. In modular character theory it is always necessary to
fix a prime $p$ and then make a fixed choice of ring homomorphism $\ast$ defined on the ring of
algebraic integers $\mathbb{A}$ with kernel equal to a maximal ideal which contains $p\mathbb{A}$.
Then $\ast$ maps $\mathbb{A}$ onto an algebraically closed field of characteristic $p$. Let
$\Irr(G)$ be the set of irreducible ordinary characters of $G$ and let $\mathrm{IBr}(G)$ be the set
of irreducible Brauer characters of a group $G$. We do not require a formal definition of Brauer
character only that a Brauer character is a map from the $p$-regular elements (elements of order
coprime to $p$) of $G$ to $\mathbb{C}$,  and that every Brauer character can be written as a sum of irreducible Brauer characters. See \cite[p263]{Isaacs} for a full definition. Given an ordinary character
$\chi$ of $G$, restricting $\chi$ to the $p$-regular elements of $G$ gives a Brauer character of
$G$ which can be written as a sum of irreducible Brauer characters. If some irreducible Brauer
character $\rho \in \mathrm{IBr}(G)$ appears in this sum we say $\rho$ \textit{appears in the
Brauer decomposition of} $\chi$.

\begin{defi}
Let $G$ be a group with $p \mid |G|$. A $p$-block of $G$ is a subset
$B \subseteq \mathrm{Irr}(G) \cup \mathrm{IBr}(G)$ satisfying
\begin{enumerate}[$(i)$]
\item for $\chi$, $\phi$ in $B \cap \mathrm{Irr}(G)$,
\[\left(\frac{\chi(g)|G|}{\chi(1)|C_G(g)|}\right)^\ast =
\left(\frac{\phi(g)|G|}{\phi(1)|C_G(g)|}\right)^\ast\] for every $g
\in G$; and
\item $B \cap \mathrm{IBr}(G)$ is the set of irreducible Brauer characters
which appear in the Brauer decomposition of some $\chi \in B \cap
\mathrm{Irr}(G)$.
\end{enumerate}
\end{defi}
The $p$-blocks define a partition of $\mathrm{Irr}(G)$ and the
principal $p$-block, denoted $B_0(G)$, is the $p$-block containing the
principal character. We are mostly interested in the ordinary
characters and so given a $p$-block $B$ we often refer to $B$ in
place of $B \cap \Irr(G)$. If there is no ambiguity in $p$ we often
refer simply to the block rather than the $p$-block.

\begin{lemma}\label{non-vanishing principal block on p-central}
Let $G$ be a group, $S \in \syl_p(G)$ and $1\neq x \in \mathcal{Z}(S)$. If
$\chi$ is an irreducible character in the principal $p$-block of $G$
then $\chi(x) \neq 0$.
\end{lemma}
\begin{proof}
Since $\chi\in B_0(G)$,
\[\left(\frac{\chi(x)|G|}{\chi(1)|C_G(x)|}\right)^\ast =
\left(\frac{|G|}{|C_G(x)|}\right)^\ast.\] Since $x$ is $p$-central,
$|G|/|C_G(x)|$ is an integer which is coprime to $p$ and so its
image under $\ast$ is non-zero. Therefore $\chi(x) \neq 0$.
\end{proof}

The following result was proved independently however a proof can also be found in \cite{Laska}.
\begin{lemma}\label{Characters mod p}
Let $G$ be a group and $\chi$ a character of $G$. Suppose that $g \in G$ is a $p$-element and
$\chi(g)\in \mathbb{Z}$. Then $\chi(g) \equiv \chi(1) ~\mathrm{mod} ~p$.
\end{lemma}
\begin{proof}
For any integer $n \geq 1$, if $\e$ is a primitive $p^n$'th root of unity then $\e$ is a root of
the cyclotomic polynomial, \[\Phi_{p^n}(X)=1+X^{p^{n-1}}+X^{2p^{n-1}}+\hdots +X^{(p-1)p^{n-1}}\]
and this polynomial is well known to be irreducible and hence the minimal polynomial of
$\e$.

Now suppose that there exist integers $a_0,a_1,\hdots, a_{p^n-1}$  such that $\sum_{i=0}^{p^{n}-1}
a_i\e^i=0$. Then $\e$ is a root of the equation $\sum_{i=0}^{p^n-1} a_iX^i=0$ and so
$\sum_{i=0}^{p^n-1} a_iX^i=\Phi_{p^n}(X)F(X)$ for some polynomial $F(X)\in \mathbb{Z}[X]$. In fact we may calculate,
$F(X)=a_0+a_1X+a_2X^2+ \hdots +a_{p^{n-1}-1}X^{p^{n-1}-1}$ and it follows that for $i \in
\{0,\hdots,p^{n-1}-1\}$ and $j \in \{1,\hdots, p-1\}$, $a_i=a_{jp^{n-1} +i}$ and
$\sum_{i=0}^{p^n-1} a_iX^i=\sum_{i=0}^{p^{n-1}-1} a_iX^i\Phi_{p^n}(X)$.

So now we suppose that $G$ is a group, $\chi$ is a character of $G$ of degree $m$ and  $g \in G$
has order $p^n$ with $\chi(g)=z\in \mathbb{Z}$. Then $\chi(g)$ is a sum of $m$  $p^n$'th roots of
unity. Therefore we may suppose that $z =\chi(g)=\sum_{i=0}^{p^n-1} b_i\e^i$ where each $b_i \in
\mathbb{Z}$ and $m=b_0+\hdots+ b_{p^n-1}$. By setting $a_0:=b_0-z$ and $a_i:=b_i$ otherwise,  we
may describe an equation for which $\e$ is a root, $0=\sum_{i=0}^{p^n-1}
a_iX^i=\sum_{i=0}^{p^{n-1}-1} a_iX^i\Phi_{p^n}(X)$ and as above for $i \in \{0,\hdots,p^{n-1}-1\}$
and $j \in \{1,\hdots, p-1\}$, $a_i=a_{jp^{n-1} +i}$.

Now $\chi(1)=m=b_0+\hdots+b_{p^n-1}$ and since for  $i \in \{0,\hdots,p^{n-1}-1\}$ and $j \in
\{1,\hdots, p-1\}$, $a_i=a_{jp^{n-1} +i}$ and $b_i=a_i$ except when $i=0$, we have that
$m=b_0+(p-1)a_0+pa_1+pa_2+\hdots +pa_{p^{n-1}-1}$. Therefore $m \equiv b_0-a_0~\mathrm{mod}~p$ and
since $b_0-a_0=z=\chi(g)$, we have $\chi(g) \equiv \chi(1) ~\mathrm{mod} ~p$.
\end{proof}

To every $p$-block of $G$ we associate a $G$-conjugacy class of $p$-subgroup (see
\cite[p278-9]{Isaacs} for a description of how we do this and why it is possible). Given a
$p$-block $B$ with associated $p$-subgroup $D$, we say $D$ is a \textit{defect group} of the
$p$-block $B$. The Sylow $p$-subgroups are the defect groups of the principal block.


\begin{lemma}[Generalized Decomposition Numbers]\label{gen
decomp numbers} Let $x\in G$ have order $p^n$. For  $\chi \in
\Irr(G)$ and $\varphi \in \mathrm{IBr}(C_G(x))$ there exist unique
algebraic integers  $d^x_{\chi \varphi}\in\mathbb{Q}(e^{2\pi i
/p^n})$ such that
\[\chi(xf)=\sum_{\varphi \in \mathrm{IBr}(C_G(x))}d^x_{\chi \varphi}
\varphi(f)\] for every $p$-regular element $f$ in $C_G(x)$.
\end{lemma}
\begin{proof}
See \cite[15.47, p283]{Isaacs}.
\end{proof}
The algebraic integers $d^x_{\chi \varphi}$ are called \textit{generalized decomposition numbers}.
The following result is used in Section \ref{section case 2} when we have a group with restricted
$3$-local subgroups and it allows us to restrict our calculations entirely within the principal
$3$-block of the character table. The proof uses Brauer's first and second main theorems (see
\cite[p282, p284]{Isaacs}).

\begin{lemma}\label{brauer's theorems}
Let $G$ be a group, $1\neq x\in S \in \syl_p(G)$ and $\chi \in
\Irr(G) \bs B_0(G)$. Suppose $N_G(D)$ is $p$-soluble and
$O_{p'}(N_G(D))=1$ for each $1\neq D \leq S \in \syl_p(G)$. Then

\begin{enumerate}[$(i)$]
\item $B_0(G)$ is the only block of $G$ with non-trivial defect
group; and

\item $\chi(fx)= 0$ for each $p$-regular $f \in C_G(x)$.

\end{enumerate}
\end{lemma}
\begin{proof}
By \cite[15.40]{Isaacs}, for each $1\neq D \leq S \in \syl_p(G)$,
$N_G(D)$ has only one block and the block necessarily has defect
group $D_1 \in \syl_p(N_G(D))$. Hence if $D \neq S$ then $D \neq
D_1$ and by Brauer's first main theorem (\cite[15.45]{Isaacs}), $G$
has no block with defect group $D$. On the other hand, if $D=S$ then
by Brauer's first main theorem, $G$ has exactly one block with
defect group $S$ and this must be $B_0(G)$. Thus every block of $G$
has defect group $S$ or $1$. It therefore follows immediately from a
corollary to Brauer's second main theorem (\cite[15.49]{Isaacs})
that for $1 \neq x \in S$ and any $p$-regular $f \in C_G(x)$,
$\chi(fx)=0$.
\end{proof}

Given a group $G$ with $\Irr(G)=\{\chi_1,\hdots, \chi_n\}$ define a
\textit{column} of $G$ to be a sequence of numbers indexed by
$\Irr(G)$, $(a_i)_{i=1,..,n}$. For example given an element $x \in
G$ the column of the character table of $G$ corresponding to $x$
forms a column $(\chi(x))_{i=1,..,n}$ as do the columns of
generalized decomposition numbers if $x$ is a $p$-element and
$\varphi \in \mathrm{IBr}(C_G(x))$, $(d^x_{\chi_i
\varphi})_{i=1,..,n}$. We define the inner product of columns
$(a_i)_{i=1,..,n}$ and $(b_i)_{i=1,..,n}$ to be the usual dot
product $((a_i),(b_i))=\sum_{i=1}^n a_i b_i$. We further define a
\textit{($p$-)principal column} of $G$ to be  a sequence indexed by the
principal ($p$-)block characters of $G$.

We use the following lemma in the proof of  Theorem A to calculate part of the character table of
$G$. The method involves producing an invertible matrix $M$ that satisfies the hypothesis of this
lemma. In Section \ref{section case 2}, we restrict our character calculations to the principal
$3$-block of $G$ and we use this lemma in place of the Suzuki method.

\begin{lemma}\label{CT-Integer Matrix}
Let $G$ be a group and $H \leq G$ with $\{\psi_1,\hdots,\psi_r\}=:\mathrm{Irr}(H)$. Let
$\{\chi_1,\hdots,\chi_s\}\subseteq \mathrm{Irr}(G)$ and for any $n\leq r$ let $x_1,\hdots,x_n$ be
representatives in $H$ of any $n$ of the conjugacy classes of $H$. Set
\[N:=(\psi_i(x_j))_{1\leq i \leq r,\hspace{1pt} 1 \leq j \leq n}
\hspace{10pt} \mathrm{and} \hspace{10pt} L=(\chi_i(x_j))_{1\leq i
\leq s ,\hspace{1pt} 1 \leq j \leq n}.\] If $M=(m_{ij})\in
\mathrm{M}_n(\mathbb{C})$ is such that $NM$ is a matrix with integer
entries then $LM$ is a matrix with integer entries.
\end{lemma}
\begin{proof}
The $(i,j)$'th entry of $LM$ is
\[ \sum_{k=1}^n \chi_i(x_k)m_{kj}.\] We restrict $\chi_i$ to $H$ to
obtain integers $a_1, \hdots, a_r$ such that $\chi_i
\big|_H=a_1\psi_1 + \hdots + a_r \psi_r.$ Hence
\[\begin{array}{rcl}
\sum_{k=1}^n \chi_i(x_k)m_{kj} &=&\sum_{k=1}^n (a_1\psi_1 + \hdots
+ a_r \psi_r)(x_k) m_{kj}\\
\;&\;&\; \\
 \;&=& a_1\sum_{k=1}^n \psi_1(x_k)m_{kj} +
\hdots + a_r\sum_{k=1}^n \psi_r(x_k)m_{kj}.\end{array}\]This is an
integer sum of integer entries of the matrix $NM$. Therefore the
$(i,j)$'th entry of $LM$ is integral.
\end{proof}
Lemma \ref{CT-Integer Matrix} allows us to choose an invertible
matrix $M$ in a nice way such that $LM$ is an integer matrix with
few entries.  The idea is that we are able to calculate the column inner products in $L$ but not the specific entries. The matrix $K:=LM$ contains columns which are linear combinations of columns of $L$ chosen such that entries are integral. Provided the matrix $K$ is sufficiently sparse, we can determine possibilities for the matrix $K$ and
then calculate $KM\inv=L$. The procedure for making a suitable
choice for $M$ amounts to choosing $M$ to be a matrix of column
operations such that $NM$ has few entries and such that these
entries are as small as possible. We will demonstrate these ideas in Section \ref{section case 2}.


\section{The Hypothesis}\label{proof of main thm}

From now on we work under the hypothesis of Theorem A. Recall that a group $G$ is said to have a
$3$-local subgroup of $\alt(9)$-type if $P\leq G$ is elementary abelian of order $27$ and $N_G(P)$
is isomorphic to the normalizer in $\alt(9)$ of an elementary abelian subgroup of order $27$.

\begin{hypA}\label{alt9 hyp} Let $G$ be a finite group with $J\leq G$ such that $J$ is
elementary abelian of order $27$. Suppose $H:=N_G(J)$ is a $3$-local subgroup of $G$ of
$\alt(9)$-type and suppose $O_{3'}(C_G(x))=1$ for every element of order three $x \in H$.
\end{hypA}
The proof of Theorem A is highly character theoretic. Many of the calculations involve the
character table of $H$ (Table \ref{char table H}). We note in particular that each irreducible
character of $H$ gives integral values on $3$-elements in $H$. We prove in Lemma \ref{order of
sylow 3's and H is self normalizing} that $H$ contains a Sylow $3$-subgroup of $G$ and hence every
irreducible character of $G$ gives integral values on all $3$-elements. We label the conjugacy
classes of $H$ by $\CC_1,\hdots, \CC_{14}$ and we continue this notation throughout Section
\ref{proof of main thm}. Furthermore, we use the notation $\CC_i^G$ to be the conjugacy class in
$G$ containing $\CC_i$ for $1 \leq i \leq 14$. Observe that we may have $\CC_i^G=\CC_j^G$ for $i
\neq j$. For the purpose of understanding the group structure of $H$, we consider a permutation
representation of $H$ in $\alt(9)$:
\[\langle
(123),(456),(789),(147)(258)(369),(14)(2536)\rangle\] from which we
see that $J=C_H(J)=O_3(H)$. In this representation, elements in
$\CC_4$ are $3$-cycles, elements in $\CC_6$ have cycle shape
$3^21^3$ and elements in $\CC_5$ have cycle shape $3^3$.

\begin{table}[h]\[\begin{tabular}{|c|c|c|c|c|c|c|c|c|c|c|c|c|c|c|}
\hline
Class &  $\CC_1$& $\CC_2$&$\CC_3$&$\CC_4$&$\CC_5$&$\CC_6$&$\CC_7$&$\CC_8$&$\CC_9$& $\CC_{10}$&$\CC_{11}$&$\CC_{12}$&$\CC_{13}$ &$\CC_{14}$\\

$|C_H(x)|$&  648 &   24 &   12  &108&81 &54 &9 &12& 12& 6 &   9  & 9 & 12 & 12\\
Order &     1 &    2 &    2 &  3 & 3&  3& 3 & 4&  6 &   6 &  9 & 9 & 12&12\\
\hline
$\psi_1$  &   1 & 1 & 1 & 1 &  1 &  1 & 1 & 1 & 1 & 1 &   1 & 1 & 1&1\\

$\psi_2$  &   1 &1 &-1&  1&  1&  1&  1& -1&  1&  -1&  1&  1&  -1&
-1\\

$\psi_3$  &    2& 2 & 0&  2&  2&  2&-1&  0&  2&   0& -1& -1&   0&
0\\

$\psi_4$  &  3& -1& -1&  3&  3&  3& 0&  1& -1&  -1& 0&  0&   1&
1\\

$\psi_5$  &   3& -1& 1 & 3&  3&  3& 0 &-1 &-1 &  1 & 0 &0 & -1&
-1\\

$\psi_6$  &    6&  2& 0 & 3 &-3 & 0 & 0& -2& -1&   0&  0& 0 &  1&
1\\

$\psi_7$  & 6& 2&  0& 3& -3&  0 & 0 & 2& -1&   0&  0&  0& -1&  -1\\

$\psi_8$  &   6& -2&  0& 3 &-3 & 0 &0 & 0 & 1& 0&  0&  0&  $\sqrt{3}$& $-\sqrt{3}$\\

$\psi_{9}$  & 6 &-2&  0&  3& -3&  0&  0&  0& 1 &  0& 0 & 0
&$-\sqrt{3}$&
$\sqrt{3}$\\

$\psi_{10}$  &   8& 0&  0 &-4& -1& 2& -1&  0 & 0& 0& 2& -1& 0&   0\\

$\psi_{11}$  &   8&  0&  0& -4& -1&  2& -1& 0 & 0 &0 &-1 &2 &0  &
0\\

$\psi_{12}$  &  8 &0 & 0& -4& -1& 2 & 2&  0 & 0 &  0& -1 &-1&   0&
0\\

$\psi_{13}$  & 12& 0 &-2 & 0 & 3 &-3 &0 & 0 & 0&   1&  0& 0 &  0&
0\\

$\psi_{14}$  & 12&  0& 2&  0& 3& -3& 0 &0 & 0 &-1& 0&  0& 0& 0\\
\hline
\end{tabular}\]\caption{The character table of $H$.}\label{char table H}\end{table}

The following lemma is a consequence of $J$ being the Thompson
subgroup of a Sylow $3$-subgroup of $G$. However for completeness we
include a proof.
\begin{lemma}\label{J is characteristic}
$J=C_H(J)=O_3(H)$ is a characteristic subgroup of every $S \in
\syl_3(H)$ and for any $a,b \in J$, $a$ is conjugate to $b$ in $G$
if and only if $a$ is conjugate to $b$ in $H$.
\end{lemma}
\begin{proof}
We observe, by considering a representation of $H$ in $\alt(9)$,
that $H/J$ acts faithfully on $J$. Thus $J=C_H(J)=O_3(H)$. Now
suppose $J \neq J_0\leq S$ for $S \in \syl_3(H)$  with $J_0 \cong
J$. Then $S=JJ_0$ and $J \cap J_0\leq \mathcal{Z}(S)$ has order nine. However
again calculating in $\alt(9)$ gives $|\mathcal{Z}(S)|=3$. Thus $J$ is
characteristic in $S$.

Suppose $a^g=b$ for some $g \in G$. Then $J,J^g\leq C_G(b)$. Let
$P,Q \in \syl_3(C_G(b))$ such that $J\leq P$ and $J^g\leq Q$. By
Sylow's Theorem, there exists $x \in C_G(b)$ such that $Q^x=P$ and
so $J^{gx}=J$. Thus $gx \in H$ and $a^{gx}=b^x=b$ as required.
\end{proof}

\begin{lemma}\label{order of sylow 3's and H is self normalizing}
\begin{enumerate}[$(i)$]
\item $H=N_G(H)$.

\item  $N_G(S) \leq H$ for $S\in \syl_3(H)$, in particular, $\syl_3(H)
\subseteq \syl_3(G)$.

\item $\CC_i^G \cap
\CC_j^G =\emptyset$ for $\{i,j\}\subset \{4,5,6\}$.

\item $C_G(x) \leq H$ for $x \in \CC_5 \cup \CC_6 \cup
\CC_{10} \cup \CC_{11} \cup \CC_{12}$ and $\CC_{11}^G\neq \CC_{12}^G$.

\item If $T \leq J$ has order nine then $N_G(T) \leq H$ and $C_G(T)
=J$.

\item Fix $x \in \CC_4$ then $H$ has two conjugacy classes of
subgroups of order nine containing $x$, $P_1^H$ and $P_2^H$ say
where $|P_1 \cap \CC_4|=|P_1 \cap \CC_6|=4$ and $|P_2 \cap
\CC_4|=|P_2 \cap \CC_6|=2$ and $|P_2 \cap \CC_5|=4$.
\end{enumerate}\end{lemma}
\begin{proof}
Since $J=O_3(H)$ and $H=N_G(J)$,  $H$ is self normalizing in $G$.
Also, since $J$ is characteristic in $S$ by Lemma \ref{J is
characteristic},  $N_G(S) \leq N_G(J)=H$ and so $S \in \syl_3(G)$.
Therefore $(i)$, $(ii)$ and $(iii)$ follow
immediately from Lemma \ref{J is characteristic}.

Let $x \in \CC_6$ and set $X:=C_G(x)$. Then $J \in \syl_3(X)$. Observe that $|N_X(J)|=|H \cap
C_G(x)|=3^32$. Let $\bar{X}=C_G(x)/\<x\>$ and let $t$ be an involution in $C_H(x)$ then
$C_J(t)=\<x\>$ so by coprime action, $\bar{t}$ acts fixed-point-freely on $\bar{J}$. Therefore
\[|N_{\bar{X}}(\bar{J})/C_{\bar{X}}(\bar{J})|=|N_{\bar{X}}(\bar{J})/\bar{J}|=2\]
and so $\bar{X}$ satisfies the hypothesis of Theorem \ref{Smith-Tyrer}. By Hypothesis A,
$O_{3'}(X)=1$ and so $O_{3'}(\bar{X})=1$. Moreover $\bar{J}=[\bar{J},\bar{t}] \leq O^3(\bar{X})$.
Hence $\bar{X}$ is $3$-soluble of length one with trivial $3'$-radical. Therefore
$\bar{X}=N_{\bar{X}}(\bar{J})$ and so $X \leq H$.

Let $T\leq J$ have order nine. A calculation in $H$ verifies that
$C_H(T)=J$ for each such $T$. Notice that every choice of $T$
contains an element in $\CC_6$ and so $C_G(T)\leq H$. Thus
$C_G(T)=J$. It is therefore immediate that $N_G(T)$ normalizes $J$
and so $N_G(T)\leq H$. This proves $(v)$.

Now let $w \in \CC_5$ and set $W:=C_G(w)$ and $\wt{W}=W/\<w\>$. Then $w\in \mathcal{Z}(S)$ for some
$S \in \syl_3(H)\subseteq \syl_3(G)$. We calculate in the image of $S$ in $\alt(9)$ to see that $|S'|=9$ and every element of order nine in $S$ cubes into $\<w\>$.  Therefore $\wt{S}$ is non-abelian of exponent three and so $\wt{S}\cong 3_+^{1+2}$. Furthermore, $N_G(S) \cap W\leq H \cap W=S$ and so $\wt{W}$ has a
self-normalizing and non-abelian Sylow $3$-subgroup isomorphic to $3_+^{1+2}$.
Let $T<S$ such that $w \in T$ and
$\wt{T}=\mathcal{Z}(\wt{S})$  then $|T|=9$ so $N_G(T)\leq H$. Therefore, $N_W(T)\leq H \cap W=S$
and so $N_{\wt{W}}(\mathcal{Z}(\wt{S}))=\wt{S}$.  Thus $\wt{W}$ satisfies the hypothesis of Lemma
\ref{self-normalizing extraspecial 3^3}. However, by hypothesis, $O_{3'}(W)=1$ and so
$O_{3'}(\wt{W})=1$. Thus $\wt{W}=\wt{S}$ and so $W=S \leq H$.

Now, if $v \in  \CC_{10} \cup \CC_{11} \cup \CC_{12}$ then either $v^2$ or $v^3$ is in $\CC_5 \cup
\CC_6$  and so it follows immediately that $C_G(v)\leq H$. So suppose that $\CC_{11}^G=\CC_{12}^G$
then let $a \in \CC_{11}$ and $b \in \CC_{12}$ such that $a^3=b^3\in \CC_5$. Then there exists $g
\in G$ such that $a^g=b$ and so $(a^3)^g=(a^g)^3=b^3$ and so $g \in C_G(a^3)\leq H$. Which implies
that $\CC_{11}= \CC_{12}$. This contradiction completes the proof of $(iv)$.

Finally to prove $(vi)$ we allow $x$ to be represented in $\alt(9)$ by $(123)$. Any group of order nine
containing $x$ necessarily centralizes $x$ and so is a subgroup of $J$. Therefore we need only
consider $P_1=\<(123),(456)\>$ and $P_2=\<(123),(456)(789)\>$ and count the $H$-orbits of the
subgroups of $J$.
\end{proof}

\begin{lemma}
If $x \in \CC_4$ then $C_G(x) \leq H$ or $C_G(x)\cong 3 \times
\alt(6)$.
\end{lemma}
\begin{proof}
Suppose $C_G(x) \nleq H$ and set $\bar{X}:=C_G(x)/\<x\>$. Since $x$ is not $3$-central in $G$,
$C_G(x)$ has Sylow $3$-subgroups of order $3^3$ and so $\bar{X}$ has Sylow $3$-subgroups of order
nine. Observe that $\bar{J} \in \syl_3(\bar{X})$ and $|N_{\bar{X}}(\bar{J})|=3^22^2$. We show that
$\bar{X}$ must be simple. So let $\bar{N}$ be a minimal normal subgroup of $\bar{X}$ then $\bar{N}$
is a direct product of isomorphic simple groups. If $3 \nmid |\bar{N}|$ then $O_{3'}(C_G(x))\neq 1$
which is not possible. Therefore $3\mid |\bar{N}|$. Now $\bar{X}$ has Sylow $3$-subgroups of order
nine and so either $\bar{N}$ is simple or $\bar{N}$ is a direct product of two isomorphic simple
groups each with Sylow $3$-subgroups of order three. Suppose the latter then $\bar{N}=\bar{N_1}
\times \bar{N_2}$ where $\bar{N_1} \cong \bar{N_2}$ are simple groups. Choose $\bar{T}\in
\syl_3(\bar{N_2})$ then $[\bar{T},\bar{N_1}]=1$. Let $T$ and $N_1$ be the preimages in $C_G(x)$ of
$\bar{T}$ and $\bar{N_1}$ respectively. Then $N_1$ splits over $\<x\>$ by Gasch\"{u}tz's Theorem
(\ref{Gaschutz}). Let $M_1$ be a complement to $\<x\>$ in $N_1$ then $M_1\cong \bar{N_1}$ is simple
and normalized by $T$. Since $[T,N_1] \leq \<x\>$, $[T,M_1] \leq \<x\> \cap M_1=1$. Thus $M_1 \leq
C_G(T)$. Now $T$ is conjugate to a subgroup of $J$ of order nine and so by Lemma \ref{order of
sylow 3's and H is self normalizing} $(v)$, $M_1\leq H^g$ for some $g \in G$. However $M_1$ is
simple and $H^g$ is soluble which implies $M_1$ has prime order. Since $3\mid |M_1|$, $M_1\cong
\bar{N_1}$ is cyclic of order three. Therefore $\bar{N}$ has order nine and so $\bar{X}$ has a
normal Sylow $3$-subgroup $\bar{J}$ and so $C_G(x) \leq H$ contradicting our assumption.

So we may assume $\bar{N}$ is simple. Let $f \in C_H(x)$ have order four then
$C_J(f)=C_J(f^2)=\<x\>$. By coprime action, $\bar{f}$ and $\bar{f^2}$ act fixed-point-freely on
$\bar{J}$. Suppose $\bar{N}$ has Sylow $3$-subgroups of order three. Then $\bar{J}\cap \bar{N}$ has
order three. However $\bar{f}$ has order four which implies that $\bar{J}=[\bar{J} \cap
\bar{N},\bar{f}]\leq \bar{N}$. Thus $\bar{J}$ is a Sylow $3$-subgroup of $\bar{N}$. Also since $J$
has more than one conjugacy class of subgroup of order nine containing $x$ by Lemma \ref{order of
sylow 3's and H is self normalizing} $(vi)$, $\bar{J}$ has more than one conjugacy class of order
three. Let $1 \neq \bar{j}\in \bar{J}$ and suppose that $\bar{a} \in \bar{N}$ is a $3'$-element and
centralizes $\bar{j}$. Then $\<a\>$ normalizes $\<j,x\>\leq J$ which has order nine. By Lemma
\ref{order of sylow 3's and H is self normalizing} $(v)$, $a\in H$ has even order. Let $t \in
\<a\>$ be an involution then we have seen that $\bar{t}$ acts fixed-point-freely on $\bar{J}$. This
contradicts our assumption that $\bar{t}$ centralizes $ \bar{j}$. Thus
$C_{\bar{X}}(\bar{j})=\bar{J}$ and so $\bar{N}$ satisfies Theorem \ref{Higman Alt6} so
$\bar{N}\cong \mathrm{Alt}(6)$. In particular, $N_{\bar{X}}(\bar{J})\leq \bar{N}$. Since
$\bar{N}\trianglelefteq \bar{X}$ and $\bar{J} \in \syl_3(\bar{N})$, we may use a Lemma
\ref{frattini} (Frattini argument) to write $\bar{X}=\bar{N} N_{\bar{X}}(\bar{J})=\bar{N}$.
Therefore $\bar{X}=\bar{N}\cong \alt(6)$ and so by Gasch\"{u}tz's Theorem (\ref{Gaschutz}), $C_G(x)
\cong 3 \times \alt(6)$.
\end{proof}

\begin{lemma}\label{3-elements in S not J}
If $x \in \CC_7$ then either $C_G(x) \leq H$ or $\CC_7^G=\CC_5^G$.
If $\CC_7^G=\CC_5^G$ then for $S\in\syl_3(H)$ and $T \in
\syl_2(N_H(S))$, $C_G(T)$ contains a subgroup isomorphic to
$\SL_2(3)$ and $T^\# \subset \CC_3$.
\end{lemma}
\begin{proof}
We see from the character table of $H$ (Table \ref{char table H})
that $C_H(x)$ has order nine and therefore $C_H(x)=\<x,z\>$ where $z
\in \mathcal{Z}(S)$ for some $S \in \syl_3(H)$. Moreover $C_H(x)=\<x,z\>$
contains two $H$-conjugates of $z$ and six $H$-conjugates of $x$. If
$x$ is not conjugate in $G$ to $z$ then $N_{C_G(x)}(\<z,x\>) \leq
N_G(\<z\>)\cap C_G(x)\leq H\cap C_G(x)=\<z,x\>$. Thus $C_G(x)$ has a
Sylow $3$-subgroup of order nine which is self-normalizing.
Therefore $C_G(x)$ has a normal $3$-complement. However, by
hypothesis, $O_{3'}(C_G(x))=1$. We conclude that $C_G(x)=\<z,x\>$ or
$x$ is conjugate to $z$.

Assume $\CC_7^G=\CC_5^G$ and let $S \in \syl_3(H)$ then $N_H(S)=N_G(S)=N_G(\mathcal{Z}(S))$ has
order $3^42$ and contains involutions from the $H$-conjugacy class $\CC_3$. Let $a \in
\mathcal{Z}(S)^\#$ and $b \in S\bs J$ have order three then $b \in \CC_7$. By assumption, $b$ is
conjugate in $G$ to $a$. Set $X:=\<a,b\>$. Let $g\in H$ such that $a^g=b$. Then $b \in
\mathcal{Z}(S^g)$ and $S=C_G(a)$ implies $S^g=C_G(b)$. Since $C_S(b)=X$, $S \cap S^g=X$. So let
$A:=N_S(X)$ and $B:=N_{S^g}(X)$ and set $Y:=\<A,B\>$. Observe that $X$ is self-centralizing in $G$
and $A$ and $B$ are distinct. Thus $Y/X\cong \SL_2(3)$. Let $TX/X= \mathcal{Z}(Y/X)$ where $T \cong
2$. Then $T$ inverts $a$ and so $T\in \syl_2(N_G(S))\leq H$. Let $1 \neq t \in T$ then $t\in
\CC_3^G$. By coprime action and an isomorphism theorem, $\SL_2(3)\cong C_{Y/X}(TX/X)\cong
C_{Y}(T)X/X\cong C_{Y}(T).$ Since $T \in \syl_2(N_G(S))$, the result is true for every subgroup in
$\syl_2(N_G(S))$.
\end{proof}

We have three scenarios to consider in more detail in the following three subsections.

\noindent $\bullet$ Case 1: For $x \in \CC_7$, $C_G(x) \leq H$.

\noindent $\bullet$ Case 2: For $x \in \CC_4$, $C_G(x) \leq H$ and $\CC_7^G=\CC_5^G$.

\noindent $\bullet$ Case 3:  For $x \in \CC_4$, $C_G(x) \cong 3 \times \alt(6)$
and $\CC_7^G=\CC_5^G$.


\subsection{Case 1}
In this case we  hypothesize that for $x \in \CC_7$, $C_G(x) \leq H$. We do not need to consider
the possibilities for $C_G(y)$ for $y \in\CC_4$ since the assumption on $\CC_7$ proves to be very
powerful. By Lemma \ref{3-elements in S not J}, a more succinct way to describe this scenario is to
hypothesize in addition to Hypothesis A that, as sets, $J^G \cap H=J$. Throughout this section we
assume $G$ satisfies the following hypothesis.
\begin{hyp}\label{case 1 hyp}
Let $G$ satisfy Hypothesis A and in addition assume that, as sets, $J^G \cap H=J$.
\end{hyp}
\begin{thm}\label{case 1 thm}
If $G$ satisfies Hypothesis \ref{case 1 hyp} then $G=H$.
\end{thm}

We suppose for a contradiction that $G\neq H$. In particular,
$J\ntrianglelefteq G$.

\begin{lemma}\label{index of H in G}
$|G|/|H| \geq 28$.
\end{lemma}
\begin{proof}
The index of $H$ in $G$ equals the number of conjugates of $J$ in
$G$. So consider the action of $J$ on the set, $\Omega:=\{J^g|g \in
G \bs H\}$,  of its distinct conjugates in $G$. Since $H \neq G$,
this set is non-empty. Moreover, if $j \in J$ fixes some $J^g \in
\Omega$ then $j \in H^g \cap J=1$. Hence $|\Omega| \geq 27$.
\end{proof}

\begin{lemma} $\CC:=\CC_6\cup  \CC_7\cup
\CC_{11}\cup \CC_{12}$ is a set of special classes in $H$.
\end{lemma}
\begin{proof}
By Lemma \ref{order of sylow 3's and H is self normalizing} and Hypothesis \ref{case 1 hyp},
$C_G(x) \leq H$ for each $x \in\CC$. Now suppose for $i \in \{6,7,11,12\}$, $\mathcal{C}_i^G \cap
\mathcal{C}\neq \mathcal{C}_i$. Then there exists $x \in \CC_i$, $y \in \CC_j$ ($i \neq j$) such
that $x^g=y$ for some $g \in G$. However this implies that $C_G(x)\cong C_G(y)$ and so the only
possibility is that $\{i,j\}=\{11,12\}$. However by Lemma \ref{order of sylow 3's and H is self
normalizing}  $(iv)$, $\CC_{11}^G\neq \CC_{12}^G$. Finally we need to satisfy condition $(iii)$ of
Definition \ref{defi-special classes}. However this is immediate since any element in  $\CC_6$,
$\CC_7$ is conjugate in $H$ to its inverse and every element of order nine in $H$ is contained in
the set.
\end{proof}

\begin{lemma}\label{Structure constants 1}
Suppose that $a \in \CC_6$, $b \in \CC_7$ and $c \in \CC_6 \cup \CC_7$ such that for some $h,g \in G$,  $a^hb^g=c$. Then $X:=\<a^h,b^g\>
\leq H$.
\end{lemma}
\begin{proof}
By Lemma \ref{triangle group}, there exists $z \in X$ of order three
such that $X\leq C_G(z)$. Since $z$ commutes with $c \in \CC_6 \cup
\CC_7$, $z \in C_G(c) \leq H$. Since $z$ commutes with $a^h\in J^h$,
$z \in J^h \cap H\leq J$.  If however $z \in \CC_4$, then since $[z,b^g]=1$, $z \in H^g \cap J\leq J^g$. So $z \in \mathcal{Z}(S^g)$ and then $z \in \CC_5^G$ which is a contradiction. So $z \in \CC_5 \cup \CC_6$ and since $a^h\in J$, $X \leq H$.
\end{proof}

Recall from Section \ref{CharacterTheory-prelims} the definition of the structure constants (labeled $a^G_{xyz}$ and $\a^G_{xyz}$ for $G$ a group with elements $x,y,z$).
\begin{lemma}\label{almost structure constants}
Let $y \in \CC_6$ and $z \in \CC_7$. Then
\begin{enumerate}[$(i)$]
    \item $\a^G_{yzy}=0$; and
    \item $\a^G_{yzz}=\frac{2^23^6}{|G|}$.
\end{enumerate}
\end{lemma}
\begin{proof}
By Lemma \ref{Structure constants 1}, the number of pairs $(a,b) \in
y^G \times z^G$ such that $ab=y$ equals the number of pairs $(a,b)
\in y^H \times z^H$ such that $ab=y$ and from the character table of
$H$ we calculate this number to be
$a^G_{yzy}=a^H_{yzy}=\frac{2^33^4}{3^4.3^2}.0=0$. Hence
$\a^G_{yzy}=0$.

By the same argument we have
$a^G_{yzz}=a^H_{yzz}=\frac{2^33^4}{3^32.3^2}.\frac{9}{2}=6$ since
$C_G(y)=C_H(y)$ and $C_G(z)=C_H(z)$. Hence
$\a^G_{yzz}=6\frac{3^32.3^2}{|G|}$.
\end{proof}

We now apply Suzuki's Theory  with the set of special classes $\CC:=\CC_6\cup \CC_7 \cup
\CC_{11} \cup \CC_{12}$. We begin by finding a basis for the space
$W=\{\phi \in \mathrm{CF}(H)\mid \phi(h)=0 ~\mathrm{for~ all}~ h \in H
\bs \mathcal{C}\}$. By Lemma \ref{special classes 2}, $W$ has
dimension four over $\mathbb{C}$.

\begin{enumerate}[$(i)$]
\item $\lambda_1=\psi_1+\psi_2-\psi_3$;
\item $\lambda_2=\psi_3+\psi_4+\psi_5+\psi_{10}+\psi_{11}-\psi_{13}-\psi_{14}$;
\item $\lambda_3=\psi_{10}-\psi_{12}$;
\item $\lambda_4=\psi_{11}-\psi_{12}$.
\end{enumerate}

We calculate $\gamma_i:=\sum_{j=1}^{14} \psi_j(x_i){\psi_j}$ where
$x_i \in \CC_i$ for $i=6,7,11,12$ and write each as a linear
combination of the class functions $\{\lambda_1,\hdots,\lambda_4\}$
to give the following:
\begin{enumerate}[$(i)$]
\item $\g_6=\l_1+3\l_2-\l_3-\l_4$;
\item $\g_7=\l_1-\l_3-\l_4$;
\item $\g_{11}=\l_1+2\l_3-\l_4$; and
\item $\g_{12}=\l_1-\l_3+2\l_4$.
\end{enumerate}
Thus we have the matrix:
\[C:=\left(
  \begin{array}{cccc}
    1 & 3 & -1 & -1  \\
    1 & 0 & -1 & -1  \\
    1 & 0 & 2 & -1 \\
    1 & 0 & -1 & 2  \\
  \end{array}
\right).\]

Let $\theta_1=1_G,\theta_2,\hdots,\theta_s$ be the irreducible
characters of $G$ where the numbering is not fixed except that
$\theta_1$ always represents the principal character. Let
$\mu_i=\lambda_i^G$ for $i=1,\hdots,4$. Table \ref{inner prods mu
's} gives the pairwise inner products $(\mu_i,\mu_j)_G$ using Lemma
\ref{special classes 1}.

\begin{table}[h]\[\begin{tabular}{|c|c c c c|}
  \hline
  $(,)_G$ & $\m_1$ & $\m_2$ & $\m_3$ & $\m_4$ \\
  \hline
  $\m_1$ & 3 &   &   &     \\
  $\m_2$ & -1 & 7 &   &    \\
  $\m_{3}$ & 0 & 1 & 2 &    \\
  $\m_{4}$ & 0 & 1 & 1 & 2   \\
  \hline
\end{tabular}\]\caption{The table of inner products $(\mu_i,\mu_j)_G$ for $1\leq i,j \leq
4$.}\label{inner prods mu 's}\end{table} We also have, using Lemma \ref{special classes 1} that
$\mu_i(1)=\lambda_i(1)=0$ for each $1\leq i\leq 4$ and
$(\mu_i,\theta_1)=0$ for $i=2,3,4$ and $(\mu_1,\theta_1)=1$. We
express $\mu_1$, $\mu_{3}$ and $\mu_{4}$ as linear combinations of
irreducible characters of $G$.
\begin{lemma}\label{Alt9-case1-mu1,3,4}
\begin{enumerate}[$(i)$]
    \item $\mu_{1}=\theta_1+\d\theta_5-\d\theta_6$;
    \item $\mu_{3}=\e\theta_2-\e\theta_3$;
    \item $\mu_{4}=\e\theta_4-\e\theta_3$;
\end{enumerate}
where $\e,\d=\pm 1$.
\end{lemma}
\begin{proof}
Since $\mu_3$ and $\mu_4$ each involve exactly two irreducible characters of $G$ and have inner
product $1$, it is clear that we can write $\mu_{3}=\e_2\theta_2+\e_3\theta_3$ and
$\mu_{4}=\e_4\theta_4+\e_3\theta_3$ ($\e_2,\e_3,\e_4=\pm 1$). However since $\mu_3(1)=\mu_4(1)=0$,
we have $\e:=\e_2=\e_4=-\e_3$. Now since $\mu_1$ involves two irreducible characters together with
the principal character and has inner product 0 with $\mu_3$ and $\mu_4$, we have
$\mu_{1}=\theta_1+\e_5\theta_5+\e_6\theta_6$ ($\e_5,\e_6=\pm 1$). Suppose $\e_5=\e_6$. Then
$0=\mu_1(1)=1+\e_5(\theta_6(1)+\theta_5(1))$. However $\theta_6(1)+\theta_5(1)$ is an integer
greater than $1$ and so we may take $\d:=\e_5=-\e_6$.
\end{proof}

Using Suzuki's Theorem, we can now calculate part of the character
table of $G$. So far we have not induced the character $\mu_2$ to
$G$ and so we let $b_i$ represent unknown constants such that
$\mu_2=\sum_{1 \leq i \leq s} b_i\theta_i$ in the following matrix:

\[B=\left(
  \begin{array}{cccccccc}
    1 & 0 & 0 & 0 & \d & -\d & 0 & \hdots \\
    0 &  b_2 & b_3 & b_4 & b_5 & b_6  & b_7 & \hdots \\
    0 & \e & -\e & 0 & 0 & 0 & 0  & \hdots\\
    0 & 0 & -\e & \e & 0 & 0 & 0  & \hdots \\
  \end{array}
\right)\] Therefore we calculate a portion of the character table
$(CB)^t$ (Table \ref{part of char table}) and we let
$d_i=\theta_i(1)$ and we avoid calculating the entries $\theta_i(x)$
for $x \in \CC_6$ for the moment.

\begin{table}[h]\[\begin{tabular}{|c|c |r r r r|}
  \hline
  \; & 1 & $\CC_6^G$ & $\CC_7^G$ & $\CC_{11}^G$ & $\CC_{12}^G$ \\
  \hline
  $\theta_1$   &     1 & 1 &    $1$ & 1      &  1       \\
  $\theta_2$   & $d_2$ &   & $-\e$  &  $2\e$ &  $-\e$   \\
  $\theta_{3}$ & $d_3$ &   & $2\e$  &  $-\e$ &  $-\e$   \\
  $\theta_{4}$ & $d_4$ &   & $-\e$  &  $-\e$ & $2\e$    \\
  $\theta_{5}$ & $d_5$ &   & $\d$   &  $\d$  & $\d$     \\
  $\theta_{6}$ & $d_6$ &   & $-\d$  &  $-\d$ & $-\d$    \\
  $\theta_{7}$ & $d_7$ &   & 0      & 0      &   0      \\
  $\vdots$ & $\vdots$ &   & $\vdots$  & $\vdots$& $\vdots$   \\
  $\theta_{s}$ & $d_s$ &   & 0      & 0      &   0      \\
  \hline
\end{tabular}\]\caption{Part of the character table of
$G$}\label{part of char table}\end{table}

\begin{lemma}\label{degrees of irreducibles for G}
$d:=d_2=d_3=d_4$ and $1+\d d_5-\d d_6=0$.
\end{lemma}
\begin{proof}
By Lemma \ref{special classes 1}, $\mu_1(1)=\mu_3(1)=\mu_4(1)=0$. Using Lemma
\ref{Alt9-case1-mu1,3,4} we see that $\e d_2-\e d_3=\e d_3-\e d_4=0$  and $1+\d d_5- \d d_6=0$ and
so $d:=d_2=d_3=d_4$ and $1+\d d_5-\d d_6=0$.
\end{proof}

We calculate structure constants for $G$ which involve the
$G$-conjugacy class  $\CC_7^G$ and the class $\CC_6^G$. Hence we
only need to know some of the character values for $\CC_6^G$
provided we know all the character values for $\CC_7^G$. Therefore
we need to calculate $\theta_i(x)$ ($x \in \CC_6$) only for
$i=2,3,4,5,6$.

\begin{lemma}\label{5 cases-the unknowns a2-a6 in B}
Let $b_2,b_3,b_4,b_5,b_6$ be the constants as in the matrix $B$. One
of the following hold.
\begin{enumerate}[$(i)$]
    \item $(b_2,b_3,b_4,b_5,b_6)=(\e,0,\e,-\d,0)$;
    \item $(b_2,b_3,b_4,b_5,b_6)=(\e,0,\e,-2\d,-\d)$;
    \item $(b_2,b_3,b_4,b_5,b_6)=(0,-\e,0,-\d,0)$;
    \item $(b_2,b_3,b_4,b_5,b_6)=(0,-\e,0,-2\d,-\d)$; or
    \item $(b_2,b_3,b_4,b_5,b_6)=(-\e,-2\e,-\e,-\d,0)$.
\end{enumerate}
\end{lemma}
\begin{proof}
We begin to induce $\l_2$ to $G$. Since
$(\mu_2,\mu_3)=(\mu_2,\mu_4)=1$ we have that $\mu_2$ either involves
$\e\theta_2+\e\theta_4$ or $-\e\theta_3$ or
$-\e\theta_2-2\e\theta_3-\e\theta_4$. In the first case in order to
satisfy $(\mu_2,\mu_1)=-1$ we have that $\mu_2$ involves either
$-\d\theta_5$ or $-2\d\theta_5-\d\theta_6$. In the second case we
again see that to satisfy $(\mu_2,\mu_1)=-1$, $\mu_2$ involves
either $-\d\theta_5$ or $-2\d\theta_5-\d\theta_6$. Finally in the
third case the only possibility is that
$\mu_2=-\e\theta_2-2\e\theta_3-\e\theta_4-\d\theta_5$.
\end{proof}

We now calculate using Lemma \ref{almost structure constants} which
says that \[0=\a^G_{yzy}= \sum_{1 \leq i \leq s} \frac{\theta_i(y)
\theta_i(z)\bar{\theta(z)}}{\theta_i(1)}=\sum_{1 \leq i \leq 6}
\frac{\theta_i(y) \theta_i(z){\theta(z)}}{\theta_i(1)}\] and

\[\frac{2^23^6}{|G|}=\a^G_{yzz}= \sum_{1 \leq i \leq s} \frac{\theta_i(y)
\theta_i(z)\bar{\theta(z)}}{\theta_i(1)}=\sum_{1 \leq i \leq 6}
\frac{\theta_i(y) \theta_i(z){\theta(z)}}{\theta_i(1)}\] where $y
\in \CC_6$ and $z \in \CC_7$. Notice that all $H$-characters are
integral on elements in $\CC_6$ and in $\CC_7$. Therefore all $G$-characters
are integral on $\CC_6$ and $\CC_7$ also. We now consider each of
the five cases described in Lemma \ref{5 cases-the unknowns a2-a6 in
B}.

\subsubsection{Candidate 1:
$\bf (b_2,b_3,b_4,b_5,b_6)=(\e,0,\e,-\d,0)$}

The missing entries from the character table are displayed in Table
\ref{suzuki-case i}.
\begin{table}[h]\[\begin{tabular}{|c|c |c|c|}
  \hline
  \; & 1 & $\CC_6^G$  \\
  \hline
  $\theta_1$   &     1 & 1       \\
  $\theta_2$   & $d$ &  $2\e$    \\
  $\theta_{3}$ & $d$ &  $2\e$    \\
  $\theta_{4}$ & $d$ &  $2\e$    \\
  $\theta_{5}$ & $d_5$ &  $-2\d$     \\
  $\theta_{6}$ & $d_6$ &  $-\d$     \\
  \hline
\end{tabular}\]\caption{A candidate for part of the character table
of G.}\label{suzuki-case i}\end{table} So we calculate (where $y \in
\CC_6$ and $z \in \CC_7$):
\[0=\a^G_{yzy}=1-\frac{4\e}{d}+\frac{4.2\e}{d}-\frac{4\e}{d}+\frac{4\d}{d_5}-\frac{\d}{d_6}=1+\frac{4\d}{d_5}-\frac{\d}{d_6}.\]
We now rearrange using the relation $1-\d d_6=-\d d_5$ (by Lemma
\ref{degrees of irreducibles for G}) to get $0=1+d_5d_6 +3\d d_6$.
Since $d_5$ and $d_6$ are positive, we have $\d=-1$ and so
$d_5=1+d_6$ and we simplify further to get
$0=1+d_6(1+d_6)-3d_6=d_6^2-2d_6+1$. This quadratic in $d_6$ has the
repeated root $d_6=1$ and so $d_5=2$. Now we calculate:
\[\frac{2^23^6}{|G|}=\a^G_{yzz}= 1+
\frac{2\e}{d}+\frac{2\e}{d}+\frac{8\e}{d}+\frac{2}{2}+\frac{1}{1}=3+\frac{12\e}{d}.\]
We simplify to get $|G|(d+4\e)=2^23^5d$. Since $d \neq0$, $d+4\e
\neq 0$. Therefore we rearrange to see that $|G|/|H|=3d/2(d+4\e)$.
By Lemma \ref{index of H in G}, $|G|/|H|\geq 28$. We rearrange
$3d/2(d+4\e)\geq 28$ to get $53d\leq -224\e$. Therefore $\e=-1$ and
$0<d\leq 4$. However this is a contradiction since this implies
$d+4\e< 0$ and so  $|G|=2^23^5d/(d+4\e)<0$.

\subsubsection{Candidate 2:
$\bf (b_2,b_3,b_4,b_5,b_6)=(\e,0,\e,-2\d,-\d)$} The missing entries
from the character table are displayed in Table \ref{suzuki-case
ii}.
\begin{table}[h]\[\begin{tabular}{|c|c |c |}
  \hline
  \; & 1 & $\CC_6^G$  \\
  \hline
  $\theta_1$   &     1 & 1       \\
  $\theta_2$   & $d$ &  $2\e$    \\
  $\theta_{3}$ & $d$ &  $2\e$    \\
  $\theta_{4}$ & $d$ &  $2\e$    \\
  $\theta_{5}$ & $d_5$ &  $-5\d$     \\
  $\theta_{6}$ & $d_6$ &  $-4\d$     \\
  \hline
\end{tabular}\]\caption{A candidate for part of the character table
of G.}\label{suzuki-case ii}\end{table} So we calculate (where $y
\in \CC_6$ and $z \in \CC_7$):
\[0=\a^G_{yzy}=1-\frac{4\e}{d}+\frac{4.2\e}{d}-\frac{4\e}{d}+\frac{25\d}{d_5}-\frac{16\d}{d_6}=1+\frac{25\d}{d_5}-\frac{16\d}{d_6}.\]
As before we rearrange and use $1-\d d_6=-\d d_5$ to get
$0=16+d_5d_6 +9\d d_6$. Now since $d_5$ and $d_6$ are positive, we
have $\d=-1$ and so $d_5=1+d_6$ and we simplify further to get
$0=16+d_6(1+d_6)-9d_6=d_6^2-8d_6+16$. This quadratic in $d_6$ has
the repeated root $d_6=4$ and so $d_5=5$. Now we calculate:
\[\frac{2^23^6}{|G|}=\a^G_{yzz}= 1+
\frac{2\e}{d}+\frac{8\e}{d}+\frac{2\e}{d}+\frac{5}{5}+\frac{4}{4}=3+\frac{12\e}{d}.\]
We simplify to get $|G|(d+4\e)=2^23^5d$ or $|G|=2^23^5d/(d+4\e)$.
Since $d \neq0$, $d+4\e \neq 0$. Therefore we rearrange to see that
$|G|/|H|=3d/2(d+4\e)$. By Lemma \ref{index of H in G}, $|G|/|H|\geq
28$. We rearrange $3d/2(d+4\e)\geq 28$ to get $53d\leq -224\e$.
Therefore $\e=-1$ and $0<d\leq 4$. However this is a contradiction
since this implies $d+4\e< 0$ and so  $|G|=2^23^5d/(d+4\e)<0$.

\subsubsection{Candidate 3:
$\bf (b_2,b_3,b_4,b_5,b_6)=(0,-\e,0,-\d,0)$}

The missing entries from the character table are displayed in Table
\ref{suzuki-case iii}.
\begin{table}[h]\[\begin{tabular}{|c|c |c |}
  \hline
  \; & 1 & $\CC_6^G$  \\
  \hline
  $\theta_1$   &     1 & 1       \\
  $\theta_2$   & $d$ &  $-\e$    \\
  $\theta_{3}$ & $d$ &  $-\e$    \\
  $\theta_{4}$ & $d$ &  $-\e$    \\
  $\theta_{5}$ & $d_5$ & $-2\d$     \\
  $\theta_{6}$ & $d_6$ & $-\d$     \\
  \hline
\end{tabular}\]\caption{A candidate for part of the character table
of G.}\label{suzuki-case iii}\end{table} So we calculate (where $y
\in \CC_6$ and $z \in \CC_7$):
\[0=\a^G_{yzy}=1-\frac{\e}{d}+2\frac{\e}{d}-\frac{\e}{d}
+\frac{4\d}{d_5}-\frac{\d}{d_6}=1+\frac{4\d}{d_5}-\frac{\d}{d_6}.\]
Again we rearrange and substitute for $\d d_5$ to get $0=1+d_5d_6
+3\d d_6$ and again see that  $\d=-1$ and so $d_5=1+d_6$ and so we
simplify further to get $0=1+d_6(1+d_6)-3d_6=d_6^2-2d_6+1$. This
quadratic in $d_6$ has the repeated root $d_6=1$ and so $d_5=2$. Now
we calculate:
\[\frac{2^23^6}{|G|}=\a^G_{yzz}= 1-
\frac{\e}{d}-\frac{4\e}{d}-\frac{\e}{d}+\frac{2}{2}+\frac{1}{1}=3-\frac{6\e}{d}.\]
We simplify to get $|G|(d-2\e)=2^23^5d$. Since $d \neq0$, $d-2\e
\neq 0$. Therefore we rearrange to see that $|G|/|H|=3d/2(d-2\e)$.
By Lemma \ref{index of H in G}, $|G|/|H|\geq 28$. We rearrange
$3d/2(d-2\e)\geq 28$ to get $53d\leq 112\e$. Therefore $\e=1$ and
$0<d\leq 2$. However this is a contradiction since this implies
$d-2\e< 0$ and so  $|G|=2^23^5d/(d-2\e)<0$.

\subsubsection{Candidate 4:
$\bf (b_2,b_3,b_4,b_5,b_6)=(0,-\e,0,-2\d,-\d)$}

The missing entries from the character table are displayed in Table
\ref{suzuki-case iv}.
\begin{table}[h]\[\begin{tabular}{|c|c |c |}
  \hline
  \; & 1 & $\CC_6^G$  \\
  \hline
  $\theta_1$   &     1 & 1       \\
  $\theta_2$   & $d$ &  $-\e$    \\
  $\theta_{3}$ & $d$ &  $-\e$    \\
  $\theta_{4}$ & $d$ &  $-\e$    \\
  $\theta_{5}$ & $d_5$ & $-5\d$     \\
  $\theta_{6}$ & $d_6$ & $-4\d$     \\
  \hline
\end{tabular}\]\caption{A candidate for part of the character table
of G.}\label{suzuki-case iv}\end{table} So we calculate (where $y
\in \CC_6$ and $z \in \CC_7$):
\[0=\a^G_{yzy}=1-\frac{\e}{d}+2\frac{\e}{d}-\frac{\e}{d}
+\frac{25\d}{d_5}-\frac{16\d}{d_6}=1+\frac{25\d}{d_5}-\frac{16\d}{d_6}.\]
Rearrange and substitute for $\d d_5$ to get $0=16+d_5d_6 +9\d d_6$.
Observe again that $\d=-1$ and so $d_5=1+d_6$ and simplify further
to get $0=16+d_6(1+d_6)-9d_6=d_6^2-8d_6+16$. This quadratic in $d_6$
has the repeated root $d_6=4$ and so $d_5=5$. Now we calculate
\[\frac{2^23^6}{|G|}=\a^G_{yzz}= 1-
\frac{\e}{d}-\frac{4\e}{d}-\frac{\e}{d}+\frac{5}{5}+\frac{4}{4}=3-\frac{6\e}{d}.\]
We simplify to get $|G|(d-2\e)=2^23^5d$. Since $d \neq0$, $d-2\e
\neq 0$. Therefore we rearrange to see that $|G|/|H|=3d/2(d-2\e)$.
By Lemma \ref{index of H in G}, $|G|/|H|\geq 28$. We rearrange
$3d/2(d-2\e)\geq 28$ to get $53d\leq 112\e$. Therefore $\e=1$ and
$0<d\leq 2$. However this is a contradiction since this implies
$d-2\e< 0$ and so  $|G|=2^23^5d/(d-2\e)<0$.

\subsubsection{Candidate 5:
$\bf (b_2,b_3,b_4,b_5,b_6)=(-\e,-2\e,-\e,-\d,0)$}

The missing entries from the character table are displayed in Table
\ref{suzuki-case v}.
\begin{table}[h]\[\begin{tabular}{|c|c |c |}
  \hline
  \; & 1 & $\CC_6^G$  \\
  \hline
  $\theta_1$   &     1 & 1       \\
  $\theta_2$   & $d$ &  $-4\e$    \\
  $\theta_{3}$ & $d$ &  $-4\e$    \\
  $\theta_{4}$ & $d$ &  $-4\e$    \\
  $\theta_{5}$ & $d_5$ & $-2\d$     \\
  $\theta_{6}$ & $d_6$ & $-\d$     \\
  \hline
\end{tabular}\]\caption{A candidate for part of the character table
of G.}\label{suzuki-case v}\end{table} So we calculate (where $y \in
\CC_6$ and $z \in \CC_7$):
\[0=\a^G_{yzy}=1-\frac{16\e}{d}+2\frac{16\e}{d}-\frac{16\e}{d}
+\frac{4\d}{d_5}-\frac{\d}{d_6}=1+\frac{4\d}{d_5}-\frac{\d}{d_6}\]
which this time reduces to $0=1+d_6(1+d_6)-3d_6=d_6^2-2d_6+1$ when
we observe that $\d=-1$. This quadratic in $d_6$ has the repeated
root $d_6=1$ and so $d_5=2$. Now we calculate:
\[\frac{2^23^6}{|G|}=\a^G_{yzz}= 1-
\frac{4\e}{d}-\frac{16\e}{d}-\frac{4\e}{d}+\frac{2}{2}+\frac{1}{1}=3-\frac{24\e}{d}.\]
We simplify to get $|G|(d-8\e)=2^23^5d$. Since $d \neq0$, $d-8\e
\neq 0$. Therefore we rearrange to see that $|G|/|H|=3d/2(d-8\e)$.
By Lemma \ref{index of H in G}, $|G|/|H|\geq 28$. We rearrange
$3d/2(d-8\e)\geq 28$ to get $53d\leq 448\e$. Therefore $\e=1$ and
$0<d\leq 8$. However this is a contradiction since this implies
$d-8\e< 0$ and so $|G|=2^23^5d/(d-8\e)<0$.

Thus our calculations in each case given by Lemma \ref{5 cases-the
unknowns a2-a6 in B} give a contradiction. Therefore we may
conclude that $G=H$. This completes the proof of Theorem \ref{case
1 thm}.


\subsection{Case 2}\label{section case 2} In this second case we
hypothesize that for  $x \in \CC_4$, $C_G(x) \leq H$ and $\CC_7^G=\CC_5^G$. Under the assumptions
of Hypothesis A  it is clear that this is equivalent to the following hypothesis.

\begin{hyp}\label{case 2 hyp}
Let $G$ satisfy Hypothesis A and in addition assume $C_G(j) \leq H$ for each $j \in J^\#$ and for
$x\in H \bs J$ of order three $x^G \cap J \neq \emptyset$.
\end{hyp}
\begin{thm}\label{case 2 thm}
No group satisfies Hypothesis \ref{case 2 hyp}.
\end{thm}

We  suppose for a contradiction that $G$ is a group satisfying
Hypothesis \ref{case 2 hyp}. We use the $p$-block theory from Section
\ref{block theory} with $p=3$. Recall that $B_0(G)$ is the set of ordinary characters in the principal block of $G$. We will show in Lemma \ref{normalizers of D} that we only need to calculate character values from characters in $B_0(G)$.

\begin{lemma}\label{No normal subgroup containing S}
If $N\trianglelefteq G$ and $\syl_3(N) \subseteq \syl_3(G)$ then
$G=N$.
\end{lemma}
\begin{proof}
Suppose $N\trianglelefteq G$ with $N \neq G$ and assume $S \in \syl_3(N)\cap \syl_3(G)\neq
\emptyset$. Then by Lemma \ref{frattini} (Frattini argument), $G=N N_G(S)$. Let $T \in
\syl_2(N_G(S))$ then $|T|=2$ and $N_G(S)=ST$ follows from Lemma \ref{order of sylow 3's and H is
self normalizing}. Thus $G=NT$ and since $G \neq N$, $T \nleq N$. Thus, by an isomorphism theorem,
$G/N=NT/N\cong T/N\cap T \cong T$. By Lemma \ref{3-elements in S not J}, there exists $A \leq G$
such that $A \cong \SL_2(3)$ and $T=\mathcal{Z}(A)$. In particular, there exists a cyclic group of order four
$F \leq A$ such that $T<F$. Since $N \cap T=1$, $N \cap F=1$ and so $G/N=NF/N \cong F$ which is a
contradiction.
\end{proof}

\begin{lemma}\label{normalizers of D}
Let $D \leq H$ be a non-trivial $3$-subgroup. Then $N_G(D)$ is $3$-soluble and $O_{3'}(N_G(D))=1$.
In particular, if $\chi \in \Irr(G) \bs B_0(G)$ then $\chi(a)=0$ for any element $1 \neq a \in G$ of order
a multiple of $3$.
\end{lemma}
\begin{proof}
If $D=S$ then $N_G(S) \leq H$ and the conclusion is clear. If $D=J$ then $N_G(D)= H$ and again it
is clear. Therefore $D$ has one of the following isomorphism types: $3$, $3 \times 3$, $3_+^{1+2}$,
$3_-^{1+2}$, $C_9$. In each case $N_G(D)/C_G(D)$ is $3$-soluble. Now since $C_G(x)$ is $3$-soluble
for each element of order three in $D$, $C_G(D)$ is $3$-soluble and thus $N_G(D)$ is $3$-soluble.
So suppose $O_{3'}(N_G(D))\neq 1$. Then $[O_{3'}(N_G(D)),D]=1$ and so $O_{3'}(C_G(D))\neq 1$. If $1
\neq x \in D$ then $x$ commutes with $O_{3'}(C_G(D))$ and so $x$ is not $3$-central as elements in
$\CC_5^G$ have centralizers of order $3^4$. Therefore $x \in J$. If $D=\<x\>$ is cyclic then we
have a contradiction since $O_{3'}(C_G(x))=1$ for every element of order three in $J$. So we must
have $D \leq J$ of order nine. However by Lemma \ref{order of sylow 3's and H is self normalizing}
$(v)$,  $C_G(D)= J$ gives us a contradiction.

Hence we may apply Lemma \ref{brauer's theorems} to say that for any  $\chi \in \Irr(G) \bs B_0(G)$
and any element non-identity $a$ of order a multiple of $3$, $\chi(a)=0$
\end{proof}

Recall from Section \ref{CharacterTheory-prelims} the definition of the structure constants (labeled $a^G_{xyz}$ and $\a^G_{xyz}$ for $G$ a group with elements $x,y,z$).
\begin{lemma}\label{part2-structure constants}
Let $a \in \CC_4$, $b \in \CC_5$, $c \in \CC_6$. Then
\begin{enumerate}[$(i)$]
\item $\a^G_{abc}=2\cdot 108\cdot 81/|G|$;
\item $\a^G_{aac}=2\cdot 108\cdot 108/|G|$; and
\item $\a^G_{cca}=4\cdot 54\cdot 54/|G|$.
\end{enumerate}
\end{lemma}
\begin{proof}
Suppose $(w,x,y) \in \{ (a,b,c),(a,a,c), (c,c,a) \}$. Suppose for
some $g,h \in G$, $w^gx^h=y$  and set $X:=\<w^g,x^h\>$. By Lemma
\ref{triangle group}, there exists an element of order three $z
\in X$ such that $X \leq C_G(z)$. Since $z$ commutes with $y \in
\CC_4 \cup \CC_6$, $z \in J$. Therefore $X \leq C_G(z) \leq H$.
Notice that $w$ lies in $\CC_4 \cup \CC_6$ and so $w^g\in J$.
Therefore $X=\<w^g,y\> \leq J$ and $w^g,x^h \in J$. Now elements
in $J$ are $G$-conjugate if and only if they are conjugate in $H$
(Lemma \ref{J is characteristic}) so we may assume $g,h \in H$.
Therefore $a^G_{wxy}=a^H_{wxy}$ and we may calculate $\a^G_{wxy}$
from Table \ref{char table H}.
\end{proof}

Recall from Section \ref{block theory} the definition of a
\textit{principal column} of $G$ as a sequence indexed by the
principal block characters of $G$.

\begin{lemma}\label{principal block order}
There are at most $14$ characters in the principal $3$-block of $G$
and if $z\in G$ is $3$-central then $\chi(z)\neq 0$ for each
$\chi\in B_0(G)$.
\end{lemma}
\begin{proof}
By Lemma \ref{normalizers of D}, we may apply Lemma \ref{brauer's
theorems} to $G$ to say that if $x$ is a $3$-element in $G$ and
$\chi \in \Irr(G)$ with $\chi(x) \neq 0$ then $\chi \in \B_0(G)$.

Let $d \in \CC_{11}$ and $e \in \CC_{12}$. By Lemma \ref{order of sylow 3's and H is self
normalizing},  $\CC_{11}^G \neq \CC_{12}^G$. Since any character of $G$ restricts to a sum of
characters of $H$, $\chi(d),\chi(e)\in \mathbb{Z}$ for any $\chi \in B_0(G)$. Furthermore $\chi(d)
\equiv \chi(e) ~\mathrm{mod}~ 3$ by Lemma \ref{Characters mod p}. Let $\{\chi_1, \hdots,
\chi_m\}=B_0(G)$ then by Lemma \ref{normalizers of D}, $\sum_{i=1}^m \chi_i(d)= \sum_{i=1}^m
\chi_i(e)=9$ and since $\CC_{11}^G \neq \CC_{12}^G$, $\sum_{i=1}^m \chi_i(d)\chi_i(e)=0$. Define
the following principal column $C:=(\frac{1}{3}(\chi_i(d)-\chi_i(e)) )_{i=1..m}$. Then the inner
product $(C,C)$ equals:
\[\sum_{i=1}^m
(\frac{1}{3}(\chi_i(d)-\chi_i(e)))^2=\frac{1}{9}\sum_{i=1}^m
\chi_i(d)^2-2\chi_i(d)\chi_i(e)+\chi_i(e)^2=\frac{1}{9} (9+0+9)=2.\]
So $C$ has just two non-zero entries. Let $\chi_j \in B_0(G)$ such
that $(\chi_j(d)-\chi_j(e))/3 =\pm 1$. Then
$\chi_j(d)-\chi_j(e)=\pm 3$. In particular there exists $f \in
\{d,e\}$ such that $|\chi_j(f)|>1$. We can therefore conclude that
the number of irreducible characters of $G$ which are non-zero on
$f$ is at most six.

So, for $z \in \CC_5$, it follows from Lemma \ref{normalizers of D} that
$\sum_{i=1}^m\chi_i(z)^2=81=|C_G(z)|$. By Lemma \ref{non-vanishing principal block on p-central}, $\chi_i(z)
\neq 0$ for each $\chi_i\in B_0(G)$. If $\chi_i\in B_0(G)$ and $\chi_i(f) =0$ then $\chi_i(z)$ is a
multiple of three since $\chi(f) \equiv \chi(z) ~\mathrm{mod} ~3$. The number of characters in
$B_0(G)$ on which $z$ is a multiple of three is at most $8$ since $8*3^2=72<81$.  It therefore
follows that $m$ is at most $8+6$ and so $m\leq 14$.
\end{proof}

Set $B_0(G):=\{\chi_1,\hdots, \chi_m\}$ where $m \leq 14$. We define the following principal
columns of character values:
\begin{enumerate}[$(i)$]
\item $\mathbf{A}=(\chi_i(x_4))_{1\leq i \leq m}$;
 \item $\mathbf{B}=(\chi_i(x_5))_{1\leq i \leq m}$;
 \item $\mathbf{C}=(\chi_i(x_6))_{1\leq i \leq m}$;
 \item $\mathbf{D}=(\chi_i(x_9))_{1\leq i \leq m}$;
 \item $\mathbf{E}=(\chi_i(x_{10}))_{1\leq i \leq m}$;
 \item $\mathbf{F}=(\chi_i(x_{11}))_{1\leq i \leq m}$;
 \item $\mathbf{G}=(\chi_i(x_{12}))_{1\leq i \leq m}$;
 \item $\mathbf{H}=(\chi_i(x_{13}))_{1\leq i \leq m}$; and
 \item $\mathbf{I}=(\chi_i(x_{14}))_{1\leq i \leq m}$.
\end{enumerate}

\begin{lemma}\label{alt9-case 2 inner prods table}
Table \ref{character inner products} shows the pairwise inner products of the principal columns
$\mathbf{A}$,  $\mathbf{B}$, $\mathbf{C}$, $\mathbf{D}$, $\mathbf{E}$, $\mathbf{F}$, $\mathbf{G}$,
$\mathbf{H}$, $\mathbf{I}$.
\end{lemma}
\begin{proof}
For $\psi \in \Irr(G)\bs B_0(G)$ and $x_i \in \CC_i^G$ ($i \in \{4,5,6,9,10,11,12,13,14\}$),  by
Lemma \ref{normalizers of D}, $\psi(x_i)=0$. In particular this implies that  $\sum_{\chi \in
B_0(G)}\chi(x_i)^2=|C_G(x_i)|$ and that for $i \neq j \in \{4,5,6,9,10,11,12,13,14\}$, $\sum_{\chi
\in B_0(G)}\chi(x_i)\chi(x_j) =0$.
\end{proof}

\begin{table}[h]\[\begin{tabular}{|c|c c c c c c c c c|}
  \hline
  $(,)$ & $\bf{A}$ & $\bf{B}$ & $\bf{C}$ & $\bf{D}$ & $\bf{E}$ & $\bf{F}$ & $\bf{G}$ & $\bf{H}$ & $\bf{I}$\\
  \hline
   $\bf{A}$& 108 & 0 & 0  & 0  & 0  & 0 & 0  & 0  & 0  \\
   $\bf{B}$& 0 & 81 & 0  & 0  & 0  & 0 & 0  & 0  & 0     \\
   $\bf{C}$& 0 & 0 & 54 & 0  & 0  & 0 & 0  & 0  & 0     \\
   $\bf{D}$& 0 & 0 & 0  & 12  & 0  & 0 & 0  & 0  & 0     \\
   $\bf{E}$& 0 & 0 & 0  & 0  & 6  & 0 & 0  & 0  & 0     \\
   $\bf{F}$& 0 & 0 & 0  & 0  & 0  & 9 & 0  & 0  & 0     \\
   $\bf{G}$& 0 & 0 & 0  & 0  & 0  & 0 & 9  & 0  & 0     \\
   $\bf{H}$& 0 & 0 & 0  & 0  & 0  & 0 & 0  & 12  & 0     \\
   $\bf{I}$& 0 & 0 & 0  & 0  & 0  & 0 & 0  & 0  & 12     \\
  \hline
\end{tabular}\]\caption{The pairwise inner products of the principal
columns $\bf{A}$-$\bf{I}$}\label{character inner
products}\end{table}

Recall Lemma \ref{CT-Integer Matrix} and consider the following
invertible matrix:

\[M:=\left(\begin{smallmatrix}
0&0&1/12&1/12&1/36&1/36&1/18&1/9&-1/
9\\
 \noalign{\medskip}0&0&0&0&1/9&1/9&-1/9&-2/9&-1/9
\\
 \noalign{\medskip}0&0&0&0&-1/18&1/9&1/18&1/9&2/9\\
  \noalign{\medskip}0
&0&1/4&1/4&1/4&-1/4&0&0&0\\
 \noalign{\medskip}0&0&0&0&1/2&0&1/2&0&0
\\
 \noalign{\medskip}-1/3&0&-1/3&1/3&1/3&0&0&0&0\\
  \noalign{\medskip}1/3
&0&0&1/3&1/3&0&0&0&0\\
 \noalign{\medskip}0&\sqrt
{3}/6&-\sqrt {3}/12&-\sqrt {3}/12&1/4&\sqrt {3}/12&-\left(3+ \sqrt
{3}
 \right)/12&0&0\\
  \noalign{\medskip}0&-\sqrt {3}/6&
\sqrt {3}/12&\sqrt {3}/12&1/4&-\sqrt {3}/12&-\left( 3-\sqrt {3}
 \right)/12&0&0\end{smallmatrix} \right) .\]
Set $N:=(\psi_i(x_j))_{ij}$ to be the $14\times 9$ matrix of
$H$-character values where $i=1,\hdots, 14$ and
$j=4,5,6,9,10,11,12,13,14$ with $x_j \in \CC_j$. We calculate $NM$
to be a matrix with integer entries. Furthermore, to address the obvious question related to how $M$ is found,  $M$ has been chosen
as a matrix of row operations of $N$ in such a way that $NM$ has few
entries and the entries are small integers. The method is somewhat ad hoc and a different choice of $M$ may potentially work just as well. Now consider the $m\times 9$ matrix $L:=(\chi_i(x_j))_{ij}$ where
$\{\chi_1,\hdots,\chi_m\}=B_0(G)$ and again $x_j \in \CC_j$ for
$j=4,5,6,9,10,11,12,13,14$.  By Lemma \ref{CT-Integer Matrix},
$LM=:K$ is a matrix with integer entries.

Now let $M_j$ be the $j$'th column of $M$. Then $L M_j$ equals the
$j$'th column of $K$, $K_j$ say.  View $K_j$ as  a principal
column. We calculate the following:

\begin{enumerate}[$(i)$]
\item $K_1=\frac{1}{3}(-\mathbf{F}+\mathbf{G})$;

\item $K_2=\frac{\sqrt{3}}{6}(\bf{H}-\bf{I})$;

\item
$K_3=\frac{1}{12}(\bf{A}+3\bf{D}-4\bf{F}-\sqrt{3}\bf{H}+\sqrt{3}\bf{I})$;

\item
$K_4=\frac{1}{12}(\bf{A}+3\bf{D}+4\bf{F}+4\bf{G}-\sqrt{3}\bf{H}+\sqrt{3}\bf{I})$;

\item
$K_5=\frac{1}{36}(\bf{A}+4\bf{B}-2\bf{C}+9\bf{D}+18\bf{E}+12\bf{F}+12\bf{G}+9\bf{H}+9\bf{I})$;

\item
$K_6=\frac{1}{36}(\bf{A}+4\bf{B}+4\bf{C}-9\bf{D}+3\sqrt{3}\bf{H}-3\sqrt{3}\bf{I})$;

\item
$K_7=\frac{1}{36}(2\bf{A}-4\bf{B}+2\bf{C}+18\bf{E}-(9+3\sqrt{3})\bf{H}-(9-3\sqrt{3})\bf{I})$;

\item $K_8=\frac{1}{9}(\bf{A}-2\bf{B}+\bf{C})$;

\item $K_9=\frac{1}{9}(-\bf{A}-\bf{B}+2\bf{C})$.
\end{enumerate}

We use the bi-linearity of the column inner product to calculate the pairwise principal column
inner products $(K_i,K_j)$ as displayed in Table \ref{Ki's- inner prods}.
\begin{table}[h]\[\begin{tabular}{|c|c c c c c c c c c|}
  \hline
  $(,)$ & $K_1$ & $K_2$ & $K_3$ & $K_4$ & $K_5$ & $K_6$ & $K_7$ & $K_8$ & $K_9$\\
  \hline
   $K_1$&  2 & \; &\; &\; &\; &\; &\; &\; &\; \\
   $K_2$&  0&2&\; &\; &\; &\; &\; &\; &\;  \\
   $K_3$&  1&-1&3 &  \; &\; &\; &\; &\; &\;   \\
   $K_4$&  0&-1&1&4&  \; &\; &\; &\; &\;  \\
   $K_5$&  0&0&0&3&7& \; &\; &\; &\; \\
   $K_6$&  0&1&-1&-1&0&3 & \; &\; &\;  \\
   $K_7$&  0&-1&1&1&-1&-1&5&  \; &\;  \\
   $K_8$&  0&0&1&1&-2&-1&3&6&  \;  \\
   $K_9$&  0&0&-1&-1&-2&0&1&2&5  \\
  \hline
\end{tabular}\]\caption{The table of principal column inner
products $(K_i,K_j)_{1 \leq i,j \leq 9}$.}\label{Ki's- inner prods}\end{table} Note that we can
easily calculate the first row of $K$ since the first entry in  each column $\bf{A}$-$\bf{I}$ is 1.
\begin{table}[h]\[\begin{tabular}{|c|c c c c c c c c c|}
  \hline
  $(,)$ & $K_1$ & $K_2$ & $K_3$ & $K_4$ & $K_5$ & $K_6$ & $K_7$ & $K_8$ & $K_9$\\
  \hline
   $1_{G}$& 0 &0 &0 & 1 &  2 &0 &0 &0 & 0\\
  \hline
\end{tabular}\]\caption{The first row of $K$.}\label{First row of K}\end{table}

The aim therefore is to find each possibility for $K$ by finding
each possibility for the columns $K_1,\hdots, K_9$ which have
integer entries. Each column has length at most $14$ and each entry
is an integer with bounded modulus. Thus there are a finite number
of possible candidates for the matrix $K$. Of course given a
candidate matrix $K$ any permutation of the rows gives another
solution. Similarly, given a candidate $K$, multiplying any row or
rows by $-1$ gives a further solution. Therefore any strategy for
finding candidates for $K$ must take this into account.

The necessary calculations were done by hand and then again with
the aid of a computer algebra package. The code is available on request.
The calculations provide thirteen
candidates for $K$. In each case we calculate $K
M\inv=L=(\chi_i(x_j))$ which gives us a candidate for part of the
principal $3$-block of the character table of $G$.

\subsubsection{Candidate 1}
\[ K= \left[ \begin{smallmatrix} 0&0&0&1&2&0&0&0&0
\\\noalign{\medskip}1&0&0&0&0&0&0&0&1\\\noalign{\medskip}1&0&1&0&0&0&0
&0&-1\\\noalign{\medskip}0&1&0&0&0&0&0&1&0\\\noalign{\medskip}0&1&-1&-
1&0&1&-1&-1&0\\\noalign{\medskip}0&0&1&0&0&0&0&0&0\\\noalign{\medskip}0
&0&0&1&0&0&0&0&0\\\noalign{\medskip}0&0&0&1&1&0&0&0&-1
\\\noalign{\medskip}0&0&0&0&1&0&0&-1&-1\\\noalign{\medskip}0&0&0&0&1&0
&-1&-1&0\\\noalign{\medskip}0&0&0&0&0&1&-1&0&0\\\noalign{\medskip}0&0&0
&0&0&1&1&0&0\\\noalign{\medskip}0&0&0&0&0&0&1&1&1\\\noalign{\medskip}0
&0&0&0&0&0&0&1&0\end{smallmatrix}\right],
 L=\left[ \begin{smallmatrix} 1&1&1&1&1&1&1&1&1
\\\noalign{\medskip}-4&-1&2&0&0&-1&2&0&0\\\noalign{\medskip}4&1&-2&0&0
&-2&1&0&0\\\noalign{\medskip}3&-3&0&1&0&0&0&\sqrt {3}&-\sqrt {3}
\\\noalign{\medskip}-3&3&0&-1&0&0&0&\sqrt {3}&-\sqrt {3}
\\\noalign{\medskip}2&2&2&2&0&-1&-1&0&0\\\noalign{\medskip}1&1&1&1&-1&
1&1&-1&-1\\\noalign{\medskip}4&1&-2&0&0&1&1&0&0\\\noalign{\medskip}0&3
&-3&0&1&0&0&0&0\\\noalign{\medskip}-3&3&0&1&0&0&0&1&1
\\\noalign{\medskip}3&3&3&-1&-1&0&0&1&1\\\noalign{\medskip}3&3&3&-1&1&0
&0&-1&-1\\\noalign{\medskip}0&-3&3&0&1&0&0&0&0\\\noalign{\medskip}3&-3
&0&-1&0&0&0&1&1\end{smallmatrix} \right].\]

Observe that the matrix $L$ up to rearrangement and sign changes of
the rows is equal to part of the character table of $H$.

\subsubsection{Candidates 2 and 3}
\[ K= \left[ \begin{smallmatrix} 0&0&0&1&2&0&0&0&0
\\\noalign{\medskip}1&0&1&0&0&0&0&0&-1\\\noalign{\medskip}1&0&0&0&0&0&0
&0&1\\\noalign{\medskip}0&1&0&0&0&0&0&1&0\\\noalign{\medskip}0&1&-1&-1
&0&1&-1&-1&0\\\noalign{\medskip}0&0&1&0&0&0&0&0&0\\\noalign{\medskip}0
&0&0&1&0&0&0&0&0\\\noalign{\medskip}0&0&0&1&1&0&0&0&-1
\\\noalign{\medskip}0&0&0&0&1&0&-1&-1&-1\\\noalign{\medskip}0&0&0&0&1&0
&0&-1&0\\\noalign{\medskip}0&0&0&0&0&1&1&0&0\\\noalign{\medskip}0&0&0&0
&0&1&-1&0&0\\\noalign{\medskip}0&0&0&0&0&0&1&1&0\\\noalign{\medskip}0&0
&0&0&0&0&0&1&1\end{smallmatrix} \right],
 L=\left[ \begin{smallmatrix} 1&1&1&1&1&1&1&1&1
\\\noalign{\medskip}4&1&-2&0&0&-2&1&0&0\\\noalign{\medskip}-4&-1&2&0&0
&-1&2&0&0\\\noalign{\medskip}3&-3&0&1&0&0&0&\sqrt {3}&-\sqrt {3}
\\\noalign{\medskip}-3&3&0&-1&0&0&0&\sqrt {3}&-\sqrt {3}
\\\noalign{\medskip}2&2&2&2&0&-1&-1&0&0\\\noalign{\medskip}1&1&1&1&-1&
1&1&-1&-1\\\noalign{\medskip}4&1&-2&0&0&1&1&0&0\\\noalign{\medskip}0&3
&-3&0&0&0&0&1&1\\\noalign{\medskip}-3&3&0&1&1&0&0&0&0
\\\noalign{\medskip}3&3&3&-1&1&0&0&-1&-1\\\noalign{\medskip}3&3&3&-1&-
1&0&0&1&1\\\noalign{\medskip}3&-3&0&-1&1&0&0&0&0\\\noalign{\medskip}0&
-3&3&0&0&0&0&1&1\end{smallmatrix} \right]; \]
\[K= \left[ \begin{smallmatrix} 0&0&0&1&2&0&0&0&0
\\\noalign{\medskip}1&0&0&0&0&0&0&0&1\\\noalign{\medskip}1&0&1&0&0&0&0
&0&-1\\\noalign{\medskip}0&1&-1&-1&0&1&-1&-1&0\\\noalign{\medskip}0&1&0
&0&0&0&0&1&0\\\noalign{\medskip}0&0&1&0&0&0&0&0&0\\\noalign{\medskip}0
&0&0&1&0&0&0&0&-1\\\noalign{\medskip}0&0&0&1&1&0&0&0&0
\\\noalign{\medskip}0&0&0&0&1&0&-1&-1&-1\\\noalign{\medskip}0&0&0&0&1&0
&0&-1&-1\\\noalign{\medskip}0&0&0&0&0&1&1&0&0\\\noalign{\medskip}0&0&0
&0&0&1&-1&0&0\\\noalign{\medskip}0&0&0&0&0&0&1&1&0\\\noalign{\medskip}0
&0&0&0&0&0&0&1&0\end{smallmatrix} \right],
 L= \left[ \begin{smallmatrix} 1&1&1&1&1&1&1&1&1
\\\noalign{\medskip}-4&-1&2&0&0&-1&2&0&0\\\noalign{\medskip}4&1&-2&0&0
&-2&1&0&0\\\noalign{\medskip}-3&3&0&-1&0&0&0&\sqrt {3}&-\sqrt {3}
\\\noalign{\medskip}3&-3&0&1&0&0&0&\sqrt {3}&-\sqrt {3}
\\\noalign{\medskip}2&2&2&2&0&-1&-1&0&0\\\noalign{\medskip}4&1&-2&0&-1
&1&1&-1&-1\\\noalign{\medskip}1&1&1&1&0&1&1&0&0\\\noalign{\medskip}0&3
&-3&0&0&0&0&1&1\\\noalign{\medskip}0&3&-3&0&1&0&0&0&0
\\\noalign{\medskip}3&3&3&-1&1&0&0&-1&-1\\\noalign{\medskip}3&3&3&-1&-
1&0&0&1&1\\\noalign{\medskip}3&-3&0&-1&1&0&0&0&0\\\noalign{\medskip}3&
-3&0&-1&0&0&0&1&1\end{smallmatrix} \right].\]

In both cases observe the third and the fifth columns of $L$. These
entries correspond (up to sign) to character values on $x_6 \in
\CC_6$ and $x_{10} \in \CC_{10}$. The Brauer character table of $C_G(x_6)$ is given in Table
\ref{char table CG(c)} (note that rows of the table correspond to the two irreducible Brauer characters of $C_G(x_6)$ and the columns correspond to the two $3$-regular elements of $C_G(x_6)$ of order one and two).

\begin{table}[h]\[\begin{tabular}{|c|c c|}
\hline

Order &   1 & 2 \\
\hline
$\phi_1$  & 1 & 1  \\
$\phi_2$  & 1 & -1  \\
\hline
\end{tabular}.\]\caption{The Brauer character table of $C_G(x_6)$.}\label{char table
CG(c)}\end{table} By Lemma \ref{gen decomp numbers},  for any $\chi\in \Irr(G)$, there exists
algebraic integers (generalized decomposition numbers) $c_1=d^{x_6}_{\chi,\phi_1}$ and
$c_2=d^{x_6}_{\chi,\phi_2}$ in $\mathbb{Q}(e^{2\pi i /p^n})$ such
that $\chi(x_6)=c_1+c_2$ and $\chi(x_{10})=c_1-c_2$. Since all
$H$-characters are integral on $x_6$ and $x_{10}$,
$\chi(x_{6}),\chi(x_{11})\in \mathbb{Z}$. Also, since $c_1$ and
$c_2$ are algebraic integers, $c_1,c_2\in \mathbb{Z}$. Thus, if
$\chi_i(x_6)=0$ then $\chi_i(x_{10})$ is an even integer. Now we
observe the thirteenth row of each candidate for $L$ and see that in
each case we have calculated a character in $B_0(G)$ which vanishes
on $\CC_6^G$ and gives $\pm 1$ on $\CC_{10}^G$. Hence both
candidates give a contradiction.

\subsubsection{The Remaining Candidates}

\tiny{\[\tiny{ K= \left[ \begin{smallmatrix}0&0&0&1&2&0&0&0&0
\\\noalign{\medskip}1&0&0&0&0&0&0&0&0\\\noalign{\medskip}1&0&1&0&0&0&0
&0&0\\\noalign{\medskip}0&1&0&0&0&0&0&1&0\\\noalign{\medskip}0&1&-1&-1
&0&1&-1&-1&0\\\noalign{\medskip}0&0&1&0&0&0&0&0&-1\\\noalign{\medskip}0
&0&0&1&0&0&0&1&0\\\noalign{\medskip}0&0&0&1&1&0&0&-1&-1
\\\noalign{\medskip}0&0&0&0&1&0&-1&-1&0\\\noalign{\medskip}0&0&0&0&1&0
&0&0&-1\\\noalign{\medskip}0&0&0&0&0&1&-1&0&0\\\noalign{\medskip}0&0&0
&0&0&1&1&0&0\\\noalign{\medskip}0&0&0&0&0&0&1&1&1\\\noalign{\medskip}0
&0&0&0&0&0&0&0&1\end{smallmatrix} \right], L= \left[
\begin{smallmatrix} 1&1&1&1&1&1&1&1&1
\\\noalign{\medskip}-1&-1&-1&-1&0&-1&2&0&0\\\noalign{\medskip}1&1&1&1&0
&-2&1&0&0\\\noalign{\medskip}3&-3&0&1&0&0&0&\sqrt {3}&-\sqrt {3}
\\\noalign{\medskip}-3&3&0&-1&0&0&0&\sqrt {3}&-\sqrt {3}
\\\noalign{\medskip}5&2&-1&1&0&-1&-1&0&0\\\noalign{\medskip}4&-2&1&0&-
1&1&1&0&0\\\noalign{\medskip}1&4&-2&1&0&1&1&-1&-1\\\noalign{\medskip}-
3&3&0&1&0&0&0&1&1\\\noalign{\medskip}3&0&-3&-1&1&0&0&1&1
\\\noalign{\medskip}3&3&3&-1&-1&0&0&1&1\\\noalign{\medskip}3&3&3&-1&1&0
&0&-1&-1\\\noalign{\medskip}0&-3&3&0&1&0&0&0&0\\\noalign{\medskip}-3&0
&3&1&0&0&0&0&0\end{smallmatrix} \right];}\]}

\tiny{\[ K= \left[ \begin{smallmatrix} 0&0&0&1&2&0&0&0&0
\\\noalign{\medskip}1&0&0&0&0&0&0&0&1\\\noalign{\medskip}1&0&1&0&0&0&0
&0&-1\\\noalign{\medskip}0&1&0&0&0&0&0&1&0\\\noalign{\medskip}0&1&-1&-
1&0&1&-1&-1&0\\\noalign{\medskip}0&0&1&0&0&0&0&0&0\\\noalign{\medskip}0
&0&0&1&1&0&-1&-1&-1\\\noalign{\medskip}0&0&0&1&0&0&1&1&0
\\\noalign{\medskip}0&0&0&0&1&0&0&-1&0\\\noalign{\medskip}0&0&0&0&1&0&0
&0&-1\\\noalign{\medskip}0&0&0&0&0&1&1&0&0\\\noalign{\medskip}0&0&0&0&0
&1&-1&0&0\\\noalign{\medskip}0&0&0&0&0&0&0&1&1\end{smallmatrix}
\right],
 L=\left[  \begin{smallmatrix} 1&1&1&1&1&1&1&1&1
\\\noalign{\medskip}-4&-1&2&0&0&-1&2&0&0\\\noalign{\medskip}4&1&-2&0&0
&-2&1&0&0\\\noalign{\medskip}3&-3&0&1&0&0&0&\sqrt {3}&-\sqrt {3}
\\\noalign{\medskip}-3&3&0&-1&0&0&0&\sqrt {3}&-\sqrt {3}
\\\noalign{\medskip}2&2&2&2&0&-1&-1&0&0\\\noalign{\medskip}1&4&-2&1&-1
&1&1&0&0\\\noalign{\medskip}4&-2&1&0&0&1&1&-1&-1\\\noalign{\medskip}-3
&3&0&1&1&0&0&0&0\\\noalign{\medskip}3&0&-3&-1&1&0&0&1&1
\\\noalign{\medskip}3&3&3&-1&1&0&0&-1&-1\\\noalign{\medskip}3&3&3&-1&-
1&0&0&1&1\\\noalign{\medskip}0&-3&3&0&0&0&0&1&1\end{smallmatrix}
\right];\]}

\tiny{\[ K= \left[\begin{smallmatrix} 0&0&0&1&2&0&0&0&0
\\\noalign{\medskip}1&0&0&0&0&0&0&0&0\\\noalign{\medskip}1&0&1&0&0&0&0
&0&0\\\noalign{\medskip}0&1&0&0&0&0&0&1&0\\\noalign{\medskip}0&1&-1&-1
&0&1&-1&-1&0\\\noalign{\medskip}0&0&1&0&0&0&0&0&-1\\\noalign{\medskip}0
&0&0&1&1&0&-1&-1&-1\\\noalign{\medskip}0&0&0&1&0&0&1&1&0
\\\noalign{\medskip}0&0&0&0&1&0&0&-1&0\\\noalign{\medskip}0&0&0&0&1&0&0
&0&-1\\\noalign{\medskip}0&0&0&0&0&1&-1&0&0\\\noalign{\medskip}0&0&0&0
&0&1&1&0&0\\\noalign{\medskip}0&0&0&0&0&0&0&1&1\\\noalign{\medskip}0&0
&0&0&0&0&0&0&1\end{smallmatrix}\right],
 L=\left[ \begin{smallmatrix} 1&1&1&1&1&1&1&1&1
\\\noalign{\medskip}-1&-1&-1&-1&0&-1&2&0&0\\\noalign{\medskip}1&1&1&1&0
&-2&1&0&0\\\noalign{\medskip}3&-3&0&1&0&0&0&\sqrt {3}&-\sqrt {3}
\\\noalign{\medskip}-3&3&0&-1&0&0&0&\sqrt {3}&-\sqrt {3}
\\\noalign{\medskip}5&2&-1&1&0&-1&-1&0&0\\\noalign{\medskip}1&4&-2&1&-
1&1&1&0&0\\\noalign{\medskip}4&-2&1&0&0&1&1&-1&-1\\\noalign{\medskip}-
3&3&0&1&1&0&0&0&0\\\noalign{\medskip}3&0&-3&-1&1&0&0&1&1
\\\noalign{\medskip}3&3&3&-1&-1&0&0&1&1\\\noalign{\medskip}3&3&3&-1&1&0
&0&-1&-1\\\noalign{\medskip}0&-3&3&0&0&0&0&1&1\\\noalign{\medskip}-3&0
&3&1&0&0&0&0&0\end{smallmatrix} \right];\]}

\tiny{\[ K= \left[\begin{smallmatrix} 0&0&0&1&2&0&0&0&0
\\\noalign{\medskip}1&0&0&0&0&0&0&-1&0\\\noalign{\medskip}1&0&1&0&0&0&0
&1&0\\\noalign{\medskip}0&1&0&0&0&0&-1&0&0\\\noalign{\medskip}0&1&-1&-
1&0&1&0&0&0\\\noalign{\medskip}0&0&1&0&0&0&1&0&-1\\\noalign{\medskip}0
&0&0&1&0&0&1&1&0\\\noalign{\medskip}0&0&0&1&1&0&0&0&-1
\\\noalign{\medskip}0&0&0&0&1&0&-1&-1&-1\\\noalign{\medskip}0&0&0&0&1&0
&0&-1&0\\\noalign{\medskip}0&0&0&0&0&1&0&0&1\\\noalign{\medskip}0&0&0&0
&0&1&-1&-1&-1\end{smallmatrix} \right],
 L=\left[ \begin{smallmatrix} 1&1&1&1&1&1&1&1&1
\\\noalign{\medskip}-4&2&-1&0&0&-1&2&-1&-1\\\noalign{\medskip}4&-2&1&0
&0&-2&1&1&1\\\noalign{\medskip}0&0&0&2&-1&0&0&\sqrt {3}&-\sqrt {3}
\\\noalign{\medskip}0&0&0&-2&1&0&0&\sqrt {3}&-\sqrt {3}
\\\noalign{\medskip}5&2&-1&1&1&-1&-1&-1&-1\\\noalign{\medskip}4&-2&1&0
&0&1&1&-1&-1\\\noalign{\medskip}4&1&-2&0&0&1&1&0&0\\\noalign{\medskip}0
&3&-3&0&0&0&0&1&1\\\noalign{\medskip}-3&3&0&1&1&0&0&0&0
\\\noalign{\medskip}0&3&6&0&0&0&0&0&0\\\noalign{\medskip}3&6&0&-1&-1&0
&0&0&0\end{smallmatrix} \right];\]}

\tiny{\[K= \left[ \begin{smallmatrix} 0&0&0&1&2&0&0&0&0
\\\noalign{\medskip}1&0&1&0&0&0&0&0&-1\\\noalign{\medskip}1&0&0&0&0&0&0
&0&1\\\noalign{\medskip}0&1&0&0&0&0&0&0&0\\\noalign{\medskip}0&1&-1&-1
&0&1&-1&0&0\\\noalign{\medskip}0&0&1&0&0&0&0&1&0\\\noalign{\medskip}0&0
&0&1&0&0&0&1&0\\\noalign{\medskip}0&0&0&1&1&0&0&0&-1
\\\noalign{\medskip}0&0&0&0&1&0&-1&-1&0\\\noalign{\medskip}0&0&0&0&1&0
&0&-1&-1\\\noalign{\medskip}0&0&0&0&0&1&-1&-1&0\\\noalign{\medskip}0&0
&0&0&0&1&1&0&0\\\noalign{\medskip}0&0&0&0&0&0&1&1&1
\end{smallmatrix} \right],
 L=\left[\begin{smallmatrix} 1&1&1&1&1&1&1&1&1
\\\noalign{\medskip}4&1&-2&0&0&-2&1&0&0\\\noalign{\medskip}-4&-1&2&0&0
&-1&2&0&0\\\noalign{\medskip}0&0&0&2&0&0&0&\sqrt {3}-1&-\sqrt {3}-1
\\\noalign{\medskip}0&0&0&-2&0&0&0&\sqrt {3}+1&1-\sqrt {3}
\\\noalign{\medskip}5&-1&2&1&0&-1&-1&1&1\\\noalign{\medskip}4&-2&1&0&-
1&1&1&0&0\\\noalign{\medskip}4&1&-2&0&0&1&1&0&0\\\noalign{\medskip}-3&
3&0&1&0&0&0&1&1\\\noalign{\medskip}0&3&-3&0&1&0&0&0&0
\\\noalign{\medskip}0&6&3&0&-1&0&0&0&0\\\noalign{\medskip}3&3&3&-1&1&0
&0&-1&-1\\\noalign{\medskip}0&-3&3&0&1&0&0&0&0\end{smallmatrix}
\right];\]}

\tiny{\[ K= \left[\begin{smallmatrix} 0&0&0&1&2&0&0&0&0
\\\noalign{\medskip}1&0&0&0&0&0&0&0&1\\\noalign{\medskip}1&0&1&0&0&0&0
&0&-1\\\noalign{\medskip}0&1&0&0&0&0&0&1&0\\\noalign{\medskip}0&1&-1&-
1&0&1&-1&-1&0\\\noalign{\medskip}0&0&1&0&0&0&0&0&0\\\noalign{\medskip}0
&0&0&1&1&0&0&-1&-1\\\noalign{\medskip}0&0&0&1&0&0&0&1&0
\\\noalign{\medskip}0&0&0&0&1&0&0&0&-1\\\noalign{\medskip}0&0&0&0&1&0&
-1&-1&0\\\noalign{\medskip}0&0&0&0&0&1&-1&0&0\\\noalign{\medskip}0&0&0
&0&0&1&1&0&0\\\noalign{\medskip}0&0&0&0&0&0&1&1&1\end{smallmatrix}
\right],
 L=\left[\begin{smallmatrix}1&1&1&1&1&1&1&1&1
\\\noalign{\medskip}-4&-1&2&0&0&-1&2&0&0\\\noalign{\medskip}4&1&-2&0&0
&-2&1&0&0\\\noalign{\medskip}3&-3&0&1&0&0&0&\sqrt {3}&-\sqrt {3}
\\\noalign{\medskip}-3&3&0&-1&0&0&0&\sqrt {3}&-\sqrt {3}
\\\noalign{\medskip}2&2&2&2&0&-1&-1&0&0\\\noalign{\medskip}1&4&-2&1&0&
1&1&-1&-1\\\noalign{\medskip}4&-2&1&0&-1&1&1&0&0\\\noalign{\medskip}3&0
&-3&-1&1&0&0&1&1\\\noalign{\medskip}-3&3&0&1&0&0&0&1&1
\\\noalign{\medskip}3&3&3&-1&-1&0&0&1&1\\\noalign{\medskip}3&3&3&-1&1&0
&0&-1&-1\\\noalign{\medskip}0&-3&3&0&1&0&0&0&0\end{smallmatrix}\right];\]}

\tiny{\[K= \left[ \begin{smallmatrix} 0&0&0&1&2&0&0&0&0
\\\noalign{\medskip}1&0&0&0&0&0&0&0&1\\\noalign{\medskip}1&0&1&0&0&0&0
&0&-1\\\noalign{\medskip}0&1&-1&-1&0&1&0&0&0\\\noalign{\medskip}0&1&0&0
&0&0&-1&0&0\\\noalign{\medskip}0&0&1&0&0&0&1&1&0\\\noalign{\medskip}0&0
&0&1&0&0&1&1&0\\\noalign{\medskip}0&0&0&1&1&0&0&0&-1
\\\noalign{\medskip}0&0&0&0&1&0&0&-1&0\\\noalign{\medskip}0&0&0&0&1&0&
-1&-1&-1\\\noalign{\medskip}0&0&0&0&0&1&0&-1&0\\\noalign{\medskip}0&0&0
&0&0&1&-1&0&0\\\noalign{\medskip}0&0&0&0&0&0&0&1&1\end{smallmatrix}
\right],
 L=\left[ \begin{smallmatrix} 1&1&1&1&1&1&1&1&1
\\\noalign{\medskip}-4&-1&2&0&0&-1&2&0&0\\\noalign{\medskip}4&1&-2&0&0
&-2&1&0&0\\\noalign{\medskip}0&0&0&-2&1&0&0&\sqrt {3}&-\sqrt {3}
\\\noalign{\medskip}0&0&0&2&-1&0&0&\sqrt {3}&-\sqrt {3}
\\\noalign{\medskip}5&-1&2&1&1&-1&-1&0&0\\\noalign{\medskip}4&-2&1&0&0
&1&1&-1&-1\\\noalign{\medskip}4&1&-2&0&0&1&1&0&0\\\noalign{\medskip}-3
&3&0&1&1&0&0&0&0\\\noalign{\medskip}0&3&-3&0&0&0&0&1&1
\\\noalign{\medskip}0&6&3&0&0&0&0&-1&-1\\\noalign{\medskip}3&3&3&-1&-1
&0&0&1&1\\\noalign{\medskip}0&-3&3&0&0&0&0&1&1\end{smallmatrix}
\right];\]}

\tiny{ \[K=\left[ \begin{smallmatrix} 0&0&0&1&2&0&0&0&0
\\\noalign{\medskip}1&0&0&0&0&0&-1&-1&0\\\noalign{\medskip}1&0&1&0&0&0
&1&1&0\\\noalign{\medskip}0&1&-1&-1&0&1&-1&0&0\\\noalign{\medskip}0&1&0
&0&0&0&0&0&0\\\noalign{\medskip}0&0&1&0&0&0&-1&0&-1
\\\noalign{\medskip}0&0&0&1&0&0&0&1&0\\\noalign{\medskip}0&0&0&1&1&0&0
&0&-1\\\noalign{\medskip}0&0&0&0&1&0&0&-1&-1\\\noalign{\medskip}0&0&0&0
&1&0&-1&-1&0\\\noalign{\medskip}0&0&0&0&0&1&0&0&1\\\noalign{\medskip}0
&0&0&0&0&1&0&-1&-1\end{smallmatrix}\right],
 L=\left[ \begin{smallmatrix} 1&1&1&1&1&1&1&1&1
\\\noalign{\medskip}-4&2&-1&0&-1&-1&2&0&0\\\noalign{\medskip}4&-2&1&0&
1&-2&1&0&0\\\noalign{\medskip}0&0&0&-2&0&0&0&\sqrt {3}+1&1-\sqrt {3}
\\\noalign{\medskip}0&0&0&2&0&0&0&\sqrt {3}-1&-\sqrt {3}-1
\\\noalign{\medskip}5&2&-1&1&-1&-1&-1&1&1\\\noalign{\medskip}4&-2&1&0&
-1&1&1&0&0\\\noalign{\medskip}4&1&-2&0&0&1&1&0&0\\\noalign{\medskip}0&
3&-3&0&1&0&0&0&0\\\noalign{\medskip}-3&3&0&1&0&0&0&1&1
\\\noalign{\medskip}0&3&6&0&0&0&0&0&0\\\noalign{\medskip}3&6&0&-1&0&0&0
&-1&-1\end{smallmatrix} \right];\]}

\tiny{\[K= \left[\begin{smallmatrix} 0&0&0&1&2&0&0&0&0
\\\noalign{\medskip}1&0&0&0&0&0&0&0&1\\\noalign{\medskip}1&0&1&0&0&0&0
&0&-1\\\noalign{\medskip}0&1&0&0&0&0&0&1&0\\\noalign{\medskip}0&1&-1&-
1&0&1&-1&-1&0\\\noalign{\medskip}0&0&1&0&0&0&0&0&0\\\noalign{\medskip}0
&0&0&1&0&0&1&1&0\\\noalign{\medskip}0&0&0&1&1&0&-1&-1&-1
\\\noalign{\medskip}0&0&0&0&1&0&0&-1&-1\\\noalign{\medskip}0&0&0&0&1&0
&0&0&0\\\noalign{\medskip}0&0&0&0&0&1&-1&0&0\\\noalign{\medskip}0&0&0&0
&0&1&1&0&0\\\noalign{\medskip}0&0&0&0&0&0&0&1&0\\\noalign{\medskip}0&0
&0&0&0&0&0&0&1 \end{smallmatrix} \right],
 L=\left[\begin{smallmatrix}1&1&1&1&1&1&1&1&1
\\\noalign{\medskip}-4&-1&2&0&0&-1&2&0&0\\\noalign{\medskip}4&1&-2&0&0
&-2&1&0&0\\\noalign{\medskip}3&-3&0&1&0&0&0&\sqrt {3}&-\sqrt {3}
\\\noalign{\medskip}-3&3&0&-1&0&0&0&\sqrt {3}&-\sqrt {3}
\\\noalign{\medskip}2&2&2&2&0&-1&-1&0&0\\\noalign{\medskip}4&-2&1&0&0&
1&1&-1&-1\\\noalign{\medskip}1&4&-2&1&-1&1&1&0&0\\\noalign{\medskip}0&
3&-3&0&1&0&0&0&0\\\noalign{\medskip}0&0&0&0&1&0&0&1&1
\\\noalign{\medskip}3&3&3&-1&-1&0&0&1&1\\\noalign{\medskip}3&3&3&-1&1&0
&0&-1&-1\\\noalign{\medskip}3&-3&0&-1&0&0&0&1&1\\\noalign{\medskip}-3&0
&3&1&0&0&0&0&0\end{smallmatrix} \right]; \]}

\tiny{\[K= \left[ \begin{smallmatrix} 0&0&0&1&2&0&0&0&0
\\\noalign{\medskip}1&0&0&0&0&0&0&0&1\\\noalign{\medskip}1&0&1&0&0&0&0
&0&-1\\\noalign{\medskip}0&1&-1&-1&0&1&-1&-1&0\\\noalign{\medskip}0&1&0
&0&0&0&0&1&0\\\noalign{\medskip}0&0&1&0&0&0&0&0&0\\\noalign{\medskip}0
&0&0&1&0&0&0&1&0\\\noalign{\medskip}0&0&0&1&1&0&0&-1&-1
\\\noalign{\medskip}0&0&0&0&1&0&0&0&0\\\noalign{\medskip}0&0&0&0&1&0&-
1&-1&-1\\\noalign{\medskip}0&0&0&0&0&1&-1&0&0\\\noalign{\medskip}0&0&0
&0&0&1&1&0&0\\\noalign{\medskip}0&0&0&0&0&0&1&1&0\\\noalign{\medskip}0
&0&0&0&0&0&0&0&1\end{smallmatrix}\right],
 L= \left[ \begin{smallmatrix} 1&1&1&1&1&1&1&1&1
\\\noalign{\medskip}-4&-1&2&0&0&-1&2&0&0\\\noalign{\medskip}4&1&-2&0&0
&-2&1&0&0\\\noalign{\medskip}-3&3&0&-1&0&0&0&\sqrt {3}&-\sqrt {3}
\\\noalign{\medskip}3&-3&0&1&0&0&0&\sqrt {3}&-\sqrt {3}
\\\noalign{\medskip}2&2&2&2&0&-1&-1&0&0\\\noalign{\medskip}4&-2&1&0&-1
&1&1&0&0\\\noalign{\medskip}1&4&-2&1&0&1&1&-1&-1\\\noalign{\medskip}0&0
&0&0&1&0&0&1&1\\\noalign{\medskip}0&3&-3&0&0&0&0&1&1
\\\noalign{\medskip}3&3&3&-1&-1&0&0&1&1\\\noalign{\medskip}3&3&3&-1&1&0
&0&-1&-1\\\noalign{\medskip}3&-3&0&-1&1&0&0&0&0\\\noalign{\medskip}-3&0
&3&1&0&0&0&0&0\end{smallmatrix} \right]. \]}

\normalsize In each of the remaining ten cases we observe the second
column of $L$ which gives us (up to sign) character values for
$3$-central $z \in \CC_5^G$ in the principal $3$-block. However, by
Lemma \ref{non-vanishing principal block on p-central}, $\chi(z)\neq
0$ for all $\chi \in B_0(G)$. Therefore none of these ten cases
occur.

\subsubsection{The Proof of Theorem \ref{case 2 thm}}
\begin{table}[h]\[\begin{tabular}{|c|c |r r r r r r r r r|}
  \hline
  \; & 1 & $\CC_4^G$ & $\CC_5^G$ & $\CC_{6}^G$ &$\CC_{9}^G$& $\CC_{10}^G$& $\CC_{11}^G$& $\CC_{12}^G$& $\CC_{13}^G$& $\CC_{14}^G$\\
  \hline

 $\chi_1$         &  1&    1&1&1&1&1&1&1&1&1\\
 $\e_2\chi_2$     &  $d_2$ &    -4&-1&2&0&0&-1&2&0&0\\
 $\e_3\chi_{3}$   &  $d_{3}$ &    4&1&-2&0&0&-2&1&0&0\\
 $\e_4\chi_{4}$   &  $d_{4}$&    3&-3&0&1&0&0&0&$\sqrt {3}$&$-\sqrt {3}$\\
 $\e_5\chi_{5}$   &  $d_{5}$&    -3&3&0&-1&0&0&0&$\sqrt {3}$&$-\sqrt {3}$\\
 $\e_6\chi_{6}$   &  $d_{6}$&    2&2&2&2&0&-1&-1&0&0\\
 $\e_7\chi_{7}$   &  $d_{7}$&    1&1&1&1&-1&1&1&-1&-1\\
 $\e_8\chi_{8}$   &  $d_{8}$&    4&1&-2&0&0&1&1&0&0\\
 $\e_9\chi_{9}$   &  $d_{9}$&    0&3&-3&0&1&0&0&0&0\\
 $\e_{10}\chi_{10}$& $d_{10}$&   -3&3&0&1&0&0&0&1&1\\
 $\e_{11}\chi_{11}$& $d_{11}$&   3&3&3&-1&-1&0&0&1&1\\
 $\e_{12}\chi_{12}$& $d_{12}$&   3&3&3&-1&1&0&0&-1&-1\\
 $\e_{13}\chi_{13}$& $d_{13}$&   0&-3&3&0&1&0&0&0&0\\
 $\e_{14}\chi_{14}$& $d_{14}$&   3&-3&0&-1&0&0&0&1&1 \\
\hline
\end{tabular}\]\caption{Part of the principal $3$-block of the
character table of $G$ (up to sign, $\e_i=\pm 1$).}\label{part of principal block}\end{table}The
only candidate for $L$ we need to consider is the first one. Therefore Table \ref{part of principal
block} displays part of the principal $3$-block of $G$. Note that characters are displayed up to
sign so $B_0(G)=\{\chi_1,\hdots, \chi_{14}\}$ and $\e_i=\pm 1$ for each $2 \leq i \leq 14$. Also we
set $d_i:=\e_i\chi_i(1)$ so each $d_i$ is an integer which may be positive or negative.

\begin{lemma}The character degrees satisfy the following
congruences.
\begin{enumerate}[$(i)$]
 \item $d_{6} \equiv -52~\mathrm{mod}~ 81$;
 \item $d_{7} \equiv 1~\mathrm{mod}~ 81$;
 \item $d_{11} \equiv -51~\mathrm{mod}~ 81$;
 \item $d_{12} \equiv -51
~\mathrm{mod}~ 81$.
\end{enumerate}
\end{lemma}
\begin{proof}
For each $i\in \{6,7,11,12\}$ we restrict $\e_i\chi_i$ to $H$ and
calculate the character inner product $(\e_i\chi_i, \psi_{12})_H$
which is integral. We calculate:
\[(\e_6\chi_6, \psi_{12})_H=\frac{1}{|H|}(8d_6-48-16+48+288+72+72)=\frac{d_6+52}{81}.\]
Thus $d_{6} \equiv -52~\mathrm{mod}~ 81$. Similarly:
\[(\e_7\chi_7, \psi_{12})_H=\frac{1}{|H|}(8d_7-24-8+24+144-72-72)=\frac{d_7-1}{81};\]
\[(\e_{11}\chi_{11},
\psi_{12})_H=\frac{1}{|H|}(8d_{11}-72-24+72+432)=\frac{d_{11}+51}{81};\]and
\[(\e_{12}\chi_{12},
\psi_{12})_H=\frac{1}{|H|}(8d_{12}-72-24+72+432)=\frac{d_{12}+51}{81}.\]
Hence $d_{7} \equiv 1~\mathrm{mod}~ 81$ and $d_{11} \equiv d_{12}
\equiv -51~\mathrm{mod}~ 81$.
\end{proof}

\begin{lemma}\label{alt9-case2 d7 isn't 1}
$d_{11}\neq -51$, $d_{12}\neq -51$ and $d_7\neq 1$.
\end{lemma}
\begin{proof}
Suppose $d_{11}=-51$ and consider the characters of $H$
\[\psi_1+\psi_2=(2,2,0,2,2,2,2,0,2,0,2,2,0,0)\] and
\[\psi_4+\psi_5=(6,-2,0,6,6,6,0,0,-2,0,0,0,0,0).\] We calculate the $H$-character inner products
$(\e_{11}{\chi_{11}}{{\big|_H}},\psi_1+\psi_2)$ and
$(\e_{11}{\chi_{11}}{{\big|_H}},\psi_4+\psi_5)$ both of which are
integral. Also, since $d_{11}\leq 0$, $\e_{11}=-1$ and so both inner
products are non-positive. Let $n=\e_{11}\chi_{11}(x_2)$ for $x_2 \in
\CC_2$ (notice this is the only character value which appears in the
inner products which we have not calculated). We calculate:

\[(\e_{11}\chi_{11}|_H,\psi_1+\psi_2)=\frac{1}{|H|}(-102+54n+36+48+72+432-108)=\frac{7+n}{12},\]
which implies $n \leq -7$, and,
\[(\e_{11}\chi_{11}|_H,\psi_4+\psi_5)=\frac{1}{|H|}(-306-54n+108+144+216+108)=\frac{5-n}{12},\]
which implies $n \geq 5$. Therefore we have a contradiction and
$d_{11} \neq 51$.

The same calculation assuming $d_{12}=-51$ gives a similar
contradiction and so $d_{12}\neq -51$.

Suppose $d_7=1$. Then $G' \neq G$. Moreover if $\rho$ is the
representation of $G$ affording $\chi_7$ then $\CC_i \leq
\mathrm{Ker}(\rho)$ for $i=4,5,6,7,11,12$. Thus $G'$ is a proper
normal subgroup of $G$ which contains a Sylow $3$-subgroup of $G$.
This contradicts Lemma \ref{No normal subgroup containing S}. Hence
$d_7 \neq 1$.
\end{proof}

Let $a \in \CC_4$, $b \in \CC_5$ and $c \in \CC_6$. We
find expressions involving the unknown character degrees $d_i$ using
Lemma \ref{part2-structure constants} to calculate $\a^G_{abc}$,
$\a^G_{aac}$ and $\a^G_{cca}$. We calculate:
\[\a^G_{abc}-\frac{\a^G_{aac}-\a^G_{cca}}{6} = 1+ \frac{8}{d_{6}}+\frac{1}{d_7}+\frac{27}{d_{11}}+\frac{27}{d_{12}} =
\frac{15552}{|G|}.\] Since $d_{6} \equiv -52~\mathrm{mod}~ 81$,
either $d_6 \leq -52$ or $d_6\geq 0$. In either case we have $1/d_6
\geq -1/52$. Similarly  $1/d_7\geq -1/80$. Now $d_{11}\neq -51 \neq
d_{12}$ so we have $1/d_{11} \geq -1/132$ and $1/d_{12} \geq
-1/132$. Therefore
\[1+\frac{8}{d_{6}}+ \frac{1}{d_7}+\frac{27}{d_{11}}+\frac{27}{d_{12}}
\geq \frac{4857}{11440}\] and so $\frac{15552}{|G|}\geq
\frac{4857}{11440}$. This gives us that $|G| < 36630$. Now a formula
due to Frobenius (see \cite[s10, p28]{Zassenhaus}) says that
\[| \{x \in G |x^{3^n}=1, n \in \mathbb{N}\}| \equiv 0
~\mathrm{mod}~ 81.\] Therefore, since we have three conjugacy classes of elements of order three and two conjugacy  classes of elements of order nine, we have
\[1+|G|(1/108+1/81+1/54+1/9+1/9)=1+|G|85/324\equiv 0
~\mathrm{mod} ~81\] and  so \[|G|(1/108+1/81+1/54+1/9+1/9)=|G|85/324\equiv -1
~\mathrm{mod} ~81.\] It follows that $|G|/81 \equiv -1
~\mathrm{mod}~ 81$.

Consider $|G|/81$. This is an integer which is a multiple of 8 and
lies between $8$ and $452$. It is easy to check that the only such
integer is 80 so $|G|=6480$. By Lemma \ref{alt9-case2 d7 isn't 1}, $d_7 \neq 1$  and since $d_7 \equiv 1~\mathrm{mod}~ 81$, $d_7^2\geq 80^2$. Also $d_6^2 \geq 29^2$  so
$|G|\geq d_6^2 +d_7^2 \geq 29^2+80^2=7241$. This is our final
contradiction in Case 2 and concludes the proof of Theorem \ref{case
2 thm}.


\subsection{Case 3}

In this final case we identify the group $\alt(9)$. We  assume that
for $x \in \CC_4$, $C_G(x) \cong 3 \times \alt(6)$ and
$\CC_7^G=\CC_5^G$. Since no group satisfies Hypothesis \ref{case 2
hyp}, it is clear that this is equivalent to an assumption that $H\neq
G$.

\begin{hyp}\label{case 3 hyp}
Let $G$ satisfy Hypothesis A and in addition assume that $H$ is a proper subgroup of $G$.
\end{hyp}
\begin{thm}\label{case 3 thm}
If $G$ satisfies Hypothesis \ref{case 3 hyp} then $G \cong \alt(9)$.
\end{thm}

We set up some further notation for this section. Fix $x \in
\CC_4$ and set $C:=C_G(x) \cong 3 \times \alt(6)$ and
$K=N_G(\<x\>)$. Since $x$ is conjugate in $H$ to its inverse,
$[K:C]=2$.

\begin{lemma}\label{F1,F2,P1 and P2}
There exist subgroups $F_1,F_2,P_1,P_2 \leq C$ such that, for $i
\in \{1,2\}$, $F_i \cong 2 \times 2$, $C_C(F_i)=\<x,F_i\>$ and
$3 \times 3\cong P_i \in \syl_3(N_C(F_i))$ where
$N_C(F_i)\cong 3 \times \sym(4)$. Furthermore, $|P_1 \cap
\CC_4^G|=|P_1 \cap \CC_6^G|=4$, $|P_2 \cap \CC_4^G|=|P_2 \cap
\CC_6^G|=2$ and $|P_2 \cap \CC_5^G|=4$.
\end{lemma}
\begin{proof}
First we  observe that any group $Y \cong \alt(6)$ has Sylow
$2$-subgroups which are dihedral of order eight and  any fours group
in $Y$ is self-centralizing with normalizer in $Y$ isomorphic to
$\sym(4)$. Moreover, $Y$ has two conjugacy classes of fours groups
$A^Y$ and $B^Y$ say where $AB\in \syl_2(Y)$. Furthermore $Y$ has two
conjugacy classes of subgroups of order three $\syl_3(N_Y(A))^Y$ and
$\syl_3(N_Y(B))^Y$.

Now we fix two fours groups $F_1,F_2 \leq C$ such that $\dih(8)\cong F_1F_2\in \syl_2(C)$ and
choose $P_1,P_2 \leq C$ such that $x \in P_i \in \syl_3(N_C(F_i))$. It is clear from the structure
of $\alt(6)$ that $C_C(F_i)=\<x,F_i\>$ and $N_C(F_i)\cong 3 \times \sym(4)$. Also $P_1$ and $P_2$
are not conjugate in $C$. Since $\alt(6)$ has two conjugacy classes of elements of order three, $C$
has two  conjugacy classes of subgroups of order nine containing $x$. Thus $P_1$ and $P_2$ are
representatives of these two classes. Now by Lemma \ref{order of sylow 3's and H is self
normalizing} $(vi)$, $H$ also has two conjugacy classes of subgroups of order nine in $J \leq C$
containing $x$ and these are non-conjugate in $G$. It follows that $P_1$ and $P_2$ are
representatives of these classes and so are non-conjugate in $G$.  Therefore, using Lemma
\ref{order of sylow 3's and H is self normalizing}, we may assume that $|P_1 \cap \CC_4^G|=|P_1
\cap \CC_6^G|=4$, $|P_2 \cap \CC_4^G|=|P_2 \cap \CC_6^G|=2$ and $|P_2 \cap \CC_5^G|=4$.
\end{proof}

We fix notation such that $P_2=\{1,x,x^2,y,y^2,z,z^2,w,w^2\}$ where
$y,z\in \CC_5^G$ and $w \in \CC_6^G$.

\begin{lemma}\label{finding sym6}
$K/\<x\> \cong \sym(6)$.
\end{lemma}
\begin{proof}
Set $\bar{K}:=K/\<x\>$. Then $\bar{K}$ has an index two subgroup $\bar{C}\cong \alt(6)$. Consider
$\bar{T}:=C_{\bar{K}}(\bar{C})$. Since $\bar{T}\cap \bar{C}=1$, $|\bar{T}|\leq 2$. Suppose
$|\bar{T}|=2$. By coprime action and an isomorphism theorem,
$\bar{J}=C_{\bar{J}}(\bar{T})=\bar{C_J{T}}\cong C_J(T)$. Thus $C_J(T)$ is a subgroup of $J$ of
order nine. Lemma \ref{order of sylow 3's and H is self normalizing} $(v)$ now gives a
contradiction. Therefore $\bar{T}=1$ and $\bar{K}$ is isomorphic to a subgroup of $\aut(\alt(6))$
and hence $\bar{K} \cong \sym(6), \mathrm{M}_{10}$ or $\PGL_2(9)$. Now $\mathrm{M}_{10}$ and
$\PGL_2(9)$ both have one conjugacy class of subgroups of order three (see \cite{atlas}). However
$\bar{P_1}$ and $\bar{P_2}$ are non-conjugate in $\bar{K}$. Thus $\bar{K} \cong \sym(6)$.
\end{proof}

By Lemma \ref{finding sym6}, $K$ has Sylow $2$-subgroups
isomorphic to $2\times \dih(8)$. Hence we may fix some further
notation by setting $E_1,E_2 \leq K$ to be the elementary abelian
subgroups of $K$ of order eight such that $E_1>F_1$, $E_2>F_2$ and
$E_1E_2\in \syl_2(K)$.

\begin{lemma}\label{Normalizer of Ei in K} Let $i \in \{1,2\}$
then $N_K(E_i) \cong 2 \times \sym(4)$ and $\mathcal{Z}(N_K(E_2)) \leq E_1 \cap
E_2$ contains an involution in $\CC_3^G$.
\end{lemma}
\begin{proof}
Since $E_i\nleq C$, $[E_i,x]=\<x\>$ and so $x$ does not normalize $E$.  Therefore $N_K(E_i)\cong
N_K(E_i)\<x\>/\<x\>\leq K/\<x\>$ which is isomorphic to the normalizer in $\sym(6)$ of an
elementary abelian subgroup of order eight. It follows that $N_K(E_i) \cong 2 \times \sym(4)$.
Notice that $F_i=E_i \cap C\vartriangleleft N_K(E_i)$ and so $\sym(4)\cong
N_{O^3(C)}(F_i)\vartriangleleft N_K(E_i)$.

Let $R_i\in \syl_3(N_K(E_i))$ then $|R_i|=3$ and we must have that $R_i \leq N_{O^3(C)}(F_i)\leq
O^3(C)$  so we may assume $R_i \leq P_i$. Recall that $P_2=\{1,x,x^2,y,y^2,z,z^2,w,w^2\}$ where
$y,z\in \CC_5^G$ and $w \in \CC_6^G$. Since $R_2=P_2 \cap O^3(C)$ is inverted in $O^3(C)\cong
\alt(6)$, the elements $xr$ and $xr^2$ ($\<r\>=R_2$) are conjugate. Thus we must have that $w \in
O^3(C)$. Therefore $R_2^\# \subset \CC_6^G$. Now for $y \in \CC_6^G$, by Lemma \ref{order of sylow
3's and H is self normalizing} $(iv)$, $C_G(y) \leq H$ has order $3^32$ and commutes with an
involution in $\CC_3$. Therefore we have that  $\mathcal{Z}(N_K(E_2))^\# \in \CC_3^G$ and since $E_1 \leq
N_K(E_2)$, $[E_1, \mathcal{Z}(N_K(E_2))]=1$. Since $N_K(E_1)\cong 2 \times \sym(4)$, we see that
$C_K(E_1)=E_1$ and so   $\mathcal{Z}(N_K(E_2)) \leq E_1 \cap E_2$.
\end{proof}

%
%

By exploiting the $3$-subgroups normalizing $F_1$ and $F_2$, we
are able to determine $N_G(F_2)$ in the following lemma. However,
we are not able to fully determine $N_G(F_1)$ until we have
control of the $2$-structure of $G$.

\begin{lemma}\label{applying FT to Fi}
\begin{enumerate}[$(i)$]
 \item $C_G(F_1)/F_1$ has a self-centralizing, but
 not self-normalizing, element of order three and $C_G(F_1)/F_1\ncong
 \PSL_2(7)$.
 \item $C_G(F_2)=E_2\<x\>$.
 \item $C_G(E_2)=E_2$.
 \end{enumerate}
\end{lemma}
\begin{proof}
Let $i \in \{1,2\}$. By Lemma \ref{F1,F2,P1 and P2}, $C_G(F_i) \cap C=\<F_i,x\>\cong 3
\times 2 \times 2$ and $N_G(F_i) \cap C \cong 3 \times
\sym(4)$. In particular this tells us that $N_G(F_i)/C_G(F_i)\cong
\sym(3)\cong \mathrm{Aut}(F_i)$.  We also see that $C_G(F_i)/F_i$
has a self-centralizing element of order three. Suppose
$C_G(F_i)/F_i\cong \PSL_2(7)$. Since $C_G(F_i)/F_i$ is normalized
by $N_G(F_i)/C_G(F_i)\cong \sym(3)$ and
$|\mathrm{Out}(\PSL_2(7))|=2$, $N_G(F_i)/F_i$ has a subgroup
isomorphic to $3 \times \PSL_2(7)$. This forces an element of
order three in $G$ to commute with an element of order seven which
is not possible. Thus $C_G(F_i)/F_i\ncong \PSL_2(7)$. Since
$E_i\leq C_G(F_i)\cap K$ and $E_i \nleq C$, $\<x\>$ is not
self-normalizing in $C_G(F_i)$. This proves part
$(i)$.

Now we fix $i=2$ and apply the Feit--Thompson Theorem (Theorem
\ref{Feit-Thompson}) to $C_G(F_2)/F_2$. Set
$X:=O_{3'}(C_G(F_2))/F_2$. Then $X$ is acted on by $P_2$ and by
coprime action, $X=\<C_X(\<r\>)|1< \<r\>< P_2\>$. Recall
$P_2=\{1,x,x^2,y,y^2,z,z^2,w,w^2\}$  where $y,z\in \CC_5^G$ and $w
\in \CC_6^G$. Since $\<x\>, \<y\>$ and $\<z\>$ act
fixed-point-freely on $X$, $X=C_X(\<w\>)$. Now $|C_G(\<w\>)|=3^32$
and so $|X|\leq 2$. However $X$ admits a fixed-point-free
automorphism of order three and so $|X|=1$ and
$O_{3'}(C_G(F_2))=F_2$.

So we have $C_G(F_2)/F_2\cong \alt(5)$ or $\sym(3)$. Suppose $C_G(F_2)/F_2\cong \alt(5)$. Since
$C_G(F_2)/F_2$ is normalized by $N_G(F_2)/F_2$ and $|\mathrm{Out}(\alt(5))|=2$, $N_G(F_2)/F_2$
contains a subgroup isomorphic to $3 \times \alt(5)$. Therefore an element of order three in
$P_2F_2/F_2 \cong P_2$ commutes with $C_G(F_2)/F_2$ and in particular with an element of order
five. However the only elements of order three in $P_2$ which commute with an element of order five
are $x$ and $x\inv$ and if $[F_2x,C_G(F_2)/F_2]=1$ then  $F_2x \in \mathcal{Z}(C_G(F_2)/F_2)=1$ which is a
contradiction.

Therefore we may conclude that $C_G(F_2)/F_2\cong \sym(3)$ and so
$C_G(F_2)=C_K(F_2)=E_2\<x\>$. It follows immediately that
$C_G(E_2)=E_2$.
\end{proof}

\begin{lemma}\label{the involution s}
Let $s \in \CC_3^G$. Then $C_G(s)/\<s\>$ has a self-centralizing
element of order three in $\CC_6^G$ and $s$ is in the centre of a subgroup of $G$
isomorphic to $\SL_2(3)$.
\end{lemma}
\begin{proof}
It is clear from Table \ref{char table H} that $s\in \CC_3$ commutes
with some $y \in \CC_6$. Also $C_G(y)=J\<s\>$ and $C_J(s)=\<y\>$.
Therefore $C_{C_G(s)}(y)=\<y,s\>$ and  $C_G(s)/\<s\>$ has a
self-centralizing element of order three $\<s\>y$.

We see that $s$ lies at the centre of a subgroup isomorphic to
$\SL_2(3)$ from Lemma \ref{3-elements in S not J}.
\end{proof}

By Lemma \ref{Normalizer of Ei in K}, there is an involution $s \in
\mathcal{Z}(N_K(E_2))$ such that $s \in \CC_3^G$ and $s \in E_1 \cap E_2$. Set
$L:=C_G(s)$.

\begin{lemma}\label{Centralizer of s}
\begin{enumerate}[$(i)$]
\item $O_2(L)\cong 2_+^{1+4}$ and $L/O_2(L)\cong
\sym(3)$.

\item If $T \in \syl_2(G)$ then $|T|=2^6$ and $|\mathcal{Z}(T)|=2$ with $\mathcal{Z}(T)^\# \in \CC_3^G$.

\item  $E_2\trianglelefteq L$.
\end{enumerate}
\end{lemma}
\begin{proof}
By Lemma \ref{the involution s},  $L/\<s\>$ satisfies Theorem
\ref{Feit-Thompson}. Suppose $L/\<s\> \cong \PSL_2(7)$. By Lemma
\ref{the involution s}, $s$ lies at the centre of a subgroup
isomorphic to $\SL_2(3)$. This implies that $L$ does not split over
$\<s\>$ and so it follows from calculation of the Schur Multiplier
of $\PSL_2(7)$ (see \cite{atlas} for example) that  $L \cong
\SL_2(7)$. However $E_2 \leq L$ is elementary abelian of order $8$
and $\SL_2(7)$ has no such subgroup. Thus $L/\<s\> \ncong
\PSL_2(7)$.

Let $Q:=O_{3'}(L)$ then by Theorem \ref{thompson-nilpotent}, $Q/\<s\>$ is nilpotent which means $Q$ is
also nilpotent. Let $R \in \syl_3(N_K(E_2))$. Then $R$ centralizes
$s$ and so $R\in \syl_3(L)$. By coprime action, $E_2=C_{E_2}(R)
\times [E_2,R]= \<s\> \times [E_2,R]$. Suppose $E_2 \nleq Q$ then $Q \cap E_2=\<s\>$.
Since $Q \cap N_K(E_2)$ is normalized by $R$ and $N_K(E_2)\cong
2 \times \sym(4)$, $N_K(E_2) \cap Q=\<s\>$. Therefore
$QN_K(E_2)/Q \cong N_K(E_2)/(N_K(E_2) \cap Q)\cong \sym(4)$.
However, by Theorem \ref{Feit-Thompson}, $L/Q\cong 3$, $\sym(3)$ or $\alt(5)$ so this is
impossible. Thus $E_2 \leq Q$. Since $Q$ is nilpotent and
$C_G(E_2)=E_2$ by Lemma \ref{applying FT to Fi}, $Q$ is a
$2$-group. Furthermore, $\mathcal{Z}(Q)\leq E_2$. If $\mathcal{Z}(Q)=E_2$ then $Q=E_2$
and $L=N_K(E_2)\cong 2 \times \sym(4)$. However this contradicts
Lemma \ref{the involution s} which says that $L$ contains a
subgroup isomorphic to $\SL_2(3)$. So $s \in \mathcal{Z}(Q) <E_2$ and $\mathcal{Z}(Q)$
is normalized by $R$. Thus $\mathcal{Z}(Q)=\<s\>$. Furthermore
$C_L(Q)=\<s\>$ which implies that $L/Q$ is isomorphic to a
subgroup of $\out(Q)$.

Suppose $L/Q\cong \alt(5)$. Then $Q/\<s\>$ is elementary abelian by Theorem \ref{Higman's SL2 Thm}.
Therefore $Q$ is an extraspecial group and $E_2\unlhd Q$. Now $1\neq Q/E_2$ embeds into $\aut(E_2)
\cong \GL_3(2)$ (as $E_2=C_Q(E_2)$) and is $R$-invariant. Therefore $|Q/E_2|=2^2$. Hence $Q$ has
order $2^5$ and contains an elementary abelian subgroup of order $2^3$ which implies that $Q\cong
2_+^{1+4}$. This is a contradiction since $\out(2_+^{1+4})$ does not contain a subgroup isomorphic
to $\alt(5)$. It is clear that $R$ is not self-normalizing in $L$ and so $L/Q\cong \sym(3)$ and
$L=QN_K(E_2)$.

Consider $Q_0:=N_Q(E_2)$. Then $Q_0/E_2$ is $R$-invariant and is
isomorphic to a subgroup of $\GL_3(2)$. Therefore $|Q_0|=2^5$.
Since $C_G(E_2)=E_2$, $Q_0$ is non-abelian. Since $Q_0/E_2$ has
order four and is acted on fixed-point-freely by $R$, $Q_0/E_2$ is
elementary abelian. Thus $\Phi(Q_0)\leq E_2$. Since $\Phi(Q_0)$ is
$R$ invariant, $\Phi(Q_0)=E_2$ or $\<s\>$.

Suppose $Q>Q_0$ then $N_Q(Q_0)>Q_0$. Thus, if $E_2=\Phi(Q_0)$ then
$E_2\trianglelefteq N_Q(Q_0)>Q_0$ which is a contradiction.
Therefore $\Phi(Q_0)=\<s\>$ which proves that $Q_0$ is
extraspecial. Furthermore, $E_2 \leq Q_0$ implies $Q_0\cong
2_+^{1+4}$. Now $Q_0 N_L(Q_0)/Q_0$ is isomorphic to a
subgroup of $\mathrm{Out}(2_+^{1+4})\cong \sym(3)\wr \sym(2)$ and
contains a proper normal $2$-subgroup $N_Q(Q_0)/Q_0$ which admits
a fixed-point-free action by $RQ_0/Q_0$. This is a contradiction.
Therefore $Q=Q_0$ which is to say that $E_2 \trianglelefteq Q$ and
$|Q|=2^5$ with $\Phi(Q)\leq E_2$. Furthermore, $E_2\trianglelefteq
L=QN_K(E_2)$.

By Lemma \ref{the involution s}, there exists a subgroup $A\leq L$
such that $A \cong \SL_2(3)$. Moreover, since $R \in \syl_3(L)$, we
may assume $R \leq A$. Observe that $Q_8\cong O_2(A) \leq Q$ since
$[O_2(A),R]=O_2(A)$. Therefore $Q=\<E_2,O_2(A)\>$. Consider $O_2(A)<
N_Q(O_2(A))\leq Q$. Since $R$ acts fixed-point-freely on
$Q/N_Q(O_2(A))$ it must be trivial. So $O_2(A)\trianglelefteq Q$ and
$Q/O_2(A)$ is elementary abelian. This implies that $\Phi(Q)\leq E_2
\cap O_2(A)=\<s\>$. It is therefore clear that
$Q'=\mathcal{Z}(Q)=\Phi(Q)=\<s\>$. Hence $Q\cong 2_+^{1+4}$.

Finally, let $T \in \syl_2(L)$. Then $\mathcal{Z}(T)\leq C_T(E_2) \leq E_2\leq
Q$ and so $\<s\>\leq \mathcal{Z}(T)\leq \mathcal{Z}(Q)=\<s\>$. Therefore $\mathcal{Z}(T)=\<s\>$
which implies that $N_G(T) \leq L$ and so $T \in \syl_2(G)$.
\end{proof}

We continue to set $Q=O_2(C_G(s))$.
\begin{lemma}\label{s not weakly closed}
$s$ is not weakly closed in $Q$ with respect to $G$.
\end{lemma}
\begin{proof}
Suppose for a contradiction that $s$ is weakly closed in $Q$ with respect to $G$. Since $E_2 \leq
Q$,  $s$ is also weakly closed in $E_2$ with respect to $G$. Also, by Lemma \ref{Normalizer of Ei
in K}, $s \in E_1 \cap E_2$ so $E_1 \leq L$. Since $s$ is the unique conjugate of itself  in
$E_2$, $L \leq N_G(E_2)\leq L$ and so $L=N_G(E_2)$. Therefore $Q=O_2(N_G(E_2))$. Since $E_1 \nleq
O_2(N_K(E_2))=E_2$,  $E_1 \nleq Q$.

Let $u \in \mathcal{Z}(N_K(E_1))\cong 2 \times \sym(4)$ and suppose $u$ is conjugate to $s$. Then $C_G(u)$
has shape $2_+^{1+4}. \sym(3)$. It follows that $E_1\leq O_2(C_G(u))$. Since $E_1 \nleq Q$, $s \neq
u$ and because $s \in E_1 \cap E_2 \leq O_2(C_G(u))$,  $u$ is not weakly closed in $O_2(C_G(u))$.
This contradicts our assumption on $s$. Therefore $u$ is not conjugate to $s$.  This implies that
$E_1$ contains at least three conjugates of $s$ as $s$ is not central in $N_K(E_1)$. Moreover
$F_1=E_1\cap C_G(x)$ and every element of order two in $F_1$ commutes with $x \in \CC_4$ however,
by Lemma \ref{the involution s}, involutions in $\CC_3^G$ commute only with elements of order three
in $\CC_6^G$. Therefore $F \cap \CC_3^G=\emptyset$. It follows that $E_1$ has exactly three
conjugates of $s$, namely $\{s,s_1,s_2\}$. If $\<s,s_1,s_2\>$ has order four then
$\<s,s_1,s_2\>\cap F_1$ would have order two and contain one of $s$, $s_1$ or $s_2$ which is not
possible. Thus $\<s,s_1,s_2\>=E_1$. Therefore $N_G(E_1)/C_G(E_1)$ is isomorphic to a subgroup of
$\sym(3)$. Since $s \in E_1$, $C_G(E_1)\leq L$ and since a Sylow $3$-subgroup of $L/\<s\>$  is
self-centralizing,  $C_G(E_1)$ is a $2$-group. Let $T \in \syl_2(L)$ such that $E_1 \leq T$. Notice
that we necessarily have that $E_2\vartriangleleft T$. Thus $N_T(E_1)=N_T(E_1E_2)>E_1E_2$. Hence
$N_T(E_1)$ has order at least $2^5$. In particular, $C_G(E_1)$ has order a multiple of $2^4$ and
therefore $2^4\mid |C_G(F_1)|$.

Suppose $2^5\mid |C_G(F_1)|$. Since $N_G(F_1)/C_G(F_1)\cong \sym(3)$, $2^6\mid |N_G(F_1)|$.
Therefore, by Lemma \ref{Centralizer of s}, we may choose $T \in \syl_2(G) \cap \syl_2(N_G(F_1))$.
Now $F_1\trianglelefteq T$ which implies $F_1 \cap \mathcal{Z}(T)\neq 1$. However this is a contradiction
since $F_1 \cap \CC_3^G=\emptyset$ and by Lemma \ref{Centralizer of s},  $\mathcal{Z}(T)^\# \subset \CC_3^G$.
Thus $C_G(F_1)$ has Sylow $2$-subgroups of order $2^4$ and it follows that $\syl_2(C_G(E_1))\cap
\syl_2(C_G(F_1))\neq \emptyset$.

Recall Lemma \ref{applying FT to Fi} $(i)$ which together with
Theorem \ref{Feit-Thompson} implies that $C_G(F_1)$ has a nilpotent
normal subgroup $N$ such that  $C_G(F_1)/N \cong \sym(3)$ or $N$ is
a $2$-group and $C_G(F_1)/N \cong \alt(5)$. Suppose $C_G(F_1)/N
\cong \sym(3)$. Then $|N/F_1|$ is an odd multiple of $2$. However
$N/F_1$ has a fixed-point-free automorphism and is nilpotent. This
contradiction implies $C_G(F_1)/N \cong \alt(5)$ and $N$ is a
$2$-group. Hence $N=F_1$.

Choose $U \in \syl_2(C_G(F_1))\cap \syl_2(C_G(E_1))$. Then $U \leq L$ and $E_1 \leq \mathcal{Z}(U)$ and so
$U$ must be abelian and contain $s$.   Since $\alt(5)$ has five Sylow $2$-subgroups, $C_G(F_1)$ has
five Sylow $2$-subgroups. Since $N_G(F_1)/C_G(F_1)\cong \sym(3)$, Sylow $3$-subgroups of $N_G(F_1)$
have order nine and act on this set of order five with at least one fixed-point. Thus we may choose
$P \in \syl_3(N_{N_G(F_1)}(U))$. Consider $L \cap P$ which normalizes $\<F_1,s\>=E_1$ and
$Q=O_2(L)$. Thus $L \cap P$ normalizes $QE_1 \in \syl_2(L)$. However by Lemma \ref{Centralizer of
s}, $\mathcal{Z}(QE_1)=\<s\>$ and so $N_G(QE_1)=N_L(QE_1)=QE_1$ since $L$ has shape $2^{1+4}.\sym(3)$.
Therefore $L \cap P=1$ which means that $U$ contains at least nine conjugates of $s$.  Since $U
\leq L$,  $|U \cap Q|\geq 2^3$. Now, $s$ is weakly closed in $Q$ so $U \cap Q$ contains no distinct
conjugate of $s$. Hence every element in $U \bs Q$ must be a conjugate of $s$. However this forces
$F_1 \leq Q\cap U$ and then $E_1=\<F_1,s\> \leq Q$. This is our final contradiction and we may
conclude that $s$ is not weakly closed in $Q$ with respect to $G$.
\end{proof}

\begin{lemma}
$G \cong \alt(9)$.
\end{lemma}
\begin{proof}
We see from Lemmas \ref{Centralizer of s} and \ref{s not weakly
closed} that $G$ satisfies Theorem \ref{Ascbacher M12}. It follows
from the order of a Sylow $3$-subgroup of $G$ that $G \cong
\alt(9)$.
\end{proof}



\chapter{A Certain 3-Local Hypothesis}\label{chaper general hypothesis}

We recall that a group $H$ is said to have characteristic $p$ ($p$ an odd prime) if
$C_H(O_p(H))\leq O_p(H)$. Moreover a group $G$ has local characteristic $p$ if every $p$-local
subgroup of $G$ has characteristic $p$ and $G$ has parabolic characteristic $p$ if every $p$-local
subgroup containing a Sylow $p$-subgroup of $G$ has characteristic $p$. An interesting situation
which arises within the ongoing project to understand groups of local characteristic $p$ concerns groups
$G$ for which $C_G(z)$ has characteristic $3$ where $z$ is $3$-central in $G$ and
$O_3(C_G(z))\cong 3_+^{1+4}$. Five of the sporadic simple groups have this structure as well
as some groups of Lie type in defining characteristic three and some groups of Lie type in defining
characteristic two also. In this chapter we consider groups which satisfy this characteristic three
condition as well as a further non-weak closure hypothesis as follows.

\begin{hyp}\label{3^1+4 Hypothesis}
Let $G$ be a finite group and let $Z$ be the centre of a Sylow $3$-subgroup of $G$ with
$Q:=O_3(C_G(Z))$.  Suppose that
\begin{enumerate}[$(i)$]
\item $Q\cong 3_+^{1+4}$;
\item $C_G(Q)\leq Q$; and
\item $Z \neq Z^x \leq Q$ for some $x \in G$.
\end{enumerate}
\end{hyp}
We describe many properties which groups satisfying Hypothesis \ref{3^1+4 Hypothesis} have. These
will be used in Chapter \ref{Chapter-O8Plus2} to characterize two almost simple groups and also in
 Chapter \ref{Chapter-HN} to characterize the sporadic simple group $\HN$. Also, in
Section \ref{GenHyp-Concluding Remarks} we list the nine almost simple groups satisfying Hypothesis
\ref{3^1+4 Hypothesis}.

\section{Groups Satisfying Hypothesis \ref{3^1+4 Hypothesis}}

Fix the following notation.
\begin{enumerate}[$(i)$]
\item $x \in G\bs N_G(Z)$ such that $Z^x \leq Q$.
\item  $Y:=\<Z,Z^x\>$.
 \item  $L:=\<Q,Q^x\>$.
 \item $W:=\<C_Q(Y),C_{Q^x}(Y)\>$.
 \item $S:=\<Q,W\>$.
\end{enumerate}

\begin{lemma}\label{Prelims-EasyLemma}
\begin{enumerate}[$(i)$]
\item $|Z|=3$.
\item $C_G(Z)/Q$ is isomorphic to a subgroup of $\Sp_4(3)$.
\item $C_{C_G(Z)}(Q/Z)=Q$.
\item $Q \cap Q^x$ is elementary abelian.
\end{enumerate}
\end{lemma}
\begin{proof}
$(i)$ By hypothesis, $Z=\mathcal{Z}(P)$ for some $P \in \syl_3(G)$. Since $Q=O_3(C_G(Z))$, $Q \leq P$. Therefore $[Q,Z]=1$ which implies that $Z \leq Q$ since $C_G(Q)\leq Q$. Thus $Z \leq \mathcal{Z}(Q)$ and $|\mathcal{Z}(Q)|=3$ since $Q$ is extraspecial. Hence $Z=\mathcal{Z}(Q)$ has order three.

$(ii)$ We have that $C_G(Q) \leq Q$ and so $C_G(Z)/Q$ is isomorphic to a subgroup of the outer
automorphism  group of $Q$. By Theorem \ref{extraspecial outer automorphisms}, $\out(Q)
\sim \Sp_4(3).2$. Notice that this group has a unique subgroup of index
two which is necessarily isomorphic to $\Sp_4(3)$. Moreover, this index two subgroup is the
subgroup of outer automorphisms which centralize $Z$. Thus $C_G(Z)/Q$ embeds into $\Sp_4(3)$.

$(iii)$ Suppose that $p$ is a prime and  $g \in C_G(Z)$ is a $p$-element such that $[Q/Z,g]=1$.
Then if $p \neq 3$ we may apply coprime action to say that $Q/Z=C_{Q/Z}(g)=C_Q(g)Z/Z= C_Q(g)/Z$
and so $[Q,g]=1$ which is a contradiction as $C_G(Q) \leq Q$. Therefore $C_{C_G(Z)/Z}(Q/Z)$ is a
$3$-group and the preimage in  $C_G(Z)$ is a normal $3$-subgroup of $C_G(Z)$ and so must be
contained in $O_3(C_G(Z))=Q$. Therefore $Q \leq C_{C_G(Z)}(Q/Z)\leq Q$.

$(iv)$ Since $[Q,Q]=Z \neq Z^x =[Q^x,Q^x]$, we immediately see that $[Q \cap Q^x,Q \cap Q^x] \leq Z
\cap Z^x =1$.  Therefore $Q \cap Q^x$ is  abelian and since  $Q$ has exponent three, $Q \cap Q^x$
is elementary abelian.
\end{proof}

\begin{lemma}
$Z \leq Q^x$.
\end{lemma}
\begin{proof}
Suppose $Z \nleq Q^x$. Notice that $C_Q(Y)$ normalizes $Q$ and $Q^x$ and therefore $Q \cap Q^x
\trianglelefteq C_Q(Y)$. This implies that $Q \cap Q^x=Z^x$ for if we had $Q \cap Q^x>Z^x$ then
$Z=C_Q(Y)'\leq Q \cap Q^x$. Therefore $Q^xC_Q(Y)/Q^x\cong C_Q(Y)/Z^x$ which must be non-abelian of
exponent three and order $3^3$ and so $Q^xC_Q(Y)/Q^x\cong C_Q(Y)/Z^x\cong 3_+^{1+2}$. Since
$C_G(Z^x)/Q^x$ is isomorphic to a subgroup of $\Sp_4(3)$ with no non-trivial normal $3$-subgroup,
and a Sylow $3$-subgroup of order at least $3^3$, it follows from the maximal subgroups of
$\Sp_4(3)$ (see \cite{atlas} for example) that $C_G(Z^x)/Q^x\cong \Sp_4(3)$. Therefore
$C_G(Z)/Q\cong \Sp_4(3)$ is transitive on $(Q/Z)^\#$. Let $P \in \syl_3(C_G(Z))$ then $1 \neq
\mathcal{Z}(P/Z) \cap Q/Z$ and since $C_G(Z)/Q$ is transitive on $(Q/Z)^\#$, we may assume $Y/Z\leq \mathcal{Z}(P/Z)$.
Therefore $|C_P(Y)|=3^8$ and $[P:C_P(Y)]=3$. Therefore there exists a Sylow $3$-subgroup of
$C_G(Z^x)$, $R$ say  with $[R:C_P(Y)]=3$. In particular,  $|C_{Q^x}(Y)|=|C_P(Y) \cap Q^x|=3^4$.

Now we have that $[C_{Q^x}(Y),Q^x,C_Q(Y)]=[Z^x,C_Q(Y)]=1$ and $[C_{Q}(Y),C_{Q^x}(Y),Q^x]\leq [Q
\cap Q^x,Q^x]=[Z^x,Q^x]=1$ and so by the three subgroup lemma, $[Q^x,C_Q(Y),C_{Q^x}(Y)]=1$. Therefore
$[Q^x,C_Q(Y)]$ is a subgroup of $Q^x$ commuting with $C_{Q^x}(Y)$. Since $|C_{Q^x}(Y)|=3^4$ and
$Q^x$ is extraspecial, we have that $[Q^x,C_Q(Y)]\leq \mathcal{Z}(C_{Q^x}(Y))$ and $\mathcal{Z}(C_{Q^x}(Y))$ has order
nine. However this implies that $[Q^x/Z^x,C_Q(Y),C_Q(Y)]\leq [\mathcal{Z}(C_{Q^x}(Y))/Z^x,C_Q(Y)]=1$ and so
$C_Q(Y)$ acts quadratically on $Q^x/Z^x$. Now by Lemma \ref{quadractic action lemma},
$C_Q(Y)/C_{C_Q(Y)}(Q^x/Z^x)$ is elementary abelian. However, by Lemma \ref{Prelims-EasyLemma},
$C_{C_Q(Y)}(Q^x/Z^x)=C_Q(Y) \cap Q^x=Z^x$ and so $C_Q(Y)/Z^x$ is elementary abelian. However we
have already seen that $C_Q(Y)/Z^x \cong 3_+^{1+2}$ which is a contradiction.
\end{proof}

Notice in particular that this implies that $L \leq N_G(Y)$.

\begin{lemma}\label{Lemma-facts about L}
\begin{enumerate}[$(i)$]
\item $W$ is a $3$-group and  $S=QW$ is a $3$-group with $\mathcal{Z}(S)=Z$.
\item $L/C_L(Y) \cong \SL_2(3)$.
\item $W \trianglelefteq L$.
\item $C_L(Y)/W \leq \mathcal{Z}(L/W)$.
\item $W=O_3(L)$, $L/W$ has four Sylow $3$-subgroups and $C_L(Y)/W$
is a $2$-group.
\item $S=QW \in \syl_3(L)$.
\end{enumerate}
\end{lemma}
\begin{proof}
$(i)$ Since $Q$ is extraspecial and $Y \leq Q$ with $|Y|=9$, $|C_Q(Y)|=3^4$. Similarly $Y \leq Q^x$ and so
$|C_{Q^x}(Y)|=3^4$.  Notice that  $C_Q(Y)=Q \cap C_G(Z^x)$ is normalized by $C_G(Z) \cap C_G(Z^x)$
so $C_{Q^x}(Y)$ normalizes $C_Q(Y)$. Thus $W$ and hence $S=QW$ are $3$-groups. Moreover, $\mathcal{Z}(S) \leq
C_G(Q) \leq Q$ and so $Z \leq \mathcal{Z}(S) \leq \mathcal{Z}(Q)=Z$.

$(ii)$ Clearly $L/C_L(Y)$ embeds into $\GL_2(3)$. Moreover since $Q \nleq C_L(Y)$, $L/C_L(Y)$ is
generated by two of its Sylow $3$-subgroups $QC_L(Y)/C_L(Y)$ and $Q^xC_L(Y)/C_L(Y)$. These are
distinct since they centralize distinct subgroups of $Y$. Thus $L/C_L(Y)\cong \SL_2(3)$.

$(iii)$ Since $W \leq C_G(Y)$, we see that $[Q,W] \leq [Q, C_G(Y)] \leq Q \cap C_G(Y)=C_Q(Y) \leq
W$.  Therefore $Q$ normalizes $W$ and similarly, $Q^x$ normalizes $W$. Therefore $W\trianglelefteq
L$.

$(iv)$ We see that $[C_L(Y),Q]\leq C_L(Y) \cap  Q =C_Q(Y) \leq W$ and similarly $[C_L(Y),Q^x] \leq
W$. So $[C_L(Y),L]\leq W$ and therefore $C_L(Y)/W\leq \mathcal{Z}(L/W)$.

$(v)$ Since $L/C_L(Y)\cong \SL_2(3)$, $L_1:=O^3(L)C_L(Y)$ is a subgroup of $L$ of index three and
$L_1/C_L(Y)\cong Q_8$. Since $C_L(Y)/W\leq \mathcal{Z}(L/W)$, it follows that $L_1/W$ is nilpotent and is
therefore a direct product of its Sylow subgroups. Let $W \leq P$ and $W \leq R$ be subgroups of
$L_1$ such that $P/W$ is a Sylow $2$-subgroup of $L_1/W$ and $R/W$ is a Sylow $3$-subgroup. Clearly
$R$ is a normal $3$-subgroup of $L$ and so $QR/W \in \syl_3(L/W)$ and $L$ has four Sylow
$3$-subgroups and any two of them generate $L$. Now $P/W$ is normalized but not centralized by $S/W=QW/W$ and so $\<P/W,S/W\>$ has
more than one Sylow $3$-subgroup. Therefore $\<P/W,S/W\>=L/W$ has a normal Sylow $2$-subgroup $P/W$ and moreover, $L_1=P$, $C_L(Y)/W$ is a $2$-group
and $W=R=O_3(L)$.

$(vi)$ It now follows immediately that $W \in \syl_3(C_L(Y))$ and so $QW \in \syl_3(L)$.
\end{proof}

\begin{lemma}\label{W is centralizer of Y in L}
$W=C_L(Y)$, in particular, $L/W\cong \SL_2(3)$.
\end{lemma}
\begin{proof}
Let $\bar{L}=L/W$ and $W\leq P \leq L$ such that $\bar{P}\in \syl_2(\bar{L})$. Then $\bar{Q}$
normalizes  $\bar{P}$ and $\bar{L}=\bar{P}\bar{Q}$ so $[N_{\bar{P}}(\bar{Q}),\bar{Q}] \leq \bar{P}
\cap \bar{Q}=1$. Therefore $N_{\bar{P}}(\bar{Q})=C_{\bar{P}}(\bar{Q})$. Since $\bar{L}$ has four
Sylow $3$-subgroups, $[\bar{P}: C_{\bar{P}}(\bar{Q})]=4$. In particular, $[\bar{P}, \bar{Q}]\neq
1$. By Lemma \ref{Lemma-facts about L} $(iv)$, $\bar{C_L(Y)} \leq \mathcal{Z}(\bar{L})$ and since
$\bar{C_L(Y)}$ has index two in $C_{\bar{P}}(\bar{Q})$, it follows that $C_{\bar{P}}(\bar{Q})$ is
abelian. However since $\bar{P}/\bar{C_L(Y)}\cong \Q_8$, $\bar{P}$ is non-abelian.

Now, $\bar{L}$ is generated by any two of its Sylow $3$-subgroups. So suppose that $\bar{P_0} <
\bar{P}$ such that $\bar{P_0}$ is normalized by $\bar{Q}$ then it must be centralized by $\bar{Q}$
else $\bar{L}=\<\bar{P_0},\bar{Q}\> < \<\bar{P},\bar{Q}\>=\bar{L}$. Therefore we have that any
proper $\bar{Q}$-invariant subgroup of $\bar{P}$ is contained in $C_{\bar{P}}(\bar{Q})$. So suppose
$[\bar{P}, \bar{Q}]<\bar{P}$. Then $1 \neq [\bar{P}, \bar{Q}]<C_{\bar{P}}(\bar{Q})$ which, using
coprime action, implies $[\bar{P}, \bar{Q}]=[\bar{P}, \bar{Q},\bar{Q}]=1$. This contradiction
proves that $[\bar{P}, \bar{Q}]=\bar{P}$.

Now, $\Phi(\bar{P})$ is also a proper $\bar{Q}$-invariant subgroup of $\bar{P}$ and so we have that
$\Phi(\bar{P}) \leq C_{\bar{P}}(\bar{Q})$. Now by coprime action,
\[\frac{\bar{P}}{\Phi(\bar{P})}=C_{\bar{P}/\Phi(\bar{P})}(\bar{Q}) \times
[\frac{\bar{P}}{\Phi(\bar{P})},\bar{Q}].\] However
\[[\frac{\bar{P}}{\Phi(\bar{P})},\bar{Q}]= \frac{[{\bar{P}},\bar{Q}]{\Phi(\bar{P})}}{{\Phi(\bar{P})}}=\frac{\bar{P}}{{\Phi(\bar{P})}}.\]
Therefore $1=C_{\bar{P}/\Phi(\bar{P})}(\bar{Q})\cong C_{\bar{P}}(\bar{Q})/\Phi(\bar{P})$. Thus
$C_{\bar{P}}(\bar{Q})=\Phi(\bar{P})$. The same argument with $\bar{P}'$ in place of $\Phi(\bar{P})$
gives $C_{\bar{P}}(\bar{Q})=\bar{P}'$. Moreover, since $\mathcal{Z}(\bar{P})\neq \bar{P}$ is normalized by $\bar{Q}$, we also have that  $\mathcal{Z}(\bar{P})\leq C_{\bar{P}}(\bar{Q})=\bar{P}'$ and so $[\bar{P},C_{\bar{P}}(\bar{Q}),\bar{Q}]\leq [\bar{P}',\bar{Q}]=1$. Moreover $[C_{\bar{P}}(\bar{Q}),\bar{Q},\bar{P}]=1$  and so by the three subgroup lemma, we have $[\bar{Q},\bar{P},C_{\bar{P}}(\bar{Q})]=1$ and so $[\bar{P},C_{\bar{P}}(\bar{Q})]=1$ which implies that $C_{\bar{P}}(\bar{Q})= \mathcal{Z}(\bar{P})$.

Now, since $\bar{P}/C_{\bar{P}}(\bar{Q})$ is elementary abelian of order four, we may choose, $a,b \in
\bar{P}\bs C_{\bar{P}}(\bar{Q})$ such that $\bar{P}=\<a,b\>$. Notice that $a^2 \in C_{\bar{P}}(\bar{Q})$. Since
$C_{\bar{P}}(\bar{Q})$ is central in $\bar{P}$, it follows that, $C_{\bar{P}}(\bar{Q})=\bar{P}'=\<[a,b]\>$.
Furthermore, $[a,b]^2=[a^2,b]=1$. Therefore $|C_{\bar{P}}(\bar{Q})|=2$ and so $|\bar{P}|=8$. Thus
$|\bar{L}|=24$ and so $W=C_L(Y)$.
\end{proof}

\begin{lemma}\label{facts about W}
\begin{enumerate}[$(i)$]
\item $Y < Q \cap Q^x$.

\item $Y$ and $W/Q \cap Q^x$ are natural $L/W$-modules and $Q\cap
Q^x /Y$ is the trivial $L/W$-module. In particular $|W|=3^5$.

\item If $Z \neq Z^{x'}\leq Y$ then $Q \cap Q^x=Q \cap Q^{x'}$ and
$\<Q,Q^{x'}\>=\<Q,Q^x\>=L$.

\item $\mathcal{Z}(W)=Y$.

\item $W$ has exponent three.
\end{enumerate}
\end{lemma}
\begin{proof}
$(i)$ Suppose that $Y=Q \cap Q^x=C_Q(Y) \cap C_{Q^x}(Y)$. Then $|W|=3^43^4/3^2=3^6$. Since $W$
centralizes $Y$, $C_Q(Y),C_{Q^x}(Y)\trianglelefteq W$. Therefore $W'\leq C_Q(Y) \cap C_{Q^x}(Y)=Q
\cap Q^x=Y$ and so $W/Y$ is abelian of order $3^4$. Furthermore $W/Y$ is generated by the
elementary abelian groups $C_{Q}(Y)/Y$ and $C_{Q^x}(Y)/Y$ and so is elementary abelian. Since
$W'\geq Z$ and $W'\geq Z^x$, we see that $Y=W'=\Phi(W)$. Choose an involution $s\in L$ such that
$Ws \in \mathcal{Z}(L/W)$. Since $W$ is non-abelian, $C_W(s) \neq 1$ and since $Y$ is a natural $L/W$-module,
$C_Y(s)=1$ and so $C_{W/Y}(s)\neq 1$. By Lemma \ref{Burnside-p'-automorphism}, $[W/Y,s] \neq 1$
else $s$ acts trivially on $W$. So $W/Y=C_{W/Y}(s) \times [W/Y,s]$ and since $[W/Y,s]$ is a
non-trivial $L/W$-module which $\mathcal{Z}(L/W)$ inverts, it has order $3^2$. Therefore $|C_{W/Y}(s)|=3^2$.
Now $Ws \in \mathcal{Z}(L/W)$ so $Ws$ normalizes $S/W=QW/W\cong Q/(Q \cap W)$. Furthermore $s$ inverts $Z$
and so normalizes $Q$ and therefore also $Q \cap W$. Since $C_Q(Y)=Q \cap W$ is non-abelian $C_{(Q
\cap W)}(s) \neq 1$ and so $C_{(Q \cap W)/Y}(s) \neq 1$. Similarly $C_{(Q^x \cap W)/Y}(s) \neq 1$
and clearly $C_{(Q \cap W)/Y}(s) \neq C_{(Q^x \cap W)/Y}(s)$. Thus $C_{W/Y}(s)=C_{(Q \cap W)/Y}(s)
C_{(Q^x \cap W)/Y}(s)$.  We may assume that $x \in L$ as $L$ is transitive on $Y$. Therefore $L/W$
permutes $C_{(Q \cap W)/Y}(s)$ and  $C_{(Q^x \cap W)/Y}(s)$ and so acts non-trivially on
$C_{W/Y}(s)$. Hence by Lemma \ref{prelims-natural sl23 mods}, $C_{W/Y}(s)$ is a natural
$L/W$-module. This is a contradiction as $s$ clearly centralizes $C_{W/Y}(s)$. Thus we may conclude
that $Q \cap Q^x \neq Y$.

$(ii)$ Since $Q \cap Q^x$ is abelian and a subgroup
of an extraspecial group, it has order at most $3^3$ and, by $(i)$, it has order exactly $3^3$.
Therefore $|W|=3^43^4/3^3=3^5$. Now every subgroup of $Q$ containing $Z$ is normalized by $Q$ and every
subgroup of $Q^x$ containing $Z^x$ is normalized by $Q^x$ and so $Q \cap Q^x$ is normalized by
$L=\<Q,Q^x\>$. Therefore $(Q \cap Q^x)/Y$ is normalized by $L/W$ and so must be a trivial module. Now
$L/W$ also acts on the elementary abelian group $W/(Q \cap Q^x)$ which is the direct product of the
groups $C_Q(Y)/(Q \cap Q^x)$ and $C_{Q^x}(Y)/(Q \cap Q^x)$. Since we can assume as before that $x \in
L$, $L/W$ acts non-trivially on $W/(Q \cap Q^x)$ and so this must be a natural $L/W$-module.

$(iii)$ Since $Y$ is a natural $L/W$-module, $L$ is transitive on $Y^\#$.
Moreover there exists an element of $L$ which preserves $Z$ and swaps $Z^x$ and $Z^{x'}$. This
element of course maps $Q \cap Q^x$ to $Q \cap Q^{x'}$ but $Q \cap Q^x$ is normal in $L$ and so the
two groups must be equal.

$(iv)$ Clearly $Y \leq \mathcal{Z}(W)$ so suppose $Y<\mathcal{Z}(W)$ then $\mathcal{Z}(W)$ has index at most nine in $W$. Since
$C_Q(Y)$ is non-abelian and contained in $W$, $W$ is non-abelian. Therefore $[W:\mathcal{Z}(W)]=9$. Notice
that  $\mathcal{Z}(W)\neq Q \cap Q^x$ otherwise $C_Q(Y)$ is abelian. So we have that $(Q\cap Q^x )\mathcal{Z}(W)/(Q
\cap Q^x)$ is a proper and non-trivial $L/W$ invariant subgroup of the natural $L/W$-module, $W/(Q
\cap Q^x)$. This is a contradiction. Thus $Y=\mathcal{Z}(W)$.

$(v)$ Since $Q$ has exponent three, $Q \cap Q^x$ does also. Choose $a \in Q  \bs (Q \cap Q^x)$ then
$(Q \cap Q^x) a$ is a non-identity element of the natural $L/W$-module, $W/Q \cap Q^x$. Moreover,
every element in the coset has order dividing three since $Q$ has exponent three. Since $L/W$ is
transitive on the non-identity elements of the natural module $W/(Q \cap Q^x)$, every element of
$W$ has order dividing three.
\end{proof}

Let $Z \leq Z_2\leq S$ such that $Z_2/Z=\mathcal{Z}(S/Z)$. Since $L/W\cong \SL_2(3)$ and $W$ is a $3$-group,
we may fix an involution $s$ in $L$ such that $W\<s\>/W=\mathcal{Z}(L/W)$. Set $J=[W,s]$.

\begin{lemma}\label{lemma structure of Y and W}
\begin{enumerate}[$(i)$]
 \item $Y \leq Z_2\leq W$ and $Z_2$ is abelian of order $3^3$ but distinct from $Q \cap Q^x$.
 \item $W'=Y$.
 \item $Q \cap Q^x=Y C_W(s)$ and $|C_W(s)|=3$.
 \item $Q \cap
J=Z_2\leq J=[W,s]=[S,s]$ is an elementary abelian subgroup of $W$ of order $3^4$ that is inverted by $s$.
 \item $J=J(S)=J(W)$ and $Y \leq S' \leq Z_2$.
 \item If $J<S_0<S$ then $Z_2> \mathcal{Z}(S_0)$  and $|\mathcal{Z}(S_0)|=9$.
 \item $L/J \cong 3\times \SL_2(3)$ and $J/Y$ is a natural $L/W$-module.
\end{enumerate}
\end{lemma}
\begin{proof}
$(i)$ By Lemma \ref{Prelims-EasyLemma} $(iii)$, $C_G(Z)/Q$ acts faithfully on $Q/Z$. Since $S/Q$ is isomorphic to a cyclic subgroup of $\GL_4(3)$ of order three, we may consider the Jordan blocks of elements of order three to see that the action of any such cyclic subgroup on $Q/Z$ is not indecomposable. Thus, there exist $S/Q$-invariant, proper, non-trivial subgroups, $N_1,N_2$ of $Q/Z$ such that $Q/Z=N_1N_2$. Hence for $i\in \{1,2\}$, $1 \neq C_{N_i}(S/Q)\leq \mathcal{Z}(S/Z)$.  Therefore $|Z_2/Z|\geq 9$ and so $|Z_2|\geq 27$. 
Since $[Q/Z,Z_2]=1$, $Z_2 \leq Q$ by Lemma \ref{Prelims-EasyLemma} $(iii)$. Now suppose $Z_2 \nleq W$. Then $S=Z_2W \in \syl_3(L)$.
Since $Z_2/Z=\mathcal{Z}(S/Z)$, $[S,Z_2] \leq Z$ and so $[W,Z_2] \leq Z \leq Q \cap Q^x$. Therefore $[W/Q
\cap Q^x, Z_2]=1$, however this implies that $S/W=Z_2W/W$ acts trivially on the natural
$L/W$-module $W/(Q \cap Q^x)$ which is a contradiction. So $Z_2 \leq W\cap Q$. Suppose  $Q \cap Q^x
\leq Z_2$. Then $Z^x=[C_{Q^x}(Y), Q \cap Q^x]\leq Z$ which is a contradiction. Therefore $|Z_2|=27$
and in particular, $Z_2 \neq Q \cap Q^x$. Furthermore $Y\vartriangleleft QW=S$ and so $Y/Z$ is
central in $S/Z$. Therefore $Y \leq Z_2$ and since $Y$ is central in $Z_2 \leq W$, $Z_2$ is
abelian.

$(ii)$ Now $Z=C_Q(Y)'\leq W'$ and $Z^x=C_{Q^x}(Y)'\leq W'$ and so $Y \leq W'$. Moreover, we have
just observed that $Z_2 \neq Q \cap Q^x$ and so  $Q \cap Q^x$ and $Z_2$ are distinct  normal
subgroups of $W$ both of index nine. Thus $Y \leq W' \leq Q \cap Q^x \cap Z_2$. It follows from the
group orders that $Y=W'=Q \cap Q^x \cap Z_2$.

$(iii)$ By coprime action on an abelian group, $W/Y=C_{W/Y}(s) \times [W/Y,s]$. By Lemma \ref{facts
about W}  $(ii)$, $Y$ and $W/(Q \cap Q^x)$ are natural $L/W$-modules. Therefore $s$ inverts $Y$ and
$W/(Q \cap Q^x)$. It follows from coprime action that $|C_W(s)|=3$ and that $C_W(s)\leq Q \cap Q^x$
with $Q \cap Q^x=YC_W(s)$.

$(iv)$ We have that $Y$ is inverted by $s$ and so $Y=[Y,s] \leq [W,s]=J$. Therefore
$[W/Y,s]=Y[W,s]/Y \cong [W,s]/([W,s] \cap Y)=[W,s]/Y$ has order nine. This implies that $[W,s]$ has
order $3^4$.  Now $s$ normalizes $W$ so we use coprime action again to see that $W=C_W(s) [W,s]$.
Notice  that this product is split since $J=[W,s]$ has order $3^4$ and $C_W(s)$ has order three. In
particular this implies that $s$ acts fixed-point-freely on $J$ and so $J$  is abelian and inverted
by $s$. By Lemma \ref{facts about W} $(v)$, $W$ has exponent three and so $J$ is elementary
abelian.

Observe that the involution $s$ normalizes $S$. Now $S\leq L=\<Q,Q^x\>$ and so $S$ normalizes
$W\<s\>$ and therefore $[S,s] \leq S \cap W\<s\>=W$. So by coprime action, $[S,s]=[S,s,s]\leq
[W,s]=J$. Since $s$ normalizes $Z$ and  $S$, we must have that $s$ normalizes $Z_2$. Moreover, $Z_2
\leq W$ and so if $C_{Z_2}(s)$ is non-trivial then $C_{Z_2}(s)=C_W(s)\leq Q \cap Q^x$ and then
$Z_2=YC_W(s)=Q \cap Q^x$. However by $(i)$, $Z_2 \neq Q \cap Q^x$. Thus $s$ acts fixed-point-freely
on $Z_2$ and so $Z_2\leq J$. Since $J \nleq Q$, we have $Z_2=Q \cap J$.

$(v)$ Suppose there was another abelian subgroup of $W$ of order $3^4$, $J_0$ say. Then $|J \cap
J_0|=3^3$ and  $J \cap J_0$ would be central in $W$. This contradicts Lemma \ref{facts about W}
which says that $\mathcal{Z}(W)=Y$. It follows therefore that $J(W)=J$.

Clearly $3^4$ is the largest possible order of an abelian subgroup of $S$ (else $Q$ would contain
abelian subgroups of order $3^4$). So suppose $J_1$ is an abelian subgroup of $S$ distinct from
$J$. Then $J_1 \nleq W$ and $J_1 \nleq Q$. Therefore, $S/Z$ contains three distinct abelian
subgroups  $Q/Z$, $J/Z$ and $J_1/Z$. We must have that $S=QJ=QJ_1$. Hence,  $(Q/Z) \cap (J/Z)$ and
$(Q/Z) \cap (J_1/Z)$ both have order nine and are both central in $S/Z$. We must have that  $Q/Z
\cap J/Z =Q/Z \cap J_1/Z=Z_2/Z$. Thus $Y \leq Z_2 \leq J_1$ and so $J_1\leq C_S(Y)=W$. However we
have seen that $J=J(W)$ is the unique abelian subgroup of order $3^4$. Thus $J=J(S)$. In
particular, $J$ is a normal subgroup of $S$ of index nine so $Y=W' \leq S' \leq Q \cap J=Z_2$.

$(vi)$ Suppose  $J<S_0<S$ then $|S_0|=3^5$. Since $J\nleq Q$, $S_0 \nleq Q$ and so $|S_0 \cap
Q|=3^4$. Therefore $\mathcal{Z}(Q \cap S_0)$ has order nine. Since $J\leq S_0$, $Z_2=J \cap Q\leq Q \cap
S_0$. Hence,  $\mathcal{Z}(Q \cap S_0)\leq Z_2$ otherwise $Q \cap S_0=\<\mathcal{Z}(Q \cap S_0),Z_2\>$ would be
abelian. Thus $\mathcal{Z}(Q \cap S_0)\leq J$ and so $\mathcal{Z}(Q \cap S_0)$ commutes with $S_0=\<Q \cap S_0,J\>$ and
$\mathcal{Z}(Q \cap S_0)=\mathcal{Z}(S_0) \cap Q$ has order nine. So suppose $\mathcal{Z}(S_0)$ has order greater than nine. Then
there exists $g\in \mathcal{Z}(S_0) \bs Q$ such that $S_0 \cap Q\leq C_Q(g)$. Therefore $S=\<Q,g\>$ and
$[S_0 \cap Q, S]=[S_0 \cap Q,Q][S_0 \cap Q,g]= Z$  which implies that $S_0 \cap Q \leq Z_2$ which
is a contradiction.

$(vii)$ We have that $L/W\cong \SL_2(3)$ and  $L/J$ has a normal subgroup of order three $W/J$.
Since $\SL_2(3)$ has no index two subgroup, $W/J$ is central in $L/J$ (else $C_{L/J}(W/J)$ would
have index two). Now $J$ is not a subgroup of $Q$ and so $S/J=QJ/J\cong Q/(Q \cap J)$ and since $Q$
has exponent three, so does $S/J$. Therefore $S/J$ splits over $W/J$ and so by Gasch\"{u}tz's
Theorem (\ref{Gaschutz}), $L/J$ splits over $W/J$ and so $L/J \cong 3 \times \SL_2(3)$. Finally it
is clear that $L/W$ acts on $J/Y=J/(J \cap Q \cap Q^x)\cong J(Q \cap Q^x)/(Q \cap Q^x)=W/(Q \cap
Q^x)$ which is a natural $L/W$-module.
\end{proof}

\begin{lemma}\label{central involutions don't fix Y}
Suppose $u \in C_G(Z)$ is an involution and  $[Qu,S/Q]=1$. Then either
 \begin{enumerate}[$(i)$]
  \item $u$ inverts $Q/Z$, $|C_J(u)|=3^2$ and $u$ inverts $S/J$; or
  \item $Q$ is a central product of the two groups $C_Q(u),[Q,u] \cong 3_+^{1+2}$, $u$ does
not normalize $Y$, $|C_{S}(u)|=3^4$, $|C_J(u)|=3^3$ and $[J,u]=[Y,u]$ has order 3.
\end{enumerate}
\end{lemma}
\begin{proof}
If $C_{Q/Z}(u)=1$ then $C_Q(u)=Z$ and $u$ inverts $Q/Z$. We have that $u$
centralizes $S/Q=QJ/Q\cong J/(J \cap Q)=J/Z_2$. Since $Z \leq C_{Z_2}(u)\leq C_Q(u)=Z$, we see that
$C_{J/Z_2}(u)=C_J(u)Z_2/Z_2\cong C_J(u)/Z$ and so $|C_J(u)|=3^2$. We also see that $u$ inverts
$Q/Z_2=Q/(Q \cap J)\cong QJ/J=S/J$.

So suppose that $C_{Q/Z}(u)\neq 1$. We have,
$Q/Z=C_{Q/Z}(u)\times [Q/Z,u]=(C_Q(u)/Z) ([Q,u]Z/Z)$. If $Q/Z=C_{Q/Z}(u)$ then $[Q,u]=1$ by coprime
action which is a contradiction. So $C_{Q/Z}(u)$ and $[Q/Z,u]$ are both proper non-trivial
subgroups of $Q/Z$. Notice that $[Q,C_Q(u),u] \leq [Z,u]=1$ and $[C_Q(u),u,Q]=[1,Q]=1$. By the
Three Subgroup Lemma, $[Q,u]$ commutes with $C_Q(u)$. We therefore see that both $C_Q(u)$ and
$[Q,u]$ are non-abelian else they would be central in $Q$. It follows immediately that $C_Q(u)\cong
[Q,u]$ are extraspecial and must have exponent three  as $Q\cong 3_+^{1+4}$ does. Therefore
$C_Q(u)\cong [Q,u]\cong 3_+^{1+2}$.

We have that $Qu$ commutes with $S/Q$ so  $3\sim
C_{S/Q}(u)=C_S(u) Q / Q \cong C_S(u)/ C_Q(u)$ and so $C_S(u)$ has order $3^4$.

Suppose that $u$ centralizes $Y$ then $u$ normalizes $\<Q, Q^x\>=L$ and $W\<u\>=C_{L\<u\>}(Y)$.
Hence $[L,u]\leq L \cap W\<u\>=W$ and so $u$ commutes with $L/W$ and in particular with $S/W=QW/W
\cong Q /(Q \cap W)$. Also $u$ centralizes $QW/Q\cong W/ (Q \cap W)$, therefore \[3\sim C_{W /(Q
\cap W)}(u)=\frac{C_W(u) (Q \cap W)}{Q \cap W} \cong \frac{C_W(u)}{C_{Q \cap W}(u)}\] and so
$|C_W(u)/C_{Q \cap Q^x}(u)|\geq 3$ and  $C_W(u)/C_{Q \cap Q^x}(u)\cong C_W(u)(Q \cap Q^x)/(Q \cap
Q^x)=C_{W /Q \cap Q^x}(u)$. Since $W/Q \cap Q^x$ is a natural $L/W$-module and $u$ commutes with
$L/W$, we have that $C_{W /Q \cap Q^x}(u)=W /Q \cap Q^x$. In particular, $C_W(u)$ has order $3^4$
and therefore $C_W(u)=C_S(u)\geq C_Q(u)\geq Y$. However this implies $Y \leq C_Q(u)\leq W\leq
C_G(Y)$ which is a contradiction as $C_Q(u)$ is extraspecial.

So suppose that $u$ induces a non-trivial automorphism on $Y$. We can assume without loss of
generality that $u$ inverts $Z^x$ and so normalizes $L$ and $Q \cap Q^x$. We have that $u$ does not
invert $Y$ as $u$ commutes with $Z$ and so $\<u\>L/W \cong \GL_2(3)$. This forces
$1=C_{S/W}(u)=C_S(u)W/W \cong C_S(u)/C_W(u)$ and so $C_W(u)=C_S(u)$ has order $3^4$. Since $W/Q
\cap Q^x$ is a natural $L/W$-module, we must have $3\sim C_{W/Q \cap Q^x}(u)\cong C_W(u)/C_{Q \cap
Q^x}(u)$. Since $Q \cap Q^x$ is abelian, we have $Q \cap Q^x=C_{Q \cap Q^x}(u) \times [Q \cap
Q^x,u]$ and since $u$ inverts $Z^x$, $C_{Q \cap Q^x}(u)$ has order at most $3^2$. However this
together with $3\sim C_W(u)/C_{Q \cap Q^x}(u)$ implies $C_W(u)$ has order at most $3^3$ which is a
contradiction. Thus we may conclude that $u$ does not normalize $Y$.

We have that $S' \leq Z_2$ so $S/Z_2$ is abelian and normalized by $u$ so again by coprime action,
$S/Z_2=C_{S/Z_2}(u) \times [S/Z_2,u]$ and $C_{S/Z_2}(u)=C_S(u)Z_2/Z_2 \cong C_S(u)/C_{Z_2}(u)$. Now
$Z_2$ is abelian so $Z_2 \neq C_Q(u)$. Therefore $Z_2=C_{Z_2}(u) \times [Z_2,u]$ where $Z \nleq
[Z_2,u]\neq 1$. If $|[Z_2,u]|=9$ then $[Z_2,u]\cap [Q,u]>Z$ follows as $[Q,u]$ is extraspecial
which is a contradiction. Therefore $|[Z_2,u]|=|Z_2 \cap [Q,u]|=3$ and $|C_{Z_2}(u)|=9$. Hence
$|C_{S/Z_2}(u)|=9$ and so $|[S/Z_2,u]|=3$. Now $[S/Z_2,u]=[S,u]Z_2/Z_2 \cong [S,u]/(Z_2\cap
[S,u])$. Suppose $Z_2 \leq [S,u]$ then $Q \geq Z_2[Q,u]=[S,u]$ has order $3^4$. By coprime action,
$[S,u]=[S,u,u] \leq [Q,u]$ has order $3^3$ which is a contradiction. Thus $Z_2\nleq [S,u]$ and so
$Z_2 \cap [S,u]$ has order at most $3^2$ and so $[S,u]$ has order at most $3^3$ and so
$[S,u]=[Q,u]$. Hence $[J,u] \leq [S,u] \cap J =[Q,u]\cap J$ has order at most nine as $J$ is
abelian. Since $J=J(S)$ is abelian and normalized by $u$, $J=C_J(u) \times [J,u]$ by coprime action
and $[J,u]\neq 1$ else $u$ centralizes $Y$ which is not the case. Furthermore $Z \leq C_J(u)$ and
$Z\leq [Q,u] \cap J$. Therefore $[J,u]<[Q,u]\cap J$ and so $|[J,u]|=3$ and $|C_J(u)|=27$. Finally,
since $u$ does not normalize $Y$, $[Y,u]=[J,u]$.
\end{proof}

\begin{lemma}\label{3' group normalized by Y}
Let $N$ be a $3'$-subgroup of $G$ which is normalized by $Y$. Then $[N,Y]=1$.
\end{lemma}
\begin{proof}
By coprime action, $N=\<C_N(y)| y \in Y^\#\>$. However for each $y \in Y^\#$, $C_N(y)$ is a
$3'$-group commuting with $y$ which is normalized by $Y \leq O_3(C_G(y))$ and so $[C_N(y),Y]\leq
C_N(y) \cap O_3(C_G(y))=1$ so $C_N(y) \leq C_G(Y)$ for each $y \in Y^\#$ and therefore $[N,Y]=1$.
\end{proof}

\section{Concluding Remarks on The Hypothesis}\label{GenHyp-Concluding Remarks}
We summarize the results of the previous section in the following theorem.

\begin{thm}\label{A general 3^1+4 theorem}
Let $G$ satisfy the Hypothesis \ref{3^1+4 Hypothesis}. Then the following hold.
\begin{enumerate}[$(i)$]
\item $Z \leq Q^x$ and so $Y:=ZZ^x \leq Q \cap Q^x$.
\item $L:=\<Q,Q^x\>\leq N_G(Y)$;
\item $W:=C_L(Y)=O_3(L)$ has order $3^5$ with $L/W\cong \SL_2(3)$.
\item $\mathcal{Z}(W)=W'=Y$ and $W$ has exponent three.
\item $S:=QW\in \syl_3(L)$  and there is an involution, $s$ such that $Ws \in \mathcal{Z}(L/W)$ and then
$J:=J(S)=J(W)=[S,s]=[W,s]=[J,s]$ is elementary abelian of order $3^4$ and inverted by $s$.
\item $Z_2:=J \cap Q$ has order $27$ and $Z_2/Z=\mathcal{Z}(S/Z)$.
\item $Y\leq S' \leq Z_2$.
\item If $J<S_0<S$ then $|\mathcal{Z}(S_0)|=9$.
\item $Y$, $W/(Q \cap Q^x)$ and $J/Y$ are natural $L/W$-modules.
\item $L/J \cong 3 \times \SL_2(3)$.
\item $Q \cap Q^x=Y\times C_W(s)$ has order $3^3$.
\item If $N \leq G$ is a $3'$-subgroup of $G$ which is normalized by $Y$ then $[Y,N]=1$.
\item If $u \in C_G(Z)$ is an involution and $[Qu,S/Q]=1$ then either  $u$ inverts $Q/Z$,
$|C_J(u)|=3^2$ and $u$ inverts $S/J$; or $Q$ is a central product of the two groups $C_Q(u),[Q,u]
\cong 3_+^{1+2}$, $u$ does
not normalize $Y$, $|C_{S}(u)|=3^4$, $|C_J(u)|=3^3$ and $[J,u]=[Y,u]$ has order 3.
\end{enumerate}
\end{thm}

Three simple groups satisfy Hypothesis \ref{3^1+4 Hypothesis} as well as several more almost  simple groups as
displayed in Table \ref{Table-AlmostSimpleGps}.

\begin{table}[h]\[\begin{tabular}{|c|c|}
  \hline
  $G$ & $C_G(Z)/Q$ \\
  \hline
    $\PSL_4(3)$ & $\SL_2(3)$ \\
    $\PGL_4(3)$ & $\GL_2(3)$ \\
    $\PSL_4(3).2$ & $\SL_2(3) \times 2$ \\
    $\PSL_4(3).2$ & $\SL_2(3).2$ \\
    $\aut(\PSL_4(3))$ & $\SL_2(3).(2 \times 2)$ \\
\hline
    $\Omega_8^+(2).3$ & $\SL_2(3)$ \\
     $\aut(\Omega_8^+(2))$ & $\SL_2(3) \times 2$ \\
\hline
   $F_4(2)$ & $(Q_8 \times Q_8):3$ \\
    $\aut(F_4(2))$ & $(Q_8 \times Q_8):\sym(3)$ \\
\hline
    $\HN$ & $2^.\alt(5)$\\
    $\aut(\HN)$ & $2^.\sym(5)$\\
\hline
\end{tabular}\]\caption{Almost simple groups satisfying Hypothesis \ref{3^1+4 Hypothesis}.}
\label{Table-AlmostSimpleGps}\end{table}

We note that there is scope for extending the results in this chapter. Observe, in particular, that in each of the cases shown in Table \ref{Table-AlmostSimpleGps} the Sylow $3$-subgroup has order $3^6$. It is hoped that future work will prove that
the only possibilities for the structure of $C_G(Z)$ are those that appear in Table
\ref{Table-AlmostSimpleGps}.

\chapter{Two Extensions of the Simple Orthogonal Group $\Omega_8^+(2)$}\label{Chapter-O8Plus2}

The two almost simple groups of shape $\Omega_8^+(2).3$ and
$\Omega_8^+(2).\sym(3)$ are both examples of groups satisfying Hypothesis \ref{3^1+4
Hypothesis}. Moreover, despite being extensions of classical groups defined over a field of order
two, they are both groups of parabolic characteristic three. Recall that a group $G$ is of
parabolic characteristic $p$ ($p$ a prime) if any $p$-local subgroup of $G$ which contains a Sylow
$p$-subgroup of $G$ is of characteristic $p$. As part of the ongoing project to understand the
groups of local characteristic $p$, Parker and Stroth will characterize the group $^2E_6(2)$. This
exceptional group of Lie type over $\GF(2)$ also has parabolic characteristic three. Moreover it
has a $3$-centralizer of shape $3 \times \Omega_8^+(2).3$. In order to $3$-locally
recognize $^2E_6(2)$ and its almost simple extensions, one needs to be able to $3$-locally recognize  both
$\Omega_8^+(2).3$ and $\Omega_8^+(2).\sym(3)$. The hypothesis we consider and
the theorem we prove are as follows.

\begin{hypB}
Let $G$ be a finite group and let $Z$ be the centre of a Sylow $3$-subgroup of $G$ with $Q:=O_3(C_G(Z))$.
Suppose that
\begin{enumerate}[$(i)$]
\item $Q\cong 3_+^{1+4}$;
\item $C_G(Q)\leq Q$; and
\item $Z \neq Z^x \leq Q$ for some $x \in G$.
\end{enumerate}
Furthermore assume that $C_G(Z)/Q \cong \SL_2(3)$ or $C_G(Z)/Q \cong
\SL_2(3)\times 2$ and the action of $O^2(C_G(Z)/Q)\cong \SL_2(3)$ on $Q/Z$ has exactly one non-central
chief factor.
\end{hypB}

\begin{thmB}\label{thmO8}
If $G$ satisfies Hypothesis B then $G \cong \Omega_8^+(2).3$ or $G \cong
\Omega_8^+(2).\sym(3)$.
\end{thmB}

We use the results from Chapter \ref{chaper general hypothesis} to determine more fully the
$3$-local structure of a group $G$ satisfying Hypothesis B. Observe that the hypothesis gives two
potential structures of $C_G(Z)$ and therefore we must consider each possibility at each stage of
our analysis. In both cases we identify five conjugacy  classes of subgroups of order three. We
label sets of  elements in these classes by $3\mathcal{A}$, $3\mathcal{B}$, $3\mathcal{C}$,
$3\mathcal{D}$ and $3\mathcal{E}$ and we apply a theorem due to Prince to recognize that elements
in $3\mathcal{A}$ have centralizer isomorphic to either $3\times \Omega^-_6(2)$ or $3\times
\mathrm{SO}^-_6(2)$. In fact, we observe that the $3$-local structure of $\mathrm{SO}^-_6(2)$ is
very similar to the $3$-local structure of $\mathrm{SO}_7(2)$ and so we require further work
involving the subgroup structure of both groups to determine, from a local perspective, which
isomorphism types appear. In fact when $C_G(Z)/Q \cong \SL_2(3)\times 2$ we see that $G$  contains
subgroups isomorphic to $\mathrm{SO}^-_6(2)$ which centralize an element of order three and
subgroups isomorphic to  $\mathrm{SO}_7(2)$ which centralize an involution.  When $C_G(Z)/Q \cong
\SL_2(3)$ we are able to show that $G$ has an index three subgroup relatively easily using a
transfer theorem of Gr\"{u}n. We see that the index three subgroup only contains elements in the
classes  $3\mathcal{A}$, $3\mathcal{C}$ and $3\mathcal{D}$ and so these are the focus of our
attention in this case. However in the case that $C_G(Z)/Q \cong \SL_2(3)\times 2$ we are unable to
recognize that $G$ has an index two subgroup until we have a good understanding of the centralizer
of an involution, $C_G(t)$. Fortunately the structure of the involution centralizer is fairly easy
to see partly due to the relatively large Sylow $3$-subgroup. We use our knowledge of the
$3$-structure and a theorem due to Goldschmidt to show that the involution centralizer has a normal
subgroup which is extraspecial of order $2^9$. In the case when $C_G(Z)/Q\cong \SL_2(3) \times 2$
we are able to recognize the index two subgroup of $G$ and so we are left only to recognize the
simple subgroup $\Omega_8^+(2)$. A theorem of Smith \cite{SmithOrthogonal} which characterizes
$\Omega_8^+(2)$ by the structure of an involution centralizer allows us to do this and therefore
completes the proof.

\section{Determining the 3-Local Structure of $G$}\label{O8-Section-3Local}

We continue notation and apply the results from Chapter \ref{chaper general hypothesis}. In
particular note that $Y:=ZZ^x$ (some $x \in G \bs N_G(Z)$ such that $Z^x \leq Q$), $L:=\<Q,Q^x\> \leq N_G(Y)$, $W:=C_L(Y)$,
$s$ is an involution such that $Ws \in \mathcal{Z}(L/W)$. We have that $S:=QW$ and $J:=[W,s]=[J,s]=J(W)=J(S)$
is elementary abelian (and inverted by $s$) and $Z_2:= J \cap Q$ where $Y \leq Z_2$ and
$Z_2/Z=\mathcal{Z}(S/Z)$.

Furthermore set $X:=O^2(C_G(Z))$ then $X/Q\cong \SL_2(3)$. Choose an involution $t$ such that
$Qt\in \mathcal{Z}(X/Q)$ and such that $s$ and $t$ commute.  In the case where $C_G(Z)/Q \cong
\SL_2(3)\times 2$ we choose $w\in C_G(Z)\bs X$ of order two such that $[w,t]=1$ and $\<t,s,w\>$ is
a $2$-group and $Q\<t,w\>/Q=\mathcal{Z}(C_G(Z)/Q)$. Finally set $N_1:=[Q,t]$ and $N_2:=C_Q(t)$.


\begin{lemma}\label{O8-Easy Lemma}
\begin{enumerate}[$(i)$]
\item $S \in \syl_3(G)$.
\item $Q=N_1N_2$ where $[N_1,N_2]=1$ and for $i\in \{1,2\}$, $3_+^{1+2}\cong N_i\vartriangleleft
N_G(Z)$.
\item For $i \in \{1,2\}$, $Y \nleq N_i$
and $|Z_2 \cap N_i|=9$, in particular $Z^G \cap N_i=Z$.
\end{enumerate}
\end{lemma}
\begin{proof}
$(i)$ It is clear that $C_G(Z)$ has Sylow $3$-subgroups of order $3^6$ and, by hypothesis, $Z$ is
central  in a Sylow $3$-subgroup of $G$. We have that $|Q|=|W|=3^5$ and $S=QW$ is a $3$-group with
$Q \neq W$. Thus $|S|=3^6$ and so $S \in \syl_3(C_G(Z))\subset \syl_3(G)$.

$(ii)$ By Lemma \ref{central involutions don't fix Y},  $Q$ is a
central product of the two groups $N_1\cong N_2\cong 3_+^{1+2}$ and since $Qt\in \mathcal{Z}(X/Q)$,
$N_1$ and $N_2$ are $N_G(Z)/Q$-invariant and so are normal subgroups of $N_G(Z)$.

$(iii)$ By Lemma \ref{central involutions don't fix Y}, $t$ does not normalize $Y$ so $Y$ is not
contained in either $N_1$ or $N_2$. Since $Z^x$ was chosen arbitrarily in $Q$, the only
$G$-conjugate of $Z$ in $N_1 \cup N_2$ is $Z$ itself. Since $N_1/Z$ and $N_2/Z$ are $S/Q$-invariant
$C_{N_i/Z}(S)$ is non-trivial for $i=1,2$. Thus $Z_2\cap N_i>Z$. Since $Z_2$ is abelian,  $N_i\neq
Z_2$. Thus $|N_i \cap Z_2|=9$.
\end{proof}
Set $A:=N_2 \cap Z_2$.  In the following Lemma we see that $A\vartriangleleft C_G(Z)$. Note that
elements in $A \bs Z$ play an important role in our proof of Theorem B.

\begin{lemma}\label{O8-first results}
\begin{enumerate}[$(i)$]
 \item If
$C_G(Z)/Q \cong \SL_2(3)\times 2$ then, without loss of generality, we may assume that $N_1=C_Q(w)$, $N_2=[Q,w]$ and $Q=[Q,tw]$. In particular we may assume that $tw$
acts fixed-point-freely on $Q/Z$.
 \item We have that
$C_G(Z)\leq N_G(A)\leq N_G(Z)$. Furthermore $C_G(A) \leq X$ and $[X:C_{G}(A)]=3$.
 \item $|C_S(t)|=3^4$, $J=C_J(t) \times
[J,t]$ is normalized by $t$, $C_J(t)$ has order $3^3$ and $3\cong [J,t]=[Y,t]\neq Z$.
 \item $S'=Z_2$.
 \item $X\<s\>/Q\cong \GL_2(3)$.
\end{enumerate}
\end{lemma}
\begin{proof}
$(i)$ We have that $\<t,w\> \cong 2 \times 2$ and so using coprime action and that
$C_{N_1/Z}(t)=1$, we have $N_1/Z=\<C_{N_1/Z}(w),C_{N_1/Z}(tw)\>$. Since $C_{N_1/Z}(w)$ and
$C_{N_1/Z}(tw)$ are preserved by $X$, we have (without loss of generality) that $1=[N_1,w]$ and
$N_1=[N_1,tw]$. Similarly $C_{N_2/Z}(w)$ and $C_{N_2/Z}(tw)$ are preserved by $X$ and $w$ does not
centralize $N_2/Z$ else $w$ centralizes $Q$. Therefore $tw$ does not centralize $N_2$ either. It
follows that $N_2=[N_2,tw]=[N_2,w]$. Therefore $N_1=C_Q(w)=[Q,t]$, $N_2=[Q,w]=C_Q(t)$ and
$Q=[Q,tw]$.

$(ii)$ Since $N_2/Z$ is a $C_G(Z)/Q$-module on which $t$ acts trivially (and if applicable $w$ acts
fixed-point-freely), it contains a trivial $X$-submodule. This trivial submodule has order  three and
is necessarily contained in $\mathcal{Z}(S/Z)=Z_2/Z$. Since $A=N_2 \cap Z_2$ has order nine, we
have that $A/Z$ is this trivial submodule and so $A\trianglelefteq C_G(Z)$. By Lemma \ref{O8-Easy
Lemma} $(iii)$, the only subgroup of $A$ which is conjugate to $Z$ is $Z$ itself. Thus $N_G(A)\leq
N_G(Z)$. However $A\leq C_Q(t)\cong 3_+^{1+2}$ so $C_{X}(A)$ has index at least three in $X$. Let
$S_2\cong Q_8$ be a Sylow $2$-subgroup of $X$ then $S_2$ acts trivially on $A/Z$ and therefore acts
trivially on $A$. Thus $C_{X}(A)$ has index exactly three in $X$. If $C_G(Z)=X$ then we clearly
have $C_G(A) \leq X$. Suppose $C_G(Z)>X$. Then a Sylow $2$-subgroup of $C_G(Z)$ does not centralize
$A$ since  $A/Z$ is inverted by $Qw$. Thus $C_G(A) \leq X$ and in either case we have
$[X:C_G(A)]=3$.

$(iii)$ This is just Lemma \ref{central involutions don't fix Y}.

$(iv)$ By Lemma \ref{lemma structure of Y and W} $(v)$, we have $Y\leq S' \leq Z_2$. Suppose
$S'=Y$. Then $[N_1,S] \leq N_1 \cap Y =Z$. Therefore $N_1 \leq Z_2$ which contradicts Lemma
\ref{O8-Easy Lemma} $(iii)$.  Thus $Y<S'=Z_2$.

$(v)$ We have that $X/Q\cong \SL_2(3)$ and $s$ inverts $Z$ so $s \in N_G(Z)\bs C_G(Z)$ and
normalizes $X$. Notice that $s$ does not invert $N_1/Z$ which is a natural $X/Q$-module else $s$
would invert $N_1$ which is not possible as $N_1$ is non-abelian. Also $s$ does not centralize
$N_1/Z$ by Theorem \ref{Burnside-p'-automorphism}. It follows from this action that $X\<s\>/Q\cong
\GL_2(3)$.
\end{proof}

By Lemma \ref{O8-first results} $(i)$, when $C_G(Z)>X$, we choose the involution $w\in C_G(Z)\bs X$
such that $[N_1,w]=1$.

\begin{lemma}\label{O8 lemma w and tw}The following hold.
\begin{enumerate}[$(i)$]
\item $W=C_G(Y)$.
\item If $C_G(Z)/Q\cong \SL_2(3)$ then $L=N_G(Y)$ and $N_G(Y)/C_G(Y)\cong \SL_2(3)$.
\item If $C_G(Z)/Q\cong \SL_2(3)\times 2$ then
$L\<tw\>=N_G(Y)$ and $N_G(Y)/C_G(Y)\cong \GL_2(3)$.
\end{enumerate}
\end{lemma}
\begin{proof}
By Lemma \ref{facts about W}, $|W|=3^5$  and therefore $W \in \syl_3(C_G(Y))$. Since $C_G(Y)\leq
C_G(Z)$, we only need to check that $C_G(Y)$ has odd order to show that $C_G(Y)=W$. However by
Lemma \ref{O8-Easy Lemma} $(iii)$ and Lemma \ref{O8-first results} $(i)$, no involution in
$C_G(Z)/Q$ centralizes $Y$ and therefore $C_G(Y)$ is a $3$-group. This proves $(i)$.

By Lemma \ref{Lemma-facts about L} and Lemma \ref{W is centralizer of Y in L}, we have that
$\SL_2(3)\cong L/W$ and further $L/W\leq N_G(Y)/C_G(Y)$ which is isomorphic to a subgroup of
$\GL_2(3)$. Suppose $N_G(Y)/W\cong \GL_2(3)$. Then there exists an involution $r\in N_G(Y)$ such
that $Wr$ centralizes $Z$ whilst inverting $Y/Z$. Therefore $r \in C_G(Z)$. If $C_G(Z)/Q\cong
\SL_2(3)$ then $Qr=Qt$ and so $Z^x \leq N_1$ which is a contradiction. Hence, if $C_G(Z)/Q\cong
\SL_2(3)$ then $L=N_G(Y)$. If $C_G(Z)/Q\cong \SL_2(3)\times 2$ then we have seen that $tw$ inverts
$Q/Z$ and so $N_G(Y)/C_G(Y)=L\<tw\>/W\cong \GL_2(3)$. This proves $(ii)$ and $(iii)$.
\end{proof}
\begin{lemma}\label{O8-3' subgroup normalized by Y}
No non-trivial $3'$-subgroup of $G$ is normalized by $Y$.
\end{lemma}
\begin{proof}
By Lemma \ref{3' group normalized by Y} such a group would commute with $Y$. However $C_G(Y)=W$ by Lemma \ref{O8 lemma w and tw} and $W$ is a $3$-group.
\end{proof}

\begin{lemma}\label{o8-Order of N_G(S)}
 \begin{enumerate}[$(i)$]
 \item If $C_G(Z)/Q \cong \SL_2(3)$ then $N_G(S)=S\<s,t\>$ has order $3^62^2$.
 \item If $C_G(Z)/Q \cong \SL_2(3)\times 2$ then $N_G(S)=S\<s,t,w\>$ has order $3^62^3$
 (where $\<s,t,w\>\cong 2^3$),
 $|C_J(tw)|=3^2$ and $|C_J(w)|=3^3$. 
\end{enumerate}
\end{lemma}
\begin{proof}
We have that $N_G(Z)/Q\cong \GL_2(3)$ or $N_G(Z)/Q\sim 2.\GL_2(3)$. Therefore $|N_G(S)|=3^62^2$ or
$|N_G(S)|=3^62^3$ respectively. Furthermore $s$ normalizes $W$ and $Z$ and therefore normalizes
$\<Q,W\>=QW=S$. Therefore we have that if $C_G(Z)/Q \cong \SL_2(3)$ then $N_G(S)=S\<s,t\>$ and if
$C_G(Z)/Q \cong \SL_2(3)\times 2$ then $Qw$ normalizes $S/Q$ and so $N_G(S)=S\<s,t,w\>$ where by
choice $\<s,t,w\>$ is a $2$-group. Moreover $\<s,t,w\>$ is a $2$-subgroup of $N_G(J(S))=N_G(J)$.
Since $s$  inverts $J$, $C_G(J)s$ is central in $N_G(J)/C_G(J)$. Hence $[s,t],[s,w] \in C_G(J)\leq
C_G(Y)=W$. Therefore $[s,t]=[s,w]=1$. Furthermore $t$ is central in $\<s,t,w\>$ since $Qt$ is
central in $N_G(Z)/Q$. Therefore $\<s,t,w\>$ is elementary abelian.

We have seen that $tw$ acts fixed-point-freely on $Q/Z$ and centralizes $S/Q$. So by Lemma
\ref{central involutions don't fix Y}, we have that $|C_J(tw)|=3^2$ and that $tw$ inverts $S/J$.
Also, by Lemma \ref{central involutions don't fix Y}, we have that $|C_J(w)|=3^3$.
\end{proof}

We intend to count conjugacy classes of elements of order three in $S$. For this we need some
notation. Recall that $A=Z_2 \cap N_2$. Recall also that by Lemma \ref{lemma structure of Y and W}
$(iii)$,  $3 \cong C_W(s)\leq Q \cap Q^x$. We fix the following elements of order three. Let $a \in
A\bs Z$, $b \in N_2 \bs A$, $z \in Z^\#$, $d \in [Y,t]^\#$ and $e\in C_W(s)^\#$.  We define the
following sets of elements of order three in $G$.

\begin{enumerate}[$(i)$]
 \item $3\mathcal{A}=\{a^g|g \in G\}\cup \{ (a\inv )^g|g \in G\}$;
 \item $3\mathcal{B}=\{b^g|g \in G\}\cup \{ (b\inv )^g|g \in G\}$;
 \item $3\mathcal{C}=\{z^g|g \in G\}\cup \{ (z\inv )^g|g \in G\}$;
 \item $3\mathcal{D}=\{d^g|g \in G\}\cup \{ (d\inv )^g|g \in G\}$;
 \item $3\mathcal{E}=\{e^g|g \in G\}\cup \{ (e\inv )^g|g \in G\}$.
\end{enumerate}
Clearly each of the sets $3\mathcal{A}$, $3\mathcal{B}$, $3\mathcal{C}$, $3\mathcal{D}$ and
$3\mathcal{E}$ is either a conjugacy class in $G$ or a union of two conjugacy classes. Note that
the labeling has been chosen to be consistent with {\sc Atlas} \cite{atlas} notation such that the
classes are ordered by the size of the centralizer in our target groups. Note that the classes
which play the greatest role in our proof are $3\mathcal{A}$, $3\mathcal{C}$ and $3\mathcal{D}$. We
will observe that these classes lie in a proper normal subgroup of $G$.

\begin{lemma}\label{o8-A not eq C not eq D}
$3\mathcal{A} \neq 3\mathcal{C} \neq 3\mathcal{D}$.
\end{lemma}
\begin{proof}
By Lemma \ref{O8-Easy Lemma} $(iii)$, $N_1 \cap Z^G=Z$ and $N_2 \cap Z^G=Z$. Since $a \in A \bs Z  \subset N_2 \bs Z$ and $d \in [Y,t] \bs Z \subset N_1 \bs Z$, it is clear that $3\mathcal{A} \neq 3\mathcal{C} \neq 3\mathcal{D}$.
\end{proof}

\begin{lemma}\label{O8-things about YU=Z2}
$Z_2^\#\subseteq 3\mathcal{A}\cup 3\mathcal{C} \cup  3\mathcal{D}$ and $|Z_2 \cap 3\mathcal{C}|=14$.
Furthermore, $N_G(Z) \cap N_G(J)=N_G(S)$, $C_G(Z_2)=C_G(J)=J$ and $N_G(Z_2)=N_G(S)$.
\end{lemma}
\begin{proof}
We have that $3\mathcal{A} \neq 3\mathcal{C} \neq
3\mathcal{D}$. Since $a,d \in J$ and $s$ inverts $J$, $|\{a^{\<Q,s\>}\}|=|\{d^{\<Q,s\>}\}|=6$.
Therefore $|Z_2\cap (3\mathcal{A}\cup 3\mathcal{D})|\geq 12$. Moreover, since $Y$ is not normalized
by $t$ but $Z_2$ is, $Z_2=YY^t=Y[Y,t]$. Thus  $Z_2\cap 3\mathcal{C}\subseteq Y \cup Y^t$ and $|Z_2
\cap 3\mathcal{C}|=14$. Furthermore, $N_G(Z_2)$ normalizes $Y \cup Y^t$ and thus $Y \cap Y^t=Z$. So
$N_G(Z_2) \leq N_G(Z)$.

Observe that $Z=\mathcal{Z}(S)$ and $J=J(S)$, therefore $N_G(S) \leq N_G(Z) \cap N_G(J)$.  However
$N_G(Z) \cap N_G(J)$ normalizes $JQ=S$. Therefore $N_G(S)=N_G(Z) \cap N_G(S)$. Since $J$ is
abelian, $J \leq C_G(J) \leq  C_G(Z_2) \leq C_G(Y)=W$ and since $\mathcal{Z}(W)=Y$,
$J=C_G(J)=C_G(Z_2)$. In particular his implies that $N_G(Z_2) \leq N_G(Z) \cap N_G(J)=N_G(S)$.
\end{proof}

\begin{lemma}\label{O8-elements of order nine}
There are four subgroups lying strictly between $J$ and $S$ namely $W$, $W^t$, $C_S(A)=C_S(N_2 \cap
Z_2)$ and $C_S([Y,t])=C_S(N_1 \cap Z_2)$. Every element of order three in $S$ lies in the set
$Q\cup W \cup W^t$. Moreover $C_S(A)$ and $C_S([Y,t])$ are not conjugate in $G$.
\end{lemma}
\begin{proof}
Since $J=J(S)\vartriangleleft S$, $S/J$ has order 9 and is either cyclic or has four proper
non-trivial subgroups. We have seen that $Y=\mathcal{Z}(W)$ is not normalized by $t$. Thus $W^t\neq
W$ and $W \cap W^t=J$. In particular $S/J$ is not cyclic. Now let $P_1=C_S([Y,t])$ and
$P_2=C_S(A)$. Then $P_1\supseteq JN_2$ and $P_2\supseteq JN_1$ and $|JN_1|=|JN_2|=3^5$ so it
follows that $P_1=JN_2$ and $P_2=JN_1$. By Lemma \ref{lemma structure of Y and W},
$|\mathcal{Z}(P_i)|=9$ and so $\mathcal{Z}(P_i)=N_i \cap Z_2$ for each $i \in \{1,2\}$.  Thus we
have found the four proper subgroups of $S$ strictly containing $J$. Suppose for some $g \in G$,
$P_1^g=P_2$. Then $\mathcal{Z}(P_1)^g=\mathcal{Z}(P_2)$. Since for $i \in \{1,2\}$, $Z^G \cap
N_i=Z$  (by Lemma \ref{O8-Easy Lemma} $(iii)$) and $\mathcal{Z}(P_i)\leq N_i$, we have that
$Z^g=Z$. Therefore $g \in N_G(Z)$. However $N_1,N_2\vartriangleleft N_G(Z)$ so we cannot have
$(N_1\cap Z_2)^g=N_2 \cap Z_2$.

We have that $|Q \cap W|=|C_Q(Y)|=3^4$ and similarly $|Q \cap W^t|=3^4$. Also, $W \cap W^t \cap
Q=J \cap Q=Z_2$ has order $3^3$ so the set $Q \cup W \cup W^t$ has order $(3^5\times 3) - (3^4\times
3)+3^3=513$. By hypothesis, $Q\cong 3_+^{1+4}$ has exponent three and by Lemma \ref{facts about W} $(v)$,  $W$ also has
exponent three. Observe that  $P_i=(P_i \cap Q)J$. Now let $ g \in J\bs Q$
and $h \in (P_i \cap Q)\bs J$ then every element in $P_i\bs (Q \cup J)$ can be written
as a product of such a $g$ and $h$. Suppose $(hg)^3=1$. Then we calculate using the identity
$h[g,h][g,h,h]=[g,h]h$ and using that $g \in J$ so commutes with all commutators in $S'=Z_2\leq J$.

\begin{eqnarray*}
1&=&hghghg \\
&=&h^2g[g,h]ghg \\
&=&h^2g^2[g,h]hg \\
&=&h^2g^2h[g,h][g,h,h]g \\
&=&h^2g^2hg[g,h][g,h,h] \\
&=&[h,g][g,h][g,h,h] \\
&=&[g,h,h].
\end{eqnarray*}
Now $S=Q\<g\>$ and so $S' =Q'[Q,g]$. Notice that $Q=(P_i \cap Q)N_i=Z_2N_i\<h\>$ and so it follows
from a commutator relation that
\[[Q,g]=[Z_2N_i\<h\>,g]\leq \<[Z_2,g]^{N_i\<h\>}\> \<[N_i,g]^{\<h\>}\> \<[h,g]\>=\<[N_i,g]^{\<h\>}\> \<[h,g]\>\leq (N_i\cap Z_2)\<[h,g]\>.\]
Thus we have that $Z_2\leq\<[h,g]\>(N_i\cap Z_2)$ which is centralized by  $h$ as $h \in
P_i=C_S(N_i \cap Z_2)$. However $C_Q(Z_2)=Z_2$ and $h \notin Z_2$ which gives us a contradiction.
Thus every such element $gh$ has order at least nine. This accounts for $108=3^5-(3^4\times 2)
+3^3=|P_i\bs (Q \cup J)|$ elements and $513+108+108=3^6=|S|$. Thus there are exactly $513-1=512$
elements in $S$ of order three and every such element lies in $Q \cup W \cup W^t$.
\end{proof}

We begin to gather some information about the conjugacy classes of elements of order three.  In
particular, in the following lemma we determine the order of a Sylow $3$-subgroup of the
centralizer of elements in $3\mathcal{B}$ and $3\mathcal{E}$. Note that we will see later that $G$
has a simple normal subgroup which does not contain these classes. Recall that $b \in N_2\bs A$
with $b \in 3\mathcal{B}$ and that $e \in C_W(s)^\#$ with $e \in 3\mathcal{E}$.
\begin{lemma}\label{O8- centralizer of e and v}
The following hold.
\begin{enumerate}[$(i)$]
 \item $C_Q(b)\in \syl_3(C_G(b))$, in particular, $3\mathcal{B}\notin \{3\mathcal{A},3\mathcal{C},3\mathcal{D}\}$.
 \item $C_Q(e)\in \syl_3(C_G(e))$, in particular, $3\mathcal{E}\notin \{3\mathcal{A},3\mathcal{C},3\mathcal{D}\}$.
 \item $N_X(\<e\>)=C_X(e)=C_Q(e)$ and $|Q \cap 3\mathcal{E}|\geq 144$.
\end{enumerate}
\end{lemma}
\begin{proof}
$(i)$ We have that $b \in N_2\bs A$ and by Lemma \ref{O8-Easy Lemma}, $N_2 \cap Z^G=Z$.  Therefore
$\<Z,b\>\cap 3\mathcal{C} =Z^\#$. As $Q$ is extraspecial, $C_Q(b)$ has order $3^4$ and
$\<Z,b\>=\mathcal{Z}(C_Q(b))$. Therefore $N_G(C_Q(b))\leq N_G(Z)$. Notice that $[A,b]\neq 1$ else
$N_2=A\<b\>$ is abelian. Recall that $A\vartriangleleft N_G(Z)$ and $A\leq S'=Z_2$ so $A$ is
contained in the derived subgroup of every Sylow $3$-subgroup of $N_G(Z)$. Suppose that $S_0\in
\syl_3(N_G(Z))$ such that $C_Q(b)<C_{S_0}(b)$ has order $3^5$. Then $A \leq S_0' \leq C_{S_0}(b)$
which is a contradiction. Thus $C_Q(b)$ is  a Sylow $3$-subgroup of $C_G(b)$. By Lemma
\ref{O8-elements of order nine},  since $a \in A=N_2 \cap Z_2$ and $d \in [Y,t]=[J,t]\leq N_1 \cap
Z_2$, both $a$ and $d$ commute with $3$-subgroups of order $3^5$. Thus $3\mathcal{A}\neq
3\mathcal{B}\neq 3\mathcal{D}$. Clearly $b$ is not conjugate to $Z$ which commutes with a $3$-group
of order $3^6$. Thus $3\mathcal{C}\neq 3\mathcal{B}$.

$(ii)$ Recall $e \in C_W(s)$ and $Y=\mathcal{Z}(W)$ and so $[Y,e]=1$. Notice that $e \notin J$ as
$J=[J,s]$ is inverted by $s$. Suppose $|C_J(e)|\geq 3^3$. Then $\<e\>C_J(e)$ would be an elementary
abelian subgroup of $W$ of order at least $3^4$ which contradicts that $J=J(W)$. Thus $C_J(e)=Y$.
In particular, $[A,e] \neq 1$. Notice that $Z^x$ is an arbitrary conjugate of $Z$ in $Q$ and lies
in $C_S(A)$. Thus every conjugate of $Z$ in $Q$ lies in $C_S(A)$. Therefore $e$ cannot be conjugate
into $Z$ and so $N_G(C_Q(e))\leq N_G(Z)$. Now we argue as before by supposing $C_{S_0}(e)>C_Q(e)$
for some $S_0\in \syl_3(N_G(Z))$ then $C_{S_0}(e)$ must have order $3^5$ and then it would contain
$S_0'>A$ which is a contradiction. Thus $C_Q(e)$ is a Sylow $3$-subgroup of $C_G(e)$.

$(iii)$ We have that $e \notin N_1$ else $[e,A]=1$ and we also have $e \notin N_2$ else $e$
commutes with $Z_2 \cap N_1$ which contradicts that $C_J(e)=Y$. Therefore $Ze$ is not preserved  by
an involution in $X/Q$. Hence $N_X(e)=C_X(e)=C_Q(e)$.  In particular this implies that $Q$ contains
$3^62^3/3^4=72$ subgroups conjugate to $\<e\>$. Thus $|Q \cap 3\mathcal{E}|\geq 144$.
\end{proof}

We now consider  $N_G(J)/J$. Given our target groups we would expect $N_G(J)/J$  to be isomorphic
to $\mathrm{SO}_4^+(3)$ or $\mathrm{GO}_4^+(3)$. We could in fact recognize this in our abstract
group $G$ by considering a quadratic form on $J$ in which conjugates of $z$ are singular. However
we require only the order of $N_G(J)/J$ which we calculate in the following lemma.
\begin{lemma}\label{O8-normalizer of J}
$J^\# \subset 3\mathcal{A} \cup 3\mathcal{C} \cup 3\mathcal{D}$ with $|J
\cap 3\mathcal{C}|=32$.

In particular, if $C_G(Z)/Q \cong \SL_2(3)$, then $|N_G(J)|=3^62^6$ and if $C_G(Z)/Q \cong
\SL_2(3)\times 2$, then $|N_G(J)|=3^62^7$.
\end{lemma}
\begin{proof}
By Lemma \ref{lemma structure of Y and W}, $J/Y$ is a natural $L/W$-module. Since $Y<Z_2<J$, there
are four $L$-conjugates of $Z_2$ in $J$ which intersect at $Y$ and every element of $J$ lies in at
least one such conjugate. By Lemma \ref{O8-things about YU=Z2}, $Z_2^\# \subset 3\mathcal{A} \cup
3\mathcal{C} \cup 3\mathcal{D}$ and hence $J^\# \subset 3\mathcal{A} \cup 3\mathcal{C} \cup
3\mathcal{D}$. Also by Lemma \ref{O8-things about YU=Z2}, $|Z_2 \cap 3\mathcal{C}|=14$. Therefore
the four $L$-conjugates of $Z_2$ allow us to see that $8+6+6+6+6=32$ elements of $J^\#$ are in $3
\mathcal{C}$ and the remaining $12+12+12+12=48$ elements are in $3\mathcal{A} \cup 3 \mathcal{D}$.
By Lemma \ref{o8-A not eq C not eq D}, these remaining elements are not conjugate to $z$.  Thus
$|J\cap 3 \mathcal{C}|=32$. By Lemma \ref{conjugation in thompson subgroup}, elements in $J$ are
conjugate if and only if they are conjugate in $N_G(J)$. Hence $[N_G(J):N_G(J) \cap C_G(Z)]=32$.

Now by Lemma \ref{O8-things about YU=Z2}, $N_G(Z) \cap N_G(J)=N_G(S)$. Moreover, by Lemma \ref{o8-Order
of N_G(S)}, if $C_G(Z)/Q \cong \SL_2(3)$ then $|N_G(S)|=3^62^2$ and if $C_G(Z)/Q \cong
\SL_2(3)\times 2$ then $|N_G(S)|=3^62^3$. Hence we have $|N_G(J)|=3^62^6$ or $|N_G(J)|=3^62^7$
respectively.
\end{proof}

\begin{lemma}\label{O8-no inverting z and centralizing u}
If $C_G(Z)/Q \cong \SL_2(3)$, then $z$ is not conjugate to $z^2$ in $C_G(a)$.
\end{lemma}
\begin{proof}
Consider $(N_G(Z) \cap N_G(\<a\>))/C_G(A)$. As $A=\<Z,a\>$ has order nine, $(N_G(Z) \cap N_G(\<a\>))/C_G(A)$ is elementary abelian of order at most four and is non-trivial since $C_G(A)s$ inverts $A$.
Suppose that $z$ is conjugate to $z^2$ in $C_G(a)$. Then $1 \neq (N_G(Z) \cap C_G(a))/C_G(A)\leq
(N_G(Z) \cap N_G(\<a\>))/C_G(A)$ which is therefore elementary abelian of order four. Therefore
$|(C_G(Z) \cap N_G(\<a\>))/C_G(A)|=2$. This is a contradiction since by Lemma \ref{O8-first
results} $(ii)$, $[C_G(Z):C_G(A)]=3$ and so $[C_G(Z) \cap N_G(\<a\>):C_G(A)]=1$ or $3$.
\end{proof}

In the proof of the following lemma we will be gathering the required hypotheses and then applying
a theorem due to Prince (Theorem \ref{prince}). Recall that $a \in A \bs Z$ and $A\vartriangleleft
C_G(Z)$.
\begin{lemma}\label{O8-PSp43}
We have that $N_G(J) \cap C_G(a) \nleq N_G(Z)$ and the following hold.
\begin{enumerate}[$(i)$]
\item If $C_G(Z)/Q \cong \SL_2(3)$ then $C_G(\<a\>) \cong 3 \times \Omega^-_6(2)$ and $N_G(\<a\>)$ is
isomorphic to the diagonal subgroup of index two in $\sym(3) \times \mathrm{SO}^-_6(2)$.

 \item If  $C_G(Z)/Q \cong \SL_2(3)\times 2$ then either $C_G(\<a\>) \cong 3 \times \mathrm{SO}^-_6(2)$ and
$N_G(\<a\>)\cong \sym(3) \times \mathrm{SO}^-_6(2)$ or $C_G(\<a\>) \cong 3 \times \mathrm{SO}_7(2)$
and $N_G(\<a\>)\cong \sym(3) \times \mathrm{SO}_7(2)$.
\end{enumerate}In either case $|J \cap 3 \mathcal{A}|=24$ and $3\mathcal{A} \neq 3\mathcal{D}$.
\end{lemma}
\begin{proof}
Let $C_a=C_G(a)$ and $\bar{C_a}=C_a/\<a\>$. By Lemma \ref{o8-A not eq C not eq D}, $\<a\>$ is not
conjugate to $Z$ so is not central in a Sylow $3$-subgroup. By Lemma \ref{O8-elements of order
nine}, $C_S(A)=C_S(N_2 \cap Z_2)$ has order $3^5$. Hence $C_S(a)\in \syl_3(C_a)$. Observe that
$\<a\>z\cap 3\mathcal{C}=\{z\}$ by Lemma \ref{O8-Easy Lemma} $(iii)$. We hence see that if $g \in
C_a$ and $\bar{g}$ centralizes $\bar{z}$ then $g$ centralizes $\<a,z\>=A$. Therefore
$C_{\bar{C_a}}(\bar{z})=\bar{C_G(A)}$.

Now $C_G(A) \leq X=C_G(A)Q$ so we calculate using an isomorphism theorem that
\[\SL_2(3) \cong X/Q=C_G(A)Q/Q \cong C_G(A)/C_Q(A).\] Furthermore $C_Q(A)=N_1\<a\>$ and so
\[3_+^{1+2}\cong N_1\cong N_1/(N_1\cap \<a\>)\cong N_1\<a\>/\<a\>=C_Q(A)/\<a\>.\]
So we have that $C_{\bar{C_a}}(\bar{z})=\bar{C_G(A)}$ has shape $3_+^{1+2}.\SL_2(3)$.  Suppose $J$
normalizes a $3'$-subgroup $N$ of $C_a$. Then $Y$ normalizes $N$ and so $N$ is trivial by Lemma
\ref{O8-3' subgroup normalized by Y}. Therefore $\bar{J}$ (which is an elementary abelian subgroup
of $C_{\bar{C_a}}(\bar{z})$ of order $27$) normalizes no non-trivial $3'$-subgroup of $\bar{C_a}$.

Suppose $N_{C_a}(J)\leq N_G(Z)$. We have $[N_G(J): N_{N_G(Z)}(J)]=16$ as $J$ contains 16 conjugates
of $Z$ and therefore $[N_G(J): N_{C_a}(J)]$ is a multiple of $16$. Moreover $[N_G(J): N_{C_a}(J)]$
is a multiple of $3$ since $\<a\>$ is not central in a Sylow $3$-subgroup of $G$. Thus $[N_G(J):
N_{C_a}(J)] \geq 48$ and so $J^\# \subseteq 3\mathcal{A} \cup 3\mathcal{C}$. Therefore there exists
$h\in N_G(J)$ such that $a^h=d$. Moreover, by Sylow's Theorem, we may choose $h$ such that
$C_S(a)^h=C_S(d)$. However this contradicts Lemma \ref{O8-elements of order nine}. Thus
$N_{C_a}(J)\nleq N_G(Z)$. In particular this implies that ${C_a} \neq N_{{C_a}}({Z})$.

Consider the cosets $\<a\>z$ and $\<a\>z^2$. Since $3\mathcal{C}\cap A=Z^\#$, it follows that $\<a\>z$
is conjugate to $\<a\>z^2$ in $C_a$ if and only if $z$ is conjugate to $z^2$ in $C_a$. If
$C_G(Z)=X$ then by Lemma \ref{O8-no inverting z and centralizing u}, $\<a\>z$ is not conjugate to
$\<a\>z^2$ in $C_a$. Therefore $\bar{z}$ is not conjugate to its inverse in $\bar{C_a}$. Now we may
apply Theorem \ref{prince} to say either $\bar{C_a}$ has a normal subgroup of index three or
$\bar{C_a} \cong \Omega^-_6(2)$. Suppose that $C_a$ has a normal subgroup of index three, $a \in K$,
such that $[\bar{C_a}:\bar{K}]=3$. Then $C_Q(a)$ is a Sylow $3$-subgroup of $K$ and $J \cap K=J
\cap C_Q(a)=Z_2$. In particular $N_{C_a}(J)\leq N_{C_a}(J \cap K)=N_{C_a}(Z_2)\leq N_G(Z)$ by Lemma
\ref{O8-things about YU=Z2} which we have seen is not the case.  Thus $\bar{C_a}\cong \Omega^-_6(2)$.

Now suppose $C_G(Z)>X$ then consider the element $sw$. By Lemma \ref{O8-first results}, $A \leq
N_2=[N_2,w]$. We may hence assume that $w$ inverts $a$. Clearly $s$ also inverts $a$ and so
$[sw,a]=1$. Therefore $sw\in \C_a$ and $s$ inverts $Z$ whilst $w$ centralizes $Z$. Hence $\bar{z}$
is conjugate to its inverse in $\bar{C_a}$. Again we apply Theorem \ref{prince} to $\bar{C_a}$.
Since $\bar{C_a} \neq N_{\bar{C_a}}(\bar{Z})$, we see that $\bar{C_a}\cong \mathrm{SO}^-_6(2)$ or $\bar{C_a}\cong
\mathrm{SO}_7(2)$.

In either case of $C_G(Z)$ we calculate, using Lemma \ref{O8-normalizer of J}, that
$[N_G(J):N_{C_G(a)}(J)]=24$ and therefore by Lemma \ref{conjugation in thompson subgroup}, $|J \cap
3 \mathcal{A}|=24$ and by Lemma \ref{O8-normalizer of J}, $3\mathcal{A} \neq 3\mathcal{D}$. Moreover, using \cite{atlas} for example we see that in any case the Schur
Multiplier of $\bar{C_a}$ has order two. Therefore $C_a$ splits over $\<a\>$.

Finally, in the case when $C_G(Z)/Q \cong \SL_2(3)$ we suppose for a contradiction that $N_G(\<a\>)\cong \sym(3)\times \Omega^-_6(2)$. Then an involutions, $u$ say inverts $a$ whilst normalizing $A$ and therefore $u$ normalizes $Z$. If $u$ inverts $Z$ then $u$ inverts $A$ which contradicts our assumed structure of $N_G(\<a\>)$. Therefore $[Z,u]=1$ and so $Qt=Qu$ which implies that $[a,u]=1$ which is a contradiction. The structure of $N_G(\<a\>)$ now follows since
$\Omega^-_6(2)$ has automorphism group $\mathrm{SO}^-_6(2)\sim
\Omega^-_6(2):2$. In the case when $\bar{C_a}$ is isomorphic to $\mathrm{SO}^-_6(2)$ or
$\mathrm{SO}_7(2)$, we observe that both groups are isomorphic to their automorphism groups and so
we have $N_G(\<a\>)\cong \sym(3) \times \mathrm{SO}^-_6(2)$ or $\sym(3) \times \mathrm{SO}_7(2)$.
\end{proof}

Recall that $d \in [Y,t]^\#=[J,t]^\# \subset N_1\bs Z$.
\begin{lemma}\label{O8-C_G(d) is sym9 or normalies J}
If $C_G(Z)/Q \cong \SL_2(3)\times 2$ then $C_G(d)/\<d\>\cong \sym(9)$ or $C_G(d) \leq N_G(J)$.
\end{lemma}
\begin{proof}
Assume that $C_G(Z)/Q \cong \SL_2(3)\times 2$  and set $C_d:=C_G(d)$ and $\bar{C_d}:=C_d/\<d\>$. By
Lemma \ref{O8-elements of order nine}, $C_S(d)=C_S([Y,t])=C_S(N_1 \cap Z_2)$ has order $3^5$. Hence
$C_S(d)\in \syl_3(C_d)$. Now $d \in N_1\cong 3_+^{1+2}$ and $N_1/Z$ is a natural $X/Q$-module.
Therefore no element of order two in $X$ centralizes $d$ and $C_X(d)=C_S(d)$. However $[N_1,w]=1$
and so $C_{C_G(Z)}(d)=C_G(N_1 \cap Z_2)=C_S(d)\<w\>$ has order $3^52$. The only conjugate of $z$ in
$\<d\>z$ is $z$ itself. This implies that the preimage of any element in $C_{\bar{C_d}}(\bar{z})$
centralizes $z$ and $d$. Hence $C_{\bar{C_d}}(\bar{z})=\bar{C_G(\<d,z\>)}$ which has order $3^42$.
The Sylow $3$-subgroup of $\bar{C_G(\<d,z\>)}$ is $\bar{C_S(d)}$. Notice that $C_S(d)'\leq Z_2 \cap
N_2=A$ and so $\bar{C_S(d)}' \leq \bar{A}$. Consider the centre of $\bar{C_S(d)}$. We have
$\mathcal{Z}(\bar{C_S(d)})\geq \bar{Z}$. Suppose this containment were proper. Then
$\bar{C_S(d)}'\leq \mathcal{Z}(\bar{C_S(d)})$. Hence $[C_S(d)',C_S(d)] \leq \<d\> \cap C_S(d)'\leq
\<d\> \cap A=1$. Therefore $C_S(d)'$ commutes with $\<C_S(d), C_S(A)\>=S$ which implies that
$C_S(d)'=Z$. However this means that $S/Z$ has two distinct abelian subgroups of index three, $Q/Z$
and $C_S(d)/Z$. Therefore $Q/Z \cap C_S(d)/Z \leq \mathcal{Z}(S/Z)$ has order at least $3^3$ which
contradicts that $Z_2/Z=\mathcal{Z}(S/Z)$ has order nine. Hence $\mathcal{Z}(\bar{C_S(d)})=
\bar{Z}$.

Now we have that $\bar{C_G(\<d,z\>)}$ has order $3^42$ and a Sylow $3$-subgroup with centre of
order three. We have further that $\bar{J}\vartriangleleft \bar{C_G(\<d,z\>)}$ so suppose that
$[\bar{w}, \bar{C_S(d)}]=1$. Then $[w, C_S(d)]=1$. This is a contradiction since $|C_J(w)|= 3^3$ by
Lemma \ref{o8-Order of N_G(S)}. Therefore we may apply Lemma \ref{Prelims-sym9} to say that
$C_{\bar{C_d}}(\bar{z})\cong C_{\sym(9)}(\hspace{0.5mm} (1,2,3)(4,5,6)(7,8,9) \hspace{0.5mm} )$.

If $J$ normalizes a $3'$-subgroup $N$ of $C_d$. Then $Y$ normalizes $N$ and so $N$ is trivial by
Lemma \ref{O8-3' subgroup normalized by Y}. Therefore $\bar{J}$ normalizes no non-trivial
$3'$-subgroup of $\bar{C_a}$. We may now apply Theorem \ref{princeSym9} to say that either
$\bar{C_d}\leq N_{\bar{C_d}}(\bar{J})$, in which case $C_d \leq N_G(J)$, or $\bar{C_d}\cong
\sym(9)$.
\end{proof}

\begin{lemma}\label{O8-Conjugates in centralizer of u}
Every element of order three in $C_S(a)$ is in the set $3\mathcal{A} \cup 3\mathcal{C} \cup
3\mathcal{D}$.
\end{lemma}
\begin{proof}
By Lemma \ref{O8-elements of order nine}, every element of order three in $S$ lies in the set $Q
\cup W \cup W^t$. Since $C_W(a)=C_{W^t}(a)=J$, every element of order three in $C_S(a)$ is in $Q
\cup J$. Now $C_Q(a)=AN_1$ and $Z_2=S' \leq C_Q(a)$. By Lemma \ref{O8-things about YU=Z2},
$N_G(Z_2)=N_G(S)$.  We have that $A\vartriangleleft X$ and so $C_Q(a)=C_Q(A)\vartriangleleft X$ and
since $S'=Z_2\ntriangleleft X$, $C_Q(A)$ contains four proper subgroups of order 27 properly
containing $A$, namely $\{Z_2^g=[S,S]^g|g \in X\}$. Every element $C_Q(a)$ therefore lies in a
subgroup conjugate to $Z_2$. By Lemma \ref{O8-things about YU=Z2}, $Z_2^\#\subseteq 3\mathcal{A}
\cup 3\mathcal{C}\cup 3\mathcal{D}$ and hence $C_Q(a)^\#\subseteq 3\mathcal{A} \cup
3\mathcal{C}\cup 3\mathcal{D}$.

By Lemma \ref{O8-normalizer of J}, $J^\# \subseteq 3\mathcal{A} \cup 3\mathcal{C}\cup
3\mathcal{D}$. Hence every element of order three in $C_S(a)$ is also in the set.
\end{proof}

Recall that $b \in N_2 \bs A$ and $e \in C_W(s)\leq Q \cap Q^x$. Recall also that we have defined
the sets $b \in 3\mathcal{B}$ and $e \in 3\mathcal{E}$.
\begin{lemma}\label{O8-Derived p-groups meeting S}
For any $P \in \syl_3(G)$, $P' \cap S \leq C_S(a)$.
\end{lemma}
\begin{proof}
Set $\mathcal{X}:=3\mathcal{A} \cup 3\mathcal{C} \cup 3\mathcal{D}$. By Lemma \ref{O8- centralizer
of e and v}, $3\mathcal{B}\cap 3\mathcal{X}=\emptyset$ and $3\mathcal{E}\cap
3\mathcal{X}=\emptyset$. In particular neither $b$ nor $e$ are conjugate into $C_S(a)$ by Lemma
\ref{O8-Conjugates in centralizer of u}. By Lemma \ref{O8- centralizer of e and v}, $|Q \cap 3
\mathcal{E}|\geq 144$. Since $C_Q(b) \in \syl_3(C_G(b))$, there are at least $9$ conjugates of
$\<b\>$ in $N_2$. Thus we have
counted $144 + 18=162$ distinct elements in $Q \bs C_Q(a)$. Hence  $Q \cap  \mathcal{X}\subseteq C_Q(a)$. 

Notice that $C_Q(Y) \cap \mathcal{X}\subseteq Z_2$ since $C_Q(YA)=C_Q(Z_2)=Z_2$. Also,
$N_G(Y)/C_G(Y)$ is  transitive on subgroups of order three of the natural module, $W/(Q \cap Q^x)$
(by Lemma \ref{facts about W} $(ii)$). Therefore $W=\<C_Q(Y)^{N_G(Y)}\>$ and so $W \cap \mathcal{X}
\subseteq \<Z_2^{N_G(Y)}\>=J$. It follows of course that $W^t \cap \mathcal{X} \subseteq J$.

By Lemma \ref{O8-things about YU=Z2}, every element of order three in $Z_2=S'$ is in $\mathcal{X}$. So let $P \in \syl_3(G)$ then $P' \cap S \subseteq S \cap \mathcal{X}\subseteq \<C_Q(A),J\>=C_S(a)$.
\end{proof}

\begin{lemma}\label{O8-index three subgroup}
If $C_G(Z)/Q \cong \SL_2(3)$, then $G$ has a normal subgroup of index three $\wt G$ and $\wt G \cap
S = C_S(a)$.
\end{lemma}
\begin{proof}
Assume that  $C_G(Z)/Q \cong \SL_2(3)$. By Lemma \ref{o8-Order of N_G(S)}, $N_G(S)=S\<s,t\>$.
Recall that $s$ centralizes $S/J\leq L/J$ and so $s$ centralizes $S/C_S(A)$. Also $S=C_Q(t)C_S(A)$
and so $S/C_S(A)\cong C_Q(t)/(C_Q(t)\cap C_S(A))=C_Q(t)/A$ is also centralized by $t$. Therefore
$S\<s,t\>/C_S(A)$ is abelian and so $N_G(S)'\leq C_S(A)$.

We may now apply Theorem \ref{GrunsThm} which says that $S \cap G'=\<S \cap N_G(S)',S \cap P'\mid P
\in \syl_3(G)\>$. So by Lemma \ref{O8-Derived p-groups meeting S}, $S \cap P'\leq C_S(a)$ for every
$P \in \syl_3(G)$ and we have just seen that $S \cap N_G(S)'\leq C_S(a)=C_S(A)$. Therefore $S \cap
G' \leq C_S(a)$. So $G$ has a proper derived subgroup with index a multiple of three. We let $\wt G
=O^3(G)$. Notice that $C_S(a) \leq \wt G$ since $N_G(\<a\>)$ has no normal $3$-subgroup and so
$C_S(a) \leq O^3(N_G(\<a\>)) \leq \wt G$. Thus $\wt G \cap S = C_S(a)$.
\end{proof}

Of course it follows now that when $C_G(Z)/Q \cong \SL_2(3)$ every element of order three in $\wt
G$ lies in $3\mathcal{A} \cup 3\mathcal{C} \cup 3\mathcal{D}$. We use notation such that for any $H
\leq G$, $\wt H= H \cap \wt G$.

\begin{lemma}\label{O8-AutPSp43 and Centralizer of d}
If $C_G(Z)/Q \cong \SL_2(3) \times 2$ then $C_G(a)\cong 3 \times \mathrm{SO}^-_6(2)$ and
$C_G(d)\leq N_G(J)$.
\end{lemma}
\begin{proof}
Assume that $C_G(Z)/Q \cong \SL_2(3) \times 2$. By Lemma \ref{O8-C_G(d) is sym9 or normalies J},
$C_G(d)/\<d\>\cong \sym(9)$ or $C_G(d) \leq N_G(J)$.  So suppose that $C_G(d)\cong 3 \times
\sym(9)$. Observe that every element of order three in $\sym(9)$ is conjugate into the Thompson
subgroup of a Sylow $3$-subgroup of $\sym(9)$ which is elementary abelian of order 27. Recall that
$d\in N_1$ and $b \in N_2$ and $[N_1,N_2]=1$  and so $b \in C_G(d)$. This implies that there exists
$h \in C_G(d)$ such that $\<d,b\>^h \leq J$. However by Lemma \ref{O8- centralizer of e and v},
$C_G(b)$ has non-abelian Sylow $3$-subgroups of order $3^4$ and so $b^h$ is not in $J^\# \subseteq
3\mathcal{A} \cup 3\mathcal{C} \cup 3\mathcal{D}$ which is a contradiction. Hence, by Lemma
\ref{O8-C_G(d) is sym9 or normalies J}, $C_G(d)\leq N_G(J)$.

Now recall that by Lemma \ref{O8-PSp43} $(ii)$, $C_G(\<a\>) \cong 3 \times \mathrm{SO}^-_6(2)$ or
$3 \times \mathrm{SO}_7(2)$. So suppose that  $C_G(a)\cong 3 \times \mathrm{SO}_7(2)$. By Lemma
\ref{Prelims-distinguishPSp62}, there exist three subgroups of $J$ of order nine, $J_1$, $J_2$,
$J_3$ say with $a \in J_i$ and $C_G(J_i)\cong 3 \times 3 \times \sym(6)$ for each $1\leq i \leq 3$.
Notice that no conjugate of $z$ or $d$ commutes with an element of order five and so we have
$J_i^\# \subset 3\mathcal{A}$. By Lemma \ref{O8-PSp43}, $|J \cap 3 \mathcal{A}|=24$. Since $|J_1
\cup J_2 \cup J_3|=21$, there exist 18 elements of $J \cap 3\mathcal{A}$ which, together with $a$
generate $J_i$ for some $1\leq i \leq 3$. The remaining conjugates of $a$ in $J$ are therefore in
$A$ which implies that $A\trianglelefteq C_G(a) \cap N_G(J)$. By Lemma \ref{O8-first results}
$(ii)$, $N_G(A) \leq N_G(Z)$ so we have that $C_G(a)\cap N_G(J)\leq N_G(Z)$. However this
contradicts Lemma \ref{O8-PSp43}. Thus $C_G(a)\cong 3 \times \mathrm{SO}^-_6(2)$.
\end{proof}

\begin{lemma}\label{O8-centralizer of w}
Suppose $C_G(Z)/Q\cong \SL_2(3)\times 2$. Then $C_G(w)/\<w\>\cong \mathrm{SO}^-_6(2)$ or
$\mathrm{SO}_7(2)$.
\end{lemma}
\begin{proof}
Assume that $C_G(Z)/Q\cong \SL_2(3)\times 2$. Since $N_G(\<a\>)\cong \sym(3) \times
\mathrm{SO}^-_6(2)$  we may choose an element of order two $w'$ say such that $w'$ inverts $a$ and
$w' \in C_G(O^3(C_G(a)))$ where $O^3(C_G(a))\cong \mathrm{SO}^-_6(2)$. Hence $C_G(w')$ has a
subgroup isomorphic to $\mathrm{SO}^-_6(2)$. Notice that $|N_1 \cap O^3(C_G(a))|\geq 9$ and since
$N_1$ is extraspecial, $Z \leq N_1 \cap O^3(C_G(a))$. Therefore $w' \in C_G(Z)$ however $Qw' \neq
Qt$ since $w'$ inverts $a$. Since $w$ commutes with $N_1 \cap O^3(C_G(a))>Z$, we also have $Qw'\neq
Qwt$ as $wt$ acts fixed-point-freely on $Q/Z$ by Lemma \ref{O8-first results} $(i)$. Hence $Qw'=Qw$
and so $w'$ is conjugate to $w$. Therefore $C_G(w)$ also has a subgroup isomorphic to
$\mathrm{SO}^-_6(2)$.

Recall that $N_1=C_Q(w)$. Since $N_1$ is extraspecial, we have that $N_{C_G(w)}(N_1)\leq N_G(Z)$.
By coprime action and an isomorphism theorem, we have
\[2 \times \SL_2(3)\cong C_{C_G(Z)/Q}(w)=C_{C_G(Z)}(w)Q/Q\cong C_{C_G(Z)}(w)/C_Q(w)\]
and so $C_{C_G(Z)}(w)\sim 3^{1+2}_+.(2\times\SL_2(3))$. Therefore $C_{C_G(w)}(Z)/\<w\>\sim
3^{1+2}_+.\SL_2(3)$.  Suppose $N$ is a $3'$-subgroup of $C_G(w)$ that is normalized by $C_J(w)$. By
Lemma \ref{o8-Order of N_G(S)} $(ii)$, $C_J(w)$ has order $27$ and by coprime action, $N=\<C_N(g)|g
\in (N_1 \cap Z_2)^\#\>$ since $N_1 \cap Z_2 \leq C_J(w)$. Let $g\in (N_1 \cap Z_2)\bs Z$ then $g
\in 3\mathcal{D}$ and $C_N(g)$ is normalized by $C_J(w)$. By Lemma \ref{O8-AutPSp43 and Centralizer
of d}, $C_G(g) \leq N_G(J)$. Therefore $C_N(g)$ normalizes $J \cap C_G(w)=C_J(w)$. Hence
$[C_N(g),C_J(w)]=1$. So $N=C_N(Z)$. However the only $3'$-subgroup of $C_{C_G(w)}(Z)$ which is
normalized by $C_J(w)$ is $\<w\>$. Thus $C_G(w)/\<w\>$ satisfies Theorem \ref{prince} and so
$C_G(w)/\<w\>\cong \mathrm{SO}^-_6(2)$ or $\mathrm{SO}_7(2)$.
\end{proof}






\section{The Structure of the Centralizer of $t$}

We now have sufficient information concerning the $3$-local structure of $G$ to determine an
involution centralizer. We set $H:=C_G(t)$, $P:=C_S(t)$, $R:=C_J(t)$ and $\bar{H}:=H/\<t\>$.

\begin{lemma}\label{O8-N_H(Z), P,R are sylow}
\begin{enumerate}[$(i)$]
\item If $C_G(Z)/Q \cong \SL_2(3)$ then $C_H(Z) \sim 3_+^{1+2}.\SL_2(3)$.
\item If $C_G(Z)/Q \cong \SL_2(3)\times 2$ then $C_H(Z) \sim 3_+^{1+2}.(\SL_2(3)\times 2)$.
\item $\mathcal{Z}(P)=Z$, $P \in \syl_3(H)$.
\item $O_2(C_H(A))=O_2(C_H(Z))\cong Q_8$ commutes with $C_Q(t)$.
\item $C_H(R)=\<R,t\>$.
\end{enumerate}
\end{lemma}
\begin{proof}
By Lemma \ref{O8-first results} $(iii)$, $|P|=|C_S(t)|=3^4$. We apply coprime action and an isomorphism
theorem to see that
\[C_{C_G(Z)/Q}(t)=C_{C_G(Z)}(t)Q/Q \cong C_{C_G(Z)}(t)/C_Q(t).\] Since $C_Q(t)=N_2 \cong
3_+^{1+2}$ and $C_{C_G(Z)/Q}(t)\cong \SL_2(3)$ or $\SL_2(3)\times 2$, we have $C_H(Z)=C_{C_G(Z)}(t)
\sim 3_+^{1+2}.\SL_2(3)$ or $3_+^{1+2}.(\SL_2(3)\times 2)$. Recall that $b \in C_Q(t) \bs A=C_Q(t)
\bs J$ therefore $b\in P \bs R$ and by Lemma \ref{O8- centralizer of e and v}, $C_Q(b)=C_S(b)\in
\syl_3(C_G(b))$. Hence $\mathcal{Z}(P)\leq C_P(b)\leq Q \cap P=C_Q(t)$. This implies that
$\mathcal{Z}(P)=Z$  and so $N_H(P)\leq N_H(Z)$. Suppose $P_0\in \syl_3(H)$ and $P<P_0$. Then
$P<N_{P_0}(P)$. Therefore $N_{P_0}(P)\leq C_G(Z)$ and $N_{P_0}(P)$ has order at least $3^5$. Thus
$|Q \cap N_{P_0}(P)|\geq 3^4$. This is a contradiction since $|C_Q(t)|=27$. Hence $P \in
\syl_3(H)$. This proves $(i)-(iii)$.

Consider $C_H(C_Q(t))$.  We have that $C_G(C_Q(t))\leq C_G(A)$ and by Lemma \ref{O8-first results}
$(ii)$, $C_G(A) \leq X$. Furthermore, it follows from above that $C_X(t)/C_Q(t)\cong \SL_2(3)$. Now
let $T$ be a Sylow $2$-subgroup of $H \cap X$ then $T\cong Q_8$. It follows also from Lemma
\ref{O8-first results} $(ii)$ that $[T,A]=1$. Suppose that $[C_Q(t),T] \neq 1$. Then by coprime
action, $C_Q(t)=A[C_Q(t),T]$. Observe that $[C_Q(t),A,T]=[Z,T]=1$ and $[A,T,C_Q(t)]=1$. Thus by the
three subgroup lemma, $[C_Q(t),T,A]=1$ which implies that $A$ commutes with
$A[C_Q(t),T]=C_Q(t)\cong 3_+^{1+2}$ which is a contradiction. Hence, $[T,C_Q(t)]=1$. Since $P$ is
non-abelian and $C_H(C_Q(t)) \leq X$, we have that $C_H(C_Q(t))=Z\times T$. In particular, $T$ is
normal in $C_H(Z)$.

Now, observe that $[O_2(C_H(Z)),C_Q(t)]\leq O_2(C_H(Z)) \cap C_Q(t)=1$.  Hence $O_2(C_H(Z))\leq
C_H(C_Q(t))$. This implies that $T = O_2(C_H(Z))$. Moreover, since $[T,A]=1$ and $T\in
\syl_2(C_H(A))$, $T=O_2(C_H(A))$. This proves $(iv)$.

Finally, since $P$ is non-abelian, $R\in \syl_3(C_H(R))$ and so $C_H(R)$  has a normal
$3$-complement, $M$ say, by Burnside's Theorem (\ref{Burnside-normal p complement}). Clearly $M
\leq C_H(Z)$ and is normalized by $P$. Thus $[M,C_Q(t)] \leq M \cap C_Q(t)=1$. Hence $M \leq
C_H(C_Q(t))$ and so $M=\<t\>$ and therefore $C_H(R)=\<R,t\>$ which completes the proof.
\end{proof}

\begin{lemma}\label{O8-t not conj tw}
Suppose that $C_G(Z)\neq X$. Then $t$ is not conjugate to $w$ in $G$.
\end{lemma}
\begin{proof}
Assume that $C_G(Z)\neq X$. By Lemma \ref{O8-N_H(Z), P,R are sylow}, $P\in \syl_3(C_G(t))$ and
$\mathcal{Z}(P)=Z$. Furthermore, $C_H(Z) \cap C_G(t) \sim (3_+^{1+2}\times Q_8).(3 \times 2)$ which
in particular contains an element of order four which squares to $t$. Suppose $t$ is conjugate to
$w$ in $G$. By Lemmas \ref{O8-first results} and \ref{o8-Order of N_G(S)}, $C_S(w)\geqslant N_1
C_J(w)$ and $|N_1C_J(w)|= 3^4$ so, since $C_G(t)$ has Sylow $3$-subgroups of order $3^4$,
$C_S(w)=N_1C_J(w)\in \syl_3(C_G(w))$ and clearly $Z\leq \mathcal{Z}(C_S(w))$. Thus $C_G(Z) \cap
C_G(w)$ must contain an element of order four which squares to $w$. However this is clearly not the
case. Thus $w$ is not conjugate to $t$ in $G$.
\end{proof}

Recall that $Y$ is a natural $N_G(Y)/W$-module and so $N_G(Y)/W$ is transitive on $Y^\#$. Fix
$x_1$, $x_2$, $x_3$ in $L$ such that $Y=Z \cup Z^{x_1}\cup Z^{x_2}\cup Z^{x_3}$ and set $x_0:=1$ so
that $Z^{x_0}=Z$. We will now find an appropriate conjugate of $t$ in each $C_G(Z^{x_i}) \cap H$.

\begin{lemma}\label{O8-defining ti's}
For each $i \in \{1,2,3\}$ there exists $t_i \in O^2(C_G(Z^{x_i}))$ such that  $[t_i,t]=1$ and $t$
is conjugate to $t_i$ in $N_G(J)$.
\end{lemma}
\begin{proof}
Recall that $t$ does not normalize $Y$ and so for each $i \in \{1,2,3\}$, $Z^{x_i} \nleq R=C_J(t)$.
Therefore, $|A^{x_i}\cap R|=3$ and so there exists some $1 \neq a' \in A^{x_i} \cap R$. Since $a'
\in  A^{x_i} \bs Z^{x_i}$, $a' \in 3 \mathcal{A}$. Thus by Lemma \ref{O8-PSp43} and Lemma
\ref{O8-AutPSp43 and Centralizer of d}, $C_G(a')\cong 3\times \Omega^-_6(2)$ or $3 \times
\mathrm{SO}^-_6(2)$. By Lemma \ref{prelims-PSp43 normalizer J}, we have further that
$C_{N_G(J)}(a')\sim 3^4.\sym(4)$ or $3^4 .(\sym(4) \times 2)$ respectively. In either case, we
observe that a Sylow $3$-subgroup of $C_{N_G(J)}(a')$ is transitive on the set of Sylow
$2$-subgroups of $C_{N_G(J)}(a')$. It follows from Lemma \ref{O8-first results} that $|C_S(A)|=3^5$
and of course $C_S(A)$ normalizes $J$. Therefore, $C_{S^{x_i}}(A^{x_i})$ normalizes $J$ and
$C_{S^{x_i}}(A^{x_i})\in \syl_3(C_{N_G(J)}(a'))$. Now $t^{x_i} \in O^2(C_G(Z^{x_i}))$ and so
$t^{x_i}$ centralizes $a'$ and normalizes $J$ which means that $t^{x_i}\in C_{N_G(J)}(a')$. Thus,
there exists $g \in C_{S^{x_i}}(A^{x_i})$ such that $\<t^{x_ig},t\>$ is a $2$-group. Set
$t_i:=t^{x_ig}$ then $t_i \in O^2(C_G(Z^{x_i}))$. If $\<t_i,t\>$ is non-abelian then it must be
dihedral of order eight. In particular, $J\<t,t_i\>/J$ is contained in a subgroup of
$C_{N_G(J)}(a')/J$ which is isomorphic to $\sym(4)$. However, since $|C_J(t)|=3^3$, $t \notin
O^2(N_G(J))$. Therefore $t,t_i \notin O^2(C_{N_G(J)}(a'))$ so the image of both $Jt$ and $Jt_i$ in
$\sym(4)$ is a transposition. This contradicts that $\<t,t_i\>\cong \dih(8)$ and so $[t,t_i]=1$.
Finally, $x_i \in N_G(Y) \leq N_G(J)$ and $g \in S^{x_i}\leq N_G(J)$. Therefore $t$ is conjugate to
$t_i$ in $N_G(J)$.
\end{proof}
We set $t_0:=t$ and continue notation from Lemma \ref{O8-defining ti's} by fixing an involution
$t_i$ in  $O^2(C_G(Z^{x_i})) \cap H$ for each $i \in \{0,1,2,3\}$.

\begin{lemma}
For $\{i,j\} \subset\{1,2,3\}$, $[J,t_i] \leq R$ and $[J,t_i] \neq [J,t_j]$. Moreover, either
$[t_i,t_j]=1$ or the following hold.
\begin{enumerate}[$(i)$]
\item $\<t_i,t_j\> \cong \dih(8)$; and
\item $\<t_i,t_j\>$ acts transitively on subgroups of $[J,t_i][J,t_j]$ of order three.
\end{enumerate}
\end{lemma}
\begin{proof}
By Lemma \ref{O8-first results} $(iii)$,  $[J,t_0]=[Y,t_0]\leq Q$ has order three and
$[Y,t_0]^\#\subset 3\mathcal{D}$. Let $\{i,j\} \subset \{0,1,2,3\}$  and suppose that
$[J,t_i]=[J,t_j]$. Since $Q^{x_i}t_i$ is central in $C_G(Z)^{x_i}/Q^{x_i}$, we have that $[J,t_i] \leq
\<Q^{x_i},t_i\> \cap J \leq Q^{x_i} \cap J$. Therefore $[J,t_i]=[J,t_j] \leq Q^{x_i} \cap Q^{x_j}
\cap J=Q \cap Q^x\cap J=Y$ (by Lemma \ref{facts about W}). However $Y^\# \subseteq 3 \mathcal{C}$
whereas $[J,t_i]^\#\subseteq 3\mathcal{D}$ and by Lemma \ref{o8-A not eq C not eq D}, $3\mathcal{C}
\neq 3\mathcal{D}$. Hence $[J,t_i]\neq [J,t_j]$. Now for $i \in \{1,2,3\}$, $[t_i,t_0]=1$, and so
$t_0$ normalizes $[J,t_i]$. If $t_0$ inverts $[J,t_i]$ then $[J,t_0]=[J,t_i]$ which we have just
seen is not the case. Therefore $[J,t_i,t_0]=1$ and so $[J,t_i] \leq R$.

Now let $\{i,j\} \subset \{1,2,3\}$ and suppose $[t_i,t_j]\neq 1$. Then $D:=\<t_i,t_j\>$ is a
non-abelian dihedral group and  $D \leq H \cap N_G(J)$. Set $V:=[J,t_i][J,t_j]$ then $|V|=9$ and
$V$ is normalized by $D$ since $[V,t_i] \leq [J,t_i] \leq V$. Suppose $J \cap D \neq 1$. Then $D
\cap J$ must be inverted by $t_i$ and $t_j$ which implies that $[J,t_i] \cap [J,t_j]\neq 1$ which
is a contradiction. So suppose that $3 \mid |D|$. Then for some $n \in \mathbb{Z}$, $g:=(t_it_j)^n$
has order three and $\<R,g\>\in \syl_3(H)$. Therefore $C_R(g)=\mathcal{Z}(\<R,g\>)$ is conjugate to
$Z$ by Lemma \ref{O8-N_H(Z), P,R are sylow}.  Now $1 \neq C_R(t_i) \cap C_R(t_j)\leq C_R(g)$  and
so we must have that $C_R(t_i) \cap C_R(t_j)=C_R(g)=Z^h$ for some $h \in N_H(J)$. In particular, $2
\times 2 \cong \<t,t_i\> \leq C_G(Z^h)$. Recall that $\<t,w\>$ is a fours subgroup of $C_G(Z)$ and
so $\<t,t_i\>$ must be conjugate to $\<t,w\>$ in $N_G(J)$. We also have that $t$ is conjugate to
$t_i$ in $N_G(J)$ (where $J \leq C_G(Z_\ast)$). Since $t$ is not conjugate to $w$ (by Lemma
\ref{O8-t not conj tw}), we must have that $t$ is conjugate to $tw$ by an element of $N_G(J)$.
However this is a contradiction since by Lemma \ref{o8-Order of N_G(S)} $(ii)$,
$|C_J(tw)|=3^2<3^3=|C_J(t)|$. Thus $3\nmid |D|$.

Since $D \leq N_G(J)$ and $N_G(J)$ is a $\{2,3\}$-group by Lemma \ref{O8-normalizer of J},  we have
that $D$ is a $2$-group. In particular, we may apply coprime action to see that $V=C_V(D) \times
[V,D]$.  Suppose that $C_V(t_i)=C_V(t_j)$. Then $C_V(D) \neq 1$ which implies that $[V,t_i],
[V,t_j] \leq [V,D]$. However this forces $[J,t_i]=[J,t_j]$ which is a contradiction. Therefore,
$C_V(t_i)\neq C_V(t_j)$ and so $V=C_V(t_i)C_V(t_j)$.

Now let $r \in \mathcal{Z}(D)$ be an involution. Then $r \in \<t_it_j\>$. Suppose $r$ centralizes
$V$.  Then $r$ centralizes $C_J(t_i)\cap C_J(t_j)$ which has order at least $3^2$ and has trivial
intersection with $V$ (as $C_V(t_i)\neq C_V(t_j)$). So $r$ commutes with $J$. However by Lemma
\ref{O8-things about YU=Z2}, $C_G(J)=J$.  Thus $r$ acts non-trivially on $V$ which implies that $D$
is isomorphic to a subgroup of $\aut(V)\cong \GL_2(3)$. Therefore we must have that $D\cong
\dih(8)$ and so $D$ necessarily acts transitively on the subgroups of $V$ of order three.
\end{proof}

\begin{lemma}\label{o8-finally t's commute}
For $\{i,j\} \subset\{0,1,2,3\}$, $[t_i,t_j]=1$. In particular, $[R,t_i,t_j]=1$.
\end{lemma}
\begin{proof}
Let $\{i,j\} \subset\{0,1,2,3\}$ and suppose that  $[t_i,t_j]\neq 1$. Then  $\{i,j\}
\subset\{1,2,3\}$ and $D:=\<t_i,t_j\>\cong \dih(8)$ acts transitively on the subgroups of
$V:=[J,t_i][J,t_j]$ of order three. Since $t_i$ is conjugate to $t$ and $[J,t_i]^\# \subset
3\mathcal{D}$, we have  $V^\#\subset 3\mathcal{D}$. Now $t_it_j$ has order four and by coprime
action, $R=[R,t_it_j]\times C_R(t_it_j)$. Since $D$ acts transitively on the subgroups of $V^\#$ of
order three, $V \leq [R,t_it_j]$ and since $1\neq C_R(t_i) \cap C_R(t_j) \leq C_R(t_it_j)$, it
follows that $|C_R(t_it_j)|=3$  and $|[R,t_it_j]|=9$ so $V= [R,t_it_j]$. Now suppose $t_it_j$
normalizes $Z$. Then either $Z = C_R(t_it_j)$ or $Z \leq [R,t_it_j]=V$. Since $V^\# \subset
3\mathcal{D}$, we must have that $Z=C_R(t_it_j)$. However $1 \neq C_R(t_i) \cap C_R(t_j)\leq
C_R(t_it_j)$ and so $t_i$ centralizes $Z$ which implies that $t_i$ centralizes $ZZ^{x_i}=Y$.
However, this contradicts Lemma \ref{O8 lemma w and tw} $(i)$ which says that $C_G(Y)$ is a
$3$-group.  Thus $t_it_j$ does not normalize $Z$ and so $R$ contains at least four conjugates of
$Z$, namely $Z,Z^{t_i},Z^{t_j},Z^{t_it_j}$ which implies that $|R\cap 3 \mathcal{C}|\geq 8$. Moreover,
since $A \leq R$, $A^{t_i}\leq R$ and since $Z^{t_i} \neq Z^{t_j}$, $A^{t_i} \neq A^{t_j}$. However
$1 \neq A^{t_i} \cap A^{t_j}$ therefore since $|A \cap 3 \mathcal{A}|=6$, we see that $|R\cap 3
\mathcal{A}|\geq 6*4-(2*6)=12$. However we now have a contradiction since $V^\# \subset 3
\mathcal{D}$, which implies that $|R\cap 3\mathcal{D}|\geq 8$ and $12+8+8>26=|R^\#|$. We can
therefore conclude that $[t_i,t_j]=1$.

In particular, we have that $[R,t_i]$ is normalized by $t_j$ and so $[R,t_i,t_j] \leq [R,t_i] \cap [R,t_j]=1$.
\end{proof}

\begin{lemma}\label{O8-U^t_i's}
\begin{enumerate}[$(i)$]
 \item For $\{i,j\} \subset\{0,1,2,3\}$,
$Z^{t_i} \neq Z^{t_j}$.

 \item For $\{i,j\}\subset\{0,1,2,3\}$ there exists $a_{ij}\in 3 \mathcal{A}$
 such that $A^{t_i} \cap A^{t_j}=\<a_{ij}\>$ and for $\{k,l\}\subset\{0,1,2,3\}$
 $\<a_{ij}\>=\<a_{jk}\>$ if and only if $\{i,j\}=\{k,l\}$.
 \item $|R\cap 3\mathcal{D}|=6$, $|R\cap 3\mathcal{C}|=8$ and  $|R\cap 3\mathcal{A}|=12$.

 \item $P$ contains at least four conjugacy classes of elements of order three.
\end{enumerate}
\end{lemma}
\begin{proof}
Recall that $t\in O^2(C_G(Z))$ does not normalize $Y$ and so $t_i\in O^2(C_G(Z^{x_i}))$ does not normalize
$Y^{x_i}=Y=Z{Z^{x_i}}$. However $t_i$ centralizes $Z^{x_i}$ and so does not normalize $Z$.
Therefore $Z^{t_i}\neq Z$ for each $i \in \{1,2,3\}$.

Let $\{i, j\} \subset \{1,2,3\}$. Since $[J,t_i]\leq R$, $[J,t_i]=[R,t_i]$. Now we have
$[R,t_it_j]=[R,t_j][R,t_i]^{t_j}=[R,t_j][R,t_i]$ as $[R,t_i,t_j]=1$. Since $[R,t_i]=[J,t_i]\neq
[J,t_j]=[R,t_j]$, $[R,t_it_j]$ has order nine. Since $t_it_j$ is an involution and coprime action
gives $R=C_R(t_it_j) \times [R,t_it_j]$, we must have $|C_R(t_it_j)|=3$. Moreover if
$\{i,j,k\}=\{1,2,3\}$ then $t_i$ and $t_j$ commute with $[R,t_k]$ so we have $C_R(t_it_j)=[R,t_k]$.
We also have that $[R,t_it_j]=[R,t_j][R,t_i]= C_R(t_k)$. Recall that $t_k$ is conjugate to $t$ by an element of $N_G(J)$ and so $[R,t_k]=[J,t_k]$ is conjugate to $[J,t]$. Therefore
$C_R(t_it_j)^\#=[R,t_k]^\# \subset 3 \mathcal{D}$. In particular, $t_it_j$ does not
centralize $Z$. Also since $t_k$ does not centralize $Z$ we have that $Z \nleq[R,t_it_j]=C_R(t_k)$
and so $t_it_j$ does not invert $Z$. Therefore $Z^{t_it_j}\neq Z$ and so $Z^{t_i} \neq Z^{t_j}$ for
every $i \neq j \in \{1,2,3\}$.

We conclude that $Z^{t_i} \neq Z^{t_j}$ for every $i \neq j\in \{0,1,2,3\}$ which proves $(i)$.
Also this gives us that $|R\cap 3\mathcal{C}|\geq 8$. Since $[R,t_i]\neq [R,t_j]$ for $\{i, j\}
\subset \{1,2,3\}$, we have $|R\cap 3\mathcal{D}|\geq 6$. Thus $|R\cap 3\mathcal{A}|\leq 12$.

Recall that $A \leq R$ and so $A^{t_i}\leq R$ for each $i\in \{0,1,2,3\}$. Now let $\{i,
j\},\{k,l\} \subset \{0,1,2,3\}$. Since $|R|=3^3$, we must have $|A^{t_i} \cap A^{t_j}|\geq 3$ and
since $Z^{t_i} \neq Z^{t_j}$, we have $A^{t_i} \cap A^{t_j}=\<a_{ij}\>$ for some $a_{ij}\in 3
\mathcal{A}$. We count conjugates of $a$ in $R$. Firstly, $A$ contains $6$ and $A^{t_1}$ contains a
further $4$. Now $A^{t_2}$ can contain only a further $2$ and so $A \cap A^{t_2}\neq A^{t_1} \cap
A^{t_2}$. The same argument gives that $A^{t_3}$ intersects each distinct conjugate at a distinct
subgroup of order three. This proves $(ii)$.

Furthermore we get that $|R\cap 3\mathcal{A}|=12$ and so $|R\cap 3\mathcal{D}|=6$ and $|R\cap
3\mathcal{C}|=8$ which proves $(iii)$.

Finally, for $(iv)$ we apply Lemmas \ref{o8-A not eq C not eq D}, \ref{O8- centralizer of e and v}
and \ref{O8-PSp43} to see that the sets $3\mathcal{A}$, $3\mathcal{B}$, $3\mathcal{C}$ and
$3\mathcal{D}$ are pairwise distinct.
\end{proof}

Set $D_i:=[R,t_i]$ and define
\begin{list}{$\bullet$}{}
 \item $\Omega_1:=\{\d_i|\d_i \in D_i^\#,i\in \{1,2,3\}\}$;
 \item $\Omega_2:=\{\d_i\d_j|\d_i \in D_i^\#, \d_j \in D_j^\#,\{i,j\}\subset\{1,2,3\}\}$;  and
 \item $\Omega_3:=\{\d_1\d_2\d_3|\d_i \in D_i^\#,1\leq i \leq 3\}.$
\end{list}

\begin{lemma}\label{O8-Omega orbits etc}
\begin{enumerate}[$(i)$]
 \item $R\<t_1,t_2,t_3\>\cong \sym(3)\times\sym(3)\times \sym(3)$.
 \item $R=D_1D_2D_3$, $|\<t_1,t_2,t_3\>|=8$ and $t_1t_2t_3$ inverts $R$.
 \item  $|\Omega_1|=6$,  $|\Omega_2|=12$ and  $|\Omega_3|=8$.
 \item $N_H(R)$ acts transitively on $\Omega_1\subset 3\mathcal{D}$, $\Omega_2\subset 3\mathcal{A}$ and
$\Omega_3\subset 3\mathcal{C}$ and irreducibly on $R$.
 \item $|N_H(R)|= 2^43^4$ if $C_G(Z)=X$ and $|N_H(R)|=2^53^4$ otherwise.
\end{enumerate}
\end{lemma}
\begin{proof}
For $i \in \{1,2,3\}$, we have seen that $|C_R(t_i)|=9$. Moreover if $\{i,j,k\}=\{1,2,3\}$ then
$t_i$ commutes with $[R,t_j][R,t_k]$ by Lemma \ref{o8-finally t's commute} and so
$C_R(t_i)=[R,t_j][R,t_k]$. By coprime action, $R=[R,t_i]\times C_R(t_i)$ and so,
$R=[R,t_i][R,t_j][R,t_k]=D_1D_2D_3$ and furthermore $R=C_R(t_i) C_R(t_j) C_R(t_k)$. Since each
$t_i$ inverts $D_i$ and centralizes $D_jD_k$, it is clear that $t_1t_2t_3$ inverts $R$. This
implies that $\<t_1,t_2,t_3\>$ does not have order $2$ or $4$ and so must have order $8$. Hence
$\prod_{1 \leq i \leq 3}\<D_i,t_i\>\cong \sym(3) \times \sym(3) \times \sym(3)$. This proves $(i)$
and $(ii)$.

Part $(iii)$ now follows immediately from $R=D_1D_2D_3$.

Recall that $b\in P \bs R$ and $C_R(b)$ commutes with $\<R,b\>=P$. Thus $C_R(b)=\mathcal{Z}(P)=Z$. This
implies that $\<b\>$ permutes the three subgroups $D_1$, $D_2$, $D_3$ transitively. We may assume
$b: D_1\mapsto D_2 \mapsto D_3$ (else we may swap $b$ for $b^2$). So we choose $d_1 \in D_1^\#$ and
then set $d_2:=d_1^b\in D_2$ and $d_3:=d_2^b\in D_3$. Therefore $d_3^b=d_1$. Now consider the
following $N_H(R)$-invariant partition of $R^\#=\Omega_1 \cup \Omega_2 \cup \Omega_3$. Since $t_1$
inverts $D_1$ and centralizes $D_2$ and $D_3$ and since $s$ inverts $R$, it is clear that
$\<b,t_1,s\>$ acts transitively on $\Omega_2$ and $\Omega_3$ and that $b$ centralizes $d_1d_2d_3
\in \Omega_3$. Thus it follows from $(iii)$ and Lemma \ref{O8-U^t_i's} $(iii)$ that  $\Omega_2 \subset 3 \mathcal{A}$ and $\Omega_3 \subset 3
\mathcal{C}$. We also clearly have $\Omega_1 \subset 3 \mathcal{D}$. In particular it is clear that
$N_H(R)$ acts irreducibly on $R$. This proves $(iv)$.

Notice that $C_H(Z)\cap N_H(R) \leq C_H(Z) \cap N_H(P)$ since $P=RN_2$ and $N_2\vartriangleleft
C_H(Z)$. Also $C_H(Z) \cap N_H(P)\leq C_H(Z)\cap N_H(R)$ since $R$ is the unique abelian subgroup
of order 27 in $P$ (otherwise $|\mathcal{Z}(P)|\geq 9$). Since  $C_H(Z) \sim 3_+^{1+2}.\SL_2(3)$ or $C_H(Z)
\sim 3_+^{1+2}.(\SL_2(3)\times 2)$, we have $| C_H(Z)\cap N_H(R)|=|C_G(Z) \cap N_H(P)|=3^42$ or
$3^42^2$ respectively. Since $N_H(R)$ acts transitively on $\Omega_3\subset 3 \mathcal{C}$, we have
that $[N_H(R):N_H(R) \cap C_G(Z)]=8$. Thus $|N_H(R)|=2^43^4$ or $2^53^4$ respectively.
\end{proof}

\begin{lemma}\label{O8-subgroups of R of order nine}
Let $V\leq R$ with $[R:V]=3$. Then $V \cap 3 \mathcal{A}\neq \emptyset$.
\end{lemma}
\begin{proof}
Since $R=D_1D_2D_3=\Omega_1\cup \Omega_2\cup \Omega_3\cup \{1\}$ we calculate the possible subgroups of $R$ of order nine to be the following.
\begin{enumerate}[$(i)$]
\item $\<\d_i,\d_j\>$ for $\{i,j\}\subset\{1,2,3\}$.
\item $\<\d_i\d_j,\d_j\d_k\>$ for $\{i,j,k\}=\{1,2,3\}$.
\item $\<\d_i,\d_i\d_j\d_k\>$ for $\{i,j,k\}=\{1,2,3\}$.
\end{enumerate}
Therefore every
subgroup of order nine contains an element in $\Omega_2\subset 3 \mathcal{A}$.
\end{proof}

Set $E_i:=O_2(C_H(A^{t_i}))$ for $i \in \{0,1,2,3\}$, Then $E_i=O_2(C_H(Z^{t_i})) \cong Q_8$ by
Lemma \ref{O8-N_H(Z), P,R are sylow}. Set $E:=\<E_0,E_1,E_2,E_3\>$ and $K:=N_H(E)$.

\begin{lemma}\label{O8-describing E}
\begin{enumerate}[$(i)$]
 \item for $\{i,j\} \subset \{0,1,2,3\}$, $[E_i,R]=E_i$ and $E_iE_j=O_2(C_H(a_{ij}))$ where
 $a_{ij}\in A^{t_i}\cap A^{t_j}$ (as in Lemma \ref{O8-U^t_i's} $(ii)$).                          
 \item $E\cong 2_+^{1+8}$ and $N_H(R)\leq K$.
 \item If $N$ is any
$3'$-subgroup of $H$ normalized by $R$ then $N \leq E$.
 \item $N_H(R) \cap
E=\<t\>$.
 \end{enumerate}
\end{lemma}
\begin{proof}
By Lemma \ref{O8-N_H(Z), P,R are sylow} $(iii)$, $[E_0,C_Q(t)]=1$. Since $P=C_Q(t)R$ and
$\<E_0,P\>\sim 3_+^{1+2}.\SL_2(3)$, we see that  $[E_0,P]\neq 1$ and so  $[E_0,R]\neq 1$. It
follows that $[E_0,R]=E_0$. Hence $[E_0,R]^{t_i}=[E_i,R]=E_i$ for each $i \in \{0,1,2,3\}$. So
suppose for some $i \neq j$, $O_2(C_H(A)^{t_i})=E_i = E_j=O_2(C_H(A)^{t_j})$. Then $E_i$ commutes
with $A^{t_i}A^{t_j}$. By Lemma \ref{O8-U^t_i's} $(ii)$, $A^{t_i}\neq A^{t_j}$. So if $i \neq j$
then $A^{t_i}A^{t_j}=R$ which implies that $[E_i,R]=1$ which is a contradiction. Thus for $i \neq
j$, $E_i \neq E_j$.

By Lemma \ref{O8-U^t_i's} $(ii)$, $(A^{t_i} \cap A^{t_j})^\#=\<a_{ij}\>^\# \subset 3 \mathcal{A}$
for $i \neq j$. Since $E_i\cong E_j\cong Q_8$ and $E_i \neq E_j$, by Lemma \ref{Prelims-PSp4(3)
involutions} $(iii)$, $[E_i,E_j]=1$ and $E_iE_j=O_2(H \cap C_G(a_{ij})')\cong 2_+^{1+4}$. Since $\{i,j\}
\subset \{0,1,2,3\}$ were arbitrary, we have $E:=E_0E_1E_2E_3\cong 2_+^{1+8}$. Furthermore $E$ is
normalized by $R$ and since $N_H(R)$ acts on $\Omega_3$, $N_H(R)$ acts on the set
$\{E_0,E_1,E_2,E_3\}$. Hence $N_H(R)$ normalizes $E=E_0E_1E_2E_3$. This proves $(i)$ and $(ii)$.

Let $N$ be a $3'$-subgroup of $H$ normalized by $R$. It follows from Lemma \ref{O8-U^t_i's} $(ii)$ that $A=Z \cup \<a_{01}\> \cup \<a_{02}\>
\cup\<a_{03}\>$. Hence, by coprime action, $N=\<C_N(z),C_N(a_{01}),C_N(a_{02}),C_N(a_{03})\>$. Now
$C_N(z)$ is a $3'$-subgroup of $C_H(z)$ normalized by $R$ and so $C_N(z)\leq E_0$. Fix $i \in
\{1,2,3\}$ and set $M:=C_N(a_{0i})$. Then $M$ is normalized by $R$. By coprime action,
$M=C_M(R)[M,R]$. By Lemma \ref{O8-N_H(Z), P,R are sylow} $(v)$, $C_M(R) \leq \<t\>$. Now $[M,R]\leq
C_G(a_{0i})'\cong \Omega^-_6(2)$ and so by Lemma \ref{Prelims-PSp4(3) involutions},  $[M,R]\leq
O_2(H \cap C_{G}(a_{0i})')=E_0E_i$. Thus $M \leq E$. Of course this argument holds for each $i \in
\{1,2,3\}$ and so this proves that $N \leq E$. This proves $(iii)$.

Finally observe that $[N_E(R),R]\leq E \cap R=1$ and so $N_E(R)=C_E(R)=\<t\>$ by Lemma \ref{O8-N_H(Z), P,R are sylow} $(v)$ which proves $(iv)$.
\end{proof}

Recall that $K=N_H(E)$. We now determine the order and structure of $K$.
\begin{lemma}\label{O8-order of KleqH}
We have that $K=EN_H(R)$ and $C_K(E)\leq E$. Moreover, if $C_G(Z)=X$, then $|K|=2^92^33^4$ whereas if $C_G(Z)>X$, then
$|K|=2^92^43^4$.
\end{lemma}
\begin{proof}
Consider $C_K(E)$. If $C_K(E)$ is a $3'$-group, then by Lemma \ref{O8-describing E} $(iii)$, $C_K(E)
\leq E$. Suppose $C_K(E)\cap P \neq 1$. Since $C_P(E)\trianglelefteq P$, $Z=\mathcal{Z}(P)\leq C_P(E)$. This
is a contradiction since $[Z,E]\neq 1$. Hence $C_K(E) \leq E$ and so $K/E$ is isomorphic to a
subgroup of $\mathrm{GO}^+_8(2)$ by Lemma \ref{extraspecial outer automorphisms}. Moreover, Lemma \ref{O8-describing E} $(iii)$ also gives us that $O_{3'}(K)=E$.

Let $N$ be a subgroup of $K$ such that $E \leq N \leq K$ and $N/E$ is a minimal
normal subgroup of $K/E$. Then $N\cap P \neq 1$. Since $N \cap P\vartriangleleft P$, $Z \leq N$ and
so $N \cap R \neq 1$. Since $N_H(R)\leq K$ acts irreducibly on $R$, $R \leq N$. Suppose $PE/E\in
\syl_3(N/E)$. If $N$ is a direct product of two or more isomorphic simple groups then $P$ is a
direct product of two or more of its subgroups which implies that $P$ is abelian. Hence $N/E$ is
simple. Using \cite{atlas} we see that the only simple subgroups of $\mathrm{GO}^+_8(2)$ with Sylow
$3$-subgroups of order $3^4$ are $\alt(9)$, $\Omega^-_6(2)$ and $\mathrm{SO}_7(2)$.
However in each case $N/E$ has just three conjugacy classes of elements of order three which implies $H$ has
at most three classes of elements of order three. However this contradicts Lemma \ref{O8-U^t_i's}
$(iv)$. So we have $RE/E\in \syl_3(N/E)$. Since $N/E$ is a minimal normal subgroup of $K/E$, $N/E$
is a direct product of isomorphic simple groups and so is either simple or a direct product of
three isomorphic simple groups. Thus, analysis of the maximal subgroups of $\mathrm{GO}^+_8(2)$
(again using \cite{atlas}) ensures that $N=ER$. Therefore by Lemma \ref{frattini} (Frattini argument), we
have $K=EN_K(R)=EN_H(R)$ since $N_H(R) \leq K$. The order of $K$ follows from Lemma \ref{O8-Omega
orbits etc} $(v)$ and since $N_H(R) \cap E=\<t\>$ by Lemma \ref{O8-describing E} $(iv)$.
\end{proof}

Recall that $\bar{H}=H/\<t\>$ and consider the following sets of elements of order two in $\bar{E}=\bar{E_0E_1E_2E_3}$.

\begin{list}{$\bullet$}{}
 \item $\Pi_1:=\{\bar{p_i} \mid p_i \in E_i, p_i^2=t, 0 \leq i \leq 3
\}$;
 \item $\Pi_2:=\{\bar{p_ip_j}\mid p_i \in E_i, p_j \in E_j,
p_i^2=p_j^2=t, \{i, j\}\subset  \{0,1,2,3\} \}$;
 \item $\Pi_3:=\{\bar{p_ip_jp_k}\mid p_i \in E_i, p_j \in E_j, p_k \in E_k,
p_i^2=p_j^2=p_k^2=t, \{i, j,k\}\subset  \{0,1,2,3\} \}$;
 \item $\Pi_4:=\{\bar{p_1p_2p_3p_4}\mid  p_i \in E_i, p_i^2=t, 0 \leq i
\leq 3 \}$.
\end{list}
Note that $|\Pi_1|=12$, $|\Pi_2|=54$, $|\Pi_3|=108$, $|\Pi_4|=81$, that $\Pi_1$ and $\Pi_3$ consist
of the images in $\bar{E}$ of elements of order four in $E$ whilst $\Pi_2$ and $\Pi_4$ consist of
the images in $\bar{E}$ of non-central elements of order two in $E$. Notice also that
$\bar{E}^\#=\Pi_1 \cup \Pi_2\cup \Pi_3 \cup \Pi_4$.

Observe that the sets $\Pi_1,\Pi_2,\Pi_3,\Pi_4$ are $N_H(R)$-invariant since the set
$\{E_0,E_1,E_2,E_3\}$ is $N_H(R)$-invariant.

Recall from Lemma \ref{O8-Omega orbits etc} $(ii)$ that $\<t_1,t_2,t_3\>$ is elementary abelian of order eight. Define $2$-groups $T^\ddagger$ and $T$ such that $T^\ddagger=E\<t_1,t_2,t_3\> \leq T\in \syl_2(K)$. It follows from the non-trivial action of
$\<t_1,t_2,t_3\>$ on $R$ that $E \cap \<t_1,t_2,t_3\>=1$ and so $|T^\ddagger|=2^{12}$. By Lemma
\ref{O8-order of KleqH}, if $C_G(Z)/Q \cong \SL_2(3)$, then $|K|=2^92^43^4$ and so $T=T^\ddagger$.
Otherwise $|K|=2^92^53^4$ and $T>T^\ddagger$.

Recall that for $i \in \{1,2,3\}$, $D_i=[R,t_i]$ and $D_i^\# \subseteq 3\mathcal{D}$.
\begin{lemma}\label{O8-D_i acts fpf on E and Centre T is t}
For $i \in \{1,2,3\}$,  $D_i$ acts fixed-point-freely on $\bar{E}$, $K$ acts irreducibly on
$\bar{E}$, $|\mathcal{Z}(\bar{T})|=2$ and $\mathcal{Z}(T)=\<t\>$.
\end{lemma}
\begin{proof}
By Lemma \ref{O8-Omega orbits etc} $(iv)$, $N_H(R)$ acts transitively on $\Omega_3\subset 3\mathcal{C}$. Therefore $N_H(R)\leq K$ acts transitively on the set
$\{E_0,E_1,E_2,E_3\}$. Moreover $[E_0,R]=E_0$ by Lemma \ref{O8-describing E} $(i)$ and so $R$ acts transitively on $\bar{E_0}^\#$. Thus $K$
acts transitively on $\Pi_1$. In particular, for $\bar{p_0} \in \bar{E_0}^\#$, $C_K(\bar{p_0})$ has
Sylow $3$-subgroups of order $3^3$. By Lemma \ref{O8-N_H(Z), P,R are sylow} $(iv)$,
$[C_Q(t),E_0]=1$ and so  $C_Q(t)\in \syl_3(C_K(\bar{p_0}))$.

Similarly, we have that $N_H(R)$ acts transitively on $\Omega_2\subset 3 \mathcal{A}$ and for
$\{i,j\} \subset \{0,1,2,3\}$, $E_iE_j=O_2(C_G(a_{ij}))$ (Lemma \ref{O8-describing E} $(i)$).
Therefore $N_H(R)\leq K$ acts transitively on the set $\{E_iE_j|\{i, j\}\subset  \{0,1,2,3\}\}$. So
consider $\bar{p_ip_j}\in \Pi_2$. We have that $[E_i,R]=E_i=C_E(A^{t_i})$ and
$[E_j,R]=E_j=C_E(A^{t_j})$ however $A^{t_i}$ and $A^{t_j}$ both normalize $E_iE_j$. Therefore
$R=A^{t_i}A^{t_j}$ acts transitively on $\bar{E_iE_j}\cap \Pi_2$ and so $K$ acts transitively on
$\Pi_2$. Since the orbit, $\{(p_ip_j)^K\}$ has length a multiple of 27 and $K\geq P\in \syl_3(H)$,
it is clear from Lemma \ref{O8-describing E} that $\<a_{ij}\> \in \syl_3(C_K(p_ip_j))$.

Recall that $b \in C_Q(t) \bs A\subseteq P \bs R$ and since $\mathcal{Z}(P)=Z$,  $C_R(b)=Z$. Hence $b$ permutes the set $\{Z^{t_1},Z^{t_2},Z^{t_3}\}$ and therefore
$b$ permutes $\{E_1,E_2,E_3\}$. Pick $p_1\in E_1\bs\<t\>$ then $\bar{p_1p_1^bp_1^{b^2}} \in \Pi_3$
and commutes with $\<b\>$. Since $R$ preserves each $E_i$, $C_R(\bar{p_ip_jp_k})= C_R(\bar{p_ip_j})\cap C_R(\bar{p_jp_k})=\<a_{ij}\> \cap \<a_{jk}\>=1$ by
Lemma \ref{O8-U^t_i's} $(ii)$.  Therefore Sylow $3$-subgroups of $C_K(\bar{p_ip_jp_k})$ are conjugate to
$\<b\>$. Since $R$ preserves each $E_i$, we see that $\{\bar{p_ip_jp_k}\}^R=\bar{E_iE_jE_k} \cap
\Pi_3$. Since $K$ is transitive on $\{E_0,E_1,E_2,E_3\}$ it is clear that
$\{\bar{p_ip_jp_k}\}^K=\Pi_3$. 

Now consider $C_{\bar{E}}(D_1)$. For $\bar{p_0} \in \bar{E_0}$, $C_R(\bar{p_0})=C_Q(t) \cap R=A$. Since
$A \cap 3\mathcal{D}=\emptyset$ (Lemmas \ref{o8-A not eq C not eq D} and \ref{O8-PSp43}), $\bar{p_0}$ commutes with no conjugate of $D_1$. We have
calculated that for $\bar{p_ip_j}\in \Pi_2$,  $\<a_{ij}\> \in \syl_3(C_K(p_ip_j))$. Furthermore,
for $\bar{p_ip_jp_k}\in \Pi_3$, $C_K(\bar{p_ip_jp_k})$ has Sylow $3$-subgroups of order three which
we have seen are conjugate to $\<b\>$ ($b \in 3\mathcal{B}\neq 3\mathcal{D}$ by Lemma \ref{O8- centralizer of e and v}). Thus $\bar{p_ip_jp_k}$ commutes with
no conjugate of $D_i$. So suppose  $C_{\bar{E}}(D_1)\neq 1$. Then there exists  some
$\bar{p_0p_1p_2p_3}\in \Pi_4$ commuting with $D_1$. However $R$ preserves each set $\bar{P_i}^\#$
and so  $C_R(\bar{p_0p_1p_2p_3})= C_R(\bar{p_0p_1})\cap C_R(\bar{p_2p_3})=\<a_{01}\> \cap \<a_{23}\>=1$. Thus
$C_{\bar{E}}(D_1)=1$.

Now we again observe that $p_1p_1^bp_1^{b^2}$ commutes with $b$ and $b\in C_Q(t)$. By Lemma \ref{O8-N_H(Z), P,R are sylow} $(iv)$, we have that $[E_0,C_Q(t)]=1$ and so $[E_0,b]=1$. Therefore for any $p_0\in E_0\bs \<t\>$, $[p_0p_1p_1^bp_1^{b^2},b]=1$ and
so $81 \nmid |\{\bar{p_0p_1p_1^bp_1^{b^2}}^K\}|$ and $K$ is not transitive on $\Pi_4$. As before,
since $R$ preserves each $E_i$, $C_R(p_0p_1p_2p_3)= C_R(p_0p_1)\cap C_R(p_2p_3)=\<a_{01}\> \cap
\<a_{23}\>=1$. Hence $|\{\bar{p_0p_1p_2p_3}^K\}|$ is a multiple of $27$ which is strictly less than
81 and so there are either two or three $K$ orbits on $\Pi_4$.

Observe that no involution in $\Pi_1\cup \Pi_2\cup\Pi_3$ lies in the centre of a Sylow
$2$-subgroup of $\bar{K}$ since each orbit has even order and so an involution in $\Pi_4$ must be
$2$-central in $\bar{K}$. Choose such an involution $\bar{q_0q_1q_2q_3}\in \mathcal{Z}(\bar{T})^\#$ then
$|\{\bar{q_0q_1q_2q_3}\}^K|$ is an odd multiple of 27 and so $|\{\bar{q_0q_1q_2q_3}\}^K|=27$.

We have that $\bar{q_0q_1q_2q_3}\in \mathcal{Z}(\bar{T})^\#$ and $t_1t_2t_3\in T$. We claim that
$\<\bar{q_0q_1q_2q_3}\>=\mathcal{Z}(\bar{T})$. Suppose not then we have $\bar{q_0q_1q_2q_3}\neq
\bar{p_0p_1p_2p_3}\in \mathcal{Z}(\bar{T})^\#$. Since $t_1t_2t_3$ inverts $R$ (Lemma \ref{O8-Omega
orbits etc} $(ii)$), it preserves the set $\{E_0,E_1,E_2,E_3\}$. Therefore
$[\bar{p_0p_1p_2p_3},\bar{t_1t_2t_3}]=1$ implies
\[[\bar{p_0},\bar{t_1t_2t_3}]=[\bar{p_1},\bar{t_1t_2t_3}]=[\bar{p_2},\bar{t_1t_2t_3}]=[\bar{p_3},\bar{t_1t_2t_3}]=1.\]
Since $\bar{q_0q_1q_2q_3}\neq
\bar{p_0p_1p_2p_3}$, for some $0\leq i \leq 3$, $\bar{p_i} \neq \bar{q_i}$ and so $\bar{E_i}=\<\bar{p_i},\bar{q_i}\>$ and furthermore, $[\bar{E_i},\bar{t_1t_2t_3}]=1$. We have seen that $D_1$ acts fixed-point-freely on $\bar{E}$ and therefore on $\bar{E_i}$. Hence
$\bar{E_i}=[\bar{E_i},\bar{D_1}]$ and so we have
$[\bar{E_i},\bar{t_1t_2t_3},\bar{D_1}]=[\bar{E_i},\bar{D_1},\bar{t_1t_2t_3}]=1$. Now, by the three subgroup
lemma, $1=[\bar{D_1},\bar{t_1t_2t_3},\bar{E_i}]=[\bar{D_1},\bar{E_i}]=\bar{E_i}$ (since $t_1t_2t_3$
inverts $D_1$) which is a contradiction. Thus $\<\bar{q_0q_1q_2q_3}\>=\mathcal{Z}(\bar{T})\cap
\bar{E}$ and so there is only one $K$-orbit of $2$-central involutions in $\bar{E}$. Hence $\Pi_4$
consists of two $K$-orbits; one of length 27 containing 2-central involutions and the other of
length 54.

Suppose $\<t\><F\leq E$ where $F\vartriangleleft K$. Then $\bar{F}$ is a union of $K$ orbits on
$\bar{E}$. However no union of $K$ orbits has order $2^n$ unless $F=E$. Hence $K$ acts irreducibly
on $\bar{E}$. Also since $C_K(E)\leq E$ (by Lemma \ref{O8-order of KleqH}), we have that $\mathcal{Z}(T)\leq E$ and so $\mathcal{Z}(T)=\mathcal{Z}(E)=\<t\>$ as $E$
is extraspecial. Since $K/E$ acts faithfully on $\bar{E}$, $\mathcal{Z}(\bar{T})\leq \bar{E}$ and so
$|\mathcal{Z}(\bar{T})|=2$.
\end{proof}

The following lemma will allow us to apply the strongly $2$-closed arguments in Lemma
\ref{prelim-strongly closed}.

\begin{lemma}\label{O8-centralizers on E}
Let $g \in T^\ddagger \bs E$ then $|C_{\bar{E}}(g)|=2^4$.
\end{lemma}
\begin{proof}
Let $g \in T^\ddagger \bs E$ and let $x \in \<t_1,t_2,t_3\>$ such that $Eg=Ex$. Then  for $\{i,j,k\}=\{1,2,3\}$, we have that $x$  equals $Et_i$, $Et_it_j$ or $Et_it_jt_k$ and so, in any case, inverts $D_i$ (see
Lemma \ref{O8-Omega orbits etc}). By Lemma \ref{O8-D_i acts fpf on E and Centre T is t}, $C_{\bar{E}}(D_i)=1$ and so by Lemma \ref{Prelims-centralizers of invs on a vspace which invert a 3}, $|C_{\bar{E}}(x)|\leq 2^4$. However $|C_{\bar{E}}(x)|\geq 2^4$ by Lemma
\ref{lem-cenhalfspace} and so $|C_{\bar{E}}(x)|= 2^4$.
\end{proof}

\begin{lemma}\label{O8-T is a sylow 2 subgroup}
We have $T \in \syl_2(G)$.
\end{lemma}
\begin{proof}
We show that $E$ is characteristic in $T$. By Lemma \ref{O8-N_H(Z), P,R are sylow} $(v)$, $C_H(R)=\<t\>R$.
Therefore $T/E$ acts faithfully on $RE/E$ and so is isomorphic to a subgroup of $\GL_3(3)$. In
particular the largest elementary abelian $2$-subgroup of $T/E$ has order $2^3$. So suppose $\a$ is
an automorphism of $T$ and $E\neq E^\a \leq T$.  Then $E^\a \vartriangleleft T$, $E^\a E/E$ is
elementary abelian of order at most $2^3$ and  $\bar{E\cap E^\a}$ has order at least $2^5$ and is
central in $\bar{EE^\a}$. If $T=T^\ddagger$ then we have that $EE^\a\leq T^\ddagger$. So suppose that
$T^\ddagger<T$. Then we have that $T/E$ is non-abelian and so $\mathcal{Z}(T/E)\leq T^\ddagger/E$. Since $EE^\a
/E\vartriangleleft T$, $1 \neq EE^\a/E \cap \mathcal{Z}(T/E) \leq EE^\a/E \cap T^\ddagger/E$.

Thus in either case we have that $1 \neq EE^\a/E \cap T^\ddagger/E$ and so $EE^\a \cap T^\ddagger>E$. By Lemma \ref{O8-centralizers on E},
$|C_{\bar{E}}(EE^\a \cap T^\ddagger)|\leq 2^4$. However $EE^\a \cap T^\ddagger$ centralizes
$\bar{E\cap E^\a}$ which has order at least $2^5$. This is a contradiction. Therefore $E$ is
characteristic in $T$. So let $S\in \syl_2(H)$ and suppose $T<S$. Then $T<N_S(T)$ and $N_S(T)$
normalizes $E$. Therefore $T<N_S(T) \leq N_G(E)=K$. This is a contradiction as $T \in \syl_2(K)$.
Therefore $T \in \syl_2(H)$. Now since $\mathcal{Z}(T)$, the same argument proves that $T \in \syl_2(G)$.
\end{proof}

Recall that when $C_G(Z)/Q \cong \SL_2(3) \times 2$ we have chosen an involution $w$ as in Lemma \ref{O8-first results}. Furthermore, recall that in Lemma \ref{O8-centralizer of w} we proved that if $C_G(Z)/Q \cong \SL_2(3) \times 2$, then either $C_G(w)/\<w\>\cong \mathrm{SO}^-_6(2)$ or $C_G(w)/\<w\>\cong \mathrm{SO}_7(2)$. We are now able to be more precise.
\begin{lemma}
If $C_G(Z)/Q \cong \SL_2(3) \times 2$, then $C_G(w)/\<w\>\cong \mathrm{SO}_7(2)$.
\end{lemma}
\begin{proof}
By Lemma  \ref{o8-Order of N_G(S)} $(ii)$, $|C_J(w)|=3^3$ and by Lemma \ref{O8-N_H(Z), P,R are sylow} $(v)$, $C_H(R)=R\<t\>$. Therefore $[R,w]\neq 1$ and $|C_J(w)\cap R|= 9$.
Hence we may apply Lemma \ref{O8-subgroups of R of order nine} to give us that $C_J(w)\cap R$ contains an element in $3 \mathcal{A}$. Let $a' \in C_J(w)\cap R \cap 3 \mathcal{A}$. Then  by Lemma \ref{O8-Omega orbits etc}, $a \in \Omega_2$ and by Lemma \ref{O8-describing E} $(i)$, $C_E(a')\cong 2_+^{1+4}$. It therefore follows that $t \in  [C_G(a'),C_G(a')]\cong \Omega^-_6(2)$ and by Lemma \ref{Prelims-PSp4(3) involutions}, $t$ is $2$-central in $[C_G(a'),C_G(a')]$.

By Lemma \ref{O8-t not conj tw}, $w$ is not $G$-conjugate to $t$. Suppose that $w \in  [C_G(a'),C_G(a')]\cong \Omega^-_6(2)$. Then, since $w$ commutes with $C_J(w)\cong 3^3$, we apply Lemma \ref{Prelims-PSp4(3) involutions} again to see that $w$ is $2$-central in $[C_G(a'),C_G(a')]$. However this would force $w$ to be conjugate to $t$ which is a contradiction.

Thus $w\notin [C_G(a'),C_G(a')]\cong \Omega^-_6(2)$ and so by Lemma \ref{prelims-PSp43 normalizer J} $(ii)$,  $C_G(a') \cap
C_G(w)\cong 3\times ( 2 \times \sym(6))$. Thus $C_G(w)/\<w\>$ contains an element of order three with
centralizer $3 \times \sym(6)$. We again apply Lemma \ref{prelims-PSp43 normalizer J} $(ii)$ to see that  $C_G(w)/\<w\>\ncong \mathrm{SO}^-_6(2)$. We can therefore conclude that $C_G(w)/\<w\>\cong
\mathrm{SO}_7(2)$.
\end{proof}

Recall that $X=O^2(C_G(Z))$ and so in the case that $C_G(Z)/Q\cong \SL_2(3) \times 2$, $C_G(Z)>X$.
\begin{lemma}
Suppose $C_G(Z)/Q \cong \SL_2(3) \times 2$. Then $w \notin O^2(G)$ and $C_{O^2(G)}(Z)=X$.
\end{lemma}
\begin{proof}
Recall that $w$ normalizes $J$ and centralizes $t$ and therefore normalizes $C_J(t)=R$. Hence $w
\in N_H(R) \leq K$. We have that $C_G(w)/\<w\>\cong \mathrm{SO}_7(2)$. We assume for a
contradiction that $w \in O^2(G)$ and so by Theorem \ref{extremal transfer}, we may suppose that
there exists $g \in G$ such that $w^g \in T^\ddagger$ and $C_T(w^g)\in \syl_2(C_G(w^g))$ which has
order $2^{10}$.

Suppose first that $w^g \in E$. Since $w$ is an involution, $\bar{w^g} \in \Pi_2 \cup \Pi_4$. So
$\bar{w^g}$ lies in a $K$-orbit of length either $54$ or $27$. Hence $|C_{\bar{T}}(\bar{w^g})|\geq
2^{11}$ which implies that $|C_{{T}}({w^g})|\geq 2^{11}$. Therefore $w^g \notin E$. So we have $w^g
\in T^\ddagger\bs E$. Since $|C_T(w^g)|=2^{10}$, $|C_E(w^g)|\geq 2^6$. Therefore
$|C_{\bar{E}}(Ew^g)|\geq 2^5$. This contradicts Lemma \ref{O8-centralizers on E} $(i)$. Thus $w \notin O^2(G)$. It is now clear that $C_{O^2(G)}(Z)=X$.
\end{proof}

We may now apply all previous results when $C_G(Z)=X$ to $O^2(G)$. Recall that when $C_G(Z)=X$, Lemma \ref{O8-index three subgroup} proves that  $G$ has an index three subgroup $\wt{G}$. Recall that for any subgroup $B \leq G$, we define $\wt{B}= B \cap \wt{G}$.

\begin{lemma}
If $C_G(Z)=X$ then $K=H$ and $\wt{H}/E\cong \sym(3) \times\sym(3) \times \sym(3)$.
\end{lemma}
\begin{proof}
Assume $K<H$. Then $E \ntrianglelefteq H$. Set $F:=\<E^H\> >E$. By Lemma \ref{O8-index three
subgroup}, $G$ has an index three normal subgroup $\wt G$ and $S \cap \wt{G}=C_S(A)$. Therefore $P \cap
\wt G=R$ and so $\wt H < H$. Clearly $F \leq \wt H$.  By Lemma \ref{O8-describing E} $(iii)$, $E$ is the unique largest $3'$-subgroup
of $H$ normalized by $R$ and so $R\cap F\neq 1$. Since $N_H(R)$ acts irreducibly on $R$, this implies $R
\leq F$. Suppose $K \cap F=ER$. Then by Lemma \ref{O8-describing E}, $N_E(R)=\<t\>$ and $N_H(R)\leq K$ and so
$N_F(R)=R\<t\>=C_F(R)$. Therefore, by Theorem \ref{Burnside-normal p complement}, $F$ has a normal $3$-complement which implies $E\trianglelefteq
H$ which is a contradiction. So $K \cap F>ER$ and it follows that $E$ is not a Sylow $2$-subgroup
of $F$. Thus $T\cap F> E$. Recall that $T=T^\ddagger=E\<t_1,t_2,t_3\>$ where $\<t_1,t_2,t_3\>$ is
elementary abelian and so $\Omega(T\cap F)=T \cap F$.

Set $N:=O_{2'}(F)$. If $N$ is a $3'$-subgroup of $F$ then by Lemma \ref{O8-describing E} $(iii)$, $N\leq E$
which implies $N=1$. So suppose $R \cap N \neq 1$. Then $[R \cap N,E] \leq N \cap E=1$. However no
element of order three in $R$ acts trivially on $E$. Thus $N=O_{2'}(F)=1$.

By Lemma \ref{O8-T is a sylow 2 subgroup}, $\bar{T}\in \syl_2(\bar{H})$ and using  Lemma \ref{O8-centralizers on E} we see that if $\bar{g} \in \bar{T} \bs \bar{E}$ then $|C_{\bar{E}}(\bar{g})|=2^4$. Therefore  we may apply \ref{prelim-strongly closed} (since $8=m(\bar{E})>m(\bar{T}/\bar{E})+m(C_{\bar{E}}(\bar{g}))=3+4$ where $m$ indicates the $2$-rank) to $\bar{H}$ to say that $\bar{E}$ is strongly
closed in $\bar{T}$ with respect to $\bar{H}$.  Hence $\bar{E}$ is strongly closed in $\bar{F} \cap
\bar{T}$ with respect to $\bar{F}$.

Now we observe that by a Frattini argument, $H=N_H(R)F$ and so $\bar{F}=\<\bar{E}^{\bar{H}}\>=\<\bar{E}^{\bar{N_H(R)F}}\>=\<\bar{E}^{\bar{F}}\>$. Finally we may apply Theorem \ref{goldschmidt} to $\bar{F}=\<\bar{E}^{\bar{F}}\>$ to get that
$\bar{E}=O_2(\bar{F})\Omega(\bar{T}\cap \bar{F})$. However $\Omega(T\cap F)\nleq E$ and so
$\Omega(\bar{T}\cap \bar{F}) \nleq \bar{E}$. This contradiction proves that $K=H$.

Since $\wt{G}\cap S=C_S(a)$, $\wt{H} \cap P=C_P(a)=R$. Therefore $\wt{H}=ER\<t_1,t_2,t_3\>$. By
Lemma \ref{O8-Omega orbits etc} $(i)$, $R\<t_1,t_2,t_3\> \cong \sym(3) \times \sym(3) \times
\sym(3)$. Thus $\wt{H}/E \cong \sym(3) \times \sym(3) \times \sym(3)$.
\end{proof}

\begin{lemma}
If $C_G(Z)/Q \cong \SL_2(3)$, then $G \cong \Omega_8^+(2).3$. If $C_G(Z)/Q \cong
\SL_2(3)\times 2$ then $G \cong \Omega_8^+(2).\sym(3)$.
\end{lemma}
\begin{proof}
Assume $C_G(Z)/Q \cong \SL_2(3)$ and we will first prove that $\wt{G}\cong \Omega_8^+(2)$. Set $N:=O_{2'}(\wt{G})$ and suppose $N \neq 1$. Then $3\mid |N|$ since
$Y$ normalizes no non-trivial $3'$-subgroup of $G$ by Lemma \ref{O8-3' subgroup normalized by Y}. Therefore $1 \neq S \cap N \trianglelefteq S$
so $\mathcal{Z}(S) \cap N \neq 1$ and therefore $Z \leq N$. Now we have that $Z$ normalizes $E$ and so $[Z,E]
\leq N \cap E=1$. However $Z$ does not centralize $E$. Thus $N=1$. This implies that $t \notin Z^*(\wt G)$ else $t \in \mathcal{Z}(\wt G)$ which is not the case. Now set
$M:=O^2(\wt G)$ and consider $\wt H \cap M$. Since $O_{2'}(\wt G)=1$ and $T \in \syl_2(G)$, $1 \neq
T \cap M\trianglelefteq T$ and so $1 \neq \mathcal{Z}(T) \cap M$. By Lemma \ref{O8-D_i acts fpf on E and Centre T is t} $\mathcal{Z}(T)=\<t\>$ and so $t \in M$. If $E \cap M=\<t\>$ then $[E,R] \leq
E \cap M=\<t\>$ which implies that $R$ acts trivially on $E/\<t\>$ and therefore $[E,R]=1$ which is a contradiction. Thus $E \cap M>\<t\>$. Since $K$
acts irreducibly on $\bar{E}$, we have, $E \leq M$. Recall that $T=E\<t_1,t_2,t_3\>$ is a Sylow
$2$-subgroup of $G$ and therefore a Sylow $2$-subgroup of $\wt G$. Recall also that each $t_i$ is
$G$-conjugate to $t$ and so $t_i \in M$. Hence $T \leq M$ and so we have $\wt G=M$.

We now apply Theorem \ref{Smith-Orthog} to $\wt G$ to say that $\wt G\cong
\Omega_8^+(2)$. Since $\out(\Omega_8^+(2))\cong \sym(3)$, $G \cong
\Omega_8^+(2).3$ is uniquely defined.

Now assume that $C_G(Z)/Q \cong \SL_2(3)\times 2$. Then we have that $O^2(G)\cong
\Omega_8^+(2).3$. Thus $G \cong \Omega_8^+(2).\sym(3)$.
\end{proof}

This completes the proof of Theorem B.

\chapter{A 3-Local Characterization of the Harada--Norton Sporadic Simple Group}\label{Chapter-HN}

In \cite{HaradaHN} in 1975, Harada introduced a  new simple group. He proved that a group with an
involution whose centralizer is a double cover of the automorphism group of the Higman--Sims
sporadic simple group is simple of order $2^{14}.3^6.5^6.7.11.19$. In 1976, in his PhD thesis,
Norton proved such a group exists and thus we have the Harada--Norton sporadic simple group, $\HN$.
The simple group was not proved to be unique until 1992. In \cite{SegevHN}, Segev proves that there
is a unique group $G$ (up to isomorphism) with two involutions $u$ and $t$ such that $C_G(u)\sim (2
^. \HS) : 2$ and $C_G(t)\sim 2_+^{1+8}.(\alt(5)\wr 2)$ with $C_G(O_2(C_G(t)))\leq O_2(C_G(t))$. We
can therefore define the group $\HN$ by the structure of two involution centralizers in this way.

In this chapter, we characterize $\HN$ by the  structure of the centralizer of a $3$-central
element of order three. The hypothesis we consider and the theorem we prove are as follows.
\begin{hypC} Let $G$ be a group and let $Z$ be the centre of a Sylow $3$-subgroup of $G$
with $Q:=O_3(C_G(Z))$. Suppose that
\begin{enumerate}[$(i)$]
\item $Q\cong 3_+^{1+4}$;
\item $C_G(Q)\leq Q$;
\item $Z \neq Z^x \leq Q$ for some $x \in G$; and
\item $C_G(Z)/Q \cong 2^{.}\alt(5)$.
\end{enumerate}
\end{hypC}
\begin{thmC}
If $G$ satisfies Hypothesis C then $G \cong \HN$.
\end{thmC}

In Section \ref{HN-Section-3Local}, we determine the structure of certain $3$-local subgroups of
$G$. We identify a subgroup $X \leq C_G(Z)$ such that $X/Q\cong \SL_2(3)$. This solvable group $X$
is isomorphic to the centralizer of a $3$-central element of order three in $\PSL_4(3)$ and so the
analysis is very similar to that required in a somewhat similar recognition of $\PSL_4(3)$
\cite{AstillPSL43}. Moreover, $3$-local arguments will often consider a subgroup of $C_G(Z)$ generated by  two distinct Sylow $3$-subgroups and of course $X$ is such a subgroup. The action of $X$ on $Q$ allows us to see the fusion of elements of order three
in $Q$. In particular, it allows us to identify a distinct conjugacy class of elements of order
three. In $3$-local recognition results, it is often necessary to determine $C_G(x)$ for each
element $x$ of order three in $G$. In this case, we have just one further centralizer to determine
which is isomorphic to $3 \times \alt(9)$. Thus we need an identification of $\alt(9)$ from its
$3$-local subgroups. Observe that in $\alt(9)$, the centralizer of a $3$-central element of order
three is just a $3$-group. This makes identification of $\alt(9)$ difficult. In Chapter \ref{chapter-Alt9}, we
describe some character and modular character theoretic methods which allow us to overcome this
difficulty. These character theoretic results together with some local arguments give a necessary
recognition of $\alt(9)$.

In Section \ref{HN-Section-CG(t)}, we determine the  structure of $C_G(t)$ where $t$ is a
$2$-central involution. This requires a great deal of $2$-local analysis, in particular, we must take full advantage of our knowledge of the $2$-local subgroups in $\alt(9)$ and use a theorem due to
Goldschmidt about $2$-subgroups with a strongly closed abelian subgroup. The determination
of $C_G(t)$ seems to be much more difficult than similar recognition results (in Chapter \ref{Chapter-O8Plus2} for example). A reason for this may
be that the $3$-rank of $C_G(t)/O_2(C_G(t))$ is just two whilst the $2$-rank is four. An
easier example may have greater $3$-rank and lesser $2$-rank.

One conjugacy class of involution centralizer is not enough to recognize $\HN$ and so in Section
\ref{HN-Section-CG(u)} we prove that $G$ also has an involution centralizer which has shape $(2^.\HS):2$  by making use of a theorem of Aschbacher. The results of Sections
\ref{HN-Section-CG(t)} and \ref{HN-Section-CG(u)} allow us to apply the uniqueness theorem by Segev
to prove that $G\cong \HN$.

It is hoped that the methods used in this chapter can soon be extended to recognize the almost
simple group $\aut(\HN)$ in a similar way.

\section{Determining the 3-Local Structure of $G$}\label{HN-Section-3Local}


We begin by recalling  Theorem \ref{A general 3^1+4 theorem} from Chapter \ref{chaper general hypothesis}  which
concerns groups which satisfy a more general hypothesis  than Hypothesis C.
Of course the conclusions of  Theorem \ref{A general 3^1+4 theorem} hold under  Hypothesis C. For
the rest of this chapter we work under Hypothesis C however we continue the notation from Theorem
\ref{A general 3^1+4 theorem}. In particular we fix a distinct conjugate of $Z$ in $Q$, $Z^x$ and
set $Y:=ZZ^x$, $L:=\<Q,Q^x\>$, $W:=C_L(Y)$, $S:=QW$, $J:=J(S)$ and $Z_2:=J \cap Q$. We continue to fix
an involution $s \in L$ such that $Ws \in \mathcal{Z}(L/W)$. Furthermore we now choose an involution $t$ such
that $Qt\in \mathcal{Z}(C_G(Z)/Q)$ and since $s$ normalizes $C_G(Z)$, we are able to choose $t$ such that $s$
and $t$ commute. We also fix an element of order three, $z$, such that $Z=\<z\>$.

\begin{lemma}\label{HN-EasyLemma}
\begin{enumerate}[$(i)$]
 \item $S \in \syl_3(G)$ and $Z=\mathcal{Z}(S)$.
 \item $C_G(Z)/Q$ acts irreducibly on $Q/Z$.

 \item $C_Q(t)=Z=C_Q(f)$ for every element of order five $f \in C_G(Z)$.

 \item There exists a group $X$ such that $S<X<C_G(Z)$ with $X/Q\cong 2^.\alt(4)\cong \SL_2(3)$ and such that $X/Q$ has  no central
 chief factors on $Q/Z$.

 \item  $C_G(Y)=W$ and $L\<t\>=N_G(Y)$ with
$L\<t\>/W \cong \GL_2(3)$.
\end{enumerate}
\end{lemma}
\begin{proof}
$(i)$ It is clear that $C_G(Z)$ has Sylow $3$-subgroups of order $3^6$ and, by hypothesis, $Z$ is
central  in a Sylow $3$-subgroup of $G$. Also  $|Q|=|W|=3^5$ and $S=QW$ is a $3$-group with $Q \neq
W$. Thus $|S|=3^6$ and so $S \in \syl_3(C_G(Z))\subset \syl_3(G)$ with $\mathcal{Z}(S)=Z$.

$(ii)$ This is because $2^.\alt(5)$ has no non-trivial modules of dimension less than four over
$\GF(3)$. We can see this, for example, from the fact that $5 \nmid |\GL_3(3)|$.

$(iii)$ By Theorem \ref{A general 3^1+4 theorem} $(xiii)$, either $C_Q(t)=Z$ and $[Q,t]/Z=Q/Z$ or
$C_Q(t)\cong [Q,t]\cong 3_+^{1+2}$. However $[Q,t]/Z$ is a non-trivial $C_G(Z)/Q$-module and so
must equal $Q/Z$. Therefore $C_Q(t)=Z$.  Now, for $f \in C_G(Z)$ of order five, by coprime action,
$Q/Z=C_{Q/Z}(f) \times [Q/Z,f]$. Since $f$ acts fixed-point-freely on $[Q/Z,f]$,  $[Q/Z,f]^\#$ has
order a multiple of five. Therefore $Q/Z=[Q/Z,f]$ and so $C_Q(f)=Z$.

$(iv)$ Observe (using \cite[33.15, p170]{Aschbacher} for example) that a group of shape
$2^{.}\alt(5)$ is uniquely defined and has Sylow $2$-subgroups isomorphic to $Q_8$ with normalizer
isomorphic to $\SL_2(3)$. Thus we may fix $S<X<C_G(Z)$ such that $X/Q\cong \SL_2(3)$. There can be no central chief factor of $X$ on $Q/Z$ because $Qt\in \mathcal{Z}(X/Q)$ inverts $Q/Z$.

$(v)$ Since $Y \neq Z=\mathcal{Z}(S)$, we have that $W \in \syl_3(C_G(Y))$. Suppose that $C_G(Y)$ contains an
involution. Since Sylow $2$-subgroups of $C_G(Z)$ are quaternion of order 8, we have that $[Y,Qt]=1$
which is a contradiction since $Qt$ inverts $Q/Z$. Suppose $C_G(Y)$ contains an element of order
five. Then we again have a contradiction since any element of order five in $C_G(Z)/Q$ acts
fixed-point-freely on $Q/Z$. Thus $C_G(Y)$ is a $3$-group and so $C_G(Y)=W$. Now $N_G(Y)/W$ is
isomorphic to a subgroup of $\GL_2(3)$ and $\SL_2(3) \cong L/W\leq N_G(Y)/W$ so $N_G(Y)/W\cong
\SL_2(3)$ or $\GL_2(3)$. Observe that $t$ centralizes $Z$ whilst inverting $Y/Z$. Therefore
$Wt\notin L/W$ and so $N_G(Y)/W\cong \GL_2(3)$ and $N_G(Y)=L\<t\>$.
\end{proof}

For the rest of this section we fix a subgroup $X$ of $C_G(Z)$ such that $S<X<C_G(Z)$ and $X/Q\cong
\SL_2(3)$.

\begin{lemma}
$Q=\<Z_2^{X/Q}\>$ and $S/Q$ acts quadratically on $Q/Z$.
\end{lemma}
\begin{proof}
First observe that since $X/Q\cong \SL_2(3)$ and there is no central chief factor of $X/Q$ on
$Q/Z$, any  proper $X/Q$-submodule of $Q/Z$ is necessarily a natural $X/Q$-module. Let $Z < V< Q$
such that $V/Z$ is an $X/Q$-submodule and is therefore a natural module. Thus $S/Q$ acts
non-trivially on $V/Z$. In particular this means $V\neq Z_2$. So $Z_2$ is not contained in any
proper $X$-invariant subgroup of $Q$. Thus $Q=\<Z_2^{X/Q}\>$.

By Theorem \ref{A general 3^1+4 theorem} $(v)$, $J=J(S)$ is abelian. Moreover $Q \leq S$ and so $Q$
normalizes  $J$. Thus $[Q,J]\leq J$ and so $[Q,J,J]=1$ as $J$ is abelian. Now $J \nleq Q$
(as $Q$ has no abelian subgroups of order $3^4$) and so $S/Q=JQ/Q$ and therefore
$[Q/Z,JQ/Q,JQ/Q]=1$.
\end{proof}

We have thus satisfied the conditions of Lemma \ref{Parker-Rowley-SL2(q)-splitting} and so we have
the following results.
\begin{lemma}\label{Lemma N1-N4}
\begin{enumerate}[$(i)$]
\item $Q/Z$ is a direct product of natural $X$-modules.
 \item There are exactly four $X$-invariant
subgroups $N_1,N_2,N_3,N_4< Q$ properly containing $Z$ such that for $i \neq j$, $N_i \cap N_j=Z$.
 \item $N_i
\cap Z_2$ has order nine for each $i$ and $S'=Z_2=\<N_i \cap Z_2|1\leq i \leq 4\>$.
 \item For some $i \in \{1,2,3,4\}$, $Y \leq N_i$ and $N_i$ is abelian.
 \item For $i \in \{1,2,3,4\}$, $X$ is transitive on $N_i\bs Z$.
\end{enumerate}
\end{lemma}
\begin{proof}
Part $(i)$ follows immediately from Lemma \ref{Parker-Rowley-SL2(q)-splitting} which says that
$Q/Z$ is a direct product of natural $X/Q$-modules. Let $N_1$ and $N_2$ be the corresponding
subgroups of $Q$. View $N_1/Z$ and $N_2/Z$ as vector spaces over $\GF(3)$. Since $N_1/Z$ and
$N_2/Z$ are isomorphic as $X$-modules, we may apply Lemma
\ref{lemma making new modules} to see that there are exactly four  $X$-invariant subgroups of
$Q/Z$. Let $N_3$ and $N_4$ be the corresponding  normal subgroups of $Q$. Then  $N_3/Z$ and $N_4/Z$
are natural $X$-modules and for $i \neq j$, $N_i \cap N_j=Z$. This proves $(ii)$.

By Theorem \ref{A general 3^1+4 theorem} $(vi)$, $Z_2/Z=\mathcal{Z}(S/Z)$ and $Y \leq Z_2=J \cap Q$. Now
for each $i\in \{1,2,3,4\}$,  $C_{N_i/Z}(S) \neq 1$ and so $Z_2 \cap N_i$ has order at least nine.
In fact the order must be exactly nine for were it greater then for some $i$, $N_i=Z_2$ and then
$N_i \cap N_j$ would have order at least nine for each $j \neq i$. Now for each  $i \neq j$, $N_i
\cap N_j=Z$ and so $N_i \cap Z_2 \neq N_j \cap Z_2$ and so $Z_2=\<N_i \cap Z_2|1\leq i \leq 4\>$.
In particular we must have (without loss of generality) that $N_1\cap J=Y$. By Lemma \ref{A general 3^1+4 theorem}, $Y\leq S' \leq Z_2$. Suppose $S'=Y$. Then for any $2\leq i \leq 4$ $Y \nleq N_i$
and so $[N_i,S] \leq N_i \cap Y =Z$. Therefore $N_i \leq Z_2$ which is a contradiction. Thus
$Y<S'=Z_2$ which proves $(iii)$.

We  have already that (without loss of generality) $N_1\cap Z_2=Y$. Suppose that $N_1$ is
non-abelian. Then $C_Q(N_1)\cong N_1\cong 3_+^{1+2}$.  Since $N_1 \nleq W$, $S=WN_1$
and so  we have that $S'\leq [W,N_1] W' N_1' \leq (W
\cap N_1)Y Z=Y$ (using Theorem \ref{A general 3^1+4 theorem} $(iv)$)  which is a contradiction since $S'=Z_2$. This proves $(iv)$.

Finally, since each $N_i/Z$ is a natural $X/Q$-module, $X$ is transitive on the non-identity elements
of $N_i/Z$. So let $Z \neq Zn \in N_i/Z$. Then $\<Z,n\>\vartriangleleft Q$ however $|C_Q(n)|=3^4$.
Therefore $n$ lies in a $Q$-orbit of length three in $Zn$. Hence every element in $Zn$ is
conjugate in $X$. Thus $X$ is transitive on $N_i \bs Z$ which completes the proof.
\end{proof}

For the rest of this section we continue the notation from Lemma \ref{Lemma N1-N4} with
$N_1,N_2,N_3,N_4$  chosen such that $Y<N_1$ and satisfying the notation set in the following lemma
also.

\begin{lemma}\label{N_1, N_2 abelian, N_3, N_4 not}
Without loss of generality we may assume that $N_1\cong N_2$ is elementary abelian and $N_3\cong N_4$ is extraspecial with $[N_3,N_4]=1$.
\end{lemma}
\begin{proof}
By Lemma \ref{Lemma N1-N4}, $N_1$ is abelian. So suppose $N_i$ is non-abelian for some $i\in
\{2,3,4\}$.  Then $C_Q(N_i)\cong N_i\cong 3_+^{1+2}$ is $X$-invariant and we may assume
$C_Q(N_i)=N_j$ for some $i \neq j\in \{2,3,4\}$. Now it follows that either $N_i$ is abelian for
every $i=1,2,3,4$ or without loss of generality $N_1\cong N_2$  and $N_3\cong N_4$ is non-abelian.
So we assume for a contradiction that $N_2$, $N_3$ and $N_4$ are all abelian.

Since $N_1/Z$ is isomorphic as a $\GF(3)X/Q$-module to $N_2/Z$, for any $m \in N_1\bs Z$ there
is an $n \in N_2 \bs Z$ such that $Zn$ is the image of $Zm$ under a module isomorphism. It then follows
(without loss of generality) that $Znm$ is an element of $N_3/Z$ and $Zn^2m$ is an element of
$N_4/Z$. In particular $x_1:=nm \in N_3$ and $x_2:=n^2m \in N_4$. Let $g \in X$ have order
four then $Qg^2=Qt$ inverts $Q/Z$ and so
\begin{equation}\label{one}Zn^{g^2}=Zn^2 ~\mathrm{and}~ Zm^{g^2}=Zm^2.\end{equation}
Also if $Z\neq Za\in N_i/Z$ and $g$ and $h$ are elements of order four in $X$ such that
$Q\<g\>\neq Q\<h\>$ then $N_i/Z=\<Za^g,Za^h\>$ and so $N_i=Z\<a^g,a^h\>$.

So consider $[x_1,x_2^g]$. We calculate the following using commutator relations and using that all
commutators are in $Z$ and therefore central.

\[\begin{array}{rclr}
    [x_1,x_2^g]&=&[nm,{(n^2)}^g
    m^g]&\:\\
    \;&=&[n,m^g][m,m^g][n,{(n^2)}^g][m,{(n^2)}^g]&\;\\
    \;&=&[n,m^g][m,{(n^2)}^g]& (\mathrm{since}~N_1~\mathrm{and}~N_2~ \mathrm{are~abelian})\\
    \;&=&[n,m^g][m,{n}^g]^2&\;\\
    \;&=&([n,m^g][m,{n}^g]^2)^g&(\mathrm{since~commutators~are~central~in~}X)\\
    \;&=&[n^g,m^{g^2}][m^g,n^{g^2}]^2&\;\\
    \;&=&[n^g,m^2][m^g,n^2]^2&(\mathrm{by~Equation~}\ref{one})\\
    \;&=&[n^g,m]^2[m^g,n]&\;\\
    \;&=&[{(n^2)}^g,m][m^g,n]&\;\\
   \;&=&[{(n^2)}^g,m][m^g,m][{(n^2)}^g,n][m^g,n]&(\mathrm{since}~N_1~\mathrm{and}~N_2~ \mathrm{are~abelian})\\
    \;&=&[{(n^2)}^g m^g,nm]&\;\\
    \;&=&[x_2^g,x_1].&\;\\
\end{array}\]Thus $[x_1,x_2^g]=[x_1,x_2^g]\inv$ and so
$[x_1,x_2^g]=1$. This holds for any element of order four in $X$. Thus $mn\in N_3$
commutes with $N_4=Z\<(n^2m)^g,(n^2m)^h\>$ where $g$ and $h$ are elements of order four as above.
Furthermore this argument works for any element of $N_3\bs Z$ and so $[N_3,N_4]=1$. However this
contradicts our assumption that $N_3$ and $N_4$ are abelian.
\end{proof}

\begin{lemma}\label{HN-a new class of three}
For $i\in \{3,4\}$, elements in $N_i\bs Z$ are not conjugate into $Z$. In particular,  there are 12
elements of order three in $Z_2$ which are not $G$-conjugate into $Z$.
\end{lemma}
\begin{proof}
Let $\{i,j\}=\{3,4\}$ and let $a\in (N_i \cap Z_2) \bs Z$. By Lemma \ref{Lemma N1-N4} $(v)$, every element in
$N_i \bs Z$ is conjugate to $a$. Suppose that $a \in Z^G$. Then we may again apply Lemmas \ref{A
general 3^1+4 theorem} and \ref{HN-EasyLemma} with $\<a\>$ in place of $Z^x$ to see that
$|C_G(\<a,Z\>)|=3^5$. Moreover $C_S(a)\geq \<N_j,J\>$ and $\<N_j,J\>$ has order $3^5$ so
$C_G(\<a,Z\>)=\<N_j,J\>$. Furthermore $S=QC_G(\<a,Z\>)$ and $Q<S<X$ so we may also apply Lemma
\ref{Lemma N1-N4} $(iv)$ to say that for some $k\in \{1,2,3,4\}$, $a \in N_k$ and $N_k$ is abelian.
By Lemma \ref{N_1, N_2 abelian, N_3, N_4 not}, $k\in \{1,2\}$. Therefore $a \in N_i \cap N_k$ and
$i \neq k$. This implies that $a \in Z$ which is a contradiction. Thus $a$ is not conjugate into
$Z$ and therefore no element in $N_i \bs Z$ is conjugate into $Z$.

Furthermore,
by Lemma \ref{Lemma N1-N4} $(iii)$, we see that $Z_2$ contains twelve  elements of order three
which are not conjugate into $Z$. These are contained in $N_3 \cap Z_2$ and $N_4 \cap Z_2$.
\end{proof}

\begin{lemma}\label{C-i's-orders and derived subgroups}
\begin{enumerate}[$(i)$]
\item Let $i \in \{1,2,3,4\}$ and set $S_i:=C_S(Z_2 \cap N_i)$ then $|S_i|=3^5$ and $|\mathcal{Z}(S_i)|=9$.
\item $S_1'=\mathcal{Z}(S_1)=Z_2 \cap N_1=Y$, $S_2'=\mathcal{Z}(S_2)=Z_2 \cap N_2$, $S_3'=\mathcal{Z}(S_4)=Z_2
\cap N_4$ and $S_4'=\mathcal{Z}(S_3)=Z_2 \cap N_3$.
\end{enumerate} In
particular $S_i \neq S_j$ for each $i \neq j$.
\end{lemma}
\begin{proof}
By Lemma \ref{Lemma N1-N4}, $|Z_2 \cap N_i|=9$ for each $i\in\{1,2,3,4\}$ and by Lemma  \ref{A
general 3^1+4 theorem}, $Z_2 \leq J$ and $J$ is elementary abelian of order $81$. Therefore $J
\leq S_i$. Hence $S_i\geq \<J,C_Q(Z_2 \cap N_i)\>$. Since $C_Q(Z_2 \cap N_i)$ has order $3^4$ and
is non-abelian, $|\<J,C_Q(Z_2 \cap N_i)\>|\geq 3^5$. Moreover since $|S|=3^6$ and
$\mathcal{Z}(S)=Z$ has order three, it follows that $S_i=\<J,C_Q(Z_2 \cap N_i)\>$ has order $3^5$.
Now by Theorem \ref{A general 3^1+4 theorem} $(viii)$, for each $i \in \{1,2,3,4\}$, $|\mathcal{Z}(S_i)|=9$
and therefore $\mathcal{Z}(S_i)=N_i \cap Z_2$.

Now for $i \in \{1,2,3,4\}$, we have that $Z \leq S_i'$. If $S_i'=Z$ then $Q/Z$ and $S_i/Z$ are two
distinct abelian subgroups of $S/Z$ of index three. This implies that $S/Z$ has centre of order at
least $3^3$. However by Theorem \ref{A general 3^1+4 theorem} $(vi)$, $Z_2/Z=\mathcal{Z}(S/Z)$ has order
nine. Thus $S_i'>Z$. Now for $i=1$, by Lemma \ref{Lemma N1-N4}, $Y \leq N_1\cap Z_2$ and so
$\mathcal{Z}(S_1)=N_1 \cap Z_2=Y$. Furthermore, for $i \in \{1,2\}$, $N_i$ is abelian and so $N_i
\leq S_i$. Therefore $S_i' \leq S' \cap N_i= Z_2 \cap N_i$ since $N_i\vartriangleleft S_i$. For
$\{i,j\} = \{3,4\}$, $[N_i,N_j]=1$ and so $N_j \leq S_i$. Therefore $S_i' \leq S' \cap N_j= Z_2
\cap N_j$ since $N_j\vartriangleleft S_i$.
\end{proof}

Continue notation such that $S_i=C_S(N_i \cap Z_2)$.
\begin{lemma}\label{elements of order nine in S}
Every element of order three in $S$ lies in the set $Q \cup S_1 \cup S_2$ and the cube of every
element of order nine in $S$ is in $Z$.
\end{lemma}
\begin{proof}
By hypothesis, $Q$ has exponent three and by Theorem \ref{A general 3^1+4 theorem} $(v)$, so does $J$.
So let $ g\in S $  such that $g \notin Q \cup J$. Then $g=cb$ for some $ c \in Q \bs J$ and some $b \in J \bs Q$. We calculate using the equality $c[b,c][b,c,c]=[b,c]c$ and using that $b \in J$ so
commutes with all commutators in $S'=Z_2 \leq J$.\begin{eqnarray*}
cbcbcb&=&c^2b[b,c]bcb \\
&=&c^2b^2[b,c]cb \\
&=&c^2b^2c[b,c][b,c,c]b \\
&=&c^2b^2cb[b,c][b,c,c] \\
&=&[c,b][b,c][b,c,c] \\
&=&[b,c,c].
\end{eqnarray*}

Since $c \in Q\bs J=Q \bs Z_2$, $Z_2\<c\>$ is a proper subgroup of $Q$ properly containing $Z_2$.
As $Z_2 \cap N_i$ has order nine for each $i\in \{1,2,3,4\}$, $Z_2N_i$ has order $81$. Thus $Z_2\<c\>=Z_2N_i$ for
some $i\in \{1,2,3,4\}$.

If $Z_2\<c\>=Z_2N_1=C_Q(Y)$ then $cb\in W$ and $W$ has exponent three.

Suppose $Z_2\<c\>=Z_2N_2$. Then $S_2=C_S(Z_2 \cap N_2)=J\<c\>$ and $S_2'=Z_2 \cap N_2$ therefore
$[b,c]\in Z_2 \cap N_2$ is central in $Z_2\<c\>=Z_2N_2$. Therefore $[b,c,c]=1$ and so $cb$ has
order three.

Now suppose $Z_2\<c\>=Z_2N_3$ (and a similar argument holds if $Z_2\<c\>=Z_2N_4$). Then
$S_4=C_S(Z_2 \cap N_4)=J\<c\>$ and $[b,c] \in S_4'=Z_2 \cap N_3$. Suppose $cbcbcb=[b,c,c]=1$. Then
$[b,c]$ commutes with $J\<c\>=S_4$ and so $[b,c] \in S_4' \cap \mathcal{Z}(S_4)=Z$. Thus
$S_4=J\<c\>=Z_2\<b,c\>$ and so $[S_4,S_4]=\<[Z_2,c],[Z_2,b],[c,b]\>$. However $[Z_2,c]\leq Z$,
$[Z_2,b]=1$  and $[c,b]\in Z$ which is a contradiction since $[S_4,S_4]=N_3\cap Z_2>Z$. Thus
$[b,c,c]\neq 1$ and $cb$ has order nine (no element can have order $27$ since $Q$ has exponent three).
Furthermore, $(cb)^3=[b,c,c] \in [Z_2 \cap N_4,c] \leq [Q,Q]=Z$ and so the cube of every such
element of order nine is in $Z$.
\end{proof}

\begin{lemma}\label{prelims-Centre of the C_i's}
For each $i\in\{3,4\}$, if $a \in \mathcal{Z}(S_i) \bs Z$ then $\mathcal{Z}(S_i/\<a\>)=\mathcal{Z}(S_i)/\<a\>$.
\end{lemma}
\begin{proof}
Let $\{i,j\}= \{3,4\}$ then by Lemma
\ref{C-i's-orders and derived subgroups}, we have that $S_i'=\mathcal{Z}(S_j)$ and $S_j'=\mathcal{Z}(S_i)$. So let $a \in \mathcal{Z}(S_i) \bs Z$ and suppose
$\mathcal{Z}(S_i/\<a\>)>\mathcal{Z}(S_i)/\<a\>$. Let $V\leq S_i$ such that $a \in V$ and
$\mathcal{Z}(S_i/\<a\>)=V/\<a\>$ then $|V|\geq 3^3$.  Therefore $S_i/V$ is abelian and so $S_i'\leq
V$. Therefore $[S_i',S_i] \leq \<a\>$. However $S_i$ normalizes $\mathcal{Z}(S_j)=S_i'$ and so
$[S_i',S_i] \leq \<a\>\cap S_i'= \<a\>\cap \mathcal{Z}(S_j)=1$ since $\mathcal{Z}(S_i) \cap
\mathcal{Z}(S_j)\leq N_i \cap N_j=Z$. However this implies that $S_i' \leq \mathcal{Z}(S_i)$ and so $N_j \cap
Z_2\leq N_i \cap Z_2$ which is a contradiction. Therefore
$\mathcal{Z}(S_i/\<a\>)=\mathcal{Z}(S_i)/\<a\>$.
\end{proof}

We fix an element of order three $a$ in $Q$ such that $a \in (N_3 \cap Z_2) \bs Z$ and therefore $a
\notin Z^G$ by Lemma \ref{HN-a new class of three}. Let $3\mathcal{A}:=\{a^g|g \in G\}$ and
$3\mathcal{B}:=\{z^g|g \in G\}$. We show in the rest of this section that these are the only
conjugacy classes of elements of order three in $G$.

\begin{lemma}\label{HN-prelims1}
$|C_S(a)|=3^5$, $|a^{G} \cap
 Q|=|a^{C_G(Z)} \cap Q|=120$ and $|z^G \cap Q|=|z^{xC_G(Z)} \cap
 Q|+2=122$. In particular, $Q^\# \subset 3\mathcal{A} \cup 3\mathcal{B}$.
\end{lemma}
\begin{proof}

%

We have chosen $a \in N_3\cap Z_2$ and so by Lemma \ref{C-i's-orders and derived subgroups},
$C_S(a)=C_S(\<Z,a\>)=C_S(N_3 \cap Z_2)=S_3$ which has order $3^5$. Now let $q \in Q\bs Z$ and
consider $[{C_G(Z)}:C_{C_G(Z)}(q)]$. By Lemma \ref{HN-EasyLemma} $(iii)$, an element of order five
acts fixed-point-freely on $Q/Z$ so we have that $5\mid [{C_G(Z)}:C_{C_G(Z)}(q)]$. Suppose $2\mid
|C_{C_G(Z)}(q)|$. Then there exists an involution $t_0\in C_{C_G(Z)}(q)$ and necessarily $Qt_0=Qt$
(since ${C_G(Z)}$ has Sylow $2$-subgroups which are quaternion of order eight). However this
implies that $q\in C_Q(t_0)=C_Q(t)=Z$ (by Lemma \ref{HN-EasyLemma} $(iii)$) which is a
contradiction. So $8\mid [{C_G(Z)}:C_{C_G(Z)}(q)]$. Furthermore $q$ is not $3$-central in $C_G(Z)$
and so $3\mid [{C_G(Z)}:C_{C_G(Z)}(q)]$. Therefore $[{C_G(Z)}:C_{C_G(Z)}(q)]$ is a multiple of 120.
Now there exists $z^x \in Q \bs Z$  which lies in a ${C_G(Z)}$-orbit in $Q$ of length at least 120
and also there exists $a \in Q$ which is not conjugate to $z$ and lies in a ${C_G(Z)}$-orbit in $Q$
of length at least 120. Since $a$ is not conjugate to $z^x$, these orbits are distinct. Thus
$|a^{G} \cap
 Q|=|a^{C_G(Z)} \cap Q|=120$ and $|z^G \cap Q|=|z^{xC_G(Z)} \cap Q|+2=122$.
\end{proof}


\begin{lemma}\label{HN-prelim-element of order four normalizing S}
\begin{enumerate}[$(i)$]
\item  $|C_J(t)|=3^2$ and $t$ inverts $S/J$.
 \item $|N_G(S) \cap C_G(Z)|=3^62^2$ and $|N_G(S)|=3^62^3$.
 \item There exists an element of order four $e \in N_G(S) \cap C_G(Z)$ such that $e^2=t$ and $e$ does not
 normalize $Y$.
\end{enumerate}
\end{lemma}
\begin{proof}
Using Theorem \ref{A general 3^1+4 theorem} $(xiii)$, $|C_J(t)|=3^2$ and $t$ inverts $S/J$. This  proves
$(i)$.

Now, $C_G(Z)/Q\cong 2^.\alt(5)$ and the normalizer of a Sylow $3$-subgroup in $2^.\alt(5)$ has
order $2^23$ with a cyclic Sylow $2$-subgroup. Thus $|N_G(S) \cap C_G(Z)|=3^62^2$ and since $s$ inverts $Z$, $|N_G(S) \cap N_G(Z)|=3^62^3$. Furthermore, we may choose
an element of order four $e \in C_G(Z)$ that squares to $t$ and normalizes $S$.  Suppose $e$
normalizes $Y$. Then $e^2=t$ centralizes $Y$ which is impossible. This completes the proof of
$(ii)$ and $(iii)$.
\end{proof}

\begin{lemma}\label{HN-prelims2}
\begin{enumerate}[$(i)$]
 \item  $J^\# \subseteq 3\mathcal{A} \cup 3\mathcal{B}$.
 \item $N_2^\# \subseteq 3\mathcal{B}$ and $C_W(s)^\# \subseteq 3\mathcal{A}$.
 \item Every element of order three in $S$ is in the set $3\mathcal{A}\cup 3\mathcal{B}$.

 \item For every $q \in Q$ there exists $P \in \syl_3(C_G(Z))$ such that
$q \in J(P)$.
 \item No non-trivial $3'$-subgroup of $G$ is normalized by $Y$.
\end{enumerate}
\end{lemma}
\begin{proof}
$(i)$ We have that $J/Y$ is a natural $L/W$-module and so there are four $L$-images of $Z_2$ in $J$
intersecting at $Y$. By Lemma \ref{HN-prelims1}, $Q^\# \subseteq 3\mathcal{A} \cup 3\mathcal{B}$.
Therefore $Z_2^\# \subseteq 3\mathcal{A} \cup 3\mathcal{B}$ which implies that $J^\# \subseteq
3\mathcal{A} \cup 3\mathcal{B}$.

$(ii)$ We have that for $i \in \{1,2,3,4\}$, by Lemma \ref{Lemma N1-N4} $(v)$, $X$ is transitive on
$N_i\bs Z$ and so either $N_i\bs Z\subseteq 3\mathcal{A}$ or $N_i\bs Z \subseteq 3\mathcal{B}$. By
Lemma \ref{HN-prelim-element of order four normalizing S} $(iii)$, there exists $e\in N_G(S)$ such
that $Y^e\neq Y$. Therefore $Y^e=N_i \cap Z_2$ for some $i \in \{2,3,4\}$. We have that  $N_i\bs
Z \subseteq 3\mathcal{A}$ for $i=3,4$ and so $Y^e=N_2 \cap Z_2$. Thus $N_2^\# \subseteq 3
\mathcal{B}$. Now there are five conjugates of $X$ in $C_G(Z)$ and therefore five images of $N_1$
and of $N_2$ in $C_G(Z)$ (since if $N_i$ was normal in two distinct conjugates of $X$ then $N_i$
would be normal in $C_G(Z)$). For each $i\in \{1,2\}$, $N_i\bs Z$ contains $24$ conjugates of $z$.
Since $Q\bs Z$ contains 120 conjugates of $Z$, there exists $i \in \{1,2\}$ and $g \in C_G(Z)$ such
that $Y\leq N_i^g\vartriangleleft X^g$ and $N_i^g\neq N_1$.  Now consider $C_Q(Y)$ which is
normalized by $s$ (as $s$ normalizes $Q$ and $Y$). By Theorem \ref{A general 3^1+4 theorem} $(xi)$, $3
\cong C_W(s) \leq Q \cap Q^x$. Therefore $|C_{C_Q(Y)}(s)|=3$. Now there are four proper subgroups
of $C_Q(Y)$ properly containing $Y$. These include $Q \cap Q^x$, $Z_2$, $N_1$ and $N_i^g$. We have
that $s$ normalizes at least two subgroups: $Z_2=S'\neq Q \cap Q^x$. Suppose that $s$ normalizes
$N_1$ and $N_i^g$. If $s$ inverts $N_1$ then $N_1\leq [W,s] \cap Q=J \cap Q=Z_2$ which is a
contradiction (as $|N_1 \cap Z_2|=9$). Therefore $N_1=YC_{W}(s)=Q \cap Q^x$ and by the same
argument $N_i^g =Q \cap Q^x$ which is a contradiction since $N_i^g \neq N_1$. Therefore at least
one of $N_1$ and $N_i^g$ is not normalized by $s$. We assume that $N_1^s \neq N_1$ (and the same
argument works if $N_i^{gs} \neq N_i^g$). Now consider $|C_Q(Y) \cap 3\mathcal{A}|$. Since $Q/N_1$
is a natural $X/Q$-module, there are four $X$-conjugates of $C_Q(Y)$ in $Q$ intersecting at $N_1$.
Each must contain exactly 120/4=30 conjugates of $a$. Thus $|C_Q(Y) \cap 3 \mathcal{A}|=30$.
Clearly $N_1\cap 3 \mathcal{A}=N_1^s \cap 3\mathcal{A}=\emptyset$ and $|Z_2\cap 3 \mathcal{A}|=12$
by Lemma \ref{HN-a new class of three}. Therefore we have $|Q \cap Q^x \cap 3\mathcal{A}|=18$. In
particular this implies $C_W(s)^\# \subseteq 3 \mathcal{A}$.

$(iii)$ By Lemma \ref{elements of order nine in S}, every element of order three in $S$ lies in $Q
\cup C_S(N_1 \cap Z_2) \cup C_S(N_2 \cup Z_2)$ and the cube of every element of order nine is in $Z$.
Since $N_1 \cap Z_2=Y$, $C_S(N_1 \cap Z_2)=W$ and since $N_2^\# \subseteq 3\mathcal{B}$ and $C_G(Z)$
is transitive on  $Q \cap 3\mathcal{B} \bs Z$,  $N_1 \cap Z_2$ is conjugate in $C_G(Z)$ to $N_2
\cap Z_2$. Therefore $S_2=C_S(N_2 \cap Z_2)$ is conjugate to $W$. Now, by Lemma \ref{A general
3^1+4 theorem} $(ix)$, $W/(Q \cap Q^x)$ is a natural $L/W$-module and so there are four $L$-conjugates of
$C_Q(Y)$ in $W$ and this accounts for every element of $W$. Since $C_Q(Y)^\# \subseteq Q^\# \subseteq
3\mathcal{A} \cup 3 \mathcal{B}$, $W^\# \subseteq 3\mathcal{A} \cup 3 \mathcal{B}$ and therefore
every element of order three in $S$ is in $3\mathcal{A} \cup 3 \mathcal{B}$.

$(iv)$ Since $z^x,a \in Z_2\leq J(S)=J$ and every element in $Q\bs Z$ is ${C_G(Z)}$-conjugate
to one of these, every element in $Q$ lies in the Thompson subgroup of a Sylow $3$-subgroup of
${C_G(Z)}$.

$(v)$ By Theorem \ref{A general 3^1+4 theorem} $(xii)$, any $3'$-subgroup of $G$ normalized by $Y$
commutes with $Y$. However $C_G(Y)=W$ is a $3$-group.
\end{proof}

\begin{lemma}\label{HN-normalizer Z2}
$N_G(Z_2)=N_G(S)=N_G(J) \cap N_G(Z)$ and $C_G(Z_2)=J=C_G(J)$.
\end{lemma}
\begin{proof}
We clearly have, $N_G(S) \leq N_G(\mathcal{Z}(S)) \cap N_G(J(S))=N_G(Z) \cap N_G(J)$. However
$N_G(Z) \cap N_G(J)$ normalizes $Q=O_3(N_G(Z))$ and $J$ and therefore normalizes $S=QJ$ and so we
have $N_G(S) =N_G(Z) \cap N_G(J)$. By Lemma \ref{Lemma N1-N4}, $|Z_2 \cap N_i|=9$ for each $i \in
\{1,2,3,4\}$. Also by Lemma \ref{Lemma N1-N4},  $Z_2=\<N_i \cap Z_2|1 \leq i \leq 4\>$ and no element of
$N_3 \bs Z$ or $N_4 \bs Z$ is conjugate to $z$ by Lemma \ref{HN-a new class of three}. Lemma \ref{HN-prelims2} $(ii)$ says that $N_2^\# \subseteq
3\mathcal{B}$ and so  $Z_2 \cap 3 \mathcal{B}=(N_1 \cap Z_2)^\# \cup (N_2 \cap Z_2)^\#$. Therefore
$N_G(Z_2)$ preserves this set and therefore also preserves the set $(N_1 \cap Z_2) \cap (N_2 \cap
Z_2)=Z$. Hence $N_G(Z_2) \leq N_G(Z)$. Since $J$ is abelian, $J \leq C_G(J) \leq C_G(Z_2) \leq C_G(Y)=W$ and
since $\mathcal{Z}(W)=Y$ (by Theorem \ref{A general 3^1+4 theorem} $(iv)$), $J=C_G(Z_2)=C_G(J)$. Therefore $N_G(Z_2) \leq  N_G(J)$ and so $N_G(Z_2)
\leq N_G(Z) \cap N_G(J)$. Clearly $N_G(S) \leq N_G([S,S])=N_G(Z_2)$ which gives us that
$N_G(Z_2)=N_G(Z) \cap N_G(J)$ therefore completing the proof.
\end{proof}

\begin{lemma}\label{HN-Normalizer of Z structure}
$N_G(Z)/Q\cong 4^{.}\alt(5)\cong 4 \ast \SL_2(5)$.
\end{lemma}
\begin{proof}
We have chosen an involution $s\in N_G(Z)\bs C_G(Z)$ such that $Ws \in \mathcal{Z}(N_G(Y)/W)$. Observe that
$C_G(Z)$ has ten Sylow $3$-subgroups and $s$ normalizes one of these, namely $S$. Clearly, $s$
must normalize at least one further Sylow $3$-subgroup of $C_G(Z)$. Let $S \neq R \in
\syl_3(C_G(Z))$ be normalized by $s$. We have that $s$ inverts $J$ and centralizes $S/J$, therefore
$9=|C_{S/J}(s)|=|C_S(s)J/J|=|C_S(s)/C_J(s)|=|C_S(s)|$. Notice that $s$ inverts $S/Q=QJ/Q\cong J/(Q
\cap J)$ and so $C_S(s) \leq Q$ and therefore $|C_Q(s)|=9$. By coprime action, we have
$Q/Z=C_{Q/Z}(s)\times [Q/Z,s]$ and $Z_2/Z \leq [Q/Z,s]$. Since $C_Q(s)\cong C_{Q/Z}(s)$ has order
nine, $Z_2/Z=[Q/Z,s]$. Suppose that $[R/Q,s]=1$. Then $R/Q$ normalizes $[Q/Z,s]=Z_2/Z$. Therefore
$Z_2\vartriangleleft \<S,R\>$ which contradicts Lemma \ref{HN-normalizer Z2} which says that
$N_G(Z_2) =N_G(S)\ngeq R$. Thus $s$ must invert $R/Q$.

Now we have that $N_G(Z)/C_G(Z)$ acts on $C_G(Z)/Q \cong 2^{.}\alt(5)$ so suppose this action is
non-trivial. Then $\<C_G(Z)/Q,Qs\> \sim 2^{.}\sym(5)$. There are two isomorphism types of group
with shape $2^{.}\sym(5)$. One of these has no involutions outside its $2$-residue which is clearly
not the case in $\<C_G(Z)/Q,Qs\>$ since $Qs$ is such an involution. The other has one class of
involutions outside its $2$-residue and these commute with a Sylow $3$-subgroup. However we have
seen that $Qs$ commutes with no Sylow $3$-subgroup of $C_G(Z)/Q$. Hence $\<C_G(Z)/Q,Qs\> \nsim
2^{.}\sym(5)$. Thus $\<C_G(Z)/Q,Qs\>$ has centre of order four. By Theorem 3.2.2 in
\cite[p64]{Gorenstein}, since $N_G(Z)/Q$ acts irreducibly on $Q/Z$, $\mathcal{Z}(N_G(Z)/Q)$ is
cyclic. Therefore $\<C_G(Z)/Q,Qs\>=N_G(Z)/Q \sim 4^{.}\alt(5)$.
\end{proof}

\begin{lemma}\label{HN-element of order four in CG(Z)}
Let $ A \in \syl_2(C_G(Z))$ such that $t \in A$ and suppose that $f \in A$ such that $f^2=t$. Then $Z \in \syl_3(C_G(f))\cap
\syl_3(C_G(A))$.
\end{lemma}
\begin{proof}
We have that $C_Q(A)=C_Q(f)=Z$ since $f^2=t$ and $C_Q(t)=Z$. By coprime action, \[4 \cong
C_{C_G(Z)/Q}(f)= C_{C_G(Z)}(f)Q/Q\cong C_{C_G(Z)}(f)/C_Q(f)\] and  \[2 \cong C_{C_G(Z)/Q}(A)=
C_{C_G(Z)}(A)Q/Q\cong C_{C_G(Z)}(A)/C_Q(A).\] Therefore $Z\in \syl_3(C_G(Z) \cap C_G(f))$ and $Z\in
\syl_3(C_G(Z) \cap C_G(A))$. Thus $Z \in \syl_3(C_G(f))\cap \syl_3(C_G(A))$.
\end{proof}

\begin{lemma}\label{HN-normalizer of J1}
$[N_G(J):C_{N_G(J)}(a)]=48$, $[N_G(J):C_{N_G(J)}(Z)]=32$ and $|N_G(J)|=3^62^7$.
\end{lemma}
\begin{proof}
By Theorem \ref{A general 3^1+4 theorem} $(ix)$, $J/Y$ is a natural $L/W$-module and so $J$ contains four
$L$-conjugates of $Z_2$ with pairwise intersection $Y$. By Lemma \ref{Lemma N1-N4}, $Z_2=\<N_i \cap
Z_2| 1 \leq i \leq 4\>$. Since the conjugates of $z$ lie in $N_1 \cup Z_2$ and $N_2 \cup Z_2$,
$|Z_2 \cap 3\mathcal{B}|=8+6=14$ and so $|J\cap 3\mathcal{B}|=8+(4*6)=32$. Therefore, by Lemma
\ref{conjugation in thompson subgroup},  $[N_G(J):C_{N_G(J)}(z)]=32$.  Now by Lemma
\ref{HN-prelim-element of order four normalizing S}, $|C_{N_G(J)}(z)|=3^62^2$ and so
$|N_G(J)|=3^62^7$. Since $J^\# \subseteq  3\mathcal{A} \cup  3\mathcal{B}$, $|J \cap
3\mathcal{A}|=48$ and so $[N_G(J):C_{N_G(J)}(a)]=48$.
\end{proof}

%
%
%
%

Recall that $L\leq N_G(Y)$ and $L/J\cong 3 \times \SL_2(3)$ using Theorem \ref{A general 3^1+4 theorem}. Recall also that a group $H$ is said to be $3$-soluble of length one $H/O_{3'}(H)$ has a normal Sylow $3$-subgroup which is to say that $H=O_{3',3,3'}(H)$.

\begin{lemma}\label{HN-normalier of J-order of the O2}
We have that $O_2(L/J) \leq O_2(N_G(J)/J) \cong 2_+^{1+4}$ and $N_G(J)/J$ is $3$-soluble of length one.
\end{lemma}
\begin{proof}
Set $K:=N_G(J)$ and $\bar{K}=K/J$. Then $\bar{K}$ has order $3^22^7$ and $\bar{S}\in
\syl_3(\bar{K})$. Consider $O_3(K)$. Recall using Theorem \ref{A general 3^1+4 theorem} that
$W=O_3(L)$. Therefore $O_3(K) \leq W$. If $O_3(K)=W$ then $K \leq N_G(W) \leq N_G(Y)$ (as
$\mathcal{Z}(W)=Y$)  and it follows from Lemma \ref{HN-EasyLemma} $(v)$, $|N_G(Y)|<|N_G(J)|$  so we have that $O_3(K)=J$.

By Burnside's $p^\a q^\b$-Theorem \cite[4.3.3, p131]{Gorenstein}, $\bar{K}$ is solvable. Let $N$ be
a subgroup of $K$ such that  $J\leq N$ and $\bar{N}=O_2(\bar{K})$. Then $\bar{N} \neq 1$ since
$\bar{K}$ is solvable and $O_3(\bar{K})=1$. Recall that $s$ inverts $J$ and so $\bar{s} \in
\mathcal{Z}(\bar{K})$, in particular, $\bar{s} \in \bar{N}$. Moreover $\bar{N}$ is the Fitting subgroup of $\bar{K}$, $F(\bar{K})$, and so
by \cite[6.5.8]{stellmacher} $C_{\bar{K}}(\bar{N})\leq \bar{N}$. If any element in $\bar{S}$
centralizes $\bar{N}/\Phi(\bar{N})$ then by Theorem \ref{Burnside-p'-automorphism}, such an element
centralizes $\bar{N}$ and so is the identity. Therefore $\bar{S}$ acts faithfully on
$\bar{N}/\Phi(\bar{N})$ and so by calculating the order of a Sylow $3$-subgroup in $\GL_n(2)$ for
$n=1,2,3$ we see that $|\bar{N}/\Phi(\bar{N})|\geq 2^4$. Moreover, since $\bar{s}$ is central in
$\bar{K}$, we have that $\bar{S}$ acts faithfully on $\bar{N}/\<\bar{s},\Phi(\bar{N})\>$ and so
$|\bar{N}| \geq 2^5$. We use Lemma \ref{HN-prelim-element of order four normalizing S}
$(iii)$ to find $e\in N_G(S)$ such that $e^2=t$ and $e$ does not normalize $Y$. Since $t$ inverts $\bar{S}$, by Lemma
\ref{HN-prelim-element of order four normalizing S} $(i)$, $\<\bar{e}\>\cap \bar{N}=1$. Thus
$|\bar{N}|\leq 2^5$ and so we have that $|\bar{N}|=2^5$ and furthermore that $\bar{K}=\bar{N}\bar{S}\<\bar{e}\>$ and so $\bar{K}$ is $3$-soluble of length one.

By Theorem \ref{A general 3^1+4 theorem}$(x)$, $W/J \trianglelefteq L/J \cong 3 \times \SL_2(3)$. It is
clear that $Q_8\cong O_2(\bar{L})\leq \bar{N}$. Since  $e$ normalizes $S$ but not $Y$,  we have
that $Q_8\cong O_2(\bar{L^e})\leq \bar{N}$ and $O_2(L) \neq O_2(L^e)$. Thus by setting $A=O_2(\bar{L})$ and $B=O_2(\bar{L^e})$
we may apply Lemma \ref{prelims-extraspecial 2^5 in GL(4,3)} to see that $\bar{N}\cong 2_+^{1+4}$ which completes the proof.
\end{proof}

Note that we may also use Lemma \ref{prelims-extraspecial 2^5 in GL(4,3)} to see that
$O_2(N_G(J)/J)$ is unique up to conjugation in $\GL_4(3)$. It therefore follows that $N_G(J)/J$ is
isomorphic to a subgroup of $N_{\GL_4(3)}(\GO^+_4(3))\sim \GO^+_4(3).2$.

\begin{lemma}\label{HN-normalizer of J2}
$C_S(s)=\<\alpha_1,\alpha_2\>\cong 3 \times 3$ where $\alpha_1,\alpha_2 \in 3\mathcal{A}$ and there exist
$\<\alpha_1,\alpha_2\>$-invariant subgroups  $Q_8 \cong X_i\leq C_G(s) \cap C_G(\alpha_i)$ for $i \in\{1,2\}$ such that $s \in X_i$
and $[X_1,X_2]=1$.
\end{lemma}
\begin{proof}
Consider $D:=C_{N_G(J)}(s)\cong C_{N_G(J)}(s)J/J=C_{N_G(J)/J}(s)=N_G(J)/J$. This is a group of
order $2^73^2$ in which $O_2(D)\cong 2_+^{1+4}$. Let $P:=C_S(s)$ then by Lemma \ref{HN-normalier of J-order of the O2},
$D=O_2(D) N_{D}(P)$. By Lemma \ref{HN-prelims2} $(ii)$, $C_W(s)^\# \subseteq 3 \mathcal{A}$ so let
$\<\alpha_1\>:=C_W(s)\leq P$. Recall that using Lemma \ref{HN-prelim-element of order four normalizing
S} $(iii)$ there is an element of order four $e\in C_G(Z)$ which normalizes $S$ but not $Y$ and so $W \neq
W^e=C_G(Y^e)\leq S$. Moreover by Lemma \ref{HN-normalizer of J2}, $N_G(S)$ has abelian Sylow $2$-subgroups and so we may assume that $[e,s]=1$ and so $e \in D$. Therefore, $\alpha_1^e=:\alpha_2\in C_{W^e}(s)$ and $P=\<\alpha_1,\alpha_2\>$. By Lemma \ref{A
general 3^1+4 theorem}$(x)$, $W/J \trianglelefteq L/J \cong 3 \times \SL_2(3)$. Therefore
$L/J=C_{L/J}(s)=C_L(s)J/J\cong C_L(s)\cong 3 \times \SL_2(3)$ and $C_W(s)\vartriangleleft C_L(s)$
implies that $\alpha_1$ is central in $C_L(s)$. It follows that $Q_8\cong X_1:=O_2(C_L(s))\leq
O_2(C_K(s))$. In the same way, $\alpha_2$ is central in $C_{L^e}(s)$ and $Q_8\cong
X_2:=O_2(C_{L^e}(s))\leq O_2(C_K(s))$.  Now we simply apply Lemma
\ref{exactly 2 q-8's in extraspecial group} to see that $[X_1,X_2]=1$.
\end{proof}

\begin{lemma}\label{HN-centralizer of a does not normalize J}
$C_G(a)\nleq N_G(J)$.
\end{lemma}
\begin{proof}
By Lemma \ref{HN-prelims2} $(ii)$ and $(iv)$,  there exists $b \in Q \cap Q^x \cap 3\mathcal{A}$
and there exists $R\in \syl_3(C_G(Z))$ such that $b \in J(R)$. The same lemma applied to $C_G(Z^x)$
says that there exists $P\in \syl_3(C_G(Z^x))$ such that $b \in J(P)$. If $Q \cap Q^x \leq J(R)$
then $Y \leq J(R)\leq C_G(Y)=W$. Hence $J(R)=J(W)=J$ (see Theorem \ref{A general 3^1+4 theorem} $(v)$) however
$Q \cap Q^x \nleq J$ (by Theorem \ref{A general 3^1+4 theorem} $(xi)$ since $1 \neq C_W(s)\leq Q \cap Q^x$ but
$J$ is inverted by $s$). Therefore $C_R(b)=J(R)(Q \cap Q^x)$ and similarly $C_P(b)=J(P)(Q \cap
Q^x)$.

Suppose $J(P)=J(R)$. Then $J(R)$ is normalized by $\<Q,Q^x\>=L$.  However $O_3(L)=W$ (see Theorem
\ref{A general 3^1+4 theorem}) and so $b \in J(R)=J(W)=J$ which is a contradiction and so $J(P)\neq
J(R)$. This implies that $C_G(b)$ has two distinct Sylow $3$-subgroups with distinct Thompson
subgroups. Since $a$ is conjugate to $b$, it follows that $C_G(a) \nleq N_G(J)$.
\end{proof}

\begin{lemma}\label{HN-finding z's in CQ(a)}
$C_Q(a) \bs J \nsubseteq 3\mathcal{A}$.
\end{lemma}
\begin{proof}
Recall that $a \in N_3 \bs Z$ and $[a,N_4]=1$ where $N_4\bs Z\subseteq 3 \mathcal{A}$. Furthermore,
$Q/N_4$ is a natural $X/Q$-module and so $X$ is transitive on the four proper subgroups of $Q$
properly containing $N_4$ of which $C_Q(a)$ is one of these. Thus $C_Q(a) \bs N_4$ contains
$(120-24)/4=24$ conjugates of $a$ and so $|C_Q(a) \cap 3\mathcal{A}|=24+24=48$. Thus $C_Q(a) \bs
J=C_Q(a) \bs Z_2 \nsubseteq 3\mathcal{A}$.
\end{proof}

In the following lemma we demonstrate the necessary hypotheses to allow us to apply Theorem A to
$C_G(a)/\<a\>$ to see that it is isomorphic to $\alt(9)$. Note that we aim to find a group of shape $3^3:\sym(4)$. In fact there are two
isomorphism types of groups with this shape and only one appears as a subgroup of $\alt(9)$. In
Chapter \ref{chapter-Alt9} we refer to such a group as a $3$-local subgroup of $\alt(9)$-type.
Recall Lemma \ref{Alt9 prelims 3cubedsym4} which allows us to recognize groups of this isomorphism
type.

\begin{lemma}
$C_G(a)\cong 3 \times \alt(9)$ and $N_G(\<a\>)$ is isomorphic to the diagonal subgroup of index two in $\sym(3)\times \sym(9)$.
\end{lemma}
\begin{proof}
Let $C_a:=C_G(a)$, $S_a:=C_S(a)\in \syl_3(C_a)$ and  $\bar{C_a}:=C_a/\<a\>$. Set $H_a:=N_G(J) \cap
C_G(a)$.

We gather the required hypotheses to apply Lemma \ref{Alt9 prelims 3cubedsym4} to $\bar{H_a}$.
Observe first that by Lemma \ref{HN-normalizer of J1}, $[N_G(J):H_a]=48$ and $|N_G(J)|=3^62^7$.
Therefore $[H_a:J]=24$ and so $|\bar{H_a}|=3^42^3$ has the required order.

By Lemma \ref{HN-prelims1}, $|Q \cap 3\mathcal{A}|=120$ and $C_G(Z)$ is transitive on the
set.  Therefore $N_G(Z)$ is also transitive on the set and so we have that
$[C_G(Z):C_{C_G(Z)}(a)]=[N_G(Z):C_{N_G(Z)}(a)]=120$. Hence $|C_{C_G(Z)}(a)|=3^5=|S_a|$ and
$|C_{N_G(Z)}(a)|=3^52$.

Now $\bar{J}\vartriangleleft \bar{H_a}$ and $\bar{J}$ is elementary abelian of order $27$. Consider $\mathcal{Z}(\bar{S_a})$. Since $a \in N_3 \cap Z_2$, we may apply  Lemma
\ref{prelims-Centre of the C_i's} to say that
$\mathcal{Z}(\bar{S_a})=\bar{\mathcal{Z}(S_a)}=\bar{Z}$ which has order three. Now the coset
$\<a\>z$ contains exactly one conjugate of $z$ and so if $h \in H_a$ and $\bar{h}$ centralizes
$\<a\>z=\bar{z}$ then $h$ centralizes $\<z,a\>$. However $C_G(z) \cap C_a=S_a$ and so we have that
$C_{\bar{H_a}}(\bar{Z})=\bar{S_a}$.

Finally, $|N_G(Z) \cap C_a|=3^52$ and so $N_G(Z) \cap C_a >C_G(Z) \cap C_a$ and so there exists an
involution $u \in C_a$ that inverts $Z$. Therefore $u$ normalizes $S_a$ and normalizes $J$ and so
$u \in H_a$. Recall that $Js$ inverts $J$ and so $\<Js,Ju\>$ is an elementary abelian subgroup of
order four and by coprime action, $J=C_J(u)C_J(us)$. Since $Jus$ centralizes $Z$,  $Jus$ is
conjugate to  $Jt$ by an element of $Q\leq N_G(J)$. Therefore $|C_J(us)|=|C_J(t)|=3^2$ (by Lemma
\ref{HN-prelim-element of order four normalizing S} $(i)$). Hence $|C_J(u)|\geq 3^2$.  Thus we see
that $\bar{u}\in \bar{H_a}$ normalizes $\bar{S_a}$ and $C_{\bar{J}}(\bar{u})\neq 1$ as required. So
by Lemma \ref{Alt9 prelims 3cubedsym4}, $\bar{H_a}$ is isomorphic to a $3$-local subgroup of
$\alt(9)$-type.

Before we may apply Theorem A we must show that for every $\bar{g}$ of order three in $\bar{S_a}$,
$O_{3'}(C_{\bar{C_a}}(\bar{g}))=1$. If $g \in J$ then this is clear since by Lemma
\ref{HN-prelims2} $(v)$, $J$ normalizes no non-trivial $3'$-subgroup of $G$. So we consider
elements of order three in $S_a\bs J$. Since $\bar{H_a}$ has one class of elements of order three
outside $\bar{J}$ (see Table \ref{char table H} for example), we may choose $g \in C_Q(a)$.
Furthermore, by Lemma \ref{HN-finding z's in CQ(a)}, there exists $h \in C_Q(a)\bs J$ such that $h
\in 3\mathcal{B}$ so we may assume that $g \in \<a,h\>$. Let $M \leq C_a$ such that $a \in M$ and
$\bar{M}=O_{3'}(C_{\bar{C_a}}(\bar{g}))$ and then set $N:=O_{3'}(M)$. Then $N$ is a $3'$-subgroup
of $C_a$ with $M=N\<a\>$ and $N$ is normalized by $C_Q(h) \cap C_Q(a)$. By Lemma
\ref{HN-prelims1}, $C_G(Z)$ is transitive on $(Q \cap 3\mathcal{B}) \bs Z$ and since $h \in 3\mathcal{B}$, $\<h,z\>$ is
$C_G(Z)$-conjugate to $Y$. By Lemma \ref{HN-prelims2} $(v)$, $Y$ normalizes no non-trivial
$3'$-subgroup of $G$. Therefore $\<h,z\>\leq C_Q(h) \cap C_Q(a)$ normalizes no non-trivial $3'$-subgroup of $G$. Thus
$N=1$ and so $O_{3'}(C_{\bar{C_a}}(\bar{g}))=1$.

Finally, we may apply Theorem A to say that either $C_a=H_a$ or $\bar{C_a}\cong
\alt(9)$. However Lemma \ref{HN-centralizer of a does not normalize J} says that $C_G(a)\nleq
N_G(J)$ and so we  conclude that $\bar{C_a}\cong \alt(9)$. Using  \cite[33.15, p170]{Aschbacher}, for example, we see
that the Schur Multiplier of $\bar{C_a}$ has order two. Therefore $C_a$ splits over $\<a\>$ and so
$C_a \cong 3 \times \alt(9)$.

To see the structure of the normalizer we need only observe that an involution $s$ inverts $J$ and
therefore inverts $Z$ whilst acting non-trivially on $O^3(C_G(a))$. Therefore since
$\aut(\alt(9))\cong \sym(9)$, the result follows.
\end{proof}

\begin{lemma}\label{HN-centralizer of the a9}
For $i \in \{1,2\}$, $\<a\>=C_G(O^3(C_G(a)))$.
\end{lemma}
\begin{proof}
Set $R:=O^3(C_G(a))\cong \alt(9)$, then $C_G(R) \cap N_G(\<a\>)=\<a\>$. Therefore $R$ has a
self-normalizing Sylow $3$-subgroup. By Burnside's normal $p$-complement Theorem (Theorem
\ref{Burnside-normal p complement}), $R$ has a normal $3$-complement to $\<a\>$, $N$ say. Since $Y
\leq C_G(a)$ and $C_G(a)$ clearly normalizes $N$, we have that $Y$ normalizes $N$ and so by Lemma
\ref{HN-prelims2} $(v)$, $N=1$.
\end{proof}

\section{The Structure of the Centralizer of $t$}\label{HN-Section-CG(t)}

We now have sufficient information concerning the $3$-local structure of $G$ to determine the
centralizer of $t$. We set $H:=C_G(t)$, $P:=C_J(t)$ and $\bar{H}:=H/\<t\>$. We will show that
$H\sim 2_+^{1+8}.(\alt(5) \wr 2)$ and so we must first show that $H$ has an extraspecial subgroup
of order $2^9$. We then show that $H$ has a subgroup, $K$, of the required shape and then finally
we apply a theorem of Goldschmidt to prove that $K=H$. Along the way we gather several results
which will be useful in Section \ref{HN-Section-CG(u)}.

\begin{lemma}\label{HN-conjugates in P}
$C_G(Z)\cap H \cong 3 \times  2^{.}\alt(5)$ and $N_G(Z)\cap H \cong 3 : 4^{.}\alt(5)$. Furthermore,
$|P \cap 3\mathcal{A}|=|P \cap 3\mathcal{B}|=4$, $C_H(P)=P\<t\>$ and $P \in \syl_3(H)$.
\end{lemma}
\begin{proof}
By coprime action and an isomorphism theorem, we have that \[2^{.}\alt(5)\cong C_{C_G(Z)/Q}(t)=
C_{C_G(Z)}(t)Q/Q\cong C_{C_G(Z)}(t)/C_Q(t)\] and
\[4^{.}\alt(5)\cong C_{N_G(Z)/Q}(t)= C_{N_G(Z)}(t)Q/Q\cong C_{C_N(Z)}(t)/C_Q(t).\] Since $C_Q(t)=Z$,
we have that $C_{C_G(Z)}(t)\sim  3 .2^{.}\alt(5)$ and $N_G(Z)\cap H \cong 3. 4^{.}\alt(5)$. By
Lemma \ref{HN-prelim-element of order four normalizing S}, $|P|=9$ and since $J$ is elementary abelian, $P$ splits
over $Z$. Thus $C_{C_G(Z)}(t)$ splits over $Z$ by Gasch\"{u}tz's Theorem (\ref{Gaschutz}) and so
$C_{C_G(Z)}(t)\cong 3 \times 2^{.}\alt(5)$ and $N_G(Z)\cap H \cong 3 :  4^{.}\alt(5)$.

Notice that $Y<PY<J$. Since $L/W\cong \SL_2(3)$ and $J/Y$ is a natural $L/W$-module, there exists $l
\in L$ such that $PY=Z_2^l$. By Lemma \ref{HN-normalizer Z2},  $N_G(Z_2) \leq N_G(Z)$. Thus
$N_G(Z_2^l)\leq N_G(Z^l)$ and $Z^l \leq Y$. Since $t$ normalizes $PY=Z_2^l$, $t$ normalizes
$Z^l\neq Z$. Since $t$ inverts $Y/Z$, $t$ inverts $Z^l$. By Lemma \ref{Lemma N1-N4}, the four proper subgroups of $Z_2$ containing $Z$ are $\{N_i\cap Z_2|i
\in \{1,2,3,4\}\}$. Since $P \leq Z_2^l$ and $Z^l \nleq P$, we have $|P \cap (N_i \cap Z_2)^l|=3$
for each $i \in \{1,2,3,4\}$. Since for $i=1,2$, $N_i\bs Z \subseteq 3\mathcal{B}$ and for $i=3,4$,
$N_i\bs Z \subseteq 3\mathcal{A}$, we see that $|P \cap 3\mathcal{A}|=|P \cap 3\mathcal{B}|=4$.

It
is clear from the structure of $C_{C_G(Z)}(t)=H \cap {C_G(Z)}$ that $C_H(P)=P\<t\>$. So suppose
$P<R\in \syl_3(H)$. Then $P<N_R(P)$ so let $x \in N_R(P) \bs P$. We have that $P$ has two subgroups
conjugate to $Z$ and two subgroups conjugate to $\<a\>$. Therefore $x$ must centralize $P \cap
3\mathcal{A}$ and $P\cap 3\mathcal{B}$ and so $x \in {C_G(Z)}$. However $P \in \syl_3({C_G(Z)} \cap
H)$ which is a contradiction. Hence $P \in \syl_3(H)$.
\end{proof}

We fix notation such that $P=\{1,z_1,z_1^2,z_2,z_2^2,a_1,a_1^2, a_2,a_2^2\}$ where $z_1=z$, $P \cap
3\mathcal{B}=\{z_1,z_1^2,z_2,z_2^2\}$ and $P \cap 3\mathcal{A}=\{a_1,a_1^2, a_2,a_2^2\}$.

\begin{lemma}\label{HN-Normalizer H of P}
Let $\{i,j\}=\{1,2\}$ then $P \cap O^3(C_G(a_i))=\<a_j\>$. Furthermore $N_H(P)$ has order $3^22^4$
and is transitive on $3\mathcal{A}\cap P$ and $3\mathcal{B}\cap P$ with $N_H(P)/\<P,t\>\cong
\mathrm{Dih}(8)$.
\end{lemma}
\begin{proof}
By Lemma \ref{HN-conjugates in P}, $C_H(P)=\<P,t\>$ and $|P\cap 3 \mathcal{A}|=|P\cap 3
\mathcal{B}|=4$. Observe that every element of order three in $\alt(9)\cong O^3(C_G(a_i))$ is conjugate to its
inverse. Therefore an element in $O^3(C_G(a_i))$ inverts $P \cap O^3(C_G(a_i))$. Thus $P \cap
O^3(C_G(a_i))=\<a_j\>$ otherwise we would have an element of $3\mathcal{A}$ conjugate to an element
in $3\mathcal{B}$. Moreover, an element in $C_G(a_i)$ permutes $\<z_1\>$ and $\<z_2\>$. Furthermore, by
Lemma \ref{HN-conjugates in P} $N_H(Z)\cong 3 : 4^{.}\alt(5)$ and so an element of order four
inverts $Z$ whilst centralizing $P/Z$. Hence an element in $N_H(Z)$ permutes $a_1$ and $a_2$. We have that
$s$ inverts $P$ and by Lemma \ref{Burnside conjugation lemma}, $N_H(P)$ controls fusion in
$P$ and so we have that $N_H(P)$ is transitive on $3\mathcal{A}\cap P$ and $3\mathcal{B}\cap P$.

Finally, since  $C_H(Z)\cong 3 \times 2^.\alt(5)$, $|N_H(P) \cap C_G(Z)|=3^22^2$. Thus, by the
orbit-stabilizer theorem, $|N_H(P)|=3^22^4$ and so $N_H(P)/C_H(P)=N_H(P)/\<P,t\>$ has order eight
and is isomorphic to a subgroup of $\GL_2(3)$ and is therefore isomorphic to $\mathbb{Z}_8$, $D_8$
or $Q_8$. Since $N_H(P)$ is not transitive on $P^\#$, we have that $N_H(P)/C_H(P)\cong
\mathrm{Dih}(8)$.
\end{proof}

\begin{lemma}\label{HN-images in alt9}
Let $\{i,j\}=\{1,2\}$ then  $a_j \in P \cap O^3(C_G(a_i))$ has cycle type $3^2$ in $\alt(9)\cong
O^3(C_G(a_i))$. Furthermore, $t$ is a $2$-central involution in $C_G(a_i)$.
\end{lemma}
\begin{proof}
We have that $C_G(a_i)\cong 3 \times \alt(9)$ and so $|P\cap O^3(C_G(a_i))|=3$. Consider
representatives for the three conjugacy classes of elements of order three in $\alt(9)$. If the
image of $P \cap O^3(C_G(a_i))$ in $\alt(9)$ is conjugate to $\<(1,2,3)\>$ then $P$ commutes with a
subgroup isomorphic to $3\times 3 \times \alt(6)$. However $z \in P$ and $C_G(z)$ has no such
subgroup. So suppose the image of $P \cap O^3(C_G(b))$ in $\alt(9)$ is conjugate to
$\<(1,2,3)(4,5,6)(7,8,9)\>$. Then $C_G(P)$ is a $3$-group which is a contradiction since $[P,t]=1$.
So we must have that the image in $\alt(9)$ of $P \cap O^3(C_G(a_i))$ is conjugate to
$\<(1,2,3)(4,5,6)\>$. Therefore $P \cap O^3(C_G(a_i))$ commutes with a $2$-central involution of
$O^3(C_G(a_i))$ which proves that $t$ is $2$-central.
\end{proof}

Let $\{i,j\}=\{1,2\}$. We fix the following notation by first fixing an injective homomorphism from
$N_G(\<a_i\>)$ into $\alt(12)$ such that $O^3(C_G(a_i))$ maps onto $\alt(\{1,..,9\})$ and $a_i$
maps to $(10,11,12)$. Note that $C_G(P)$ has Sylow $2$-subgroups of order two and so we can make a
fixed choice of $2$-central representative for $t$ in $C_G(a_i)$.

\begin{notation1}\label{HN-Alt9notation}
\begin{enumerate}[$\bullet$]
\item $a_i\mapsto (10,11,12)$.
\item $a_j \mapsto (1,3,5)(2,4,6)$.
 \item $t \mapsto (1,2)(3,4)(5,6)(7,8)$.
 \item $Q_i \mapsto \<(1,2)(3,4)(5,6)(7,8), (1,3)(2,4)(5,8)(6,7),
 (1,5)(3,8)(2,6)(7,4),
 (1,2)(3,4), (3,4)(5,6)\>$.
 \item $r_i \mapsto (1,3)(2,4)$.
 \item When $i=1$, $Q_1>E\mapsto \<(1,2)(3,4)(5,6)(7,8),(1,3)(2,4)(5,8)(6,7),(1,5)(3,8)(2,6)(7,4)\>$.
 \item When $i=2$, $Q_2\ni u\mapsto (1,2)(3,4)$ and $Q_2>F\mapsto \<(1,2)(3,4),(3,4)(5,6)\>$.
\end{enumerate}
\end{notation1}

We observe the following by calculating directly in the image of $N_G(\<a_i\>)$ in $\alt(12)$.
\begin{lemma}\label{HN-alt9 observations}
\begin{enumerate}[$(i)$]
 \item $C_H(a_i)\sim 3 \times (2_+^{1+4}:\sym(3))$ and $Q_i=O_2(C_H(a_i))\cong 2_+^{1+4}$
 with $r_i \in C_H(a_i)\bs Q_i$.
 \item $2 \times 2 \times 2\cong E\vartriangleleft C_H(a_1)$ and there exists $\GL_3(2)\cong C\leq C_G(a_1)$ such
 that $a_2 \in C$ and $C$ is a complement to $C_{C_G(a_1)}(E)$ in $N_{C_G(a_1)}(E)$.
 \item If $\<t\> <V <Q_i$ such that $V \vartriangleleft C_H(a_i)$ then $V$ is elementary
 abelian.
 \item $C_{C_H(a_i)}(Q_i)=C_{N_H(\<a_i\>)}(Q_i)=\<t,a_i\>$.
 \item $C_{C_H(a_1)}(E)=C_{N_H(\<a_1\>)}(E)=\<E,a_1\>$.
 \item $C_{C_G(a_1)}([E,P])=C_{N_G(\<a_1\>)}([E,P])\leq N_{C_G(a_1)}(E)$ and has $\<a_1\>$ as a Sylow $3$-subgroup.
 \item Any involution which inverts $P=\<a_1,a_2\>$ is conjugate to $t$ in $G$, in particular, $t$ is
conjugate to $s$ in $G$.
 \item If $Q_i \leq T \in \syl_2(N_G(\<a_i\>))$ then $Q_i$ is characteristic in $T$.
\end{enumerate}
\end{lemma}
\begin{proof}
These can all be checked by direct calculation in the permutation group. However we add the
following remarks.  Firstly $(iii)$ is a calculation within $C_H(a_i)/\<a_i\>$ and so can be
checked in a parabolic subgroup of $\GL_4(2)\cong \alt(8)$.

Secondly we calculate the image of
$[E,P]$ to be a fours subgroup of $E$ (since by coprime action, $E=[E,P] \times C_E(P)=[E,P] \times
\<t\>$). Therefore $(vi)$ amounts to calculating the centralizer of a fours subgroup (consisting of
involutions of cycle type $2^4$) in $\alt(9)$ and $\sym(9)$.

Thirdly, to prove $(vii)$ we observe
that any involution in $N_G(\<a_1\>)$ that inverts $P$ is conjugate in $N_G(\<a_1\>)$ to an
involution of shape  $(1,3)(2,4)(7,8)(10,11)$. Such an involution centralizes  an element of order
three of cycle type $3^3$ (the elements $(1,2,7)(3,4,8)(5,6,9)$ in this example). Now an element of
order three in $\alt(9)$ with cycle type $3^3$ is the cube of an element of order nine. By Lemma
\ref{elements of order nine in S}, any element of order nine in $G$ has cube in $3\mathcal{B}$ and
the only involutions to commute with elements in $3\mathcal{B}$ are conjugate to $t$ (as
$C_G(Z)/Q\cong 2^.\alt(5)$). Therefore we may assume that an involution which inverts $P$ is
conjugate in $G$ to $t$.

Finally, to verify $(viii)$ we check that a Sylow $2$-subgroup of
$N_G(\<a_i\>)$ is isomorphic to a Sylow $2$-subgroup of $\sym(9)$. Therefore we simply check that a
Sylow $2$-subgroup of $\sym(9)$ has a unique normal extraspecial subgroup of order $2^5$.
\end{proof}

\begin{lemma}\label{HN-3'-subgroups of centralizers}
Let $i \in \{1,2\}$. If $M$ is any $3'$-subgroup of $C_H(a_i)$ that is normalized by $P$ then
$M\leq Q_i$. If $M$ is any $3'$-subgroup of $C_H(z_i)$ that is normalized by $P$ then $M\in
\{1,\<t\>, A_i,B_i\}$ where $A_i\cong B_i\cong \Q_8$ are distinct Sylow $2$-subgroups of $C_H(z_i)$
with $\<A_i,B_i\>=O^3(C_H(z_i))$.
\end{lemma}
\begin{proof}
We have that $C_H(a_i)\sim 3 \times 2_+^{1+4}:\sym(3)$. Furthermore, $2_+^{1+4}\cong  Q_i=
O_2(C_H(a_i))$ and so if $M$ is a $3'$-subgroup of $C_H(a_i)$ that is normalized by $P$ then $MQ_i$ is also. Therefore
$MQ_i \leq Q_i$ and so $M \leq Q_i$.

We have that $C_H(z_i)\cong 3 \times 2^{.}\alt(5)$. Let $M$ be a $3'$-subgroup of $C_H(z_i)$ that is normalized by $P$. If $t \notin M$ then $2 \nmid |M|$ since
$C_H(z_i)$ has Sylow $2$-subgroups isomorphic to $Q_8$. Therefore $|M|=5$ or $|M|=1$. A Sylow $5$-subgroup of $C_H(z_i)$ is not normalized by $P$ and so $M=1$. So assume $\<t\> <M$. Then we
must have  $MP\cong 3 \times \SL_2(3)$. Since $P$ normalizes
precisely two Sylow $2$-subgroups of $C_H(z_i)$ we define $A_i$ and $B_i$ to be these two distinct
$2$-groups.
\end{proof}

We continue the notation for the $P$-invariant subgroups from  the previous lemma. The subgroups $\{A_i,B_i\}$ and $Q_j$ play key roles in this section.

\begin{lemma}\label{HN-swapping a_i's and z_i's-2}
Let $\{i,j\}=\{1,2\}$. The following hold.
\begin{enumerate}[$(i)$]
\item  $N_H(P) \cap C_H(a_i)$ acts transitively on the set $\{\<z_1\>,\<z_2\>\}$.
\item $N_H(P) \cap C_H(z_i)$ acts transitively on the set $\{\<a_1\>,\<a_2\>\}$.
\item $N_H(P) \cap N_H(\<a_i\>)$ acts transitively on the set $\{A_1,B_1, A_2,B_2\}$.
\item $N_H(P) \cap C_H(z_i)$ acts transitively on the set $\{A_i,B_i\}$ and
preserves $\{A_j\}$ and $\{B_j\}$.
\end{enumerate}
\end{lemma}
\begin{proof}
By Lemma \ref{HN-Normalizer H of P}, $N_H(P)/\<P,t\>\cong \mathrm{Dih}(8)$ and $N_H(P)$ is
transitive on $3\mathcal{A}\cap P$ and $3\mathcal{B}\cap P$ which both have order four. It is therefore clear
that $N_H(P) \cap C_H(a_i)$ acts transitively on  $\{\<z_1\>,\<z_2\>\}$ and $N_H(P) \cap C_H(z_i)$
acts transitively on $\{\<a_1\>,\<a_2\>\}$. This proves $(i)$ and $(ii)$.

Now by Lemma \ref{HN-3'-subgroups of centralizers}, $N_H(P)$ acts on the set $\{A_1,B_1,A_2,B_2\}$.
Recall that $s\in H$ inverts $P$. In particular, $s$ acts on the set $\{A_1,B_1\}$. If $s$
normalizes $A_1$ and $B_1$, then $\<s,A_1\>$ and $\<s,B_1\>$ are two distinct Sylow
$2$-subgroups of $N_H(Z)\cong  3 : 4^{.}\alt(5)$. However this is a contradiction since $4\cong
O_2(N_H(Z))\leq \<s,A_1\> \cap \<s,B_1\>=\<t,s\>\cong 2 \times 2$. Hence $s \in N_H(P) \cap
N_H(\<a_i\>)$ permutes $\{A_1,B_1\}$ and by the same argument $s$ permutes  $\{A_2,B_2\}$. Thus $N_H(P) \cap N_H(\<a_i\>)$ acts
transitively on $\{A_1,B_1, A_2,B_2\}$. This proves $(iii)$.

Finally, we have that $N_H(\<z_j\>)\cong  3 : 4^{.}\alt(5)$ and so there is an element of order
four, $f$ say, in $N_H(\<z_j\>)$ that inverts $z_j$ whilst centralizing $z_i$ and $A_j$ and $B_j$.
So $f \in C_H(z_i)\cong  3 \times  2^{.}\alt(5)$. If $f$ normalizes $A_i$ and $B_i$ then since
$A_i,B_i \in \syl_2(C_H(z_i))$, we have that $f \in A_i \cap B_i=\<t\>$. This contradiction proves
that $f$ permutes the set $\{A_i,B_i\}$ and so  $\<f,P\>=N_H(P) \cap C_H(z_i)$ acts transitively on $\{A_i,B_i\}$ and
normalizes  $A_j$ and $B_j$.
\end{proof}

%
%

Recall from Notation \ref{HN-Alt9notation} that $u \in Q_2$ and $r_i\in C_G(a_i)$ are involutions. Set $2\mathcal{A}=\{u^g|g\in G\}$ and $2\mathcal{B}=\{t^g|g \in G\}$.

\begin{lemma}\label{HN-at least 2 classes of involution}
$C_G(a_i)$ has two classes of involution which are not conjugate in $G$. In particular, $r_i \in 2\mathcal{A} \neq 2 \mathcal{B}$.
\end{lemma}
\begin{proof}
We have that every involution in $C_G(a_i)$ lies in $O^3(C_G(a_i))\cong \alt(9)$ and $\alt(9)$ has
two classes of involution with representatives $(1,2)(3,4)$ and $(1,2)(3,4)(5,6)(7,8)$. We have
seen that $t \in C_G(a_i)$ is 2-central in $C_G(a_i)$ and $C_G(t)$ has a Sylow $3$-subgroup of order nine which
intersects non-trivially with both $3\mathcal{A}$ and $3\mathcal{B}$. So let $v \in O^3(C_G(a_i))$
be  an involution which is not conjugate to $t$ in $C_G(a_i)$. Then the image of $v$ in $\alt(9)$ is a double transposition which necessarily commutes with a $3$-cycle. Hence,  $v$ commutes with a subgroup of
$C_G(a_i)$ of order nine, $R$ say, and $C_G(R)\cong 3\times 3 \times \alt(6)$. This implies that $R$ contains no conjugate of $Z$ since no conjugate of $Z$ commutes with a subgroup isomorphic to $\alt(6)$. Thus $v$ is not $G$-conjugate to $t$. In particular, it is now clear from the images of $u$ and $r_i$ that neither are not conjugate to $t$ and so $u,r_i \in 2\mathcal{A} \neq 2 \mathcal{B}$.
\end{proof}

The following lemma is a key step in determining the structure of $H$ since it proves that $H$
contains a subgroup which is extraspecial of order $2^9$.
\begin{lemma}\label{HN-Q_i's}
Let $\{i,j\}=\{1,2\}$ then $Q_i \cap Q_j=\<t\>$ and $C_G(Q_i)=Q_j\<a_i\>$. In particular $\<t\>$ is
the centre of a Sylow $2$-subgroup of $G$ and $Q_1Q_2\cong 2_+^{1+8}$ with $C_G(Q_1Q_2)=\<t\>$.
\end{lemma}
\begin{proof}
Let $\{i,j\}=\{1,2\}$. Since $C_H(P)=P\<t\>$ and $Q_1 \cap Q_2 \leq C_H(\<a_1,a_2\>)=C_H(P)$, we have $Q_1 \cap
Q_2=\<t\>$. Now observe that $P$ normalizes $Q_i$.  By Lemma
\ref{HN-alt9 observations},  $C_{C_H(a_i)}(Q_i)=C_{N_H(\<a_i\>)}(Q_i)=\<t,a_i\>$. Therefore
$C_G(Q_i)$ has a normal $3$-complement, $N$ say, by Burnside's normal $p$-complement  Theorem
(Theorem \ref{Burnside-normal p complement}). Furthermore $C_{N}(a_i)=\<t\>$. By coprime action we
have,
\[N=\<C_N(z_1),C_N(z_2),C_N(a_1),C_N(a_2)\>=\<C_N(z_1),C_N(z_2),C_N(a_j)\>\]
since $C_N(a_i)=\<t\>$. Suppose first that $N=\<t\>$. Then $C_G(Q_i)=\<t,a_i\>$ and so
$N_G(Q_i)\leq N_G(\<a_i\>)$. By Lemma \ref{HN-alt9 observations}, $Q_i$ is characteristic in a
Sylow $2$-subgroup of $N_G(\<a_i\>)$ and so $N_G(\<a_i\>)$ contains a Sylow $2$-subgroup of $H$. Let
$T\in \syl_2(N_G(Q_i))$ then $|T|=2^7$. By Sylow's Theorem, since $A_1 \leq H$, there exists $g \in H$ such that
$A_1^g\leq T\leq N_G(\<a_i\>)$. Therefore $|A_1^g \cap C_G(a_i)|= 4$ or $8$ and of course $A_1^g \cap C_G(a_i)$ commutes with $a_i \in 3\mathcal{A}$. However, by Lemma \ref{HN-element of order four in CG(Z)}, since $A_1^g \cap C_G(a_1)$ has order 4 or 8, $C_G(A_1^g \cap C_G(a_1))$ has a Sylow $3$-subgroup $Z^g$ which is a contradiction. Thus $N>\<t\>$.

Suppose $C_N(z_1)>\<t\>$. Then by Lemma \ref{HN-3'-subgroups of centralizers}, we may assume,
without loss of generality, that $A_1=C_N(z_1)$. By Lemma \ref{HN-swapping a_i's and z_i's-2}
$(iii)$, $N_H(P) \cap N_H(\<a_i\>)$ acts transitively on $\{A_1,B_1,A_2,B_2\}$. Clearly $N_H(P)
\cap N_H(\<a_i\>)$ normalizes $Q_i$ and therefore $\<A_1,B_1\> \leq N$ which is a contradiction
since $\<A_1,B_1\>\cong 2^.\alt(5)$ and $N$ is a $3'$-group. Thus $C_N(z_1)=\<t\>$ and by the same
argument $C_N(z_2)=\<t\>$. So we have that $\<t\><N=C_N(a_j)\leq Q_j$.

Hence $\<t\> < N\leq Q_j$. Suppose for a contradiction that $N<Q_j$. Then $N\vartriangleleft
C_H(a_i)$ and since $a_i$ acts fixed-point-freely on $N/\<t\>$, $|N|=2^3$. By Lemma \ref{HN-alt9
observations} $(iii)$, $N$ is elementary abelian. Now by Lemma \ref{HN-alt9 observations} $(vii)$, $s$ is conjugate to $t$ in $G$. Recall Lemma \ref{HN-normalizer of J2}. This,
together with the fact that $P=\<a_1,a_2\> \in \syl_3(H)$, implies that for $k\in\{1,2\}$ there
exists a $P$-invariant subgroup $Q_8\cong X_k\leq C_H(a_k)$  with $[X_1,X_2]=1$. Now by Lemma
\ref{HN-3'-subgroups of centralizers}, $X_i \leq Q_i$ and  $X_j \leq Q_j$.  We have that $X_j$ and $N$ are both $P$-invariant and furthermore we have that $X_j\cong Q_8$ where as $N$ is elementary abelian. Therefore $|X_j \cap N|=2$ and so $Q_j=NX_j$. Similarly, $Q_i=O_2(C_G(Q_j))X_i$. Therefore $X_j$
commutes with $Q_i$ which is a contradiction.

Hence we have that $N=Q_j$ and so $[Q_1,Q_2]=1$ which implies that $Q_1Q_2\cong 2_+^{1+8}$. Now let
$Q_1Q_2\leq T\in \syl_2(G)$ then $\mathcal{Z}(T)\leq C_T(Q_1Q_2)\leq C_T(Q_1) \cap C_T(Q_2)\leq Q_2 \cap Q_1=\<t\>$. Hence
$\mathcal{Z}(T)=C_G(Q_1Q_2)=\<t\>$.
\end{proof}

Set ${Q_{12}}:=Q_1Q_2\cong 2_+^{1+8}$ and recall that in Notation \ref{HN-Alt9notation} we defined
$E\leq C_G(a_1)$ such that $t \in E\trianglelefteq C_H(a_1)$ is elementary abelian of order eight. We now consider $C_G(E)$ and $N_G(E)$.

\begin{lemma}\label{HN-Describing E}
We have  $t\in E \trianglelefteq  C_H(a_1)$, $C_G(E)/E$ has a nilpotent normal $3$-complement on which $a_1$
acts fixed-point-freely. Furthermore, $N_G(E)/C_G(E)\cong \GL_3(2)$ where the extension is split and $C_H(a_1)$ contains a complement of $C_G(E)$ in $N_G(E)$ which contains $a_2$.
\end{lemma}
\begin{proof}
By Lemma \ref{HN-alt9 observations} $(v)$, $C_{C_H(a_1)}(E)=C_{N_H(\<a_1\>)}(E)=\<E,a_1\>$ and so
by Burnside's normal $p$-complement Theorem  (Theorem \ref{Burnside-normal p complement}), $C_G(E)$
has  a normal $3$-complement, $M$ say and $C_M(a_1)=E$ which implies that $a_1$ acts
fixed-point-freely on $M/E$. A theorem of Thompson says that $M/E$ is nilpotent and therefore $M$ is nilpotent. Also by Lemma \ref{HN-alt9 observations} $(ii)$, there exists a
complement to $C_G(E)$ in $C_H(a_1)$ containing $a_2$.
\end{proof}

\begin{lemma}\label{HN-info on centralizer of E}
Without loss of generality we may assume that $O_{3'}(C_G(E))=\<E,Q_2,A_1,A_2\>$
and $\<{Q_{12}},A_i\>$ is a $2$-group for $i \in \{1,2\}$
that is normalized by $P$.
\end{lemma}
\begin{proof}
Let $N:=O_{3'}(C_G(E))$.  Since $P$ normalizes $N$, we may apply coprime action again to see that
\[N=\<C_N(z_1),C_N(z_2),C_N(a_1),C_N(a_2)\>.\] By Lemma \ref{HN-3'-subgroups of centralizers}, we see that $N$ is generated by $2$-groups and since $N$ is nilpotent, by Lemma \ref{HN-Describing E},
$N$ is a $2$-group.

Since $E\leq Q_1$ and by Lemma \ref{HN-Q_i's}, $[Q_1,Q_2]=1$, we have
that $Q_2\leq N$ and so $N/E\neq 1$. Since $Q_1 \cap Q_2=\<t\>$, $Q_2 \cap E=\<t\>$. In particular,
$N$ does not split over $E$. Let $g\in N_G(E)\cap C_G(a_1)$ be an element of order seven then $g$
acts fixed-point-freely on $E$. If $[N/E,g]=1$ then $N=C_N(g)\times E$ which is a contradiction.
Thus $[N/E,g]\neq 1$ and so $|N/E|\geq 2^3$. Since $a_1$ acts fixed-point-freely on $N/E$ and
preserves $[N/E,g]$, we have $|[N/E,g]| \geq 2^6$.

If $z_1$ and $z_2$ act fixed-point-freely on $N/E$ then $N=Q_2E$ and so
$|N/E|=2^4$ which we have seen is not the case. Therefore at least one of $C_{N/E}(z_1)$ and
$C_{N/E}(z_2)$ is non-trivial.  Since $E\trianglelefteq  C_H(a_1)$  we may apply Lemma
\ref{HN-swapping a_i's and z_i's-2} $(i)$ which says that $N_H(P) \cap C_H(a_1)$ acts transitively
on the set $\{\<z_1\>,\<z_2\>\}$. Therefore $C_{N/E}(z_1)$ and
$C_{N/E}(z_2)$ are both non-trivial. So we may assume, without loss of generality, that $A_1\leq N$
and $A_2 \leq N$ and so
$N=\<E,Q_2,A_1,A_2\>$. Finally, since $Q_1$ normalizes $E$ and so normalizes $N$, we see that $\<Q_{12},A_1\>$ and $\<Q_{12},A_2\>$ are both $2$-groups which are clearly normalized by $P$.
\end{proof}
We continue the notation from this lemma for the rest of this chapter such that $A_1$ and $A_2$ commute with $E$. Set $K:=N_G(Q_{12})\leq H$. We show in the rest of this section that $K=H$.

\begin{lemma}\label{HN-three things about K}
\begin{enumerate}[$(i)$]
\item $N_H(P)\leq K$.
\item $C_{Q_{12}}(A_1)\neq C_{Q_{12}}(A_2)$.
\item For $i \in \{1,2\}$, $C_H(z_i)\leq K$.
\end{enumerate}
\end{lemma}
\begin{proof}
$(i)$ First observe that $N_H(P)$ acts on the set $\{a_1,a_2,a_1^2,a_2^2\}=P \cap 3\mathcal{A}$ and
therefore it preserves ${Q_{12}}=O_2(C_H(a_1))O_2(C_H(a_2))$ so $N_H(P) \leq K$.

$(ii)$ Suppose that $C_{Q_{12}}(A_1)=C_{Q_{12}}(A_2)$. By Lemma \ref{HN-swapping a_i's and z_i's-2}
$(iv)$, $N_H(P) \cap C_H(z_1)$ acts transitively on the set $\{A_1,B_1\}$ whilst preserving $A_2$
and $B_2$. Therefore there exists $g \in N_H(P) \cap C_H(z_1)\leq K$ such that $A_2^g=A_2$ whilst
$A_1^g=B_1$. Therefore
\[C_{Q_{12}}(A_1)=C_{Q_{12}}(A_2)=C_{Q_{12}}(A_2)^g=C_{Q_{12}}(A_1)^g=C_{Q_{12}}(B_1).\]
Therefore $E \leq C_{Q_{12}}(A_1)=C_{Q_{12}}(\<A_1,B_1\>)$. This is a contradiction.

$(iii)$ By Lemma \ref{HN-info on centralizer of E}, $T:=\<Q_{12},A_1\>$ is a $2$-group which is normalized by $P$. We consider $N_T(Q_{12})\leq K$. Since $T$ is normalized by $P$, we apply coprime action to see that \[N_T(Q_{12})=\<C_{N_T(Q_{12})}(r)\mid r \in P^\#\>.\] Since $Q_{12}$ is normalized by $N_H(P)$ which is transitive on $\{A_1,B_1,A_2,B_2\}$ (by Lemma \ref{HN-swapping a_i's and z_i's-2} $(iv)$), it is clear that $A_1 \nleq Q_{12}$. Thus $N_T(Q_{12})>Q_{12}$. Now we use Lemma \ref{HN-3'-subgroups of centralizers} to see that for $j \in \{1,2\}$, $C_{N_T(Q_{12})}(a_j)=Q_j$ and to see that for some $i \in \{1,2\}$,  $C_{N_T(Q_{12})}(z_i)\in \{A_i,B_i\}$. However we again apply Lemma \ref{HN-swapping a_i's and z_i's-2} $(iv)$ to see that since one of $A_i$ or $B_i$ is in $K$ and $N_H(P) \leq K$ is transitive on $\{A_1,B_1,A_2,B_2\}$, $\<A_1,B_1,A_2,B_2\> \leq K$. We can therefore conclude that $C_H(z_1)\leq
K$ and $C_H(z_2)\leq K$.
\end{proof}

\begin{lemma}\label{HN-invs acting on Q}
\begin{enumerate}[$(i)$]
\item Suppose that $v\in K$ such that $Q_{12}v$ is an involution which inverts ${Q_{12}}z_i$ for some $i \in \{1,2\}$.
Then $|C_{\bar{{Q_{12}}}}(v)|=2^4$. In particular, if $Q_{12}v$ inverts a Sylow $3$-subgroup of
$K/Q_{12}$ then $|C_{\bar{{Q_{12}}}}(v)|=2^4$.

\item Suppose that $v\in K$ such that $Q_{12}v$ is an involution which inverts ${Q_{12}}a_i$ for some $i
\in \{1,2\}$.
 Then $|C_{\bar{{Q_{12}}}}(v)|\leq 2^6$.
\end{enumerate}\end{lemma}
\begin{proof}
Observe that $\bar{Q_{12}}$ is elementary abelian and $\bar{Q_{12}}\bar{v}$ has order two and
inverts $\bar{Q_{12}}\bar{x}$ which has order three. Therefore we may use Lemma
\ref{Prelims-centralizers of invs on a vspace which invert a 3}. In case $(i)$,
$|C_{\bar{Q_{12}}}(z_i)|=1$ so we have that $|C_{\bar{{Q_{12}}}}(v)|\leq 2^4$ however by Lemma
\ref{lem-cenhalfspace}, $|C_{\bar{{Q_{12}}}}(v)|\geq 2^4$ so we get equality. In case $(ii)$,
$|C_{\bar{Q_{12}}}(a_i)|=2^4$. Therefore $|C_{\bar{{Q_{12}}}}(v)|\leq 2^6$.
\end{proof}

\begin{lemma}\label{HN-K/Q has shape alt(5) wr 2 i}
$N_K(P)Q_{12}/Q_{12}\cong 3^2:\mathrm{Dih}(8)$, $C_H(a_i)\leq K$ for $i \in
\{1,2\}$ and a minimal normal subgroup of $K/{Q_{12}}$ is neither a $3$-group nor a $3'$-group.
\end{lemma}
\begin{proof}
Recall from Lemma \ref{HN-three things about K} that $N_H(P) \leq K$. Clearly $N_K(P) \cap Q_{12}\leq C_H(P) \cap Q_{12}=\<t\>$ (see Lemma \ref{HN-conjugates in P}). Thus $N_K(P)\cap Q_{12}=\<t\>$. So $N_K(P)Q_{12}/Q_{12}\cong N_K(P)/(N_K(P) \cap Q_{12})\cong
3^2:\mathrm{Dih}(8)$ by Lemma \ref{HN-Normalizer H of P}. Notice also that
$C_H(a_i)=Q_iN_{C_H(a_i)}(P)$ and so $C_H(a_i)\leq K$.

Let $M \geq {Q_{12}}$ be a normal subgroup of $K$ such that $M/{Q_{12}}$ is a minimal normal
subgroup of $K/{Q_{12}}$. Suppose $M$ is a $3'$-group. Then, since $M$ is normalized by $P$, we may
apply coprime action to say that
\[M=\<C_M(a_1),C_M(a_2),C_M(z_1),C_M(z_2)\>.\] By Lemma \ref{HN-3'-subgroups of centralizers},
$Q_j \leq C_M(a_j)\leq Q_j$ for $j=1$ and $j=2$. Therefore we may assume $C_M(z_i)>\<t\>$ for some
$i \in \{1,2\}$. Now $C_M(z_i)$ is normalized by $P$ so must equal $A_i$ or $B_i$ by Lemma
\ref{HN-3'-subgroups of centralizers}. However by Lemma \ref{HN-swapping a_i's and z_i's-2} $(iii)$,
$N_H(P)$ acts transitively on the set $\{A_1,B_1, A_2,B_2\}$. Therefore $M\geq \<A_i,B_i\>\cong
\alt(5)$ and so $M$ is not a $3'$-group which is a contradiction. So suppose instead that
$M/{Q_{12}}$ is a $3$-group. Then we must have $M={Q_{12}}P$. Now by Lemma \ref{frattini} (Frattini
argument), $K={Q_{12}}N_K(P)$. Therefore $K/Q_{12}\cong N_K(P)/(N_K(P) \cap Q_{12})\cong
3^2:\mathrm{Dih}(8)$ and $|K|=2^93^22^3$. However this contradicts Lemma \ref{HN-three things about K} which says that $C_H(z_1) \leq K$ and $5\mid |C_H(z_1)|$. Hence $M/Q_{12}$  is not a $3$-group.
\end{proof}

\begin{lemma}\label{HN-K/Q has shape alt(5) wr 2 ii}
$K/{Q_{12}}\cong \alt(5) \wr 2$ and there exist subgroups $M_1,M_2\leq K$ such that for $\{i,j\}=\{1,2\}$,  $z_i \in M_i$ and $A_j \leq M_i$ with $M_i/Q_{12}\cong \alt(5)$ and $[M_1,M_2]\leq Q_{12}$.
\end{lemma}
\begin{proof}
We continue notation from the previous result by setting $Q_{12}\leq M \leq K$ such that $M/Q_{12}$
is a minimal normal subgroup of $K/Q_{12}$. By Lemma \ref{HN-K/Q has shape alt(5) wr 2 i},
$M/{Q_{12}}$ is a direct product of non-abelian isomorphic simple groups and properly contains
$P/{Q_{12}}$.  By Lemma \ref{HN-Q_i's}, $C_G(Q_{12})\leq Q_{12}$ and so $M/{Q_{12}}$ is isomorphic
to a subgroup of $\aut(Q_{12})\cong \mathrm{GO}_8^+(2)$ (Lemma \ref{extraspecial outer
automorphisms}). Suppose that $M/Q_{12}$ is simple. Then we check (using \cite{atlas} for example)
every simple subgroup of $\mathrm{GO}_8^+(2)$  to see that the only simple groups with an elementary abelian Sylow
$3$-subgroup of order nine are $\alt(6)$, $\alt(7)$ and $\alt(8)$. Note that
$C_{K/Q_{12}}(M/Q_{12})\leq
C_{K/Q_{12}}(PQ_{12}/Q_{12})=C_{K}(P)Q_{12}/Q_{12}=P\<t\>Q_{12}/Q_{12}\leq M/Q_{12}$ by
coprime action. Thus  $K/Q_{12}$ is isomorphic to a subgroup of the automorphism group of
$\alt(6)$, $\alt(7)$ or $\alt(8)$. Note that $N_H(Z)Q_{12}/Q_{12}\cong N_H(Z)/N_{Q_{12}}(Z)=N_H(Z)/\<t\>\cong \sym(3)\times \alt(5)$ which
is not the case in any such group. Thus $M/Q_{12}$ is not simple.

So we must have that $M/{Q_{12}}$ is a direct product of two non-cyclic isomorphic simple groups.
Let $M=M_1M_2$ where $Q_{12} \leq M_1 \cap M_2$ and $M_1/{Q_{12}}\cong M_2/{Q_{12}}$ is simple and $[M_1,M_2]\leq Q_{12}$. Let $3
\cong R\in \syl_3(M_1)$ then by coprime action, $M_2/Q_{12} \cong
C_{M_2/Q_{12}}(R)=C_{M_2}(R)Q_{12}/Q_{12}\cong C_{M_2}(R)/C_{Q_{12}}(R)$. Observe that
$C_{K}(a_i)/C_{Q_{12}}(a_i)=C_{K}(a_i)/Q_{i}\cong 3\times \sym(3)$, so we have without loss of
generality that $z_1 \in M_1$ and $z_2 \in M_2$ and furthermore we have that $M_1/Q_{12}\cong
M_2/Q_{12}\cong \alt(5)$. Moreover we have that $M_2/Q_{12}=C_{K/Q_{12}}(z_1)$ and so $A_1 \leq M_2$. Finally we apply a Frattini argument to see now that $K=MN_K(P)$ and it therefore
follows that $K/{Q_{12}}\cong \alt(5)\wr 2$.
\end{proof}

Let $T\in \syl_2(K)$. In the following lemma we prove that $T$ is in fact a Sylow $2$-subgroup of $G$.

\begin{lemma}\label{HN-T is sylow in G}
$C_{\bar{{Q_{12}}}}(T \cap O^2(K))<2^4$ and $T\in \syl_2(G)$.
\end{lemma}
\begin{proof}
We show that ${Q_{12}}$ is a characteristic subgroup of $T$ to conclude that $T \in \syl_2(H)$ and
since $\<t\>=\mathcal{Z}(T)$ by Lemma \ref{HN-Q_i's}, we can from there conclude that $T\in \syl_2(G)$. We
show that $\bar{{Q_{12}}}$ is characteristic in $\bar{T}$ by applying Lemma \ref{Prelims 2^8 3^2
Dih(8)} to $\bar{K}$.

We have that $O^2(K)=M_1M_2$ and $O^2(K)/Q_{12}\cong \alt(5)\times \alt(5)$. Notice that every
involution in $\alt(5)$ inverts an element of order three. Suppose that $Q_{12}v$ is an involution
in $O^2(K/{Q_{12}})$. Then $Q_{12}v$ inverts an element of order three in some $M_i/Q_{12}$ which is therefore
conjugate to $Q_{12}z_i$ ($i \in \{1,2\}$). Hence, by Lemma \ref{HN-invs acting on Q},
$|C_{\bar{{Q_{12}}}}(v)|=2^4$. In particular, $C_{\bar{{Q_{12}}}}(A_i{Q_{12}}/{Q_{12}})$ has
order at most $2^4$ ($\{i,j\}=\{1,2\}$). By Lemma \ref{HN-three things about K} $(ii)$,
$C_{\bar{{Q_{12}}}}(A_1{Q_{12}}/{Q_{12}})\neq C_{\bar{{Q_{12}}}}(A_2{Q_{12}}/{Q_{12}})$. Therefore
$|C_{\bar{{Q_{12}}}}(\<A_1,A_2\>{Q_{12}}/{Q_{12}})|<2^4$.

Now, let $R$ be a non-trivial elementary abelian normal $2$-subgroup of $T/{Q_{12}}$. If $|R|=2$
then $R\in \mathcal{Z}(T/{Q_{12}})$ and therefore $R\leq O^2(K/{Q_{12}})\cong \alt(5)\times \alt(5)$ and so
$|C_{\bar{{Q_{12}}}}(R)|\leq 2^4$. Suppose $|R|=4$ or $8$. Then $R \cap O^2(K/{Q_{12}})\neq 1$ and
so again we have $|C_{\bar{{Q_{12}}}}(R)|\leq 2^4$. Now suppose $|R|=2^4$.  Then a calculation in
$\alt(5) \wr 2$ verifies that $R=T \cap O^2(K/{Q_{12}})$. We may assume (up to conjugation) that
$R=A_1A_2{Q_{12}}/{Q_{12}}$. Therefore $|C_{\bar{{Q_{12}}}}(R)|<2^4$. Thus we may now apply Lemma
\ref{Prelims 2^8 3^2 Dih(8)} to say that $\bar{{Q_{12}}}$ is characteristic in $\bar{T}$ and we are
done.
\end{proof}


\begin{lemma}\label{HN-centralizer of F}
$N_H(E) \leq K$ and for $V<E$ such that $t \in V\cong 2 \times 2$, $C_G(V)\leq K$ and $|C_G(V)|=
2^{13}3$.
\end{lemma}
\begin{proof}
By Lemma \ref{HN-info on centralizer of E}, $C_G(E)=\<E,Q_2,A_1,A_2,a_1\>$. By Lemma  \ref{HN-K/Q
has shape alt(5) wr 2 ii}, $A_i \leq C_H(z_i) \leq K$. Thus $C_G(E) \leq K$. So we consider
$N_H(E)$. By Lemma \ref{HN-Describing E}, there exists a complement, $C$, to $C_G(E)$ in $N_G(E)$
such that $C \leq C_G(a_1)$. Now, by Lemma \ref{dedekind} (Dedekind Modular Law), $N_G(E) \cap
H=C_G(E)C \cap H=C_G(E)(C \cap H)$. Furthermore, $C \cap H \leq C_H(a_1)\leq K$ by Lemma
\ref{HN-K/Q has shape alt(5) wr 2 i}. Thus $N_H(E) \leq K$.

By Lemma \ref{HN-alt9 observations}, $C_{C_H(a_1)}([E,P])=C_{N_H(\<a_1\>)}([E,P])$ and  by coprime
action, $E=[E,P] \times C_E(P)=[E,P] \times \<t\>$. Therefore by Burnside's normal $p$-complement
Theorem (Theorem \ref{Burnside-normal p complement}), $C_G([E,P])$ has a normal $3$-complement, $N$
say, which is normalized by $P$. By coprime action,
\[N=\<C_N(a_1),C_N(a_2),C_N(z_1),C_N(z_2)\>.\] Since $[E,P]\leq E$, it follows from Lemma
\ref{HN-info on centralizer of E} that $C_N(a_2)\geq Q_2$, $C_N(z_1)\geq A_1$ and $C_N(z_2)\geq
A_2$. Since $C_N(a_2)$, $C_N(z_1)$, $C_N(z_2)$ are $3'$-groups normalized by $P$ it follows that
$C_N(a_2)= Q_2$, $C_N(z_1)= A_1$ and $C_N(z_2)= A_2$. By Lemma \ref{HN-alt9 observations} $(vi)$, $C_G(a_1) \cap C_G([E,P])\leq N_G(E)$. Thus $N \leq N_G(E)$ and therefore $C_G([E,P])\leq N_G(E)$. Finally, $N_G(E)$ is
transitive on subgroups of $E$ of order four. Therefore if we choose $t \in V<E$ of order four.
Then $C_G(V) \leq N_G(E) \cap H \leq K$.
\end{proof}

\begin{lemma}\label{HN-K is strongly 3 embedded}
$K$ is strongly $3$-embedded in $H$.
\end{lemma}
\begin{proof}
Let $h \in H$ and $y\in K \cap K^h$ be an element of order three. By Lemmas \ref{HN-three things about K} and \ref{HN-K/Q has shape alt(5) wr 2 i}, the centralizer in $H$ of every element of order three in $K$ is contained in $K$. Thus $C_H(y) \leq K \cap K^h$.
Therefore $K \cap K^h$ contains a Sylow $3$-subgroup of $H$. So assume $P\leq K \cap K^h$. Then
${Q_{12}}=O_2(K)=\prod_{p\in P^\#}O_2(C_H(p))=O_2(K^h)={Q_{12}}^h$. Therefore $h \in
N_G({Q_{12}})=K$ and so $K=K^h$.
\end{proof}


\begin{lemma}\label{HN-involutions in O^2(K)/Q}
Let $\bar{v}\in O^2(\bar{K})\bs \bar{{Q_{12}}}$ be an involution. Then either $v$ is an element of
order four squaring to $t$ and $C_H(v)$ contains a conjugate of $Z$ or $v\in 2\mathcal{B}$ and
$|C_{\bar{{Q_{12}}}}(\bar{v})|=2^4$ and $|C_{\bar{K}}(\bar{v})|=2^9$.
\end{lemma}
\begin{proof}
We have that $O^2(K/{Q_{12}})\cong \alt(5) \times \alt(5)$ and it follows that ${Q_{12}}v$ lies in
one of two $K/{Q_{12}}$-conjugacy classes of involutions in $O^2(K)$. Either ${Q_{12}}v\in
M_i/Q_{12}$ for some $i \in \{1,2\}$ or is a diagonal involution. Suppose ${Q_{12}}v \in
M_i/{Q_{12}}\cong \alt(5)$ where ${Q_{12}}z_i\in M_i/{Q_{12}}\vartriangleleft O^2(K/{Q_{12}})$.
Then up to conjugation we may assume ${Q_{12}}v$ inverts ${Q_{12}}z_i$. If ${Q_{12}}v$ is diagonal
then we may assume up to conjugation that ${Q_{12}}v$ inverts ${Q_{12}}z_1$ and ${Q_{12}}z_2$. So
in either case we may apply Lemma \ref{HN-invs acting on Q} $(i)$ to say that
$|C_{\bar{{Q_{12}}}}(v)|=2^4$ and then by Lemma \ref{lem-conjinvos}, every involution in
$\bar{{Q_{12}}}\bar{v}$ is conjugate to $\bar{v}$. We may choose an element of order four, $f\in
C_H(z_1)$ with $f^2=t$. Then $\bar{{Q_{12}}}\bar{f}$ is an involution in
$\bar{M_2}/\bar{{Q_{12}}}\cong M_2/{Q_{12}}$ and so if ${Q_{12}}v \in M_2/{Q_{12}}$ then $Q_{12}v$
is conjugate to $Q_{12}f$ and therefore $\bar{v}$ is conjugate to $\bar{f}$ which implies that $v$
has order four and is conjugate to $f$.  Suppose that ${Q_{12}}v$ is diagonal. Recall that $Q_{12}s$ inverts $P$ and so  ${Q_{12}}v$ is conjugate
to ${Q_{12}}s$ and therefore $\bar{v}$ is conjugate to $\bar{s}$ which implies $v$ is conjugate to $s$ or
$st$. By Lemma \ref{HN-alt9 observations} $(vii)$, since $s$ and $st$ invert $P$, $s,st,v \in 2\mathcal{B}$. Also by Lemma
\ref{lem-conjinvos},
$|C_{\bar{K}}(\bar{v})|=|C_{\bar{{Q_{12}}}}(\bar{v})||C_{K/{Q_{12}}}({Q_{12}}v)|=2^9$.
\end{proof}

Recall we fixed an involution $r_1 \in C_H(a_1)$ in Notation \ref{HN-Alt9notation}.
\begin{lemma}\label{HN-Transfer-O^2(H) is proper}
$r_1$ is not in $O^2(H)$. In particular, $H\neq O^2(H)$ and $O^2(H) \cap K\sim 2_+^{1+8}.(\alt(5)\times \alt(5))$.
\end{lemma}
\begin{proof}
Given the cycle type of the images of $r_1$ and $t$ in $\alt(9)\cong O^3(C_G(a_1))$ and by Lemma \ref{HN-at least 2
classes of involution}, we see that $r_1$ is not conjugate to $t$ in $G$ however the product $r_1t$
is conjugate to $t$ in $O^3(C_G(a_1))$ and therefore $r_1$ is not conjugate to $r_1t$ in $G$.

Observe that $r_1$ inverts $a_2$ therefore $r_1\notin {Q_{12}}$ else $[r_1,\<a_2\>]=\<a_2\>\leq
{Q_{12}}$. Since  $r_1$ centralizes $a_1$ whilst inverting $a_2$, we have that $r_1$ permutes
$\<z_1\>$ and $\<z_2\>$ and therefore permutes $M_1$ and $M_2$ and so $r_1\notin O^2(K)$. Recall
that $T\in \syl_2(K)$ so choose $T$ such that $r_1\in T$ and suppose that for some $h \in H$,
$r_1^h\in O^2(K) \cap T$. Suppose that $r_1^h \in {Q_{12}}$. Then $\<r_1^h,t\>\vartriangleleft
{Q_{12}}$ but is not central in ${Q_{12}}$ as $Q_{12}$ is extraspecial. Therefore $r_1^h$ is
conjugate to $r_1^ht=(r_1t)^h$ in ${Q_{12}}$ and so $r_1$ is conjugate to
$r_1t$ which is a contradiction. So $r_1^h \notin {Q_{12}}$. 
So consider $Q_{12} \neq {Q_{12}}r_1^h$. By Lemma \ref{HN-involutions in O^2(K)/Q}, either $r_1^h
\in 2 \mathcal{B}$ or has order four. However $r_1$ is an involution and is not conjugate to $t$ in
$G$ and so we have a contradiction.

Thus no $H$-conjugate of $r_1$ lies in $T \cap O^2(K)$ which is a maximal subgroup of $T\in
\syl_2(H)$. By Thompson Transfer (Lemma \ref{ThompsonTransfer}), $r_1 \notin O^2(H)$ and so $H \neq
O^2(H)$. Since $[K:O^2(K)]=2$, we must have $O^2(K)=O^2(H) \cap K\sim 2_+^{1+8}.(\alt(5)\times
\alt(5))$.
\end{proof}

\begin{lemma}\label{HN-orbits of elements in Q}
Let $f \in {Q_{12}}\bs \<t\>$. Then either $f$ has order four or one of the following occurs.
\begin{enumerate}[$(i)$]
 \item $f\in 2\mathcal{B}$, $C_H(f)\leq K$ has order
 $2^{13}3$ and $\bar{f}$ is $2$-central in $\bar{K}$.
  \item $f\in 2\mathcal{A}$, $|C_K(f)|=2^{11}35$ and
  $C_{K}({f}){Q_{12}}/{Q_{12}}\cong \alt(5) \times 2$ or $\sym(5)$.
 \end{enumerate}
\end{lemma}
In particular, $K$ acts irreducibly on $\bar{Q_{12}}$, $C_H(f) \cap 3\mathcal{A} \neq 1$ and if
$\bar{f} \in \mathcal{Z}(\bar{T})$ then $f\in 2\mathcal{B}$ and $C_H(f) \leq K$.
\begin{proof}
We have that ${Q_{12}}z_1\in M_1/Q_{12}$ acts fixed-point-freely on $\bar{{Q_{12}}}$ and so every
$M_1/Q_{12}$-chief factor of $\bar{Q_{12}}$ is non-trivial. By Lemma \ref{prelim-alt5 action},
every non-central chief factor has order  $2^4$ and is a natural module for $M_1/Q_{12}$. Let $e\in K$ such that ${Q_{12}}e$ has order five and
$M_1/Q_{12}=\<{Q_{12}}z_1,{Q_{12}}e\>\cong \alt(5)$. Then $Q_{12}e$ acts fixed-point-freely on every chief factor of $\bar{{Q_{12}}}$ and therefore acts fixed-point-freely on $Q_{12}$.
It follows that for every $1\neq \bar{f}\in \bar{{Q_{12}}}$, $\bar{f}$ lies in a $K$-orbit of
length a multiple of $15$ and therefore $f$ lies in a $K$-orbit of length a multiple of $30$.

If $|{f^K}|=30,60$ or $90$ then $|C_K(f)(Q_{12})/Q_{12}|=2^535, 2^435,$ or $2^55$. As $\alt(5) \wr 2$ has no subgroup of order $2^55$, there is no orbit of length 90. If the orbit has length 30 or 60 then $C_K(f)$ contains a conjugate of $a_1$ (the image of which is diagonal in $\alt(5) \wr 2$) however $\alt(5) \wr 2$ has no subgroups of the necessary order containing a diagonal element of order three.  Thus  $|{f^K}|$ is not equal to $30,60$ or $90$.

Recall Lemma \ref{HN-Describing E} which describes $t \in E\leq {Q_{12}}$. Every involution in $E$
is conjugate to $t$ since $N_G(E)/C_G(E)\cong \GL_3(2)$. Furthermore, by Lemma \ref{HN-centralizer
of F}, if we choose $f \in E\bs \<t\>$ then $C_H(f)=C_G(\<f,t\>)\leq K$ and $|C_H(f)|=2^{13}3$.
Since $\<f,t\> \vartriangleleft {Q_{12}}$, it follows that $|C_{\bar{H}}(\bar{f})|=2^{13}3$.
Therefore $\bar{f}$ is central in a Sylow $2$-subgroup of $\bar{K}$ and $f$ lies in a $K$-orbit of
length $|K|/(2^{13}3)=150$.

In Notation \ref{HN-Alt9notation} we fixed an image of $Q_1$ in $\alt(9)$. Observe that the image
of $Q_1$ contains involutions in $2\mathcal{A}$. So let $f \in Q_1$ be such an involution. Now
${Q_{12}}\bs\<t\>$ contains 240 elements of order four and 270 elements of order two (see
\cite[2.4.1]{SegevHN}  for example). Therefore ${f}$ lies in an orbit of length a multiple of 30
and less than 120 and not 30, 60 or 90. Therefore
$[K:C_K(f)]=|\{f^K\}|=120$ and so $|C_K(f)|=2^{11}35$
and $C_{K}({f}){Q_{12}}/{Q_{12}}\cong \alt(5) \times 2$ or $\sym(5)$.

We now suppose $f \in Q_1$ has order four. Then we have that ${f}$ lies in a $K$-orbit of length a
multiple of 30 and less than 240 and not 30, 60 or 90. Moreover $[a_1,Q_1]=1$ and so $\{f^K\}$ is not a multiple of nine. Therefore the only possibilities are
$[K:C_K(f)]=|\{f^K\}|=120,150,240$ and
$C_{K}(f){Q_{12}}/{Q_{12}}=2^335,2^53,2^235$. If $f$ lies in an orbit of length 112 or 150 then consider the remaining elements of order four in $Q_{12}$. These elements cannot lie in an orbit of length 30, 60, 90. Thus it follows that $|\{f^K\}|\neq 150$ and the elements of order four either lie in two orbits of length $120$ or one orbit of length $240$.

In particular, every $1\neq \bar{f} \in \bar{{Q_{12}}}$ commutes with an element of order three in
$\bar{K}$. Since each $z_i$ acts fixed-point-freely on $\bar{{Q_{12}}}$, we have that $\bar{f}$ is
centralized by a conjugate of ${Q_{12}}a_i$. Therefore $f$ commutes with a conjugate of $a_i$.
Furthermore we observe that if $f$ has order four or $f\in 2\mathcal{A}$ then $\bar{f}$ is not
$2$-central in $\bar{K}$ whereas if $f\in 2\mathcal{B}$ then $\bar{f}$ is $2$-central in $\bar{K}$.
Finally, suppose that $ W<Q_{12}$ with $t \in W\vartriangleleft K$. Then $W$ must be a union of
$K$-orbits. However the $K$-orbits on $Q_{12}$ have lengths in $\{1, 150,120,240\}$ and no union of
orbits  is a power of $2$ greater than $2$ and less than $2^9$. Thus $K$ acts irreducibly on
$\bar{Q_{12}}$.
\end{proof}

\begin{lemma}\label{HN-Q=Q^h}
Let $h \in H$. If $\bar{{Q_{12}}} \cap \bar{{Q_{12}}^h}$ contains a $2$-central involution
then ${Q_{12}}={Q_{12}}^h$.
\end{lemma}
\begin{proof}
We may suppose that for some $1 \neq \bar{f} \in \mathcal{Z}(\bar{T})$,  $\bar{f} \in \bar{{Q_{12}}} \cap
\bar{{Q_{12}}}^h$. By Lemma \ref{HN-orbits of elements in Q}, $f\in 2\mathcal{B}$ and $C_H(f) \leq
K$ and also $C_H(f) \leq K^h$. However this implies that $3\mid |K \cap K^h|$ and so $K=K^h$ and
${Q_{12}}={Q_{12}}^h$ by Lemma \ref{HN-K is strongly 3 embedded}.
\end{proof}

\begin{lemma}
$\bar{{Q_{12}}}$ is strongly closed in $\bar{T}$ with respect to $\bar{H}$.
\end{lemma}
\begin{proof}
Let $1\neq \bar{f}\in \bar{{Q_{12}}}$ such that $\bar{f} \in \bar{T}^h\bs \bar{{Q_{12}}}^h$ for
some $h \in H$. Since $f \in Q_{12}\leq O^2(K)\leq O^2(H)$, we must have that $f\in O^2(K^h)=O^2(H)
\cap K^h$. By Lemma \ref{HN-involutions in O^2(K)/Q} applied to $K^h$,  either $f$ is an element of
order four squaring to $t$ and commuting with a conjugate of $Z$ or $f\in 2\mathcal{B}$ and
$|C_{\bar{{Q_{12}}}^h}({f})|=2^4$ and $|C_{\bar{K}^h}({f})|=2^9$.

Suppose first that $f$ has order four. Then $f^2=t$ and ${Q_{12}}^hf$ is an involution in
$O^2(K/{Q_{12}})^h$. By Lemma \ref{HN-involutions in O^2(K)/Q}, $C_G(f)$ contains a conjugate of
$Z$ and then by Lemma \ref{HN-element of order four in CG(Z)}, a Sylow $3$-subgroup of $C_G(f)$ is
conjugate to $Z$. However, by Lemma \ref{HN-orbits of elements in Q},  $C_H(f) \cap 3\mathcal{A}
\neq 1$ which is a contradiction.

So we suppose instead that $f$ is an involution then $f \in 2 \mathcal{B}$ and by Lemma
\ref{HN-orbits of elements in Q}, $C_H(f)\leq K$.  Set $D:=C_{\bar{K}^h}({f})$ and
$V:=C_{\bar{{Q_{12}}}^h}({f})$. Clearly $|D \cap \bar{Q_{12}}| \geq 2^4$ however suppose that $V
\cap \bar{Q_{12}}=1$. Then $V \bar{Q_{12}}\in \syl_2(\bar{O^2(K)})$ and $[V, D \cap
\bar{Q_{12}}]\leq V \cap \bar{Q_{12}}=1$. However this implies that $D \cap \bar{Q_{12}}$ commutes
with  $V \bar{Q_{12}}\in \syl_2(\bar{O^2(K)})$ which contradicts Lemma \ref{HN-T is sylow in G}.

Hence $1 \neq V \cap \bar{Q_{12}}\vartriangleleft D$ and so there exists some $y \in \bar{Q_{12}}
\cap V\cap \mathcal{Z}(D)$.  Notice that $y$ commutes with $\bar{Q_{12}}^h D \in \syl_2(\bar{K}^h)$ and so $y \in \bar{{Q_{12}}}
\cap \bar{{Q_{12}}}^h$ is a $2$-central involution of $\bar{K}^h$. Now by Lemma \ref{HN-Q=Q^h},
${Q_{12}}={Q_{12}}^h$ which is a contradiction. Thus $\bar{{Q_{12}}}$ is strongly closed in
$\bar{T}$ with respect to $\bar{H}$.
\end{proof}

\begin{lemma}
$K=H$.
\end{lemma}
\begin{proof}
Assume for a contradiction that $K<H$ then ${Q_{12}} \ntriangleleft H$. Consider $O_{3'}(H)$. By
Lemma \ref{HN-orbits of elements in Q}, the only proper subgroup of $Q_{12}$ which is normalized by
$K$ is $\<t\>$. So we have that $O_{3'}(H)\cap K\leq O_{3'}(H)\cap Q_{12}=\<t\>$. Since $O_{3'}(H)$
is normalized by $P$, by coprime action, $O_{3'}(H)$ is generated by elements commuting with
elements of $P^\#$. However by Lemmas \ref{HN-K/Q has shape alt(5) wr 2 i} and \ref{HN-K/Q has
shape alt(5) wr 2 ii}, for every $p \in P^\#$, $C_H(p) \leq K$. Therefore $O_{3'}(H)\leq K$ and so
$O_{3'}(H)=\<t\>$.

Set $M:=\<{Q_{12}}^H\>\trianglelefteq H$ then $M \leq O^2(H)$. Moreover $O_{3'}(M)\leq O_{3'}(H)$
and so $O_{3'}(M)=\<t\>$.  Therefore we have $P \leq M$. Suppose $C_M(P)=N_M(P)$. Then $M$ has a
normal $3$-complement which is a contradiction since $O_{3'}(M)=\<t\>$. Since $[P, N_H(P)\cap
{Q_{12}}]\leq P \cap {Q_{12}}=1$, we see that $N_H(P) \cap {Q_{12}}=C_H(P)\cap {Q_{12}}$. Suppose
$Q_{12} \in \syl_2(M)$ then $K \cap M= Q_{12}P$. By Lemma \ref{HN-K/Q has shape alt(5) wr 2 i},
$N_H(P)\leq K$ and so $N_M(P)\leq Q_{12}P$. By Lemma \ref{dedekind} (Dedekind), $N_M(P) \cap
Q_{12}P=(N_M(P) \cap Q_{12})P=C_{Q_{12}}(P)P\leq C_H(P)$ which is a contradiction. Therefore
$Q_{12}$ is not a Sylow $2$-subgroup of $M$ and so $M \cap K>Q_{12}P$.

Set $N:=O_{2'}(M)$. If $N$ is $3'$ then $N \leq O_{3'}(H) =\<t\>$ and so $N=1$. Otherwise $P \leq
N$ and then $[P,{Q_{12}}] \leq N \cap {Q_{12}}=1$ which is a contradiction. Therefore
$O_{2'}(M)=1$. Now, since $P \leq M\vartriangleleft H$, $H=MN_H(P)$ by a Frattini argument and so
$M=\<Q_{12}^H\>=\<Q_{12}^{N_H(P)M}\>=\<Q_{12}^M\>$ since $N_H(P)\leq K=N_G(Q_{12})$. Finally, we
may apply Theorem \ref{goldschmidt} to $\bar{M}=\<\bar{{Q_{12}}}^M\>$. As required, we have that
$O_{2'}(\bar{M})=1$ and since $\bar{{Q_{12}}}$ is strongly closed in $\bar{T}$ with respect to
$\bar{H}$, we have that $\bar{{Q_{12}}}$ is strongly closed in $\bar{M} \cap \bar{T}$ with respect
to $\bar{M}$. Thus $\bar{{Q_{12}}}=O_2(\bar{M})\Omega(\bar{T} \cap \bar{M})$. Since ${Q_{12}}$ is
not a Sylow $2$-subgroup of $M \leq O^2(H)$ we may find $e \in (M \cap T)\bs {Q_{12}}$. Then by
Lemma \ref{HN-involutions in O^2(K)/Q}, ${Q_{12}}e$ contains either involutions or elements of
order four squaring to $t$. In either case $\bar{{Q_{12}}}\bar{e} \cap \Omega(\bar{T} \cap
\bar{M})\neq 1$ and so ${Q_{12}} \nleq \Omega(\bar{T} \cap \bar{M})$. This contradiction proves
that $H=K$.
\end{proof}

\section{The Structure of the Centralizer of $u$}\label{HN-Section-CG(u)}

We now know the structure of the centralizer of an involution in $G$-conjugacy class $2\mathcal{B}$
and so we must determine the structure of the centralizer of an involution in $2\mathcal{A}$.
Recall that in Notation \ref{HN-Alt9notation} we fixed an  involution $u \in Q_2$ and we defined
$2\mathcal{A}$ to be the conjugacy class of involutions in $G$ containing $u$. By Lemma \ref{HN-at
least 2 classes of involution}, $2\mathcal{A}\neq 2 \mathcal{B}$. Let $L:=C_G(u)$ and
$\wt{L}=L/\<u\>$ and we continue to set $H=C_G(t)$ and $\bar{H}=H/\<t\>$. We will show that $L\sim
(2^.\HS):2$ and so we must identify that $\wt{L}$ has an index two subgroup isomorphic to the
sporadic simple group $\HS$. We first show that $\wt{L}$ has a subgroup $2\times \sym(8)$ and later
that the centre of this subgroup will lie outside of $O^2(\wt{L})$. We will use the information we
have about $C_G(t)=H$ and $N_G(E)$ to see the structure of some $2$-local subgroups of $\wt{L}$.
Once we have used extremal transfer to find the index two subgroup of $\wt{L}$ we are then able to
use this $2$-local information to apply a theorem due to Aschbacher \cite{AschbacherHS} to
recognize $\HS$. The Aschbacher result requires us to find $2$-local subgroups of shape $(4 *2_+^{1+4}).\sym(5)$ and $(4 \times 4 \times 4). \GL_3(2)$.

Recall using  Notation \ref{HN-Alt9notation} that $u \in F \leq Q_2 \leq C_G(a_2)$ and that $a_1$
normalizes $F$.

\begin{lemma}\label{HN-fours groups centralize A8}
$C_G(F)\cong 2\times 2 \times \alt(8)$ with $C_G(F)>C_G(u) \cap C_G(a_1)\cong \alt(8)$ and
$C_{\wt{L}}(\wt{F})\cong 2 \times \sym(8)$. Moreover if $F_0$ is any fours subgroup of $C_G(a_2)$
such that  $F_0^\# \subseteq 2\mathcal{A}$ then $C_G(F_0)\cong C_G(F)$.
\end{lemma}
\begin{proof}
Set $M:=C_G(F)$. First observe that $F \leq O^3(C_G(a_2))\cong \alt(9)$ and the image of $F^\#$ in
$\alt(9)$ consists of involutions of cycle type $2^2$. Notice also that $\alt(9)$ has two classes
of such fours groups with representatives $\<(1,2)(3,4),(1,3)(2,4)\>$ and
$\<(1,2)(3,4),(3,4)(5,6)\>$. These subgroups of $\alt(9)$ have respective centralizers isomorphic
to $2 \times 2 \times \alt(5)$ and $2 \times 2 \times \sym(3)$ and respective normalizers
 $(\alt(4) \times \alt(5)):2$ and $\sym(4) \times \sym(3)$.

Given the image of $F$ in $O^3(C_G(a_2))$, we have that $M\cap C_G(a_2)\cong 3 \times 2 \times 2 \times
\sym(3)$. Let $R \in \syl_3(M \cap C_G(a_2))$ such
that $\<R,a_1\>$ is a Sylow $3$-subgroup of $N_G(F)$ (notice that $a_1$ permutes $F^\#$). Then $a_2 \in R$ and  $\<R,a_1\>$ is abelian and $R^\# \subseteq
3\mathcal{A}$ since no element of order three in $3\mathcal{B}$ commutes with a fours group.
Therefore by the earlier argument for each $r \in R^\#$, $C_G(r) \cap M\cong 3\times 2 \times 2
\times \alt(5)$ or $3\times 2 \times 2 \times \sym(3)$.

Consider $M \cap C_G(a_1)$ which is isomorphic to a subgroup of $\alt(9)\cong O^3(C_G(a_1))$. By Lemma
\ref{HN-centralizer of the a9}, $C_G(O^3(C_G(a_1)))=\<a_1\>$. In particular, $F$ does not commute
with $O^3(C_G(a_1))$ and so $M \cap C_G(a_1)$ is a proper subgroup of $O^3(C_G(a_1))$. By Lemma
\ref{HN-Q_i's}, we  have that $F\leq Q_2$ commutes with $Q_1\leq C_G(a_1)$. Also $F$ commutes with
$R\leq C_G(a_1)$ and so $|M \cap C_G(a_1)|$ is a multiple of $2^53^2$. Moreover $M \cap C_G(a_1)$ contains the subgroup $Q_1\<a_2\>\sim 2^{1+4}_+.3$.

We check the maximal subgroups of $\alt(9)$ (see \cite{atlas}) to see that $M \cap C_G(a_1)$ is
either a subgroup of $\alt(8)$ or the diagonal subgroup of index two in $\sym(5) \times \sym(4)$.
The latter possibility leads to a Sylow $2$-subgroup of order $2^5$ with centre of order four which
is impossible as $2_+^{1+4}\cong Q_2\leq M\cap C_G(a_1)$. So $M\cap C_G(a_1)$ is isomorphic to a
subgroup of $\alt(8)$. Suppose it is isomorphic to a proper subgroup of $\alt(8)$. We again check
the maximal subgroups of $\alt(8)$ (\cite{atlas}) to see that $M\cap C_G(a_1)$ is isomorphic to a
subgroup of $N_{\alt(8)}(\<(1,2)(3,4),(1,3)(2,4),(5,6)(7,8),(5,7)(6,8)\>)\sim 2^4:(\sym(3) \times
\sym(3))$. This subgroup can be seen easily in $\GL_4(2)$ as the subgroup of matrices of shape
\[\left(\begin{array}{cccc}
\ast & \ast & 0 & 0 \\
\ast & \ast & 0 & 0 \\
\ast & \ast & \ast & \ast \\
\ast & \ast & \ast & \ast
\end{array}\right).\] We calculate in this group that an extraspecial subgroup of order $2^5$ is not normalized by a element of order three. Therefore  $M \cap C_G(a_1)$ is not isomorphic to a subgroup of this matrix group. Thus $M\cap C_G(a_1)\cong \alt(8)$. In particular $M$ has a subgroup isomorphic to
$2\times 2 \times \alt(8)$.

Now we have that for every $r \in R^\#$, $C_M(r)\cong 3\times 2 \times 2 \times \sym(3)$ or
$3\times 2 \times 2 \times \alt(5)$. Now $R\leq C_M(a_1)\cong \alt(8)$ and so $R \in
\syl_3(C_M(a_1))$.  Moreover, $\alt(8)$ has two conjugacy classes of elements of order three. So we
may set $R=\{1,a_2,a_2^2,a_3,a_3^2,b_1,b_1^2,b_2,b_2^2\}$ where $a_2$ is conjugate to $a_3$ in
$C_M(a_1)$ and $b_1$ is conjugate to $b_2$ in $C_M(a_1)$ such that $C_{C_M(a_1)}(b_i)\cong 3 \times
\alt(5)$ ($i \in \{1,2\}$) and $C_{C_M(a_1)}(a_j)\cong 3 \times \sym(3)$ ($j \in \{2,3\}$). Now we
already have that $C_M(a_3) \cong C_M(a_2)\cong 3 \times 2 \times 2 \times \sym(3)$ and we have two
possibilities for the structure of the other $3$-centralizer. Therefore we must have that
$C_{M}(b_i)\cong 3 \times 2 \times 2 \times \alt(5)$. Now by coprime action $C_{M/F}(Fb_i)\cong
C_M(b_i)/F$ and $C_{M/F}(Fa_i)\cong C_M(a_i)/F$. Hence we may apply Corollary
\ref{Cor-ParkerRowleyA8} to $M/F$ to say that $M/F\cong \alt(8)$. Therefore $M\cong 2 \times 2
\times \alt(8)$.

Consider $N_L(F)$. We have seen that $N_G(F)/M\cong \sym(3)$ and so $[N_L(F):M]=2$. It follows that
$N_L(F)/F\cong 2 \times \alt(8)$ or $\sym(8)$. For $b_1\in R$, $C_{M}(b_1)\cong 3 \times 2 \times 2
\times \alt(5)$ and so  $C_{N_L(F)}(b_1)\sim 3 \times (2 \times 2 \times \alt(5)):2$ and
$C_{N_L(F)}(b_1)/F\sim 3 \times \sym(5)$ which is not a subgroup of $\alt(8) \times 2$. Thus we
must have that $N_L(F)/F\cong \sym(8)$ and so $C_{\wt{L}}(\wt{F})\cong 2 \times \sym(8)$.

Now let $F_0\leq C_G(a_2)$ have image $\<(1,2)(3,4),(1,3)(2,4)\>$ in $\alt(9)\cong O^3(C_G(a_2))$. Then
$C_G(a_2) \cap C_G(F_0)\cong 3 \times 2 \times 2 \times \alt(5)$. Now recall that $R\in \syl_3(M)$
and $M=C_G(F)\cong 2 \times 2 \times \alt(8)$ and so there exists $r \in R^\#$ such that $M \cap
C_G(r)\cong 3 \times 2 \times 2 \times \alt(5)$. Since every element in $R^\#$ is conjugate in $G$,
we have that $F_0$ is conjugate to $F$ in $G$. Thus $C_G(F_0)\cong C_G(F)$.
\end{proof}

Recall from Notation \ref{HN-Alt9notation} that $r_2$ is an involution in $O^3(C_G(a_2))$ which is
conjugate to $u$ and $r_2u$. In light of Lemma \ref{HN-fours groups centralize A8}, the following
result is a calculation in a group isomorphic to $2 \times 2 \times \alt(8)$.
\begin{lemma}\label{HN-HS-centralizer of r,t,u}
$C_{H \cap L}(r_2)\sim 2 \times 2 \times (2 \times 2 \times \alt(4)) :2$.
\end{lemma}
\begin{proof}
It is clear from Notation \ref{HN-Alt9notation} that $\<r_2,u\>^\# \subseteq 2\mathcal{A}$. Set
$F_0:=\<r_2,u\>$ then by Lemma \ref{HN-fours groups centralize A8}, $C_G(F_0)\cong 2\times 2 \times
\alt(8)$. Notice also from Notation \ref{HN-Alt9notation} that $t \in C_G(F_0) \cap C_G(a_2)\cong 3
\times 2 \times 2 \times \alt(5)$ which has an abelian subgroup containing $t$ isomorphic to $3
\times 2 \times 2 \times 2 \times 2$. Consider $\<F_0,t\> \cap C_G(F_0)'$ (of course
$C_G(F_0)'\cong \alt(8)$) which has order two. If $\<F_0,t\> \cap C_G(F_0)'$ is $2$-central in
$C_G(F_0)'$ then $a_2 \in C_G(F_0)' \cap C_G(t)$ is isomorphic to the subgroup of $\alt(8)$ of
shape $2_+^{1+4}.\sym(3)$. However this implies that $C_G(\<F_0,t\>) \cap C_G(a_2)\cong 2 \times 2
\times 2 \times 3$ which is not the case. Thus $\<F_0,t\> \cap C_G(F_0)'$ is not $2$-central in
$C_G(F_0)'$ and so $C_G(F_0)' \cap C_G(t)$ is isomorphic to a subgroup of $\alt(8)$ of shape $(2
\times 2 \times \alt(4)) :2$. Thus $C_{H \cap L}(r_2)\sim 2 \times 2 \times (2 \times 2 \times
\alt(4)) :2$.
\end{proof}

\begin{lemma}\label{HN-HS-Sylow 2}
$H\cap L$ contains a Sylow $2$-subgroup of $L$ which has order $2^{11}$ and centre $\<t,u\>$.
\end{lemma}
\begin{proof}
Let $S_u$ be a Sylow $2$-subgroup of $C_L(t)$. We have that $u \in Q_2 \leq Q_{12}$ and since $u
\in 2\mathcal{A}$,  we may apply Lemma \ref{HN-orbits of elements in Q} to see that
$|C_{{H}}({u})|=2^{11}.3.5$. Therefore $|S_u|=2^{11}$. Now, $u \in Q_2$ and $[Q_1,Q_2]=1$ (by Lemma
\ref{HN-Q_i's}) so  we have that $Q_1\leq C_{O_2(H)}(u) \leq S_u$. Moreover, $\mathcal{Z}(S_u)\leq
C_{S_u}(Q_1) \leq Q_2$. Therefore $\mathcal{Z}(S_u)\leq \mathcal{Z}(C_{Q_2}(u))=\<t,u\>$ since
$Q_2$ is  extraspecial of order $2^5$. Hence $\mathcal{Z}(S_u)=\<t,u\>$. Since $\<t,u\>\leq Q_{12}$
and $Q_{12}$ is extraspecial, $u$ is conjugate to $ut$ in $Q_{12}$. Therefore $N_G(\<t,u\>)\leq
C_G(t)$. So let $S_u \leq T_u\in \syl_2( L)$ then $N_{T_u}(S_u) \leq N_{L}(\<t,u\>) \leq H
\cap L$. Thus $S_u$ is a Sylow $2$-subgroup of $L$.
\end{proof}

\begin{lemma}
$(H \cap L)/({Q_{12}} \cap L) \cong \sym(5)$.
\end{lemma}
\begin{proof}
Using Lemma \ref{HN-orbits of elements in Q} we have that $C_H(u)/C_{Q_{12}}(u)\cong \alt(5) \times
2$ or $\sym(5)$. We suppose for a contradiction that $(H \cap L) /(Q_{12} \cap
L)=C_H(u)/C_{Q_{12}}(u)\cong C_H(u){Q_{12}}/{Q_{12}}\cong 2 \times \alt(5)$. Now set
$V:=C_{Q_{12}}(u)$ then $|V|=2^8$ and $\bar{V}$ is normalized by $C_H(u)/V\cong 2 \times \alt(5)$.

Recall from Notation \ref{HN-Alt9notation} that $r_2$ is an involution in $O^3(C_G(a_2))$ and from
Lemma \ref{HN-alt9 observations} that $r_2\in C_H(a_2)\bs  Q_2$. Since $[r_2,a_2]=1$,
$[Vr_2,Va_2]=1$ and therefore $Vr_2\in \mathcal{Z}(C_H(u)/V)$. In particular, $C_{\bar{V}}(r_2)$ is
preserved by $O^2(C_H(u)/V)\cong \alt(5)$. Since $Va_2$ acts non-trivially on $\bar{V}$,
$O^2(C_H(u)/V)$ acts non-trivially. This is to say that there exists a non-central
$O^2(C_H(u)/V)$-chief factor of $\bar{V}$. Moreover, this chief factor has order at least $2^4$.

By Lemma \ref{lem-cenhalfspace} $(ii)$, $|C_{\bar{V}}(r_2)|\geq 2^4$. Now Lemma
\ref{HN-HS-centralizer of r,t,u}  gives us that $C_{H \cap L}(r_2)\sim 2 \times 2 \times (2 \times
2 \times \alt(4)) :2$. Clearly $C_V(r_2)$ is a normal $2$-subgroup of $C_{H \cap L}(r_2)$. However
$O_2(C_{H \cap L}(r_2))$ has order $2^6$ and contains $r_2$. Therefore $|C_V(r_2)|\leq 2^5$ and so
by Lemma \ref{Prelims-p centralizers on class 2 groups}, $|C_{\bar{V}}(r_2)|\leq 2^5$. Thus
$|C_{\bar{V}}(r_2)|=2^4$ or $2^5$. Suppose first that $|C_{\bar{V}}(r_2)|=2^4$ then $\bar{u} \in
C_{\bar{V}}(r_2)$ is normalized by $O^2(C_H(u)/V)$ and so $C_{\bar{V}}(r_2)$ is necessarily a sum of
trivial $O^2(C_H(u)/V)$-modules. Moreover $\bar{V}/C_{\bar{V}}(r_2)$ has dimension three and is
therefore also a sum of trivial $O^2(C_H(u)/V)$-modules. This is a contradiction.

So suppose instead that $|C_{\bar{V}}(r_2)|=2^5$. Then $|[\bar{V},r_2]|=2^2$ by Lemma
\ref{lem-cenhalfspace} $(i)$. Furthermore $[\bar{V},r_2]$  is preserved by $O^2(C_H(u)/V)$. Thus
$[\bar{V},r_2]$ is a sum of two trivial $O^2(C_H(u)/V)$-modules. Since $[\bar{V},r_2]\leq
C_{\bar{V}}(r_2)$ (Lemma \ref{lem-cenhalfspace} $(ii)$), it follows that $C_{\bar{V}}(r_2)$ is also
a sum of trivial $O^2(C_H(u)/V)$-modules as is $\bar{V}/C_{\bar{V}}(r_2)$. Again this gives us a
contradiction. Hence we may conclude that $C_H(u)/C_{Q_{12}}(u)\cong C_H(u){Q_{12}}/{Q_{12}}\cong
\sym(5)$.
\end{proof}

\begin{lemma}\label{HN-HS-element of order five}
Let $(L \cap {Q_{12}})e \in (L \cap H)/(L \cap {Q_{12}})$ have order five then
$C_{Q_{12}}(e) \cong [{Q_{12}},e] \cong 2_{-}^{1+4}$.
\end{lemma}
\begin{proof}
By Lemma \ref{HN-Q_i's}, $C_G({Q_{12}})\leq {Q_{12}}$ and so  $e$ acts non-trivially on $Q_{12}$ and since $(L \cap {Q_{12}})e $ has order five, $e$ describes an automorphism of $Q_{12}$ of order five. We have that $e$ centralizes $u$ and so $C_{Q_{12}}(e)>\<t\>$. Hence by Lemma \ref{prelims-extraspecial and a coprime aut}, $C_{Q_{12}}(e)$ and
$[{{Q_{12}}},e]$ are both extraspecial with intersection equal to $\<t\>$ and product equal to $Q_{12}$. Since
$e$ acts fixed-point-freely on $[\bar{{Q_{12}}},e]$, we have that $|[\bar{{Q_{12}}},e]|=2^4$. Thus $C_{Q_{12}}(e)$ and $[{{Q_{12}}},e]$ are both extraspecial of order $2^5$. Since
$[{{Q_{12}}},e]$ has an automorphism of order five, $[{{Q_{12}}},e]\cong 2_-^{1+4}$ follows from
Lemma \ref{extraspecial outer automorphisms}. Finally, since ${Q_{12}}$ is extraspecial of plus
type, we have that $C_{Q_{12}}(e)\cong 2_-^{1+4}$.
\end{proof}

\begin{lemma}\label{HN-HS-1-finding 4*4*4}
There exists an element of order four $d\in C_{Q_2}(u)$ such that $d^2=t$ and $4 \times 2\cong
{\<d,u\>}\vartriangleleft {H \cap L}$.
\end{lemma}
\begin{proof}
Set $V:=C_{Q_{12}}(u)/\<u,t\>$ then $|V|=2^6$. Consider the action of $(L \cap H)/(L \cap
{Q_{12}})\cong \sym(5)$ on $V$. We have that $C_{Q_{12}}(a_2)=Q_2$ and if we set
$B:=C_{Q_2}(u)=C_{C_{Q_{12}}(u)}(a_2)$ then $|B|=2^4$. Now using coprime action we have that
$|C_V(a_2)|=|C_{C_{Q_{12}}(u)}(a_2)/\<t,u\>|=2^2$. It follows that $V$ is a sum of two trivial
$\alt(5)$-modules and a 4-dimensional natural $\alt(5)$-module. Let $V_0$ be an irreducible $((L
\cap H)/(L \cap {Q_{12}}))$-submodule of $V$ and let $W_0$ be the preimage of $V_0$ in
$C_{Q_{12}}(u)$. Suppose $V_0$ is a $4$-dimensional module. Then an element of order five and an
element of order three act fixed-point-freely on $V_0$. We have that $|W_0|=2^6$ and can be written
as a direct product in two different ways. Firstly,  $W_0=\<u\>\times [{Q_{12}},e]$ where
$({Q_{12}} \cap L)e \in (H \cap L)/({Q_{12}} \cap L)$ has order five. Secondly, $W_0=\<u\>\times
[{Q_{12}},a_2]$. However $[{Q_{12}},e]\cong 2_-^{1+4}$ and $[{Q_{12}},a_2]=Q_1\cong 2_+^{1+4}$
which is a contradiction. Thus $V_0$ is isomorphic to either a trivial $\sym(5)$-module or a sum of
two trivial $\alt(5)$-modules. In the latter case, $|W_0|=2^4$ and commutes with an element of
order five and an element of order three. Thus $W_0 \leq C_{Q_{12}}(a_2)=Q_2$ and so
$W_0=C_{Q_2}(u)$. However $\<F,t\> \leq C_{Q_2}(u)$ and $\<F,t\>$ is elementary abelian of order
eight. This implies that $\<F,t\>\leq C_{Q_{12}}(e)\cong 2_-^{1+4}$ which is a contradiction. Thus
$|V_0|=2$ and so $|W_0|=8$  and since $\<u,t\>$ is central in $W_0$, $W_0$ must be abelian.
Moreover if $({Q_{12}} \cap L)e$ is an element of order five in $(H \cap L)/({Q_{12}} \cap L)$ then
by Lemma \ref{HN-HS-element of order five}, $W_0 \leq C_{Q_{12}}(e)\cong 2_-^{1+4}$. Thus $W_0$ is
not elementary abelian and so $H \cap L \vartriangleright W_0 \cong 4 \times 2$. Thus, there is an
element of order four $d\in C_{Q_{12}}(u)$ such that $d^2=t$ and $4 \times 2\cong
{\<d,u\>}\vartriangleleft {H \cap L}$.
\end{proof}

\begin{lemma}\label{HN-HS-complement to A}
There exists  a  complement $C\cong \GL_3(2)$ to $C_L(E)$ in $N_L(E)$ such that $EC \leq C_G(F)$ and
there exists $S_u \in \syl_2(H \cap L)$ such that $E\vartriangleleft S_u$.
\end{lemma}
\begin{proof}
Recall that $u \in F \leq Q_2$ and by Lemma \ref{HN-fours groups centralize A8}, $2 \times 2 \times
\alt(8)\cong C_G(F)> C_G(u) \cap C_G(a_1)\cong \alt(8)$. Notice that $E \leq Q_1 \leq C_G(F)$ since $[Q_1,Q_2]=1$. Notice also that $t\in C_G(u)\cap C_G(a_1)$.
From notation \ref{HN-Alt9notation}, the image of $t$ in $\alt(9)\cong O^3(C_G(a_1))$ is
$(1,2)(3,4)(5,6)(7,8)$ and so clearly $t$ lies in exactly one subgroup of $O^3(C_G(a_1))$
isomorphic to $\alt(8)$.  By Lemma \ref{HN-Describing E}, $O^3(C_G(a_1))$ contains a complement,
$C$ say, to $C_G(E)$ in $N_G(E)$. Moreover the image of $EC$ in $O^3(C_G(a_1))$ lies in a subgroup
isomorphic to $\alt(8)$ containing $t$. Therefore $EC\leq C_G(u)\cap C_G(a_1)\leq C_G(F)$.

Recall from Lemma \ref{HN-info on centralizer of E} that $O^3(C_G(E))=\<E,Q_2,A_1,A_2\>$ and
consider $W=O^3(C_G(E)) \cap L$ which is normalized by $C$. Let $n \in C$ be an element of order
seven. Notice that $C_{Q_2}(u)\leq W$ which has order $2^4$ and does not split over $\<t\>=Q_2 \cap
E$. Therefore $W$ does not split over $E$. In particular, $n$ does not centralize $W/E$
else $W=C_W(n)\times E$. Now consider the action of $C$ on $W/E\<u\>$. 
We have that $C_{Q_2}(u)\leq W$ so $|W|\geq 2^6$ therefore
$|W/E\<u\>|\geq 2^2$. Suppose first that $|W/E\<u\>|=2^2$. Then $n$ necessarily acts trivially on $W/E$ which is a contradiction. So suppose $|W/E\<u\>|=2^3$. Then  $n$ must act fixed-point-freely on $W/E\<u\>$. 
Recall that $F \leq C_{Q_2}(u) \leq W$ and so $FE/\<Eu\>$ has order two. Since $[C,F]=1$,
$[FE/\<E,u\>,n]=1$ and $n$ does not act fixed-point-freely on  $W/E\<u\>$. Thus $|W/E\<u\>|\geq
2^4$ and so $|W|\geq 2^8$. Now $C$ has a Sylow $2$-subgroup of order eight which centralizes $t$.
Thus $N_G(E) \cap L$ has a Sylow $2$-subgroup of order at least $2^{11}$ which centralizes $t$.
Hence if we call this $2$-group $S_u$ then $S_u \in \syl_2(H \cap L)$ by Lemma \ref{HN-HS-Sylow 2}
and $E\vartriangleleft S_u$.
\end{proof}

\begin{lemma}\label{HN-HS-index two subgroup}
$L$ has an index two subgroup $L_1$ such that $F \nleq L_1$ and $L_1 \cap N_G(E)\sim 2.(
4^3:\GL_3(2))$. Moreover, there is an element of order four, $d\in Q_{2}$ as in Lemma
\ref{HN-HS-1-finding 4*4*4}  and $d \in L_1$.
\end{lemma}
\begin{proof}
By Lemma \ref{HN-alt9 observations} $(v)$, a Sylow $3$-subgroup of $C_G(E)/E$ is self-normalizing
and therefore $3 \nmid |C_G(u) \cap C_G(E)|$. Hence $C_L(E)$ is a $2$-group. Since
$E\vartriangleleft S_u \in \syl_2(L)$, $C_L(E) \leq S_u$. By Lemma \ref{HN-HS-1-finding 4*4*4},
there exists an element of order four $d\in C_{Q_2}(u)$ such that $d^2=t \in E$ and
${\<d,u\>}\vartriangleleft H \cap L$ which implies that $4\cong \<\wt{d}\>\vartriangleleft \wt{H
\cap L}$. Therefore $\<\wt{d}\>\vartriangleleft \wt{N_L(E) \cap H}$. Since $N_L(E)/C_L(E)\cong
\GL_3(2)$, $[N_L(E):N_L(E) \cap H]=7$. So consider $\<\wt{d}^{N_L(E)}\>$. Since $\wt{d}^2=\wt{t}\in
\wt{E}$ and $N_L(E)$ is transitive on $E^\#$, we clearly have at least seven conjugates of
$\<\wt{d}\>$ in $\wt{N_L(E)}$. Moreover since $\<\wt{d}\>\vartriangleleft \wt{C_L(E)}\leq\wt{S_u}$,
the seven conjugates of $\<\wt{d}\>$ in $\wt{C_L(E)}$ pairwise commute. Thus
$\wt{N_L(E)}\vartriangleright \<\wt{d}^{{N_L(E)}}\>=:\wt{A}\cong 4 \times 4 \times 4$ (where $u
\in A\trianglelefteq  N_L(E)$). Now by Lemma \ref{HN-HS-complement to A}, there exists a complement,
$C$, to  $C_L(E)$ in $N_L(E)$. Moreover, $\wt{C}$ acts non-trivially on $\wt{A}$ and so
$\wt{A}\wt{C}\sim 4^3:\GL_3(2)$. Since $C_L(E)$ is a $2$-group, and $L$ has Sylow $2$-subgroups of
order $2^{11}$, it follows that $|N_L(E)| \leq 2^{11}37$. Thus $\wt{A}\wt{C}$ has index at most two in
$\wt{N_L(E)}$.

Recall that  $u \in F \leq Q_2$ 
and by Lemma \ref{HN-HS-complement to A}, $F \leq C_L(E)$. Therefore $\wt{F}$ normalizes $\wt{AC}$.
Furthermore, by Lemma \ref{HN-fours groups centralize A8}, $C_{\wt{L}}(\wt{F})\cong 2 \times
\sym(8)$. In particular, $\wt{F} \nleq \wt{A}$ and $[\wt{A}, \wt{F}] \neq 1$. By Lemma
\ref{HN-HS-complement to A}, $[C,F]=1$. Thus $\wt{ACF}\sim 4^3:(2 \times \GL_3(2))$. Thus $\wt{ACF}
= \wt{N_L(E)}$  so we may apply Lemma \ref{Prelims-4*4*4 transfer} to $\wt{L}$ to  say that
$O^2(\wt{L})\neq \wt{L}$.

So we define $u \in L_1\vartriangleleft L$ such that $\wt{L_1}=O^2(\wt{L})$ then $\wt{L_1} \cap
\wt{N_L(E)}\sim 4^3:\GL_3(2)$ so clearly $\wt{d} \in \wt{A} \leq L_1$. It is also clear that  $[L:L_1]=2$.
\end{proof}

We continue the notation in the following lemma such that $u \in L_1\vartriangleleft L$ with $[L:L_1]=2$.

\begin{lemma}
$L\cong 2^.\HS :2$.
\end{lemma}
\begin{proof}
We must prove that $\wt{L_1}$ satisfies the hypotheses of  Theorem \ref{Aschbacher-HS}   to
recognize the sporadic simple group $\HS$. Now we have that $\wt{t}$ is an involution in
$\wt{L_1}$. Consider $C_{\wt{L_1}}(\wt{t})$. If $x \in L_1$ and $[\wt{t},\wt{x}]=1$ then $x$
centralizes $\<u\>t=\{t,ut\}$. It follows from Notation \ref{HN-Alt9notation} that $ut\in 2
\mathcal{A}$. Therefore $x$ centralizes $t$ and so $C_{\wt{L_1}}(\wt{t})=\wt{(H \cap L_1)}$. Notice
that $F\leq Q_{12} \cap L$ and $F \nleq L_1$ so $Q_{12} \cap L_1<Q_{12} \cap L$. Now $[Q_1,u]=1$
and $[a_2,u]=1$. Moreover, $[Q_1,a_2]=Q_1$ (otherwise $C_{Q_1}(a_2)>\<t\>$) and so $Q_1 \leq L'
\leq L_1$. By definition of $L_1$, we have that $u \in L_1$. Also Lemma \ref{HN-HS-index two
subgroup} says that an element of order four $d$ satisfying Lemma \ref{HN-HS-1-finding 4*4*4} is in
$C_{Q_2}(u) \cap L_1$ such that $\<d,u\>\cong 4 \times 2$. Thus $(2_+^{1+4} \ast 4) \times 2\sim
Q_1\<d,u\>=Q_{12} \cap L_1$.  Now $H \cap L_1/Q_{12} \cap L_1\cong \sym(5)$ follows from an
isomorphism theorem since
\[\frac{H \cap L_1}{Q_{12} \cap L_1}=\frac{H \cap L_1}{(H \cap L_1) \cap (Q_{12} \cap L)}\cong  \frac{(H \cap
L_1)(Q_{12} \cap L)}{Q_{12} \cap L}=\frac{H \cap L}{Q_{12} \cap L}\cong \sym(5).\] Thus
$C_{\wt{L_1}}(\wt{t})$ has $2$-radical, $\wt{Q_{12} \cap L_1}\cong 2_+^{1+4}*4$ with quotient
$\sym(5)$.

Now we have that $E \leq Q_{12} \cap L_1$. Suppose that $x \in L_1$ and $\wt{x}$ normalizes
$\wt{E}$. Then $x$ normalizes $E\<u\>$. Since $N_L(E)$ is transitive on $E^\#$ and we have seen
that $tu \in 2 \mathcal{A}$, we have that $\{ue|e \in E^\#\}\subseteq 2\mathcal{A}$. Therefore
$E\<u\> \cap 2\mathcal{B}=E^\#$. Hence $x$ normalizes $E$. Thus $N_{\wt{L_1}}(\wt{E})=\wt{L_1 \cap
N_G(E)}$. By Lemma \ref{HN-HS-index two subgroup}, $L_1 \cap N_G(E)\sim 2.( 4^3:\GL_3(2))$ and so
$N_{\wt{L_1}}(\wt{E})\cong 4^3:\GL_3(2)$. Thus we have satisfied the hypothesis of Theorem
\ref{Aschbacher-HS}  and therefore $\wt{L_1}\cong \HS$. Since $L=L_1F$ and since $\wt{F}$ acts
non-trivially on $\wt{L_1}$, $\wt{L}\cong \aut(\HS)\sim \HS :2$. Now notice that $L$ does not split
over $\<u\>$ for example because if we consider the image of $u$ in $\alt(9)\cong O^3(C_G(a_2))$ as
in Notation \ref{HN-Alt9notation}, $u\mapsto (1,2)(3,4)$ then we see that an element of order four
with image $(1,3,2,4)(5,6)$ squares to $u$. Thus $L\cong 2^.\aut(\HS)\cong 2^.\HS :2$.
\end{proof}

\begin{lemma}
$G \cong \HN$
\end{lemma}
\begin{proof}
We have that $G$ is a finite group with two involutions $u$ and $t$ and $L=C_G(u) \sim (2 . HS) :
2$. Also $C_G(t)\sim 2_+^{1+8}.(\alt(5)\wr 2)$ and $O_2(H)=Q_{12}$ and by Lemma \ref{HN-Q_i's},
$C_G(Q_{12})\leq Q_1\<a_2\> \cap Q_2\<a_1\>=\<t\>\leq Q_{12}$. Thus, by Theorem \ref{Segev-HN}, $G
\cong \HN$.
\end{proof}

This completes the proof of Theorem C.

\appendix

\include{AppendixCode}
\addcontentsline{toc}{chapter}{Bibliography}
\bibliographystyle{plain}
\bibliography{mybibliography}

\end{document}